\documentclass[11pt]{article}
\usepackage{color}
\usepackage[active]{srcltx}
\usepackage{graphicx}
\usepackage{tipx}
\usepackage{subcaption}
\usepackage{bbm}
\usepackage{amssymb,amsmath,amsthm}
\usepackage{upgreek}
\usepackage{yfonts}
\usepackage{calrsfs}
\usepackage{rsfso}
\usepackage{stmaryrd}
\newtheorem{remark}{Remark}
\newtheorem{definition}{Definition}

\newtheorem{proposition}{Proposition}
\newtheorem{lemma}{Lemma}
\newtheorem{corollary}{Corollary}
\newtheorem{theorem}{Theorem}

\newcommand{\Z}{\mathbb{Z}}

\def\rest{\hskip 1pt{\hbox to 10.8pt{\hfill
\vrule height 7pt width 0.4pt depth 0pt\hbox{\vrule height 0.4pt
width 7.6pt depth 0pt}\hfill}}}

\newcommand{\eps}{\varepsilon}

\def\B{{\mathbb B}}
\def\D{{\mathbb D}}
\def\C{{\mathbb C}}

\def\R{{\mathbb R}}
\def \rd {{\rm d}}
\def \mM {{\mathcal M}}
\def \mC {{\mathcal C}}
\def \pont{{\rm {\bf P}}^{\rm ont}} 
\def \tpont{\check{\rm {\bf P}}^{\rm ont}} 
\def \hscal{{\rm h}_{\rm scal}}
\def \fL{ {\mathfrak L}}
\def \fLk{\mathfrak L^k}
\def \fLkperp{ {\mathfrak L}^{k, \perp}}
\def \mL{{\mathcal L}}
\def\N{{\mathbb N}}
\def \Ext {\mathfrak E{\rm xt}}
\def \nablaq{\nabla_{_4}}
\def \nablatrois {\nabla_{_3}}
\def \ov {\overrightarrow}
\def \mN {{\mathcal N}}
\def\nuN {{\upnu_{_{\mathcal N}}}}
\def\nutrois {{\upnu_{3}}}
\def\nudeux {{\upnu_{2}}}
\def\S{{\mathbb S}}

\def \deg {{\rm deg \, } }
\def \rU {{\rm U}}
\def \ra {{\rm a}}
\def \rQ {{\rm Q }}
\def \rh {{\rm h}}
\def \Rm {{\rm M}}

\def \rh {{\rm h }}
\def \rL{{\rm L}}
\def \rC {{\rm C}}
\def \rCspag {{\rm {\bf  C}}_{\rm  spg}}
\def \rKspag {{\rm {\bf  K}}_{\rm  spg}}
\def \Spagk{ {{\rm \bf  S}_{\rm \bf pag}^k}}
\def\Gordk {{ {\rm \bf G}_{\rm \bf ord}^k}}
\def\tGordk {{ {\rm \bf  \tilde G}_{\rm \bf ord}^k}}

\def \DefLkperp{\mathcal D_{\rm ef}\mathfrak {L}^{k, \perp}}

  \def \Ocrosstar {{\mathcal O_{\rm cross, \star}^h} }
    \def \tOcrosstar {{\check{\mathcal O}_{\rm cross, \star}^h} }

\def \rq {{\rm q}}

\def \rI {{\rm I}}
\def \fp {{\mathfrak p}}
\def \fq {{\mathfrak q}}
\def \fP {{\mathfrak P}}

\def \rbJ{{\rm \bf J}}
\def \fM {{\rm  M}}
\def \Gopt{ G_{\rm opt} ^{^\upalpha }  }
\def \Thread { \mathbbmss T_{\rm hread}}
\def \Nbre {  {\rm N_{\rm ber}}}
\def \Nel{  {\rm N_{\rm el}}}
\def \Icyl { \mathbbmtt I_{\rm cyl}}
\def \Tr {{\mathbbm T^m_{\mathbbmtt r, \theta}}}
\def \charge {{\rm Ch_g}}
\def  \Vcharged{{V_{\rm chg}}}
\def  \Vbd{{V_{\rm bd}}}
\def \Grp{{\rm G}_{\rm rp} }
\def \uvee { {\overset{\star}{\curlyvee}}}
\def \sousgraphe {\overset{\star}{\Subset}}

\def\rbW { {\rm \bf W} }

\def \rul {{\rm u}_\ell}
\def \rE {{\rm E}}

\def \wedgetrois {\vee_{_3}}
\def \wedgequatre {\vee_{_4}}
\def \frameref {\mathfrak e^\perp_{\rm ref}}
\def \Up {\pont_\varrho[\mathcal C, \mathfrak e_{\rm ref}^\perp ]}
\def \tk {{\rm T}^k}
\def \dist {{\rm dist \, }}

\def \rbA{{\rm {\bf  A}}}
\def \rba{{\rm {\bf  a}}}
\def \rbx{{\rm {\bf  x}}}
\def \rby{{\rm {\bf  y}}}
\def \rbP {{\rm {\bf P}}}
\def \rbQ {{\rm {\bf Q}}}

\def\rH { {\mathbbm H}}
\def \nP{{\mathbb P}_{\rm north}}
\def\sP {{\mathbb P}_{\rm south}}

\def \Using {\mathcal U_{\rm sing}}
\def\ometrois {\omega_{_{\S^3}}}
\def\omedeux {\omega_{_{\S^2}}}

\def \Lbr {{\rm L_{\rm branch} }}

\def \Lbra {{\rm L^\upalpha_{\rm branch} }}

\def \Lbrafour {{\rm L^{\upalpha_4}_{\rm branch}}}

\def \Lbda {{\Uplambda_m^\upalpha}}
\def \Lbdam {{\Uplambda_m^{\upalpha_m}}}
\def \Upxia {\Uplambda^{m,\upalpha}_{\rm norm}}
\def \Upxiam {\Uplambda^{m, \upalpha_m}_{\rm norm}}
\def \rk {{\rm k}}
\def \rW{{ \rm W}}
\def \mg {{\mathfrak g}}
\def \mq {{\mathfrak q}}
\def \mi {{\mathfrak i}}
\def \mj {{\mathfrak j}}
\def \mr {{\mathfrak r}}

\def \mv {{\mathfrak v}}
\def \be {\vec{{\bf e}}}

\def \Singp {\Sigma_{{\star}, +}}
\def \Singm {\Sigma_{{\star}, -}}

\def \Sing {\Sigma_{{\star}}}

\def \Pom {{\rm P}_{\Omega}}
\def \Grp{{\rm G}_{\rm rp}}

\def\QED{\hbox{${\vcenter{\vbox{
   \hrule height 0.4pt\hbox{\vrule width 0.4pt height 6pt
   \kern5pt\vrule width 0.4pt}\hrule height 0.4pt}}}$}\vspace{7pt}}

\begin{document}
\author{Fabrice BETHUEL
\thanks{Sorbonne Universit\'es, UPMC Univ Paris 06, UMR 7598, Laboratoire Jacques-Louis Lions, F-75005, Paris, France } \
 \thanks{ CNRS, UMR 7598, Laboratoire Jacques-Louis Lions, F-75005, Paris, France} 
 }
 
\title{A counterexample  to  the weak density of smooth maps  between manifolds in Sobolev spaces}
\date{}
\maketitle

\begin{abstract} 
The present paper presents a  counterexample to the  sequential weak density of smooth maps between two manifolds $\mM$ and $\mN$ in the Sobolev space $W^{1, p} (\mM, \mN)$,  in the case $p$ is an integer. It has been  shown  (see e.g. \cite{Be2}) that, if $p<\dim \mM $ is not an integer and the $[p]$-th homotopy group $\pi_{[p]}(\mN)$ of $\mN$  is not trivial,  $[p]$ denoting the largest integer less then $p$, then smooth maps are not  sequentially weakly dense in $W^{1, p} (\mM, \mN)$. On the other hand, in the case $p< \dim \mM$ is an integer, examples of specific manifolds  
$\mM$ and $\mN$ have been provided where  
smooth maps are actually  sequentially weakly dense in $W^{1, p} (\mM, \mN)$  with   $\pi_{p}(\mN)\not = 0$, although they are not dense for the \emph{strong convergence}.  This is the case  for instance for $\mM=\B^m$, the standard ball in $\R^m$, and $\mN=\S^p$ the standard sphere of dimension $p$, for which
$\pi_{p}(\mN) =\Z$.  The main result of this paper  shows however that  such a property does not  holds for  arbitrary manifolds $\mN$ and integers $p$.

Our  counterexample deals with      the case  $p=3$,   $\dim \mM\geq 4$
 and $\mN=\S^2$, for which  the homotopy group  $\pi_3(\S^2)=\Z$ is related to the Hopf fibration. We explicitly construct  a map which is not  weakly approximable in $W^{1,3}(\mM, \S^2)$ by maps in $C^\infty(\mM, \S^2)$. One of the central ingredients in our argument is related to issues in  branched transportation  and irrigation theory in the critical exponent case, which are  possibly  of independent interest.  As a byproduct of our method, we also address some questions concerning the $\S^3$-lifting problem  for $\S^2$-valued Sobolev maps. 
  \end{abstract}

\bigskip
\noindent
\section{Introduction}
\label{introd}
 \subsection{Setting and statements}
 Let $\mM$ and  $\mN$ be two manifolds, with $\mN$ isometrically embedded in some euclidean space $\R^\ell$, $\mM$ possibly having  a nonempty boundary.  For given  numbers  $0<s< \infty $ and $1\leq  p <\infty$, we consider the Sobolev space  $W^{s, p} (\mM, \mN)$ of maps between $\mM$ and $\mN$ defined by
 \begin{equation*}
 W^{s, p} (\mM, \mN)=\{ u \in W^{s, p} (\mM, \R^\ell), \ u(x) \in \mN  {\rm \ for \ almost \ every \ }  x \in \mM \}. 
\end{equation*}
The study of these spaces is  motivated in particular by various problems in physics, as liquid crystal theory, Yang-Mills-Higgs or Ginzburg-Landau models, where singularities of topological nature appear, yielding maps which are hence not continuous but belong to suitable Sobolev spaces, built up in view of the corresponding variational frameworks. Starting with the seminal  works of Schoen and Uhlenbeck (\cite {SU}), this field of research has grown quite fast in the last decades. A central issue is  the approximation of maps in  $W^{s, p} (\mM, \mN)$ by smooth maps (or maps with singularities of prescribed type)  between $\mM$ and  $\mN$.   Restricting ourselves to the case $s=1$,    as we will do  in the rest of the paper, it is easily seen that, if  $p>m\equiv \dim \mM$,   then smooth maps are indeed   dense in $W^{1, p} (\mM, \mN)$, with no restriction on the target manifold $\mN$. Indeed, if $p>m$, then maps in $W^{1, p} (\mM, \mN)$ are   H\"older continuous due to  Sobolev embedding:  Standard arguments based on convolution by mollifiers and reprojections allow to conclude.  The result and  the argument extend to the limiting case $p=\dim \mM$.
It turns out that,  when $1\leq p<\dim \mM$,   the  answer to the approximation problem is  strongly related to the $[p]$-th homotopy group $\pi_{[p]}(\mN)$ of the target manifold $\mN$, where $[p]$ denotes the largest integer less  or equal to $p$. Indeed,  if $\pi_{[p]}(\mN)\not =0$, then as we will recall below, one may construct maps in $W^{1,p}(\mM, \mN)$ which cannot be approximated by smooth maps   between $\mM$ and $\mN$ for the \emph{strong topology} (see \cite{Be2}), whereas the condition $\pi_{[p]}(\mN) =0$   yields approximability be smooth maps   when the domain has a simple topology, for instance a ball (see Sections I to IV in   \cite {Be2}). When the domain $\mM$ has a more sophisticated topology, it was
shown in \cite{HL1, HL2} that the topology of $\mM$  might induce some  additional  obstructions to the approximation problem, obstructions which have  actually been missed in \cite{Be2}.  \footnote {The argument in Section V \cite{Be2}, which is aimed to extend the case of a cube to an arbitrary manifolds being erroneous.}

\smallskip

Approximation  by sequences of  smooth maps at the level of the \emph{weak convergence}  is the focus of the present  paper. In order to avoid problems with the topology of $\mM$ we restrict ourselves \emph{first}  to the case $\mM=\B^m$, the standard unit ball of $\R^m$ and, motivated 
by  the 	above discussion, we
 assume  that
\begin{equation}
\label{condition}
1\leq p <m {\rm \ and \ }  \pi_{[p]}(\mN) \not =0.  
\end{equation}
Indeed if one of the conditions in \eqref{condition} is not met, then we already now that $C^\infty (\B^m, \mN)$ is dense for the strong topology in $W^{1,p} (\B^m, \mN)$, hence \emph{also}  sequentially weakly dense. As a matter of fact, we may  even restrict ourselves to the case $p$ is an integer, since the following observation made in \cite{Be2} settles the case $p$ is not:

\begin{theorem}
\label{pasjojo}
Assume that \eqref{condition} holds and that $p$ is \emph{not} an integer. Then $C^\infty (\B^m, \mN)$ is not sequentially weakly  dense  in $W^{1,p} (\B^m, \mN)$. 
\end{theorem}

\noindent
{\it Sketch of  the proof of Theorem \ref{pasjojo} in the case $m-1<p<m$}.   The proof relies on a dimension reduction argument together with the fact that homotopy classes are preserved under weak convergence in $W^{1, p}(\S^{m-1}, \mN)$ for $p>m-1$. First, 
 since we assume in view of \eqref{condition} that $\pi_{m-1}(\mN)\not =0$, there exists some smooth map $\varphi: \S^{m-1} \to \mN$ such that $\varphi$ is not homotopic to a constant map and hence cannot be extended continuously  to the whole  ball $\B^{m}$.   Consider next the  map $\Using$ defined by 
\begin{equation}
\label{singular}
\mathcal U_{\rm sing} (x)=\varphi \left(\frac{x}{\vert x\vert}\right), {\rm \ for \ } x\in \B^m\setminus \{0\}, 
\end{equation}
which is smooth, except at the origin. Introducing  the $p$-Dirichlet energy $\rE_p$ defined by 
\begin{equation*}
\label{penergy}
 \rE_p(v, \mM)=\int_\mM \vert \nabla v \vert^p dx,  {\rm \ for \ } v: \mM \to \R^\ell, 
\end{equation*}
we observe that 
   $$\displaystyle{\rE_p(\Using, \B^m)=\int_0^1 r^{m-1-p}\left(\rE_p (\varphi, \S^{m-1} )\right) dr<+\infty,}$$
    so that $\Using$ belongs to 
   $W^{1,p}(\B^m, \mN)$, since $p<m$.   
Next  assume by contradiction that there exists a sequence $(u_n)_{n \in \N}$ of maps in $C^\infty (\B^m, \mN)$ converging weakly to $\Using$ in $W^{1,p} (\B^m, \R^\ell)$. Then there exists
  \footnote{Similar arguments, based an Fubini's theorem combined with an averaging argument,  will be detailed in Section \ref{proofmain}.}
   some radius $0<r<1$ such that the restriction of $(u_n)_{n \in \N}$  to the sphere $\S_r^{m-1}$ of radius $r$ and centered at $0$ weakly  converges, up to a subsequence, to the restriction of $\Using$ to $\S_r^{m-1}$ in $W^{1,p}(\S_r^{m-1})$. Since $p>m-1$,  the convergence is \emph{uniform} on $\S_r^{m-1}$, by compact  Sobolev  embedding,  and hence $\Using$ and $u_n$ restricted to $\S_r^{m-1}$ are in the  same homotopy class for $n$ large. This  is a contradiction,  since  the map  $u_n$ can be extended inside the sphere $\S_r^{m-1}$  to a continuous map with values into $\mN$,  whereas the restriction of $\Using$ to $\S_r^{m-1}$ does not possess this property.  This  contradiction  establishes the theorem.
   \qed

\medskip
When $p=m-1$ is an integer,  the previous arguments \emph{cannot be extended},  since weak  convergence in  $W^{1,m-1}(\S_r^{m-1})$
 does not  imply uniform convergence.  As a matter of  fact, we have in this case: 
 
 \begin{proposition} 
 \label{segment}
  There exists a sequence of maps $(U_n)_{n \in \N}$ in $C^{\infty} (\B^{m}, \mN)$ converging 
 to $\Using$ weakly  in $W^{1,m-1} (\B^m, \R^\ell)$. Moreover, the sequence $(U_n)_{n \in \N}$ has the following properties:
 \begin{itemize}
 \item The sequence $(U_n)_{n \in \N}$ converges uniformly on every compact set of $\displaystyle{\B^m\setminus \mathcal I_m}$ to $\Using$, where $\mathcal I_m$ denotes the segment 
 $\mathcal I_m=[0, \nP]$ where
$\nP$ denotes the north pole $\nP=(0, \ldots, 0, 1)\in \R^m$.

 \item We have the convergence 
\begin{equation}
\label{sergent}
\vert  \nabla U_n \vert ^{m-1} \rightharpoonup \vert  \nabla \Using  \vert ^{m-1}+ \nu \mathcal{H}^{1}\rest
[0,\nP]
{\rm \ 
in \ the \ sense \ of \ measures \  on \ }  \B^{m}, 
  \end{equation}
  where $\nu$ represents the constant
$$\displaystyle{ \nu=\upnu_{_{\mN, m-1}}(\llbracket \varphi \rrbracket )=\inf \{ \rE_{m-1} (w),  w \in C^1(\S^{m-1}, \mN) \ {\rm homotopic  \ to  \ } \varphi \}>0. 
}$$
\end{itemize}
 \end{proposition}

Since this type of results  is central in the whole discussion, we briefly sketch the argument.  The  proof of Proposition  \ref{segment} combines  a dimension reduction argument similar to the one we used for Theorem \ref{pasjojo},   together with the \emph{bubbling phenomenon} occurring in dimension $m-1$,  for which the $\rE_{m-1}$ energy is scale invariant. We discuss this property first. 

\medskip
\noindent
{\it The bubbling phenomenon}. We recall first the  scaling properties of the functional $\rE_{p}$.  Consider   an arbitrary  integer ${\rm q} \in \N$, $p>0$  and an  arbitrary map $u: \B^\rq \to \mN$. The scaling transformations yields the formula, for $r>0$,
\begin{equation}
\label{scalingprop}
 \rE_p(u_r, \B_r^{\rm q})=r^{{\rm q}-p} \rE_p(u, \B^{\rm q}), {\rm \ where \ } u_r(x)=u (\frac{x}{r}) \ {\rm \ for \ } x \in \B_r^{\rm q}\equiv \B^{\rm q}(0,r).
 \end{equation}
  In the critical case where the exponent is equal to the dimension, i.e. when we have   $p=\rm q$, then   the energy is \emph{scale invariant},  namely $ \rE_p(u_r, \B_r^{\rm q})= \rE_p(u, \B^{\rm q})$. Choosing small values for $r$, this invariance  allows for concentration  of $\rE_q$-energy  at \emph{isolated  points} for  weakly converging sequences. 
We next  replace   the domain $\B^\rq$ by the sphere  $\S^\rq$ of same dimension and consider now regular maps  from $\S^\rq$ to $\mN$ assuming that $\pi_\rq(\mN)$ is not trivial.  Given  $\varphi \in C^\infty (\S^\rq, \mN)$ we denote by $\llbracket \varphi \rrbracket $  its homotopy class. Homotopy classes  are not  preserved in $W^{1, \rq}$ under weak convergence as the next result shows. 

\begin{lemma}
\label{bulle}
 Let $\varphi : \S^{\rq } \to\mN$ be a  given smooth map. Then there exists  a sequence of smooth maps $(\varphi_n)_{n \in \N}$ from $\S^{\rq }$
to $\mN$  such that  the following holds:
\begin{itemize}
\item  $\varphi_n$ is homotopic to a constant map for any $n \in \N$
\item $\varphi_n(x)=\varphi (x) $, for any $n \in \N^\star$, 
  for any $x\in \S^{\rq }\setminus \B^{\rq+1} (\nP, (n+1)^{-1})$  where
$\nP$ denotes the north pole $\nP=(0, \ldots, 0, 1)$
\item $\vert  \nabla \varphi_n \vert ^{\rq }\rightharpoonup\vert  \nabla \varphi \vert ^{\rq }+ \nu_\rq \delta_{_P} $
in the sense of measures on  $\S^{\rq } {\rm \ as \ } n \to +\infty, $
 where we have set
 \begin{equation}
 \label{infmini}
  \nu_{\rm q}=\upnu_{_{\mN, \rm q}}(\llbracket \varphi \rrbracket )=\inf \{ E_{\rq} (w),  w \in C^1(\S^{\rq}, \mN) \ {\rm homotopic  \ to  \ } \varphi \}>0.
  \end{equation}
\end{itemize}
 \end{lemma}
 The idea of the proof of Lemma \ref{bulle} is to glue a scaled copes  minimizers or almost minimizers for \eqref{infmini} at the north pole $\nP$. 

\begin{remark}
\label{reverse}
{\rm   There is  a kind of converse to Lemma \ref{bulle}. Indeed, given any sequence $(\psi_n)_{n \in \N}$ of smooth maps from $\S^{\rq }$ to $\mN$ converging weakly to $\varphi$, there exists a subsequence still denoted $(\psi_n)_{n \in \N}$, points $a_1,\ldots, a_s$ on $\S^\rq$, positive  numbers 
$\mu_1, \ldots, \mu_s$  and a positive measure $\omega_\star$ such that 
\begin{equation}
\label{oups}
 \vert  \nabla \varphi_n \vert ^{\rq }\rightharpoonup\vert  \nabla \varphi \vert ^{\rq }+ 
 \underset{i=1}{\overset{s}\sum}\mu_i\delta_{a_i} +\omega_\star  {\rm \ in \ the \ sense \ of \ measures \ on \ } \S^\rq \ 	as \ n \to +\infty, 
 \end{equation}
 with $\sum \mu_i\geq \nu_{\rm q}$.  We consider next the minimal energy of weakly approximating sequences of smooth maps,  namely the number $\uptau_\star(\varphi)$ given by 
 \begin{equation}
 \label{defaut}
\uptau_\star(\varphi)\equiv \inf  \left \{
\begin{aligned} &\underset{n \to +\infty} \liminf \,  \rE_\rq (w_n), {\rm where \ }  (w_n)_{n \in \N} {\rm \ is \ s. t.   \ } 
w_n \in C^\infty(\S^\rq, \mN), \forall n \in \N,  \\
& {\rm \ and \ such  \ that  \ }\,  \llbracket w_n\rrbracket =0 {\rm \ and  \ }  w_n \underset{n \to +\infty} \rightharpoonup \varphi. 
\end{aligned}
 \right \}.
 \end{equation}
  We may write $\uptau_\star (\varphi)=\rE_\rq(\varphi) + \epsilon_\star(\varphi)$.  In view of Banach-Steinhaus theorem, we have $\epsilon_\star(\varphi) \geq 0$: The number $\epsilon_\star(\varphi)$ will be called \emph{the defect energy} for approximating sequences. 
 If the sequence  $(\psi_n)_{n \in \N}$  fulfills the optimality condition
 $$
\underset{n \to +\infty} \liminf  \,  \rE_\rq(\psi_n) =\rE_\rq(\varphi) + \epsilon_\star(\varphi), $$
then we may show that  we have $\omega_\star=0$ and $\sum \mu_i= \nu_{\rm q}$. Hence,  the defect energy is given by 
\begin{equation}
\label{deftopo}
\epsilon_\star (\varphi)=\nu_{\rm q}, 
\end{equation}
a number which depends only on  the  homotopy class of $\varphi$.  
}
\end{remark}

\medskip
\noindent
{\it Sketch of the proof of Proposition \ref{bulle}}.
Proposition \ref{segment} is deduced from Lemma \ref{bulle} for the choice $\rq=p=m-1$,  constructing   the  sequence $(U_n)_{n \in \N}$ as
\begin{equation}
\label{extra}
U_n (x)=\varphi_n \left(\frac{x}{\vert x\vert}\right), {\rm \ for \ } \frac 1 n \leq  \vert x \vert \leq 1,  
\end{equation}
 and extending $U_n$ inside the small ball $\B(\frac 1 n)$ in a smooth way with \emph{vanishing energy control}\footnote{A similar or related construction is given in Proposition \ref{trivialextend}  of Subsection \ref{trivial}.}: This  is possible since the map $\varphi_n$ belongs to the trivial homotopy class, and with an energetical cost  tending to $0$ and $n$ goes to $+ \infty$.   Since the energy of the map $\varphi_n$ concentrates at the North Pole $\nP$ in view of Lemma \ref{bulle}, it follows from the construction \eqref{extra}
that  the $(m-1)$-energy   of the sequence $(U_n)_{n \in \N}$ concentrates on the radial extension of the North Pole, that is, on   the segment $[0,\nP]$. 

\medskip
  After this digression, we come  back to the   general problem  of sequential weak density of smooth maps. In view of the previous discussion,   the main problem to consider is  the case  
\begin{equation}
\label{condition0}
p {\rm \ is \  an  \  integer, } \  1\leq p <m {\rm \ and \ }  \pi_{[p]}(\mN) \not =0.  
\end{equation}
  So far,  several   results establishing  \emph{sequentially weak density} of smooth maps between $\B^m$ and $\mN$ have been obtained\footnote{In several of these results, an additional boundary condition is imposed.}. For instance, when $\mN=\S^{p}$ for which  $\pi_p(\mN)=\Z$, we have:
  
 \begin{theorem}[\cite{BeZ, Be1, Be2}]
 \label{vieux}
   Let $p$ be an integer.  Then given any manifold $\mM$, $C^\infty (\mM, \S^p)$ is sequentially  weakly dense in $W^{1,p} (\mM, \S^p)$.
\end{theorem}

 In a  related  direction, a positive answer  was  given in  \cite{H, PR} for $(p-1)$-connected manifolds $\mN$ and in    \cite{PR}  in the case $p=2$, whatever  manifold $\mN$, similar results involving the $H^2$ energy are given in \cite{HR4}.  
The  main result of this paper presents an \emph{obstruction to sequential weak density of smooth maps} when \eqref{condition0} holds and deals with the special  case  $\mN=\S^2$ and $p=3$, for which  $\pi_3(\S^2)=\Z$. 
More precisely,  the main result of this paper is the following: 

 \begin{theorem}
 \label{bis}
  Given any manifold $\mM$ of dimension larger or equal to $4$, $C^\infty (\mM, \S^2)$ is  \emph{not sequentially weakly dense} in $W^{1,3} (\mM, \S^2)$.
 \end{theorem}

As a matter of  fact, the topology and the nature of the manifold $\mM$ is of little importance  in the proof. We rely indeed on the construction of a counterexample in the special case $\mM=\B^4$, imposing  however an additional condition on the boundary $\partial \B^4$.

\begin{theorem} 
\label{maintheo}
There exists a map $\mathcal U $ in $W^{1,3} (\B^4, \S^2)$ which is not the weak limit in $W^{1,3}(\B^4, \R^3)$   of smooth maps between $\B^4$ and $\S^2$. Moreover the restriction of $\mathcal U$ to the boundary $\partial \B^4=\S^3$ is a constant map.
\end{theorem}

   As far as we are aware of, this is the first  case where \emph{an obstruction to sequential weak density}  of smooth maps  between manifolds has been established  when $p$ is an integer.  Theorem \ref{bis} also answers a question explicitly raised in \cite{HR1, HR2, HR3}.
   Let us emphasize that the map $\mathcal U$ constructed in theorem \ref{maintheo}  \emph{necessarily  must have a infinite number of singularities}, and is hence very different from 
the example $\Using$ provided in \eqref{singular}. Indeed, let us recall that,  for $m-1\leq p<m$,   the set of maps with  a finite number of isolated interior  singularities 
\begin{equation}
\label{RP}
\begin{aligned}
\mathcal R^p(\B^m, \mN)=\{ & u \in W^{1,p}(\overline{\B^m}, \mN), {\rm  \ s. t  \ \ } 
 u {\rm \ is \ Lipschitz \ on \ very \ compact \ subset  \ of  \ }  \\
&\overline{ \B^m}\setminus A {\rm \ for \ a \ finite \ set \ } A \subset \B^m\}
\end{aligned}
\end{equation}
  is not only  dense in $W^{1,p} (\B^m, \mN)$  \emph{for the strong topology}, but, in the case $p=m-1$,   is also  contained in the sequential weak closure of smooth maps with values into $\mN$. The proof of this latest  fact, given in \cite{Be1, Be2, BeZ} and which will be sketched in a moment,  is   inspired   by a method introduced in the seminal work of Brezis, Coron and Lieb \cite {BCL}. It is  and along the same line   the singularity of $\Using$ was removed using  concentration of energy along lines connecting the singularity to the boundary, or possibly to other singularities with opposite topological charges.  In view of   \eqref{sergent}, the energy of the  constructed  approximating maps are controlled in the limit by  a term which is of the order of  the length of the connecting lines, multiplied by a  number depending on the topological charge. This number,  which  corresponds to a defect energy,   is obviously bounded when the number of singularities is finite, yielding hence the mentioned weak approximability of maps in $\mathcal R (\B^4, \mN)\equiv\mathcal R^3(\B^4, \mN)$ by smooth maps. We  may  however not a priori exclude the fact  that, when approximating a map in $W^{1,3}(\B^4, \mN)$ by  maps with a finite number of singularities,  the defect energy grows as the number of singularities grows. Our strategy in the proof of Theorem \ref{maintheo} is precisely   to produce  a map $\mathcal U$ for which this phenomenon occurs. 
  
  \smallskip

   \smallskip
   As  this stage, it is  worthwhile to compare, when the exponent $p$ is equal to $3$,  the results obtained for the respective  cases the target manifolds are $\S^2$ or $\S^3$. In both cases the  homotopy groups are similar, since  $\pi_3(\mN)=\Z$, for $\mN=\S^2$ or $\mN=\S^3$.  However, we obtain, provided $\dim \mM\geq 4$,  
sequentially weak density of smooth maps in the  case $\mN=\S^3$  thanks to  Theorem \ref{vieux}, whereas in the case $\mN=\S^2$, we obtain exactly the opposite result, since there are  obstructions to weak density of smooth maps in view of Theorem
\ref{maintheo}. Hence ultimately, not only the nature of the homotopy group matters, but also more subtle issues related to the way its elements behave according to the Sobolev norms and the $\rE_3$ energy.

\medskip

In the next subsection, we review with more details the constructions mentioned above and emphasize their   connections with optimal transportation theory. 
\subsection{Defect measures and optimal transportation of topological charges}
\label{defectsm}
 As in Remark \ref{reverse}, but now in a higher dimension,   given  $u \in W^{1, 3}(\B^4, \mN)$, we introduce the \emph{ defect energy $\epsilon_\star(u)$} related to its weak approximability by smooth maps  defined by
\begin{equation}
\label{defoto}
 \rE_3(u)+\epsilon_\star (u) \equiv \inf  \left \{ \underset{n \to +\infty} \liminf \,  \rE_3 (w_n), (w_n)_{n \in \N} {\rm \ s. t.   \ }
  w_n \in C^\infty(\B^4, \mN)  {\rm \ and  \ }  w_n \underset{n \to +\infty} \rightharpoonup u 
 \right \},
 \end{equation}
  with the convention that $\epsilon_\star (u)=+\infty$ if $u$  cannot be approximated weakly  by smooth maps.  In this subsection, we  focus on maps $u$ with a finite number of singularities and describe briefly how  one may approximate maps in $\mathcal R(\B^4, \mN)$, a set which is dense in $W^{1,3}(\B^4, \mN)$,   weakly by smooth maps in $W^{1, 3}$-norm and how this leads to upper bounds  for the defect energy $\epsilon_\star$. 
As for identity \eqref{deftopo} in Remark \ref{reverse}, the numbers $\upnu_{_{\mN, 3}}$ enter  directly in these estimates.  We  describe first some relevant properties of these numbers in the special case $\pi_3(\mN)=\Z$, emphasizing thereafter asymptotic properties in the  cases $\mN=\S^3$ or $\mN=\S^2$. 

\medskip
\noindent
{\it Infimum of energy in homotopy classes when $\pi_3(\mN)=\Z$}.   When $\pi_3(\mN)=\Z$, each homotopy class in $C^0(\S^3, \mN)$ is  labelled by an integer
    which  will be termed  \emph{the topological charge} of the homotopy class or of its elements.
   Setting in this case, for $d$ given in $\Z$  
   $$
   \upnu_{_\mN}(d)\equiv \upnu_{_{\mN, 3}} (\llbracket  \varphi \rrbracket),   {\rm  \ with \ } 
   \llbracket \varphi \rrbracket=d,
   $$
  We verify  that 
$\nuN(-d)=\nuN(d)$.
Concentrating  bubbles of topological charge $\pm 1$ at $\vert d \vert $ distinct  points, we are led,  for  $d \in \Z$, to the upper bound 
\begin{equation}
\label{upperside}
\nuN (d) \leq \vert d \vert \nuN (1)  
 {\rm  \ \,  and \  more  \  generally \ }
\nuN (kd) \leq k \nuN (d)  \ {\rm \ for \ } k \in \N.
 \end{equation}
 A natural question is  to determine whether this upper bound on $\nuN (d) $ \emph{is sharp or not}. It turns out that the answer to this question strongly depends on  the  target manifold $\mN$.


\medskip
\noindent
{\it Asymptotic behavior of  $\upnu_{_\mN}(d)$ as $\vert d \vert \to + \infty$ when $\mN=\S^3$ and $\mN=\S^2$}. When $\mN=\S^3$  the topological charge is  called the \emph{degree} and denoted  $\deg(\varphi)$. It can be  proved (see Section \ref{topology}, inequality  \eqref{degreenergy}) that, for any $\varphi: \S^3 \to \S^3$, one has 
$$\int _{\S^3} \vert  \nabla \varphi   \vert ^3 dx \geq  3^{\frac 32}\vert   \S^3\vert \vert \deg (\varphi)\vert, $$
 so that, setting   $\nutrois (d) = \upnu_{\S^3} (d)$ we are led to the identity
 \begin{equation}
 \label{upmud3}
  \nutrois (d) = 3^{\frac 32}  \vert \S^3\vert \vert d \vert=2 \sqrt{27}\pi^2  \vert d \vert.
  \end{equation}
When $\mN=\S^2$,    the topological charge is usually called the \emph{Hopf} invariant  and denoted  in this paper
 $\rH(\varphi)$. As  we will recall in Section \ref{topology} (see \eqref{riodeja}), one verifies  that  for any map 
$u: \S^3 \to \S^2$,  we have the lower bound
\begin{equation}
\label{lowers3}
\int _{\S^3} \vert  \nabla u  \vert ^3 dx \geq {\rm C}_{\nu} \vert d \vert ^{\frac 3 4}, d=\rH(u),  
\end{equation}
so that $\nudeux (d) \geq {\rm C}_{\nu} \vert d \vert ^{\frac 3 4}$, where ${\rm C}_{\nu}>0$ is some universal constant, and  where  we have  set  $  \nudeux (d) = \upnu_{\S^2} (d)$.
In \cite {Ri}, Rivi\`ere made the remarkable observation that the bound on the left hand side is \emph{optimal}, that is, there exist a universal constant ${\rm K}_\nu>0$ such that, for any $d\in \Z$, we have the upper bound
\begin{equation}
\label{tr}
\nudeux (d) \leq {\rm K}_{\nu} \vert d \vert^{\frac 3 4}, \,  {\rm \ with } \  \nudeux (d) = \upnu_{\S^2} (d).
\end{equation}
 It follows  that  the function $\nudeux$ is actually sublinear on $\N$.
 In other words, the minimal energy necessary for creating a map of charge $d$ is no longer proportional to $\vert d \vert $, but grows in fact sublinearily as $\vert d \vert^{\frac 3 4}$.  This   fact has  important consequences on the  way to connect optimally defect  for maps from $\B^4$ to $\S^2$ having a finite number of singularities  and the definition of  the corresponding defect measures\footnote{In \cite{HR1}, the authors extend this discussion to several other targets.}. 
 
\medskip
 \noindent
 {\it Topological charges  of singularities of maps in $\mathcal R (\B^4, \mN)$}.   Consider  an arbitrary map $v \in \mathcal R(\B^4, \mN)$,  so that $v$ is continuous in a neighborhood of the boundary $\partial \B^4$.  Given a singularity $a$ of $v$, the homotopy class of the restriction  of $v$ to any small sphere centered at $a$ does not depend on the chosen radius, provided the later is sufficiently small. We will denote 
$ \llbracket a \rrbracket$ this element in $\pi_3(\mN)$ and in the case $\pi_3(\mN)=\Z$,   the number $d$ labelling the homotopy class $ \llbracket a \rrbracket$ will be referred to as  a the \emph{topological charge of the singularity $a$}.   The homotopy classes of the singularities are related to the homotopy class of the restriction of $v$ to the boundary by the simple relation\footnote{The proof relies on the fact that the homotopy group $\pi_q(N)$  is  commutative, for $q\geq 2$. In the case $q=1$, the group might be non commutative, so that formula \eqref{sommebord} no longer holds, see. e.g. \cite{betjac} for related results.}
\begin{equation}
\label{sommebord}
\underset{ a {\rm \ singularity \ of \  } v}\sum \llbracket a \rrbracket=\llbracket u_{\vert_{\partial \B^4}} \rrbracket.
\end{equation}

 \medskip
 \noindent
 {\it Removing singularities of maps in $\mathcal R (\B^4, \mN)$ when $\pi_3(\mN)=\Z$.}   We assume  here  that   $\pi_3(\mN)=\Z$ and, 
for sake of simplicity, that  the  singularities of $v$  have either  topological charges $+1$ or $-1$\footnote{This is not an important  restriction, since the class of maps having this property is also strongly dense.}. 
We may also assume without loss of generality that the degree of $v$ restricted to $\partial \B^4$ is non-negative, so that there are, in view of \eqref{sommebord}, at least as much positive singularities as negative ones. 
 We  denote by $\{P_i\}_{i\in I}$ the set of  singularities of charge  $+1$ and $\{Q_j\}_{j \in J}$ the set of the singularities of charge $-1$, with $\ell = \sharp (J)\leq \sharp(I)=r$. In order to approximate weakly $v$ by smooth maps, we adapt the idea of the proof of Proposition \ref{segment}.   We consider  a set of  bounded curves $\{ \mathcal L_\fp\}_{\fp\in  \fP}$  connecting singularities between themselves or to the boundary as follows: First, each  singularity is connected only to one point. A positive  singularity $P_i$  is connected to a point $A_i \in \{Q_j\}_{j\in J}  \cup \partial \B^4$, that is, $A_i$   is either a negative  charge or a point on the boundary $\partial \B^4$. A negative singularity $Q_j$  which is not yet connected to a positive singularity,  is connected  to a point on the boundary.   
 This yields a sequence   $(\varphi_n)_{n \in \N}$ of maps in ${\rm Lip }(\B^4, \S^3)$, such that
  \begin{equation}
\label{major}
\vert  \nabla \varphi_n \vert ^{3}\rightharpoonup \vert  \nabla v \vert ^{3}+ \upmu_*  {\rm \ as  \ } n \to +\infty \   {\rm \ where \ }
\upmu_*=\nuN  (1)\mathcal{H}^{1}\rest
\left(
\underset {\fp \in \fP} {\cup}  \mathcal L_\fp\right), 
\end{equation}
  so that
 \begin{equation}
 \label{defaut}
\underset {n \to \infty}  \lim \rE_3 (\varphi_n)=\rE_3 (v)+\vert \upmu_* \vert {\rm \ with \ } \vert \upmu_* \vert =  \nuN (1)\left( \underset{\fp \in \fP} {\sum} \mathcal H^1( \mathcal L_\fp)\right).
 \end{equation}
The measure $\upmu_*$ represents a defect  energy measure for the above convergence. It follows from the definition of the minimal defect  energy  $\epsilon_\star(u)$   given in \eqref{defoto} that
\begin{equation}
\label{upup}
\epsilon_\star(u)\leq \vert \upmu_\star\vert, 
\end{equation}
 so that a good estimate for  $\vert \upmu_\star\vert $ yields an estimate of the minimal  defect energy  $\epsilon_\star(u)$. The formula for $\upmu_\star$  given in \eqref{defaut}  depends  not only on the position  of the singularities but  also on  the way we choose to connect them.  In order to obtain general  weak approximation results,  we choose therefore optimal   connections of the singularities, with the hope that the upper bound \eqref{upup} can be turned into a related lower bound. This program can be completed in the case $\mN=\S^3$.

\medskip
\noindent
{\it Minimal connections for $\mN=\S^3$}. 
Consider as above   $v$ in $\mathcal R (\B^4, \S^3)$, with topological charges $\pm 1$.   In order to have the value of the energy defect as small as possible, we  connect  singularities with straight segments and choose the configuration with the smallest total length.  This leads to  the notion of  length of a   minimal connection between the points $\{P_i\}_{i \in J}$, $\{Q_i\}_{i\in J}$ and the boundary $\partial \B^4$,  a notion  introduced in the present context in \cite{BCL}.  Set ${\rm \bf P}=\{P_i\}_{i\in I}$ and ${\rm \bf Q}=\{Q_j \}_{j\in J}$. We introduce the set  $\mathbb T$ of all mappings $\displaystyle{{\rm Tr} :  \rbA\equiv { \rm \bf  P} \cup { \rm \bf Q} \to \rbA\cup \partial \B^4}$ such that  
${\rm Tr}({\bf P}) \subset {\bf Q} \cap \partial \B^4$,  any  $Q_j \in {\rm \bf Q}$  has at most one pre-image  belonging to ${\rm \bf P}$,
and such that, if  $Q_j\not \in {\rm Tr} ({\rm \bf P})$ then ${\rm Tr} (Q_j) \in \partial \B^4$.  The   length of a   minimal connection
  is given by 
\begin{equation}
\label{minimal}
\rL\left(\{P_i\}, \{Q_i\}, \partial \B^4\right)= \inf  \left \{ {\underset {i\in I}\sum} \vert P_i- {\rm  Tr} (P_i) \vert  +{\underset {Q_j\not \in {\rm Tr} ({\rm \bf P}) }\sum}\vert Q_i- {\rm Tr} (Q_i) \vert, \  {\rm for \ } {\rm Tr} \in  \mathbb T  \right \}.  
\end{equation}
   A \emph{minimal connection}, corresponds to a connection related to a minimizer for \eqref{minimal}.
Going back to \eqref{defaut} we obtain hence, for a minimal connection  
 \begin{equation}
 \label{epsilonstar}
  \vert \upmu_\star \vert=3^{\frac 32}\vert \S^3 \vert \,  \rL (v),  {\rm \ where \  } \rL(v) \equiv \rL(\{P_i\}, \{Q_i\}, \partial \B^4),  {\rm  \ since \ } \nuN (1)=3^{\frac 32}\vert \S^3 \vert. 
  \end{equation}
 The important observation made in \cite{BCL} (see also \cite{ABL} for a different proof) is that, if we assume moreover that $v$ is \emph{constant on the boundary}, that is $v \in R_{\rm ct} (\B^4, \S^3)$, where
 $$
\mathcal R_{\rm ct} (\B^4, \mN)=\{u\in \mathcal R(\B^4, \mN), u { \rm \ is \ constant \ on \ } \partial \B^4\},
$$
 then the length of a minimal connection  can be estimated by  the energy of the map  as 
 \begin{equation}
 \label{crux}
 \rE_3(v) \geq  3^{\frac 32}\vert \S^3 \vert \, \rL(\{P_i\}, \{Q_i\}, \partial \B^4).
 \end{equation}
%
   In view of \eqref{upup},    the defect energy $\epsilon_*(v)$ is hence bounded  above  by the energy, that is 
 \begin{equation}
 \label{theborne}
\epsilon_*(v) \leq    \rE_3(v). 
 \end{equation}
Using the fact that $\mathcal R (\B^4, \S^3)$ is dense in $W^{1,3 }_{\rm ct} (\B^4, \S^3)$, where
$$\displaystyle{ W_{\rm ct}^{1,p } (\B^4, \S^3)=  \{ v \in \mathcal R^p(\B^4, \S^3), {\rm s. t } \  v {\rm \ is \ constant \ on \ } \partial \B^4 \},
 }$$
   we  deduce from \eqref{theborne}  that maps  in $W^{1,3}_{\rm ct} (\B^4, \S^3)$ can be approximated weakly in $W^{1,3}(\R^4, \S^3)$ be sequences of maps in $C^\infty(\B^4, \S^3)$. 
   It can even be shown that 
   $\epsilon_\star(v)=\vert \upmu_\star \vert=3^{\frac 32}\vert \S^3 \vert \,  \rL (v),  $
    so that our previous construct is  optimal (see \cite {BBC, giamosou1, giamosou2}). 
Morever,  it can be proved
  (see e.g \cite{giamosou1}) that  any sequence corresponding to a minimizer  in \eqref{defoto} behaves according to \eqref{major}. 
  
  \begin{remark}{\rm
 We  have  assumed   that     all singularities have only   topological charges of  values $\pm1$: This is indeed not a restriction since  the subset  of $\mathcal R_{\rm ct} (\B^4, \S^2)$ maps with  topological charges $\pm 1$ is also dense.    When $\mN=\S^3$, multiplicities  do not really affect the property of the $\rL$, it suffices to repeat each singularity in the collection according to its multiplicity.  
 }
 \end{remark}

\noindent
{\it Removing singularities of  maps in  $\mathcal R (\B^4, \S^2)$: Branched transportation}.   The approximation  proposed in \eqref{major} is   not optimal when   $  \nuN$ grows sublinearily,   the defect energy $\epsilon_\star$ might indeed be   much smaller than $\vert \upmu_\star \vert$ as constructed above. We illustrate this on the case $\mN=\S^2$.

\smallskip
Given $u$ in $\mathcal R_{\rm ct} (\B^4, \S^2)$ and  assuming as before that all topological charges are equal to  $\pm 1$,  we approximate weakly $u$ by smooth maps from $\B^4$ to $\S^2$ connecting  again the positive charges $(P_i)_{i \in I}$ either  to  the set of  negative charges $(Q_j)_{j \in J}$ or to the boundary, the negative charges which are not yet connected to a positive charge being connected to the boundary as above in the case of a minimal connection. In contrast with  the case $\mN=\S^3$ however,  straight lines joining  points  of opposite charges  or to the boundary  may  not represent  the optimal solution. Indeed,  it may be  energetically more favorable, in view of the subadditivity property \eqref{tr},   that some parts of the connection carry a higher topological charge and we need therefore  to introduce  \emph{branching points}. Such  connections with branching points have  been modelled by Q. Xia  in his pioneering work  \cite{xia} with the notion of \emph{transport path}.  We  adapt  this notion and  term  it in our setting \emph{branched connection}, a notion depending only on  the distribution of the charges. A  detailed presentation  is provided in Appendix A.
  A branched connection  associated to the distribution of   charged points $\displaystyle{{\bf A}=\{P_i\}_{i \in I}\cup \{Q_j\}_{j\in J}}$ in $\B^4$  is   given as a directed graph $G$ in $\B^4$ with corresponding source points\footnote{ In our context, we call equally source points positively charges or negatively charged points; so that we will not distinguish sources and sinks.}. It involves:
\begin{itemize}
\item  {a finite  vertex set $V(G)\subset \overline {\B^4}$}, such that the collection of source points  belongs to $V(G)$, that is 
${ {\bf A} \subset V(G)}$. There may also be other points, called \emph{branching points}. 
\item A set {$E(G)$}  of \emph{oriented  segments}  joining the {vertices}, possibly with \emph{multiplicity} or \emph{density} $d(\be)$: For { ${\be} \in E(G)$}, we denote by {$e^-$} and {$e^+$} the \emph{endpoints} of {$e$}, so that 
$$\be=[e^-, e^+]  {\rm \ with \ } e^-, e^+ \in V(G).$$
\end{itemize}
For   a source point $a\in V(G)$, set 
 $E^\pm (a, G)= \{  e\in E(G), e^\pm=a  \}.$  We impose 
for    $a \in V(G)\setminus   \partial \Omega$  the \emph{Kirchhoff law}
\begin{equation}
\label{balance}
\left\{
\begin{aligned}
\sharp \left(E^-(a, G)\right)&=
 \sharp \left( E^+ (a, G)\right)+ 1 {\rm \ \,  if  \ } a \in {\rm \bf P}\equiv  \{P_i\}_{i\in I} \\
  \sharp \left(E^-(a, G)\right)&= \sharp \left(E^+ (a, G)\right)- 1 {\rm \ \,  if  \ } a \in {\rm \bf Q}\equiv \{Q_j\}_{j\in J} \\
\sharp \left(E^-(a, G)\right)&=\sharp \left( E^+ (a, G)\right)
  {\rm \ \,  if   \ }  a {\rm  \ is \ a \ branching \ point \ i.e. \ } a \not \in {\bf A}  \cup \partial \B^4. 
\end{aligned} 
\right.
\end{equation}
  The  multiplicity  $d(\be)$   represents the  charge carried through the segment $\be$ and  relation \eqref{balance} expresses   conservation of this charge at the vertex points, with  signed source  provided by  the  charges at the points in $\bf A$.  We denote by $\mathcal G ( {\rm \bf P}\, {\rm \bf Q}, \partial \B^4 )$ the set of all graphs having the previous properties and introduce  the quantity
\begin{equation}
\label{Lbranch}
\Lbr(  {\rm \bf P}\,  {\rm \bf Q}, \partial \B^4)= \inf \{\rbW(G), G \in \mathcal G ( {\rm \bf P}\, {\rm \bf Q}, \partial \B^4 )\}, 
\end{equation}
 where the functional $\rbW(G)$ is the weighted length of the graph connection  defined by
 \begin{equation}
 \label{w2}
\rbW(G)= \underset { \be\in E(G)}\sum [ (d(\be)) ]^{ 3\slash 4} \mathcal H^1 (\be),  \ {\rm \ for \ }  G \in  \mathcal G ({\rm \bf P}\, {\rm \bf Q}, \partial \B^4). 
 \end{equation}
 The functional $\Lbr$ plays  a similar role   for $\S^2$-valued maps as does  the length of a minimal connection for $\S^3$-valued maps:   It yields  the defect energy when approximating maps in $\mathcal R(\B^4, \S^2)$ by sequences of smooth maps between $\B^4$ and $\S^2$. Indeed,  let $u \in \mathcal R  (\B^4, \S^2)$ as above, let  $\{P_i\}_{i\in I}$ denote the set of $+1$ singularities and    $\{Q_j\}_{j\in J}$  the set of $-1$ singularities. Let   $G$  be an arbitrary  graph in $\displaystyle{\mathcal G (\{P_i\}_{i \in I},\{Q_j\}_{j\in J}, \partial \B^4 )}$. Using concentration of maps along the segments of $G$ with corresponding multiplicity $d(\be)$ on each segment $\be$,  we   construct a sequence $(\varphi_n)_{n \in \N}$  of maps in $C^\infty (\B^4, \S^2)$ converging weakly to $u$ 
such that
\begin{equation*}
\vert  \nabla \varphi_n \vert ^{3}\rightharpoonup \vert  \nabla u \vert ^{3}+ \upmu_*  {\rm \ as  \ } n \to +\infty,  \   {\rm \ where \ }
\upmu_*=\mathcal{H}^{1}\rest
\left(
\underset {\be\in E(G)} {\cup}\upnu_2(d(\be)) \be \right). 
\end{equation*}
It follows that 
\begin{equation*}
\left\{
\begin{aligned}
&{\underset {n\to + \infty}\lim \rE_3}(\varphi_n)=\rE_3(u)+\vert \upmu_\star\vert {\rm \ where \ } \\
 &\vert \upmu_\star \vert = \underset { \be\in E(G)}\sum \upnu_2( (d(\be)) ]) \mathcal H^1 (\be) 
 \leq {\rm K}_{\upnu}  \underset { \be\in E(G)}\sum [ (d(\be)) ]^{ 3\slash 4} \mathcal H^1 (\be)= {\rm K}_{\upnu}\rbW(G),
\end{aligned}
\right. 
\end{equation*}
Choosing   $G$ as a minimizer for $\rbW(G)$ we obtain  
$$\vert \upmu_\star\vert \leq  {K}_{\upnu}\Lbr({\rm \bf P}, {\rm \bf Q}, \partial \B^4), 
{\rm \  so \ that \ }  
\epsilon_\star (u)\leq  {\rm K}_{\upnu} \Lbr( {\rm \bf P}, {\rm \bf Q}, \partial \B^4).
$$
An important point is that   the \emph{reverse inequality is also valid}:  This  has been proved in \cite{HR2}, Theorem  1.1 (see  also Theorem 6.1 and 7.2 in \cite{HR1}). As  a direct consequence of these results (see Subsection \ref{hardtriviere}), we have:

\begin{proposition}
\label{unpoco}
Let $u \in \mathcal R(\B^4, \S^2)$ be such that  $u(\rbx)=\sP$ for $\vert x \vert \geq 3 \slash 4$. Let  $(u_n)_{n \in \N}$ be  a sequence of maps  in $C^\infty (\B^4, \S^2)$  such that 
$\displaystyle{
 u_n \rightharpoonup  u {\rm \ in }  \ W^{1,3}(\B^4, \S^2)  
}.$
  Then, we have  
  \begin{equation}
  \label{leprecieux}
  \underset {n \to \infty} \liminf \, \rE_3(u_n) \geq {\rm C}_{\rm conv} \Lbr(u), 
  \end{equation}
  where  $\Lbr(u)$ is defined by $\equiv \Lbr( \{P_i\}_{i \in I}, \{Q_j\}_{j\in J}, \partial \B^4)$  for $u \in  \mathcal R_{\rm ct} (\B^4, \S^2)$, and  where   $0< {\rm C}_{\rm conv  }\leq 1$ denotes an absolute constant.
\end{proposition}
 
  The defect energy is hence again related to  a quantity involving only the location of the singularities and the sign of their topological charge.  

  \subsection{How  to produce  counterexamples}
   We have seen in the case of $\S^3$-valued that we may bound the defect energy of a map in $u \in  \mathcal R_{\rm ct}(\B^4, \S^3)$ by the 3-energy of the map itself (see inequality \eqref{theborne}) and that this upper bound, combined with the strong  density of $\mathcal R_{\rm ct}(\B^4, \S^3)$  directly leads  to the weak density of smooth maps. If an estimate similar to \eqref{theborne} would exist for $\S^2$-valued maps,  then the same line of thoughts would yield weak approximability as well. Our next result states precisely that there is \emph{no analog}  of  \eqref{theborne} for $\S^2$-valued maps.
  
   \begin{proposition}
  \label{deformons}
  Given any $k \in \N^*$, there exists a map $\mathfrak v_k \in \mathcal R_{\rm ct} (\B^4, \S^2)$  such that 
  \begin{equation}
\label{averell0}
 \rE_3 (\mathfrak v_k) \leq \rC_1 k^3 
 \end{equation} 
 and
 \begin{equation}
 \label{debranche}
     \Lbr(\mathfrak v_k) \geq  \rC_2 k^3  \log k, 
   \end{equation}
      where $\rC_1>0$ and  $ \rC_2>0$ are universal constants.
  \end{proposition}

 Notice that  inequality \eqref{debranche} shows that 
  $\displaystyle{ \Lbr (\mv_k) \geq C  \log k \, \, \rE_3 (\mathfrak v_k)}$,  so that
$$\frac{ \Lbr (\mv_k)}{\rE_3 (\mathfrak v_k)} \to +\infty {\rm \ as \ } n \to +\infty.
  $$
  The functional $\Lbr$, which,  as seen,  is related   to the defect energy $\epsilon_\star$ by \eqref{leprecieux},  is  therefore \emph{not controlled} by the Dirichlet energy  $\rE_3$, in contrast with  inequality \eqref{theborne} for $\mN=\S^3$.
  This property  is  at the heart of the paper. Indeed, not only  it shows that the argument for 
  $\S^2$-valued maps cannot be transposed, it also provides a way to construct  counterexamples.  Indeed,  the map $\mathcal  U$  in Theorem \ref{maintheo}  is obtained by  gluing together an infinite countable number of copies of  scaled and translated versions of the maps $\mathfrak v_k$, for suitable choices of the integers $k$ and of  the scaling factors.  We choose  and tune these parameters in such a way the total sum of the energies is finite, whereas the sum of the defect energies diverges.  
  
 \smallskip
  The next paragraphs present the main steps of the construction of the sequence $(\mv_k)_{k\in \N}$. 
  
\subsection{  On the construction of $\mathfrak v_k$} 

The construction of the maps $\mathfrak v_k$ faces two,  in principle opposite,  constraints:
\begin{itemize}
\item having the   functional  $\Lbr (\mv_k)$ as large as possible.   Since this functional is related to the configuration of singularities, this  task requires to have a large number of singularities, and branched transportation teaches us that  the best way to increase the functional is to have singularities  well-separated (at least if they have  the same sign).
\item having an energy as small as possible. An intuitive idea suggest that increasing the number of singularities will increases the energy. 
\end{itemize}
As we will see at the end of the construction, the number of singularities of $\mathfrak v_k$ is  of order of $k^4$, consisting of two   well-separated clouds of singularities of the same sign,  whereas the energy is  of order  $k^3$. 

\smallskip
 Related to the energy constraint, the  starting point of the construction   is to step one dimension below and  consider maps from $\S^3$ (or actually  $\R^3$ through compactification at infinity, see details in subsection \ref{compactification}) to $\S^2$ which are nearly  optimal for  the energy inequality  \eqref{tr}.  Such maps have been constructed in \cite{Ri}. These maps from $\R^3$ to $\S^2$,  denoted  $\Spagk$  and  termed in this paper  $k$-\emph{Spaghettons}, carry a topological charge of order $k^4$, with an energy of order $k^3$. For the definition  of  $\Spagk$,  
we modify somewhat  the original construction  given in   \cite{Ri}, and recast it into a more  general framework known as the  Pontryagin construction \cite{ponte}, see also \cite{kosinski} for a detailed presentation.  In order to describe briefly $\Spagk$, let us mention that these maps are constant outside   $2 k^2$  closed thin tubes of section  of order $h=k^{-1}$, of length of order $1$. The thin tubes are gathered in two distinct regular bundles which are\emph{ linked}: This linking provides the non trivial topology.

The next  step is to go to dimension 4. We construct  a  deformation denoted $\Gordk$  of $\Spagk$ on the strip $\Lambda= \R^3\times [0, 50]$, which is such  that:
\begin{equation}
\label{strip}
\left\{
\begin{aligned}
&{\rm   \ the  \ restriction \  of \  } \Gordk {\rm \ to  \ the \ slice \ }  \R^3 \times \{0\} {\rm \ is \  equal \ to \ } \Spagk \\
 &{\rm \ its \  restriction \ to \ the  \  slice \ }  \R^3 \times \{50\} {\rm \ is \  a  \  constant \ function}. 
\end{aligned}
\right.
\end{equation}
Such a deformation is of course not possible in the continuous class, since the maps on the top and on the bottom belong to different homotopy classes. In contrast, it is allowed in the Sobolev class $W^{1,3}$, with  an energy  of the same order than the energy of the map restricted to the bottom, that is  the energy of the Spaghetton $\Spagk$. In  particular, one is able to untie  the thin linked tubes thanks to crossings.
We will term  therefore this map $\Gordk$ the Gordian cut of order $k$.   Each of the \emph{cuts}   creates a singularity of the map $\Gordk$.

The construction  of the map $\mv_k$ is completed deforming  the map $\Gordk$ into a map on $\B^4$ with the desired properties, a step which is more elementary and standard than the previous ones.  

\smallskip
We next  go a little further in our description of the maps $\mv_k$.

\subsubsection{ The Pontryagin construction and the $k$-spaghetton map}
The Pontryagin construction we present next  provides  a beautiful  way to produce maps from  $\R^{m+\ell}$    to $\S^\ell$ with non trivial topology. 
This construction, introduced first in \cite{ponte},  relates   to a framed  smooth $m$-dimensional submanifold in $\R^{m+\ell}$   a map from $\R^{m+\ell}$  to $\S^\ell$. By framed submanifold, we mean  here  that for each point $a$ of the submanifold, we are given   an  orthonormal basis $\mathfrak e^\perp \equiv (\vec \tau_1(a),\vec  \tau_2(a), \ldots, \vec \tau_\ell(a))$ of  the $\ell$-dimensional  cotangent  hyperplane at the point $a$, which varies   smoothly with the point $a$.

\smallskip 
We describe   the  Pontryagin construction in the case $m=1$ and $\ell=2$,   which is  the only  situation of interest for us. The framed manifold we consider  is therefore a framed closed curve $\mathcal C$ in $\R^3$, for which we are given an orthonormal basis of its orthogonal plane $\mathfrak e^\perp(\cdot) \equiv(\vec \tau_1(\cdot), \vec \tau_2(\cdot))$. This frame in turn induces  a natural  \emph{orientation} of the curve, choosing 
$$\vec \tau_3(a)=  \vec \tau_{\rm tan}(a)\equiv \vec \tau_1(a) \times \vec  \tau_2 (a)$$
 as a unit tangent vector to the curve at the point $a$, so that \emph{any framed curve is oriented}.   Our next task is to map a small annular neighborhood of   the curve onto the sphere $\S^2$.   To that aim,  we step again one dimension below and  present  as a  preliminary    ingredient   the construction of a map from  the unit   disk onto the sphere $\S^2$ mapping the boundary of the disk to the  south pole. 

\medskip
\noindent
{\it Mapping the unit  disk to the standard sphere $\S^2$}.  We consider in the plane $\R^2$   the unit disk   
$$\D=\{(x_1, x_2) \in \R^2, x_1^2+x_2^2< 1\},$$
 and define a map $\chi$ from  the  closure of the unit  disk $\overline{\D}$  onto the  standard two-sphere $\S^2$ by setting, for $(x_1, x_2) \in \R^2$,   with $r=\sqrt{x_1^2+ x_2^2}$, 
\begin{equation}
\label{defchi1}
\chi(x_1, x_2)= (x_1f(r), x_2f(r), g(r))\ \ {\rm with \ }   \ r^2 f^2(r) +g^2(r)=1, 
\end{equation}
 where $f$ and $g$ are smooth   given real functions on $[0,1]$ such that 
 \begin{equation}
 \label{defchi1bis}
 \left\{
 \begin{aligned}
 &f(0)=f(1)=0, \  0\leq rf(r)\leq 1{\rm \ for \ any \ } r \in [0,1]\\
& -1\leq g \leq 1 {\rm  \  and  \ } g  {\rm \ decreases \  from  \ } g(0)=1 {\rm \ to \ } g(1)=-1 \\
& \, g'(r)<0 {\rm \ for \ any \ } r \in [0, 1].
 \end{aligned}
 \right.
 \end{equation}
  It follows from this 
  definition what $\chi$  maps one to one  of the open disk $\D$ to the set  $\S^2\setminus \{\sP \}$,  where $\sP=(0,0-1)$.  Thanks to the last condition in \eqref{defchi1bis}, $\chi$ is actually  a \emph{diffeomorphism} of $\D$ onto $\S^2\setminus \{\sP\}$. Moreover the boundary $\partial \D$ is mapped  onto the south pole $\sP$, whereas the origin $0$ is mapped to the North pole $\nP=(0,0,1)$. 

  \smallskip
  We  introduce   a scaled version of the previous construction: Given $\varrho>0$,  we  define  the scaled function $\chi_\varrho$ on $\R^2$  by setting
 $$\displaystyle{\chi_\varrho(x_1, x_2)=\chi( \frac{x_1} {\varrho}, \frac{x_2} {\varrho}),
 {\rm \ for \ } (x_1, x_2) \in \D_\varrho, \ \chi_\varrho(x_1, x_2)=\sP {\rm \ otherwise,  \ }
   }$$
 so that  $\chi_\rho$ is Lipschitz.  We have moreover  the gradient estimate
 \begin{equation}
 \label{blablanabla}
 \Vert \nabla \chi_\varrho \Vert_{L^\infty (\D_\varrho)} \leq C\varrho^{-1}. 
 \end{equation}

\bigskip
\noindent
{\it The annular neighborhood of the curve $\mathcal C$}. Let $\mathcal C$ be a smooth  curve in $\R^3$. For $a \in \mathcal C$,   
 let    $P_a^\perp$ denote the affine  plane orthogonal to the tangent vector $\vec \tau_{\rm tan}(a)$ and containing the point $a$, and denote  $\D^\perp_a (\varrho)$ the disk in $P^\perp _a$   centered at $a$ of radius $\varrho>0$. We consider the tubular neighborhood ${\rm T}_\varrho (\mathcal C)$ of $\mathcal C$ defined by
\begin{equation}
\label{tubular}
 {\rm T}_\varrho (\mathcal C)=\underset {a \in \mathcal C} \cup \D^\perp_a (\varrho). 
 \end{equation}
Since $\mathcal C$ is smooth,  there exists some number $\varrho_0=\varrho_0(\mathcal C)>0$ depending only on $\mathcal C$, such that, if $0<\varrho \leq \varrho_0$,  then all disks  $\D^\perp_a(\varrho)$, $a\in \mathcal C$
are mutually disjoint.  In particular, if we assume moreover that $\mathcal C$ is framed, for any $x\in {\rm T}_\varrho (\mathcal C)$, there exists a unique point $a \in \mathcal C$, and a unique point $(x_1, x_2)\in \D_\varrho$ such that $x$ may be decomposed as
\begin{equation}
\label{formix}
x= a+ x_1\vec \tau_1(a)+x_2  \vec  \tau_2(a).
\end{equation}
It follows from these definitions that the map $\mathcal C \times \D_\varrho \to  {\rm T}_\varrho (\mathcal C)$ given by $(a, x_1, x_2)\mapsto x$ where $x$ is defined in \eqref{formix}  is a diffeomorphism.

\bigskip
\noindent
{\it Definition of the  Pontryagin map related to the framed curve $\mathcal C$}.
 For given  $0<\varrho<\varrho_0$,  we construct   a  smooth map 
$\pont_\varrho[\mathcal C, \mathfrak e^\perp]: {\rm T}_\varrho (\mathcal C)  \to \S^2$ as follows: For given  $x \in {\rm T}_\varrho (\mathcal C)$, we decompose $x$  in  the form \eqref{formix} and  set 
\begin{equation}
\pont_\varrho[\mathcal C, \mathfrak e^\perp ](x)= \chi_\varrho(x_1, x_2).  
\end{equation}
 Since $\pont_\varrho[\mathcal C, \mathfrak e^\perp ]$ is equal to $\sP$ on $\partial  {\rm T}_r (\mathcal C)$ we extend this map to the whole of $\R^3$ setting
 $$
 \pont_\varrho[\mathcal C, \mathfrak e^\perp](x)=\sP, {\rm \ for \ } x \in \Omega_\varrho (\mC)\equiv \R^3 \setminus {\rm T}_\varrho (\mathcal C),  
 $$
so  that    $\pont_\varrho[\mathcal C, \mathfrak e^\perp ]$ is now a Lipschitz  map from $\R^3$ to $\S^2$.  The map  $\pont_\varrho[\mathcal C, \mathfrak e^\perp ]$ is called the Pontryagin map related to the framed curve $\mathcal C$ of order $\varrho$. 
 Since  $\pont_\varrho[\mathcal C, \mathfrak e^\perp ]$ is equal to $\sP$ outside a bounded region and in view of \eqref{blablanabla}, we have hence shown:

\begin{lemma}
\label{tropfacile}  If $0<\varrho \leq \varrho_0 (\mathcal C)$ then 
the  map $\pont_\varrho[\mathcal C, \mathfrak e^\perp]$ belongs to ${\rm Lip}  \cap C^0_0( \R^3, \S^2)$, where
\begin{equation}
 C^0_0( \R^3, \S^2)=\{ u \in C^0( \R^3, \S^2) {\rm \ such \ that \ } \underset {\vert x \vert  \to + \infty} \lim  u(x) {\rm \ exists}\}.
\end{equation}
Moreover, we have
\begin{equation}
\label{blablanabla2}
\vert \nabla \pont_\varrho[\mathcal C, \mathfrak e^\perp ] (x) \vert \leq C\varrho^{-1},  {\rm \ for \ every \ } x \in \R^3, 
\end{equation}
where $C>0$ is some constant depending possibly on  the curve $\mathcal C$ as well as  on the choice of frame $\mathfrak e^\perp$ of the orthonormal plane.
\end{lemma}

\begin{remark}
\label{reguvalpont}
{\rm We claim that, as a consequence  of      \eqref{defchi1} and \eqref{defchi1bis}  of $\chi$,    any point on $\S^2\setminus \{\sP\}$ is a  \emph{regular value} of 
 $\pont_\varrho[\mathcal C, \mathfrak e^\perp]$. Consider indeed an arbitrary point $M \in \S^2\setminus \{\sP \}$,  set $L_M=\pont_\varrho[\mathcal C, \mathfrak e^\perp]^{-1}(M)$  and let $a\in L_M$, so that $\pont_\varrho[\mathcal C, \mathfrak e^\perp](a)=M$. Since $\pont_\varrho[\mathcal C, \mathfrak e^\perp](x)=\sP$  for $x\in \R^3\setminus \overline{\Omega_\varrho (C)}$, we deduce that 
 $a \in  \overset{\circ}{ {\rm T}}_\varrho (\mathcal C)$.
 If we restrict the map $\pont_\varrho[\mathcal C, \mathfrak e^\perp](x)=\sP$  for $x\in \R^3\setminus \overline{\Omega_\varrho (C)}$  to the affine plane $P_a^\perp$,  this restriction is, near $a$, a diffeomorphism, since $\chi$ is a diffeomeorphism. Consequently, the tangent map of the restriction  $\pont_\varrho[\mathcal C, \mathfrak e^\perp]_{\vert  P_a^\perp}(a)$ is onto and the same holds for $D (\pont_\varrho[\mathcal C, \mathfrak e^\perp])(a)$. Since the property holds for any $a \in L_M$, any point in $L_M$ is a regular point, which  establishes the claim. 
}
\end{remark}

\medskip
\noindent
{\it The case of planar curves. } All curves $\mC$ that enter through the Pontryagin construction  in our later  definition of the Spaghetton map $\Spagk$  are either  planar or an  union of planar curves. Moreover,  they  lie in planes either parallel to the plane $P_{1, 2}$ or to the plane $P_{2, 3}$ where
\begin{equation}
\label{ps}
\left\{
\begin{aligned}
P_{1, 2}&=(\R\be_3)^\perp=P_{1, 2}(0), { \rm \ where \ } P_{1,2}(s)\equiv \{(x_1, x_2, x_3)\in \R^3 {\rm \ s. t  \  } x_3=s\}, \forall s\in \R,\\
P_{2, 3}&=(\R\be_1)^\perp=P_{2,3 }(0), { \rm \ where \ } P_{2,3}(s)\equiv \{(x_1, x_2, x_3)\in \R^3{ \rm \ s.t.  \   } x_1=s\}, \forall s\in \R, 
\end{aligned}
\right.
\end{equation}
 and where we have  set $\be_1=(1,0,0)$, $\be_2=(0,1,0)$ and $\be_3=(0,0,1)$.  For such planar curves,  we  define  a  \emph{canonical reference framing} as follows.  We first choose the orientation of the curves:   Curves in $P_{1,2}$ and $P_{2,3}$    are orientated   counter-clockwise  according to  the orthonormal  bases $(\be_1, \be_2)$ and $(\be_2, \be_3)$  of  $P_{1,2}$ and $P_{2,3}$  respectively. With this convention, we  denote by   $\vec \tau_{\rm tan}(a)$  the  unit tangent vector   at the point $a$ of the curve oriented accordingly.  Second, we choose the first orthonormal vector $\vec \tau_1$ as
\begin{equation}
\left\{
\begin{aligned}
\vec \tau_1(a)&=\be_3, {\rm \ for \ curves \ in \ } P_{1,2}, {\rm \  and\ }  \\
 \vec \tau_1(a)&=\be_1, {\rm \ for \ curves \ in \ } P_{2,3}. 
\end{aligned}
\right.
\end{equation} 
Finally, we set
$\vec \tau_2(a)=\vec \tau_{\rm tan}(a) \times \vec \tau_1(a)$, so that $\vec \tau_2(a)$ is a unit vector orthogonal to the vector   $\tau_{\rm tan}(a)$,  included in the plane $P_{1,2}$  or $P_{2,3}$ respectively,  and \emph{exterior to the curve} (see Figure \ref{repereref}). We  consider the frame of the orthonormal  plane given by 
\begin{equation}
\label{refframe}
\frameref(a)=(\vec \tau_1(a), \vec \tau_2(a)), \, {\rm \ for \ } a \in \mC. 
\end{equation} 
In particular,   $(\vec \tau_1(a), \vec  \tau_2(a), \vec \tau_{\rm tan}(a))$ is a direct orthonormal basis of $\R^3$. In the case $\mathcal C$ is  a planar curve  in affine planes parallel to $P_{1,2}$ or $P_{2,3}$ or a infinite union of such curves.  
we will often use the notation
$$\pont_\varrho [\mathcal C]=\Up. $$

 \begin{figure}[h]
\centering
\includegraphics[height=4.5cm]{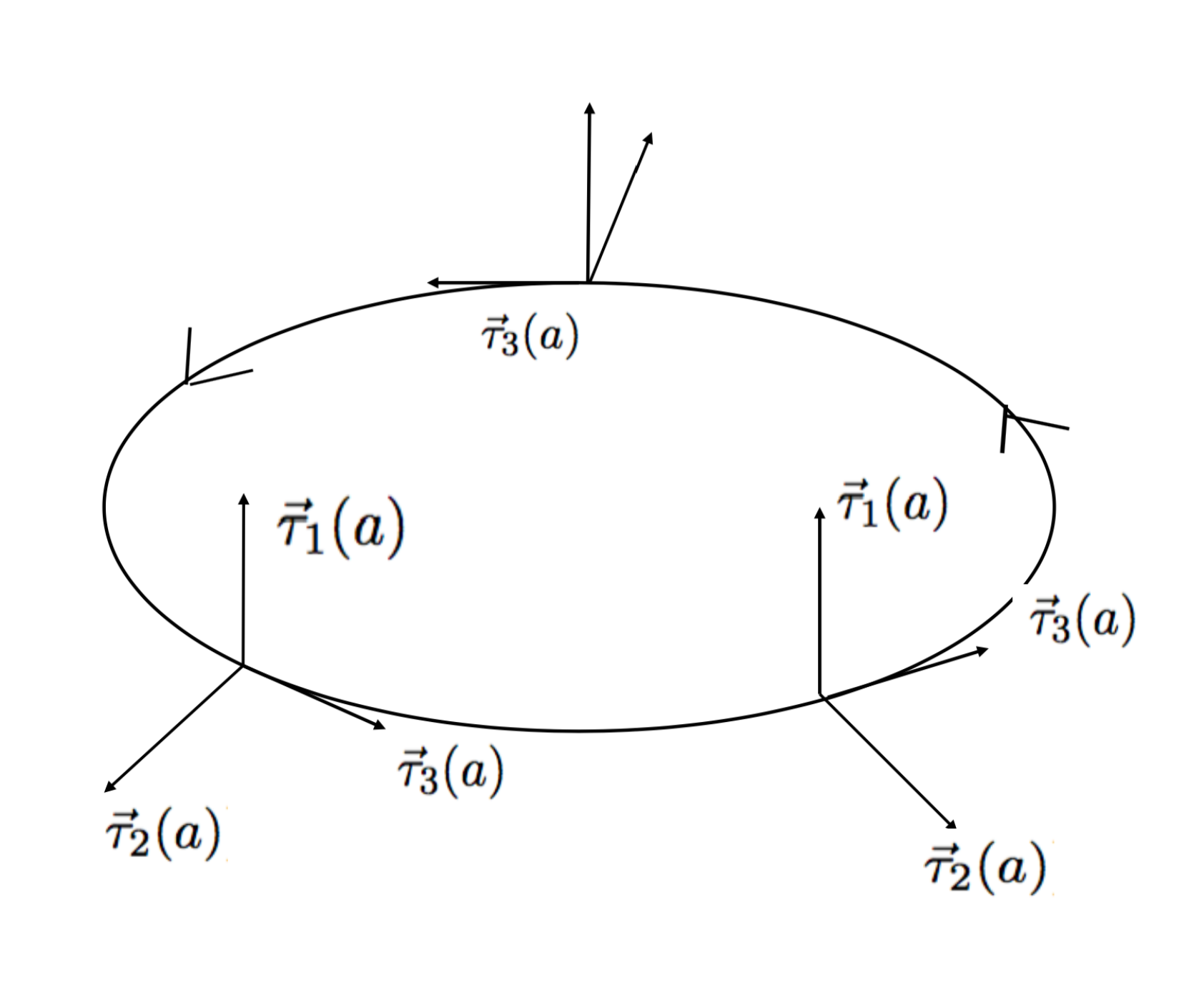}
\caption{  {\it The reference frame for a planar curve.}}
\label{repereref}
\end{figure}

\noindent
{\it Homotopy classes of Pontryagin maps}. If $\mathcal C$ is a planar curve in $P_{1,2}$ or $P_{2,3}$  \emph{framed with the reference frame}  $\frameref$ defined above,  then it turns out that the homotopy class of $\Up$ is \emph{trivial}. 
There are   two  simple ways to make non trivial homotopy classes emerge from the Pontryagin construction:
\begin{itemize}
 \item Twisting  the frame of the orthogonal plane to the curve, \emph{a method which we  will not use in this  paper}. 
 \item  Considering    \emph{planar curves as above which are linked}. 
\end{itemize}
 The idea of the construction of the  Spaghetton  map relies on this latest idea.

\bigskip
\noindent
{\it   On the construction of the $k$-Spaghetton.}  The construction   of the $k$-Spaghetton maps  involves two sheaves of planar curves  which will be denoted $\fLk$ and  $\fLkperp$ respectivement. Each of the sheaves contains   exactly $k^2$ \emph{stadium shaped} connected curves, included in $k$ parallel planes, each of  the planes containing $k$ such curves, which are concentric, so that the general idea is that each of these sheaves consist of parallel \footnote{The word "parallel" has to be taken here in an intuitive meaning and not in a rigorous mathematical sense.} planar curve.  Our construction yields actually
\begin{equation}
\label{presentflk1}
\fLk\subset \underset {q=1} {\overset{k} \cup}  P_{1,2}(qh) {\rm \ and \ }
\fLkperp \subset \underset {q=1} {\overset{k} \cup}  P_{2,3}(qh), 
{\rm \ where \ } h=k^{-1}, 
\end{equation}
where for $s\in \R$, the hyperplanes  $P_{1, 2}(s)$ and $P_{2,3}(s)$ are defined in \eqref{ps}. It follows  that the distance between neighboring parallel planes is exactly $h=k^{-1}$.   The curves in each of these planes are deduced from the other by translation that is
 \begin{equation}
 \label{presentflk2}
 \left\{
 \begin{aligned}
 \fLk\cap P_{1,2}(qh)&=\fLk\cap P_{1,2}(h)+ (q-1)h\, \be_3, {\rm \ for \ } q=1, \ldots, k, \\
 \fLkperp\cap P_{2,3}(qh)&=\fLkperp\cap P_{2,3}(h)+ (q-1)h\, \be_1, {\rm \ for \ } q=1, \ldots, k.
 \end{aligned}
 \right.
\end{equation}
The curves in each sheaf are organized in a  regular way. For instance,   the intersection  of $\fLk$ with the plane $P_{2, 3}$ is given by  a set of $2k^2$ points organized in two two-dimensional grids, namely
\begin{equation}
\label{guidon}
\fLk\cap P_{2, 3}=\{0\} \times \left(\left({\boxplus}^2_k(h)\right)\cup \left((13, 0)+ \boxplus^2_k(h)\right)\right),
\end{equation}
where  the symbol $\boxplus^2_k(h)$ represents  the   discrete sets of points located on the  regular two dimensional grid given, for  $k\in \N^*$  and $h >0$ by
\begin{equation*}
\label{boxplus2} 
 \boxplus^2_k (h)=h \boxplus^2_k=\{ hI, I \in \{1, \ldots, k\}^2.
\end{equation*}
Similarly, we have, for $P_{1, 3}=(\R\be_2)^\perp$
$$
\fLkperp\cap P_{1, 3}=\left\{
(x_1,0,  x_3),{\rm  \  with \ } (x_1, x_3) \in 
\left((0, -7)+\tilde{\boxplus}^2_k(h)\right)\cup \left((0, 6)+ \boxplus^2_k(h)\right)
\right\}.
$$
The two sheaves do not intersect and   each   curve is linked with  all curves of the other sheaf, but with  none of its own. 
On each of the planar curves, we choose the reference frame, set
 $  \mathcal S^k=\fL^k \cup\fLkperp$ and define the $k$-Spaghetton map $\mathfrak S_k$ as the Pontryagin map related to $\mathcal S^k$
\begin{equation}
\label{dececco2}
\mathfrak S_k=\pont_{\varrho_k}[\mathcal S^k, \frameref],
\end{equation}
where the parameter
parameter $\varrho_k$  is chosen suitably of order $h$.
 Details are provided in   
Section \ref{dececco}. 
 We summarize some  its main properties in the  next proposition.
 
\begin{proposition} 
\label{grammage}
The $k$-Spaghetton   $\Spagk$ is a Lipschitz map from $\R^3$ to  $\S^2$  with the following properties:
\begin{itemize}
\item $\Spagk (x)=\sP$,   if $\vert x \vert \geq 17$.
\item $\vert  \nabla  \Spagk  (x)\vert\leq \rCspag k$, for any  $x\in \R^3$, where $\rCspag>0$ is a universal constant.    
\item The Hopf  invariant of $\Spagk$ is 
 $ \rH(\Spagk)=2 k^4$.
 \item  The 3-energy verifies  the energy  bound  $ \rE_3 (\Spagk) \leq \rKspag k^3$, where $\rKspag>0$ is a universal constant.
 \end{itemize}
\end{proposition}

\subsubsection{The Gordian cut} 
 This map  represents the second step of the construction of $\mv_k$.  It is  a map from   a subset $\Lambda \subset \R^4 \to \S^2$, where  $\Lambda$ is the strip  of $\R^4$ given by
$$\Lambda= \R^3\times [0, 50]=\{ \rbx=(x,  x_4), \  x \in \R^3, \  0\leq x_4 \leq 50\}.  $$
  The gordian cut $\Gordk$ corresponds actually to a deformation of the $k$-Spaghetton to a constant map which belongs to the Sobolev class $W^{1,3}$, the fourth coordinate standing for the deformation coordinate, similar to  the time variable in usual deformations. The map  $\Gordk$ belongs to the class of maps  $w: \Lambda  \to \S^2$ such that  the following four conditions are met:
  \begin{equation}
  \label{kontiki}
\left \{
\begin{aligned}
 &w \in \mathcal R(\Lambda, \S^2)  {\rm \ and \ } \rE_3(w, \Lambda)\equiv \int_{\Lambda} \vert \nabla w \vert^3  < \infty,\\
 & w (x, 0)=\Spagk(x, 0),  \ {\rm \ for \ almost  \ every  \ } x \in \R^3, \\
 &w(x, 50)=\sP,   \ {\rm \ for \  almost  every  \ } x \in \R^3, \\
 &w (x, s)=\sP,   {\rm  \  for   \ } x \in \R^3 {\rm \ such \ that \ } \vert x \vert \geq 40 {\rm \ and  \ }   0\leq s \leq 50. 
  \end{aligned}
  \right.
  \end{equation}

\begin{proposition}
\label{deform}
 There exists a map $\Gordk:\Lambda \to \S^2$ verifying \eqref{kontiki}  such that    $\Gordk$  has exactly $k^4$  topological singularities of charge  $+2$ and such that 
\begin{equation}
\label{averell}
 \rE_3(\Gordk) \leq   \, { \rm {\bf K}}_{\rm Gord} \, k^3, 
 \end{equation}
where $ {\rm \bf  K}_{\rm Gord} >0$ is some universal constant.  Let   $\mathbb A_{\star}^k$ denotes the set of singularities of $\Gordk$.   We have 
\begin{equation}
\label{joedalton}  
\mathbb A_{\star}^k=\left(12+\frac h 4 \right)\be_4+
\boxplus_k^2(h) \times \mathbbmss T \left(\boxplus_k^2(h)\right), 
\end{equation}
where $\mathbbmss T$ is   linear and onto   from $\R^2$ into itself  given by 
$\mathbbmss T(x_3, x_4)=(x_3, -2x_3+2x_4)$.
\end{proposition}

Introducing the sets of points on a uniform grid of dimension $4$ given by
\begin{equation}
\label{boxplus} 
\boxplus^4_k=\{1, \ldots, k\}^4  {\rm \ and \ }  \boxplus^4_k (h)=h \boxplus^4_k=\{ hI, I \in \{1, \ldots, k\}^4\}, 
\end{equation}
 we observe that, as a consequence of \eqref{joedalton}, we have 
 \begin{equation}
  \label{bboxplus}
 \mathbb A_{\star}^k=\Upphi_k\left( \boxplus^4_k (h)\right),  {\rm \ with \ } \Upphi_k(\rbx)=\left(12+\frac h 4 \right)\be_4
 +(x_1, x_2, \mathbbmss T(x_3, x_4)),
   \end{equation}
 for  $\rbx=(x_1, x_2, x_3, x_4) \in \R^4$, so that $\Upphi_k$ is an affine one to one mapping on $\R^4$. 
 It follows that $\mathbb A_{\star}^k$  is a regular grid of singularities  (see Figure \ref{singlet}): This observation is crucial, in particular in relation to the minimal branched connection and the result described in Appendix A.
 
\smallskip 
  Although the detailed  argument of the proof of Proposition \ref{deform} involves  some technicalities, 
 the heuristic idea is  simple: We consider $x_4$ as a time variable, and  push down along the $x_3$-axis the sheaf $\fLkperp$, keeping however its shape  unchanged, whereas the sheaf $\fLk$ does not move. This process  presents no major  difficulty as long as the sheaf $\fLkperp$ does not encounter the sheaf  $\fLk$. When some fibers touch, we are no longer able to define the corresponding Pontryagin map $\pont_{\varrho_k}$.  To overcome this difficulty, we take advantage of  the fact that we are working in a Sobolev class where singularities are allowed: Using such singularities, the fiber in contact are able to cross, that is the  sheaf $\fL^k_\perp$ is able to pass through  the fibers of $\fLk$. Each time fibers cross, a singularity of topological charge $2$ is created at the intersection point.  The set of  singularities forms a  cloud of uniformly distributed points as stated in  
 Proposition \ref{deform}.

 \subsubsection{Construction of the sequences of  map $(\mathfrak v_k)_{n \in \N}$}
The construction of the sequences of  map $(\mathfrak v_k)_{k \in \N}$ described in Proposition \ref{deformons} is then deduced  rather directly modifying  the maps $\Gordk$ constructed in Proposition \ref{deform} using some  elementary transformations as affine mappings  or reflections. We proceed  so that, at the end of the process,  we have $\mathfrak v_k\in C_{\rm ct}^0 (\B^4\setminus \Sigma_{\rm sing}^k, \S^2)$, where the set $\Sigma_{\rm sing}^k$  is given  by 
\begin{equation}
\label{sigmasing}
\Sigma_{\rm sing}^k=\boxplus_k^4(\hscal) \cup \mathbbmss S_{\rm ym}\left( \boxplus_k^4(\hscal) \right)  {\rm \ where \ \,  } 
\hscal=\frac{h}{100}=\frac {1}{100k}. 
\end{equation}
 Here,   $\mathbbmss S_{\rm ym}$ stands for the reflection  symmetry through  the hyperplane ${x_4=0}$, i.e. is given by 
\begin{equation}
\label{bbmsym}
\mathbbmss S_{\rm ym}(x_1, x_2, x_3, x_4)=(x_1, x_2, x_3,-x_4), {\rm \ for \ any \ } (x_1, x_2, x_3, x_4) \in \R^4. 
\end{equation}
The singularities in $ \boxplus_k^4(h)$ have Hopf invariant $+2$ whereas the singularities in  $\mathbbmss S_{\rm ym}\left( \boxplus_k^4(h) \right)$ have Hopf invariant $-2$, the total charge being equal to $0$. 
The energy  estimate \eqref{averell0} for the map $\mathfrak v_k$ follows from the corresponding energy  estimate for $\Gordk$.
\subsubsection{Irrigability  of a cloud of points} 
To complete the  proof of Proposition \ref{deformons}, it remains  to establish estimate \eqref{debranche} for  the branched transportation of the map $\mathfrak v_k$, whose value  involves only the location of singularities of $\mathfrak v_k$.   The main property which we will use in the proof of \eqref{debranche}  is expressed in property  \eqref{sigmasing}, which shows that  the singularities are  located on a regular grid.  It follows from \eqref{sigmasing} that
\begin{equation}
\label{tagada}
\frac{1}{k^4} \left( \underset{ a \in \Sigma_{\rm sing}}  \sum  \rH(a) \delta_a\right) \to  f\, {\rm d}x
{\rm \ as  \   } k \to + \infty, {\rm \ where\ }  f=2\left[ \mathbbm 1_{[0, a]^4} -\mathbbm 1_{[0, a]^3\times [-a, 0]}\right], 
\end{equation}
 and where $a=1\slash 100$. It turns out, in view of \eqref{tagada}, that  the behavior of  the functional $\Lbr$  as $k$ grows is related to the irrigation problem  for  the Lebesgue measure, a central  question in the theory of branched transportation.  It has   been proven in \cite{devisolo} (see also Devillanova's thesis or the general description in \cite{bermocas}, in particular Chap 10) that the Lebesgue measure \emph{is not irrigable} for the critical exponent $\alpha_c=\frac 3 4$.  This result can be interpreted directly as the fact that  the functional $\Lbr$ 
 grows more rapidly that  the number of points at the power $\alpha_c$, hence more rapidly then $k^3$.   A lower bound for this  divergence   yields  
 \eqref{debranche}, completing  the proof of Proposition \ref{deformons}. As a matter of fact,  we will rely on a precise  lower bound  of   logarithmic  form for this divergence   which is established in a separate Appendix. 
 
\subsection{On the proof of  the main theorems}

Concerning Theorem \ref{bis},  the proof consist in adding   additional dimensions to the previous constructing and is rather standard. 


\subsection{The  lifting problem}
As a by product of our method, in particular the construction of the Spaghetton maps, we are able to address  some questions related to the lifting problem of $\S^2$-valued maps within the Sobolev context. Such question have already been raised and partially solved in \cite{BeChi, HR1, HR2}.  The main  additional remark we wish to provide in the present paper is that the question is \emph{not related in an essential way} to topological singularities, since our counterexamples  do  not have such singularities. 

\smallskip
Recall that  maps into $\S^2$  and maps into $\S^3$ are  connected through a projection map $\Pi: \S^3 \to \S^2$ termed  the Hopf map and which  we describe briefly.
To start with an  intuitive picture (but as we will see below, this picture is  not completely correct)  the sphere $\S^3$ is very close, at least from the point of view of topology, to the    group of   rotations $SO(3)$ of the space $\R^3$, 
the sphere $\S^3$  may be  in fact  identified with  its universal cover.  Any rotation $R$ in $SO(3)$  yields an element on $\S^2$  considering the image by $R$ of an arbitrary fixed point of the sphere, for instance the North pole $\nP=(0,0,1)$. We  obtain a projection from  $SO(3)$  to $\S^2$ considering the map  $R \mapsto R(P)$,  for $R\in SO(3)$.  The construction  of the projection $\Pi$ from $\S^3$ onto $\S^2$  is in the same spirit, but requires to introduce some preliminary objects.

      \medskip 
 \noindent
 {\it Identifying $\S^3$ with $SU(2)$}. Here  $SU(2)$ denotes the Lie  group of two-dimensional complex unitary matrices of determinant one, i.e. 
$SU(2)=\{ U\in \mathbb M_2(\C),\ U U^* =I_2\ {\rm and}\ \det(U)=1\}$ or 
   \begin{equation}
   SU(2)=\left \{
   \begin{pmatrix}
   a & -\bar b \\
   b  & \bar a \\
   \end{pmatrix}, \ \  a\in \C, b \in \C  {\rm  \ with  \ }   
   \vert a \vert^2+\vert b \vert ^2=1
   \right \} \simeq \S^3.
   \end{equation}
The Lie algebra 
$$su(2)=\{ X\in \mathcal M_2(\C),\ X+X^*=0 \ {\rm and} \ {\rm tr}(X)=0 \}$$
of $SU(2)$ consists of  traceless anti-Hermitian matrices. A canonical 
basis of this $3$-dimensional space,   orthonormal for the Euclidean 
norm $|X|^2\equiv \det(X)$, 
is provided  by the Pauli matrices
$$ \upsigma_1\equiv \left(\begin{array}{rc}
i & 0  \\
0 & -i\\
\end{array}\right), \indent 
\upsigma_2\equiv \left(\begin{array}{rc}
0  & 1  \\
-1 & 0\\
\end{array}\right), \indent
\upsigma_3\equiv \left(\begin{array}{rc}
0 & i  \\
i & 0\\
\end{array}\right).$$
We identify the $2$-sphere $\S^2$ with the unit sphere of $su(2)$ 
for the previous scalar product: 
$$\S^2 \simeq \{ X \in su(2),\ |X|^2=\det(X)=1\}.$$

      \smallskip 
 \noindent
 {\it The Hopf map}.
The group $SU(2)$   acts naturally on $su(2)$ by conjugation: If $g\in SU(2)$, 
then
$$su(2) \ni X \mapsto Ad_g(X) \equiv gXg^{-1} \in su(2) \
{\rm \ is \ an \ isometry \  of \  determinant \ } 1.$$
\begin{definition}
The map $\Pi : SU(2)\simeq \S^3 \to \S^2\subset su(2)$ defined by
$$\Pi(g)\equiv Ad_g(\upsigma_1)=g\, \upsigma_1g^{-1} {\rm \ for \ } g \in SU(2)$$  
 is called the Hopf map.
\end{definition}
Notice  that $\Pi(I_2)=\upsigma_1$ and that  $\Pi(g)=\upsigma_1$ if and only if $g$ is of the form 
$$g=\exp (\upsigma_1 t)=\begin{pmatrix}
\exp it &0 \\
0& \exp- it \\
\end{pmatrix}, {\rm \ for \ some \ } t \in [-\pi, \pi].
$$
 More generally, if $g$ and $g'$ are such that $\Pi(g)=\Pi(g')$, then  there exists some $t\in [-\pi,\pi]$  such that  $g'=g\exp \upsigma_1 t$.  The fiber $\Pi^{-1}(u)$ is hence  diffeomorphic to  the circle
$\S^1$ for every $u\in \S^2$. By the Hopf map, $SU(2)$ appears 
as a fiber bundle with base space $\S^2$ and fiber $\S^1$.
This bundle is not trivial, but twisted since  
$Id_{\S^2}$  \emph{does not} admit a continuous lifting 
$\Phi : \S^2 \to \S^3$ such that $Id_{\S^2}=\Pi \circ \Phi$. Indeed, 
$\Phi$ is homotopic to a constant map, but  $Id_{\S^2}$ is not.

\medskip
\noindent
{\it Projecting maps onto $\S^2$}. Given a domain $\mM$ and a map $\rU: \mM \to \S^3$, we may associate  to $\rU$   the map $u:\mM  \to \S^2$ obtained through  to the composition with $\Pi$, that is setting  
 $u=\Pi \circ \rU$. This construction works for a rather general class of maps, with mild regularity assumptions, for instance measurability. In particular,  since $\Pi$ is smooth, if $\rU$ belongs to $W^{1,3}(\mM, \S^3)$, then the same Sobolev regularity holds for $u=\Pi \circ \rU$. The correspondence 
 $U\mapsto u$ is of course not one to one. Indeed,  given any scalar function $\Theta: \mM\to \R$,  we have the identity
 $$
 u=\Pi \circ \rU= \Pi \circ \left (\rU \exp (\upsigma_1 \Theta (\cdot))\right).
 $$
 Conversely, given two maps $\rU_1$ and $\rU_2$ such that $u= \Pi \circ \rU_1=\Pi \circ \rU_2$ then there exists  a map 
 $\Theta : \mM \to \R$  such that  $\rU_2=\rU_1  \exp (\upsigma_1 \Theta (\cdot))$. The map $\Theta$ is referred to as 
the gauge freedom.

\medskip
\noindent
{\it  Lifting maps  to $\S^2$ as maps to $\S^3$.} The lifting problem corresponds to   invert the projection $\Pi$, which means that,  given  a map $u$ from $\mM$ to $\S^2$ in a prescribed regularity class, one seeks  for a map  $\rU$ from $\mM$  to $\S^3$,  if possible in the same regularity class,  such that $u=\Pi \circ \rU$. The map $\rU$ is then called a \emph{lifting}  of $u$. As seen  in the previous paragraph, if a lifting exists, then there is no-uniqueness, since if $\rU$ is a solution, then the same holds for the map $U\exp (\upsigma_1 \Theta)$, where $\Theta$ is  arbitrary scalar functions  $\Theta:\mM \to \R$ in the appropriate regularity class.

\smallskip
If   $\mM$ is   simply connected with $\pi_2(\mM)=\{0\}$, it can be shown that the lifting problem has always a solution in the continuous class, i.e.  for any continuous maps $u$ from $\mM$ to $\S^2$  there exists a continuous lifting  ${\rU}$ from $\mM$ to $\S^3$ of  $u$ such that $u=\Pi \circ {\rm  U}$.  As an example in the case $\mM=\S^3$, the identity from $\S^3$ into itself is a lifting of the Hopf map.
The fact that the lifting property holds in the continuous class allows to provide a one to one correspondance between homotopy classes in $C^0(\mM, \S^3)$ and $C^0(\mM, \S^2)$. Indeed,  
two maps $u_1$ and $u_2$   from $\mM$ to $\S^2$ are homotopic if and only if their respective liftings ${\rm U_1}$ and ${\rm U_2}$ are in the same homotopy class.  Specifying  this property to the case  $\mM=\S^3$, we obtain as  mentioned an identification of $\pi_3(\S^2)$ and $\pi_3(\S^3)$. 
On the level of Sobolev regularity,  the picture is quite different.  We will prove in this paper:
\begin{theorem}
\label{repulpage}
Let $\mM$ be a  smooth compact manifold.  For any $2\leq p < m=\dim \mM$ there exist a map $\mathcal V$ in $W^{1,p}(\mM, \S^2)$ such there exist    no map ${\rm V} \in W^{1,p} (\mM, \S^3)$ satisfying $\mathcal V= \Pi \circ {\rm V}$. Moreover $\mathcal V$ belong to the strong closure of smooth maps in $W^{1,p}(\mM, \S^2)$. 
\end{theorem}

 This results supplements earlier results obtained  in \cite {BeChi, HR1}. It is proved in \cite {BeChi} that, if $1\leq p< 2\leq \dim \mM$,  $ p \geq \dim \mM \geq 3$  or $p>\dim \mM=2$, then  any map $\mathcal V$  in $W^{1,p}(\mM, \S^2)$  admits a lifting ${\rm V}$ in $ W^{1,p} (\mM, \S^3)$, whereas  a  map was produced there in the cases $2\leq p <3\leq \dim \mM$ or $p=3< \dim \mM$, which possesses no lifting in  $W^{1,p}(\mM, \S^2)$. In the later case, however, the example produced in \cite{BeChi} is not in the strong  closure of smooth maps, in contrast with the map constructed in Theorem \ref{repulpage}.  Notice that  Theorem \ref{repulpage} gives a negative answer to  Open Question 4 in \cite{BeChi}. The only case left open  for the lifting problem in the Sobolev class $W^{1,p}$ is the case $p=\dim \mM=2$ corresponding to the Open Question 3 in \cite{BeChi}.
 

\subsection{Concluding remarks and open questions}
As perhaps the previous presentation shows, the construction of our  counterexample  relies on  several specific properties of the Hopf invariant, a homotopical  invariant   which combines  in an appealing way various aspects of topology in the three dimensional space. 
Our proof  is built  on the fact that the  related branched transportation  involves precisely the critical exponent, yielding a divergence in some estimates which are   crucial.  An analog for this exponent  for more general target  manifolds with infinite homotopy group $\pi_p(\mN)$  has been provided  and worked out in \cite{HR3}, based on more sophisticated  notions in topology.  It is likely that this exponent plays an important role in  issues  related to weak density of smooth maps. In the case the exponent provided   in \cite{HR3} is larger then the critical exponent of the related branched transportation, as described above, one may reasonably conjecture that there should exist some obstruction  the sequential weak density of smooth maps. 
However the  \emph{effective}   constructions of such obstructions, perhaps similar to the ones proposed in this work,  remain  unclear.  In particular the Pontryagin construction  used here seems  at first sight somehow restricted to  the case the target is a sphere.

In another direction, we  notice that most if not all   results related to the weak closure of smooth maps  between manifolds 
for integer exponents deal with manifolds having infinite homotopy group $\pi_p(\mN)$.  The case  when $\pi_p(\mN)$ is finite seems widely open and raises interesting questions also on the level of the  related notions of minimal connections. As a first  example, one may start  with $\mM=\S^2$, $p=4$, for with we have $\pi_4(\S^2)=\Z^2$. In this case  a   description of the homotopy classes in terms of the Pontryagin construction is also available.

\medskip
This paper is organized as follows. In the next section, we recall some  notion of topology. Section 3 is devoted to the construction of the $k$-Spaghetton map. In Section \ref{unfolding}, we describe a number of elementary deformations used in the construction of  the Gordian cut $\Gordk$. These tools as used in Section \ref{laforme} where the proof to Proposition \ref{deform}, which is the central part of the paper, is provided. In Section 6, we provide the proofs of the main results, relying also on some results provided in Appendix A, in particular  Theorem \ref{droppy}, which, beside Proposition \ref{decompp}, is the main result there.

\medskip
\noindent
{\bf Acknowledgements}. The author wishes to thank the referees for their careful reading of  the first  versions of this paper, pointing out several mistakes and indicating several important lines of improvements. 


\numberwithin{theorem}{section} \numberwithin{lemma}{section}
\numberwithin{proposition}{section} \numberwithin{remark}{section}
\numberwithin{corollary}{section}
\numberwithin{equation}{section}


\section{Some topological background}
\label{topology}
We review in this section some basic properties of maps  from $\S^3$ into $\S^2$ or $\S^3$.
 \subsection{Compactification at infinity of maps from $\R^3$ into  $\mN$}
 \label{compactification}
 Whereas  the emphasis was put   so far  in several places   on  maps defined on  the 3-sphere $\S^3$,  it is sometimes easier  to work on the space $\R^3$ instead of $\S^3$. Since our maps will have some limits at infinity or even are constant outside a large ball,  we  introduce the    space
 $$ {C^0_0} (\R^3, \R^\ell) =\{ u \in C^0 (\R^3, \R^\ell)  \  {\rm s. t } \ \underset{ \vert x \vert \to \infty} \lim u  \  {\rm exists} \}$$
 and define accordingly the space  ${C^0_0} (\R^3, \mN)$. The space $C^0_0(\R^3, \R^\ell)$ may be put in  one to one correspondence with the space  $C^0(\S^3,  \R^\ell)$ thanks to the stereographic projection ${\rm St}_3$ which is  a smooth map  from $\S^3 \setminus \{\sP\}$ onto   $\R^3$ and is   defined by 
 $$
 {\rm St}_3 (x_1, x_2, x_3, x_4)= \left( \frac{x_1}{1+x_4}, \frac{x_2}{1+x_4}, \frac{x_3}{1+x_4} \right), {\rm \ for \ } (x_1, x_2, x_3, x_4) \in \R^4  {\rm \ s.t \ } \ \underset{i=1}{\overset{4} \sum} x_i^2=1.
 $$
 For any map $u \in C^0(\S^3, \mN)$ we may define $u \circ {\rm St}_3 \in C^0_0 (\R^3, \mN) $ and conversely  given any map $v$  in $C^0_0 (\R^3, \mN)$  the map $v \circ  {\rm St}_3^{-1}$ belongs to $C^0(\S^3, \mN)$. This allows to handle maps in   $C_0^0 (\R^3, \mN)$ as maps in $C^0(\S^3, \mN)$  and yields a one to one correspondence of homotopy classes. In particular, when $\mN=\S^3$ or $\mN=\S^2$ we may define the degree in the first case or the Hopf invariant in the second for maps in $C_0^0 (\R^3, \mN)$. 
 

\subsection{Degree theory}
Degree theory yields a topological invariant which  classifies homotopy classes for maps  from $\S^3$ to $\S^3$.  For a smooth $\rU$ from $\S^3$ to $\S^3$, its analytical definition is given by 
\begin{equation}
\label{degree}
\deg \rU= \frac{1}{\vert \S^3 \vert}\int_{\S^3} \rU^\star (\omega_{_{\S^3}})=\frac{1}{(2 \pi^2)} \int_{\S^3} \det (\nabla \rU)\,  {\rm d}x, 
\end{equation}
where $\ometrois$  stands for  a standard  volume form on $\S^3$ and $\rU^*$ denotes  pullback by $\rU$.  It turns out that $\deg \rU$ is an integer  which is a homotopical invariant, that is,  two maps  in $C^1(\S^3, \S^3)$ which  are homotopic have  same degrees and conversely,  two maps with the same degree are homotopic, leading as mentioned to  a  complete  classification of homotopy classes. Notice that the degree of   the identity map of $\S^3$ whose homotopy class is the generator of $\pi_3(\S^3)$ is  $1$. The area formula yields a more geometrical interpretation, namely
\begin{equation}
\label{sardine}
\deg u=\underset{ a \in u^{-1}(z_0)} \sum {\rm sign }  (\det (\nabla u)), 
\end{equation}
 where $z_0 \in \S^3$ is any regular point, so that $u^{-1}(z_0)$ is a finite set.  Finally, an important property, is that the $\rE_3$-energy of a map 
 $\rU$ provides a bounded on its degree. Indeed,  integrating the point-wise bound\footnote{Inequality \eqref{nounours} is deduced  from the elementary inequalities  for three dimensional vectors 
 $$  \vert \det (\vec a, \vec b, \vec c )  \vert  \leq \vert \vec a  \vert \cdot \vert \vec b  \vert \cdot \vert \vec c \vert \leq 
 3^{-\frac 32}\left(\vert \vec a \vert^2+\vert \vec b\vert^2 +\vert \vec c \vert^2\right)^{\frac 32},  \forall \vec a , \vec b, \vec c {\rm \ in \ } \R^3, 
 $$
 equality holding if and only if the three vectors $\vec a, \vec b, \vec c$ are orthogonal and have the same norm.}
 \begin{equation}
 \label{nounours}
 \vert \det (\nabla
 \rU)  \vert (x) \leq 3^{-\frac32} \vert \nabla \rU\vert^3, \forall x \in \R^3, 
 \end{equation}
 we obtain the lower   lower bound
  \begin{equation}
  \label{degreenergy}
\rE_3(\rU)\equiv   \int_{\S^3} \vert \nabla \rU \vert^3 \geq   3^{\frac32}  \left \vert  \int_{\S^3} \det (\nabla \rU)\,  {\rm d}x   \right  \vert
\geq 3^{\frac32} \vert \S^3 \vert \vert  \deg \rU \vert. 
  \end{equation}
 On may check that this lower  bound is actually optimal.  Indeed, for degree one maps, equality is achieved in \eqref{degreenergy} by the identity map. For more general integers $d \in \Z$, one may invoke  the  scale invariance of the energy $\rE_3$ in  dimension $3$ and  a gluing  procedure  of  $\vert d \vert $ copies of degree one maps   to show that 
   \begin{equation}
 \label{optideg}
\nutrois (d)\equiv  \inf \left\{ E_3(u), u\in W^{1,3} (\S^3, \S^3), \deg (u)=d\right\} = 3^{\frac32} \vert \S^3  \vert \vert d \vert=2 \sqrt{27}\pi^2 \vert d \vert. 
 \end{equation}
   \subsubsection{The Hopf invariant}
   \label{grouicguinec}
     We next  turn to maps $u$  from $\S^3$ into $\S^2$  which are assumed to have sufficient regularity. Since in  this case there exists a lifting $\rU: \S^3 \to \S^3$ such that $u=\Pi \circ \rU$,  the degree theory  for  $\S^3$ valued maps allows to classify also the homotopy classes of maps from $\S^3$ to $\S^2$. Set
   $$ \rH(u)=\deg (\rU).$$
   This number is called the Hopf invariant of $u$ and  as seen before classifies homotopy classes in $C^0(\S^3, \S^2)$.   
  Notice that,   since $\Pi=\Pi\circ {\rm Id}_{\S^3}$, the Hopf invariant of the Hopf map $\Pi$ is  
  $\rH(\Pi)=1,$
   so that  its homotopy class $[\Pi]$ is a generator of $\pi_3(\S^2)$.


   \medskip
   \noindent
   {\it Integral formulations.}
Let $\mM$ be a simply connected manifold, $ \rU : \mM \to SU(2)$ sufficiently smooth and  set 
$u\equiv \Pi \circ {\rm U}$. We construct a  $1$-form $A$  with values into the Lie algebra $su(2)$ 
 setting $A\equiv \rU^{-1}d\rU$. Conversely, given any sufficiently smooth   $su(2)$ valued 1-form $A$ on $\mM$, on object also called a \emph{connection}, one may find a map $ \rU : \mM \to SU(2)$ such that $A =\rU^{-1}d\rU$, provided the zero curvature equation for connections holds, that is, provided
\begin{equation}
\label{curvature}
dA+\frac 12[A, A]=0.
\end{equation}
Decomposing $A$ on the canonical basis of $su(2)$ as  $A= A_1\sigma_1 + A_2\sigma_2 +A_3\sigma_3$, where $A_1, A_2$ and $A_3$ denote scalar $1$-forms on $\mM$, we 
are led to the relations 
$$du=U [ A,\upsigma_1] U^{-1}=A_3 \upsigma_2-A_2 \upsigma_3.$$
The components $A_2$ and $A_3$ of $A$ are therefore completely determined by    the projected  map $u=\Pi \circ {\rm U}$. On  the other hand, $A_1$ is not,  a consequence  of  the \emph{gauge freedom}  mentioned before. Indeed, for any sufficiently smooth  function  $\Theta: \mM \to \R$,  let
$\rU_\Theta(x) \equiv \exp(\Theta(x)\sigma_1)\rU(x)$, so that 
$u=\Pi \circ \rU_\Theta$ and $\rU_\Theta^{-1}d\rU_\Theta= \rU^{-1}d\rU +(d\Theta) \sigma_1=A+d\Theta \upsigma_1$. The values of $A_2$ and $A_3$ are left unchanged by the gauge transformation, and $A_1$ is changed into  $A_1^\Theta=A_1+ d\Theta$. We notice also the relations
\begin{equation}
\label{eqlabousse}
\left\{
\begin{aligned}
&u^\star (\omedeux)=A_2\wedge A_3,    \  \rU^*(\ometrois)=A_1\wedge A_2\wedge A_3, \\
&|dU|^2= |A_1|^2 + |A_2|^2 + |A_3|^2
{\rm \ and  \ } 
|du|^2 =  |A_2|^2 + |A_3|^2, 
\end{aligned}
\right.  
\end{equation}
 where $\omedeux$ stands for the standard volume form on $\S^2$. 
The curvature equation \eqref{curvature}  yields 
\begin{equation}
\label{glode}
2dA_1=A_2\wedge A_3=u^\star (\omedeux), 
\end{equation}
so that $dA_1$ is also completely determined by the projected map $u$.  Going back  to \eqref{eqlabousse} we may write
$$\rU^*(\ometrois)=A_1 \wedge u^\star (\omedeux). $$
Specifying the discussion to the case $\mM=\S^3$,  the integral formula for the degree yields an integral formula for the Hopf invariant namely, for any map $u: \S^3 \to \S^2$, we have
 \begin{equation}
 \label{hinvariant}
 \rH(u)=\frac{1}{4\pi^2}\int_{\S^3} \alpha   \wedge u^\star (\omedeux), {\rm  \  with  \ } d\alpha=u^\star (\omedeux), 
 \end{equation}
where actually $\alpha$ corresponds to the  1-form $\alpha =2A_1^\Theta$, whatever choice of gauge  $\Theta$.  

\medskip
\noindent
{\it  Choosing a good gauge.} Recall that at this stage $d\alpha=dA_1^\Theta$ is completely determined
by \eqref{glode}.
 To remove the gauge freedom,   we may  supplement \eqref{glode} imposing another constraint  in order to obtain an elliptic system. Hence are led a impose a condition on $d^\star \alpha$, for instance
  \begin{equation}
\label{glodu}
 d^\star \alpha=0, {\rm \ and \  hence\ } \alpha =d^\star \Phi, 
\end{equation}
where $\Phi$ is  some   2-form verifying  $d\Phi=0$. In view of  \eqref{glode},  \eqref{glodu} and the definition $\Delta=dd^\star+d^\star d$ of the Laplacian, we have   the identity
\begin{equation}
\label{laplace}
\Delta_{\S^3} \Phi= u^\star (\omedeux).
\end{equation}
 Hence $\Phi$  is determined up to some additive constant form. 
 
\medskip
\noindent
{\it Energy estimates and the Hopf invariant}. By standard elliptic theory, we obtain the estimates
\begin{equation}
\label{riodeja}
 \Vert \alpha \Vert_{L^3 (\S^3)} \leq  C \Vert \nabla \Phi\Vert_{L^3 (\S^3)} \leq C \Vert \nabla u \Vert_{L^3 (\S^3)}^2, 
 \end{equation}
 so that, going back to formula \eqref{hinvariant}, we deduce that
$
\rH (u) \leq  C\Vert \nabla u \Vert_3^4
$
and hence, as mentioned the lower bound \eqref{lowers3} is readily an immediat consequence of the integral formula for the Hopf invariant.  The  fact that this lower bound is optimal   is    proved \cite{Ri}  and stated here as \eqref{tr}(see also \cite{AK} for related ideas). The proof of the  bound  \eqref{tr} is   subtle  and relies on the  identity
\begin{equation}
\label{triste}
\rH(\omega \circ u) = (\deg \omega)^2 \rH(u), 
\end{equation}
for any $\omega: \S^2 \to \S^2$. Since this fact  is somewhat  central in our later arguments, we briefly indicate how \eqref{triste} may lead to the lower bound \eqref{tr}. A  first elementary  observation is that, given any integer $\ell \in \Z$,  one may construct a smooth map $\omega_\ell: \S^2  \to \S^2$ such that 
\begin {equation}
\label{grouinec}
\deg (\omega_\ell)=\ell \ {\rm \ and \ }  \vert \nabla \omega_\ell \vert_{L^\infty (\S^2)} \leq C \sqrt{\vert \ell \vert}, 
\end{equation}
 the idea being to glue together $\vert \ell \vert $ copies of  degree $\pm 1$ maps scaled down to cover disks of radii of order  $\sqrt{\vert \ell \vert}$. Set $\rul=\omega_\ell \circ \Pi$, so that $\rul$ is a map from $\S^2$  onto $\S^2$. It follows from \eqref{triste} and \eqref{grouinec} that 
$$ \rH(\rul)=\ell^2 {\rm \  and \ }   \vert \nabla \rul \vert_{L^\infty (\S^2)} \leq C \sqrt{\vert \ell \vert}.$$
 Integrating the gradient  bound, we obtain
 $$\rE_3(\rul) \leq  C \vert \ell \vert^{\frac 3 2} \ \leq  C \vert \rH (\rul) \vert^{\frac 3 4}, $$
yielding hence the proof of \eqref{tr}, at least when  the  Hopf invariant $d=\ell^2$ is a square.  The Spaghetton map which we will construct later corresponds to  a  modification of the map $\rul$ and enjoys essentially the same properties, as it  will be seen in  the light of the next paragraph.

\subsection{Linking numbers for pre-images and the Pontryagin construction}
\label{deframe}
   Properties  of the pre-images of regular points yield another,   very appealing,  geometrical interpretation of the Hopf invariant which is  parallel to \eqref{sardine} for the degree. Given  a smooth map $u:\R^3 \to \S^2$  in $C^0_0(\R^3, \S^2)$ and  a \emph{regular value}    $M$ of $u$ on the target    $\S^2$,   its pre-image  $L_M\equiv u^{-1} (M)$ is a smooth  bounded curve in $\S^3$.  The curve   $L_M$ inherits from  
 the original map  $u$ a normal framing and hence an orientation. Indeed, consider an arbitrary  point $a \in L_M$, that is such that $u(a)=M$. Since $M$ is a regular value of $u$, any point on $L_M$ is a regular point of $u$ and the differential $Du(a)$ induces an isomorphism  
of the normal plane $P(a)\equiv (\R\vec \tau_{\rm tan} (a))^\perp$ onto the tangent  space  $T_M(\S^2)$.
  If $(\ov{W}_{1, M}, \ov{W}_{2,M})$  is an orthonormal basis of $T_M(\S^2)$ such that $(\ov{W}_{1, M}, \ov{W}_{2,M}, \ov{OM})$ is a direct orthonormal basis of $\R^3$,  then  its image $\mathfrak f^\perp$ by the inverse  $T=(Du(a)_{\vert P(a)})^{-1}$  is a frame of  $P(a)$ which is however not necessarily orthonormal.   We define  a framing on $L_M$, choosing  the first vector $\vec \tau_1(a)$ of the frame as
  $$\vec \tau_1(a)=\frac{T(\ov {W}_{1, M} )}{\vert T(\ov{W}_{1, M}) \vert} $$
and then  $\vec \tau_2(a)$ as  the unique  unit vector orthogonal to $\vec \tau_1(a)$  such that $\mathfrak e_u^\perp\equiv (\vec \tau_1(a), \vec\tau_2(a))$ has the same orientation   as $\mathfrak f^\perp$. 

A first remarkable  observation (see \cite {ponte} and \cite{kosinski}, chapter XI, section 3) is that  $(L_M, \mathfrak e_u^\perp)$ completely determines the homotopy class of  $u$: Indeed, if $\varrho>0$ is sufficiently small, then 
$$ 
\rH(u)=\rH\left ( \pont_\varrho [L_M, \mathfrak e_u^\perp]\right).
$$    
  A second important property is that the linking number $\mathfrak m (L_{M_1}, L_{M_2})$ of the preimages of any two regular points   $M_1$ and $M_2$ on $\S^2$ is independent of the choice of  the two points and is  equal to the Hopf invariant, that is 
  \begin{equation}
  \label{linkhopf}
   \mathfrak m (L_{M_1}, L_{M_2})=\rH(u).
   \end{equation}
  Recall that the linking number  of two \emph{ oriented curves} $\mC_1$ and $\mC_2$  in $\R^3$ is given by the Gauss integral formula
 \begin{equation}
 \label{gauss}
 \mathfrak m (\mathcal C_1, \mathcal C_2)= \frac{1}{4\pi} \oint_{\mC_1} \oint_{\mC_2}
 \frac{\ov{a_1-a_2}}{ \vert a_1-a_2\vert^3 }\, .\,  \ov{da}_1 \times \ov{da}_2.
\end{equation}
  Notice in particular  that the linking number is always an integer, that  it is symmetric, i.e.  
  \begin{equation}
  \label{lepuc}
  \mathfrak m (\mathcal C_1, \mathcal C_2)=\mathfrak m (\mathcal C_2, \mathcal C_1),
  \end{equation}
     that  its sign changes when the orientation of one of the curves is reversed and that $\mathfrak m (\mathcal C_2, \mathcal C_1)=0$ if the two curves are not linked. In case of several connected components, we have the rule
  \begin{equation}
  \label{linksom}
  \mathfrak m (\mathcal C_{1, 1}\cup \mathcal C_{1,2}, \mathcal C_2)=\mathfrak m (\mathcal C_{1, 1}, \mathcal C_2)+
  \mathfrak m (\mathcal C_{1, 2}, \mathcal C_2).
  \end{equation}
  In practice, as we will do,  the linking number of two given curves can be computed as the half sum of the \emph{signed  crossing number} of a projection on a two dimensional plane.

\begin{remark}{\rm  Chapter IX of \cite{kosinski} offers a good general background to the topics in this section and their extensions. 
The book \cite{mona} offers a more  elementary and intuitive presentation.
}  
\end{remark} 
\subsection{ The Hopf invariant of an elementary spaghetto}

 We go back  to the Pontryagin construction and  consider    the case the curve $\mathcal C$ is planar and connected. We may assume without loss of generally that $\mathcal C$ is included in the plane $P_{1,2}$. We assume moreover that it is framed with the reference frame $\frameref$.  In that  case, the map $\pont_\varrho[\mathcal C]$ will be called an \emph{elementary spaghetto}.  We  first observe: 

\begin{lemma} 
\label{trifou}
We have   $\rH(\pont_\varrho[\mathcal C, \frameref])=0$, for any $0<\varrho<\varrho_0(\mathcal C)$.
\end{lemma}

\begin{proof}  The most direct proof is to use formula \eqref{linkhopf}  and to consider  the linking number of pre-images of any two regular points. For the Pontryagin construction, all points are regular points (see Remark \ref{reguvalpont}), except the south pole $\sP$ whose pre-image is the  set $\partial {\rm T}_\varrho (\mathcal C) \cup\Omega_\varrho (\mathcal C)=(\R^3\setminus {\rm T}_\varrho(\mathcal C)) \cup \partial {\rm T}_\varrho (\mathcal C) $.We  may hence  consider  as regular points the North pole $\nP$ and   the point  $M$ on the equator given by $M=(1,0,0)$. We have
$$ L_{(\nP)}=\mathcal C,  {\rm  \,  \ whereas  \  \, } L_M= \mathcal C +g^{-1}(0) \varrho \be_3,  $$
 where  $g$ is defined in \eqref{defchi1}-\eqref{defchi1bis}. The two curves are parallel  and hence not linked,  so that  
 $$\mathfrak m\,  (L_{ \nP}, L_M)=0.$$
 The conclusion then follows  directly from  \eqref{linkhopf}. 
\end{proof}
  \begin{remark}
  \label{loidebiot}
  {\rm An alternative,  perhaps more direct  and more illuminating though also longer  proof would be to construct 
  \emph{explicitely} a continuous deformation  with values into $\S^2$ of $\pont_\varrho[\mathcal C, \frameref]$  to a constant map. The main step in this construction is  to show that there exists a \emph{continuous}  map $\Phi$ from the exterior domain  $ \R^3 \setminus\mathcal C$ to the circle $\S^1$ such that
  \begin{equation}
  \label{biot}
   \Phi\left( a+x_1 \vec \tau_1(a) + x_2 \vec \tau_2 (a)\right)=\frac{ (x_1, x_2)}{\sqrt{x_1^2+x_2^2}} {\rm \ for \ any \ } a \in  \mC {\rm \ and \ } 
0<  x_1^2+x_2^2\leq\varrho^2.
    \end{equation}
   Assume for the moment that   $\Phi$ is constructed and let us define the deformation. We set 
$$ 
F(x, t)=\left((1-t) \frac{\mr(x)}{\varrho} f \left(\frac{\mr(x)}{\varrho}(1-t)\right) \Phi (x), \, g\left(\frac{\mr(x)}{\varrho}(1-t)\right)\right) {\rm \ for \ } x \in \R^3 { \rm \ and \ } t \in [0,1],
 $$
 where the functions $f$ and $g$ have been defined in \eqref{defchi1}-\eqref{defchi1bis}  and where  the function $\mr$ is defined as
 \begin{equation*}
 \left\{
 \begin{aligned}
 &\mr(x)=\sqrt{x_1^2+x_2^2}  {\rm \ for  \  any  \  } x= a+x_1 \vec \tau_1(a) + x_2 \vec \tau_2 (a) {\rm \ with \ } 
 a \in  \mC {\rm \ and \ } 
0<  x_1^2+x_2^2\leq\varrho^2,  \\
&\mr(x)= \varrho {\rm \ otherwise. } 
\end{aligned}
\right.
\end{equation*}
  It follows from the properties of $f$ and $g$ that 
 $F$ is continuous from  $\R^3 \times [0,1]$  to $\S^2$, that $F(\cdot, t)$ belongs to $C^0_{\rm ct}(\R^3, \S^2)$   for any $t \in [0,1]$ and that 
$$ F(\cdot, 0)=\pont_\varrho[\mathcal C, \frameref]  {\rm \ whereas \ } F(\cdot, 1)=\nP, $$
 yielding hence the desired deformation. The construction of the map $\Phi$ is obtained adapting the Biot and Savart formula, as done for instance in \cite{ABO}. 
} \end{remark}

\begin{remark}
\label{weekend}
  {\rm
A first possible way to obtain not trivial  homotopy classes  through the Pontryagin constructing with planar curves, is to twist the frame. Consider a map 
$\gamma : \mC \to SO(2) \simeq \S^1,$  and consider the twisted frame 
$$e_\gamma^\perp=\gamma (\frameref)\equiv
\left(\gamma (\cdot)(\vec \tau_1(\cdot)),\gamma (\cdot)( \vec \tau_2 (\cdot))\right),$$
 where, for $a \in \mC$, the map $\gamma(a)$ is considered as a rotation of the plane $(\vec{\tau}_{\tan }(a))^\perp$. Since $\mC$ is topologically equivalent to a circle, one may define a winding number  of $\gamma$ and  prove, for instance using the crossing numbers,  that
$$
\rH \left(\pont_\varrho \left[\mC, e_\gamma^\perp\right]\right)=\deg (\gamma).
$$
In some places, we will denote, for given $d \in \Z$,   by $e_{{\rm twist}=d}^\perp$  a framing which corresponds to a planer curves whose reference framing is twisted by a degree $d$ map. As an exercise, the reader may construct a  deformation showing that if $\mC_1$ and $\mC_2$ are two planar curves which  do not intersect and which are  not linked then we  may merge them into a single curve with a frame twisted by the sum of the twists so that 
\begin{equation}
\label{defexo}
\rH \left(\pont_\varrho \left[ (\mC_1, e_{{\rm twist}=d_1}^\perp)\cup (\mC_2, e_{{\rm twist}=d_2}^\perp)\right]\right)=d_1+d_2.
\end{equation}
}
\end{remark}
\subsection{The Hopf invariant of two linked spaghetti}
\label{linkedspag}
Another simple way to obtain non trivial homotopy classes is to consider two linked planar curves, yielding what is often called a \emph{Hopf link}. 
Consider  two planar curves without self-intersection, a curve $\mathcal C_1$ included in the plane $P_{1,2}$ of equation  $x_3=0$ and a curve $\mathcal C_2$ included in the plane $P_{2,3}$ of equation  $x_1=0$. To fix ideas,  one may take for $\mathcal C_1$ and $\mathcal C_2$ the circles
$$\mathcal C_1=\{(x_1, x_2, 0) \in \R^3, x_1^1+ x_2^2=1\} {\rm  \  and \ }  \mathcal C_2=\{(0, x_2, x_3) \in \R^3, (x_2+1)^2+ x^3=1\}.$$
The center of $\mathcal C_1$ is the origin ${\rm O}=(0,0,0)$, the center of $\mathcal C_2$ is  ${\rm O}_2=(0, -1, 0)$, both circles having radius $1$. 
We choose  for  both  circles the reference  frames $\frameref$, yielding  corresponding orientations. They are obviously linked, and  using the crossing numbers, we verify that
$$\mathfrak m (\mC_1, \mC_2)=1.$$
We   then set
$\mathcal C=\mathcal C_1 \cup \mathcal C_2.$

 \begin{lemma}
 \label{hopfion2}
 We have, for sufficiently small $\varrho>0$,  $\rH (\pont_\varrho [\mathcal C, \frameref])=2$.
 \end{lemma}
\begin{proof}  We argue as in the proof of Lemma \ref{trifou} and consider the preimages $L_{\nP}=\mC=\mC_1\cup\,  \mC_2$  and   $L_M=\mC_1'\cup \mC_2'$ of the North pole and   $M=(1,0,0)$  respectively, where we have set 
$\displaystyle{
\mC'_1=  \mathcal C_1 +g^{-1}(0) \varrho \be_3 {\rm \ and \ } \mC'_2=  \mathcal C_2 +g^{-1}(0) \varrho \be_1.
}$
It follows that 
\begin{equation}
\label{frais}
\begin{aligned}
\mathfrak m (L_{\nP}, L_M)&=\mathfrak m  (\mC_1\cup\,  \mC_2, \mC'_1\cup\,  \mC_2') \\ 
&=\mathfrak m (\mC_1, \mC'_1)+\mathfrak m (\mC_1, \mC'_2)+ \mathfrak m (\mC_2, \mC'_1)+\mathfrak m (\mC_2, \mC'_2).  
\end{aligned}
\end{equation}
 Since the curves $\mC_1$ and $\mC'_1$ are parallel and hence not linked $\mathfrak m (\mC_1, \mC'_1)=0$ and likewise $\mathfrak m (\mC_2, \mC'_2)=0$. On the other hand $\mathfrak m (\mC_1, \mC'_2)=\mathfrak m (\mC'_1, \mC_2)=\mathfrak m (\mC_1, \mC_2)=1$ so that we obtain  $\mathfrak m (L_{\nP}, L_M)=2$. Invoking \eqref{linkhopf} the conclusion follows. 
\end{proof}



\section {Linked  $k$-spaghetton map}
\label{dececco}
We provide in this section  a precise definition of the Spaghetton map $\Spagk$, which is roughly described  in the introduction. The general idea is to  extend the construction performed in Subsection \ref{linkedspag}  when each planar curve is replaced by a sheaf of such curves. The spaghetton is  obtained by the Pontryagin construction with  the corresponding reference frame. 

Each of the curves  with which we will perform the Pontragyin construction is  stadium shaped. Let us recall that a stadium is a closed curve whose interior  consists of the interior of a rectangle,  with two parallel ends capped off with semi-disks. 
   Given an integer $k \in \N^\star$, the total number of curves  will be $k^2$ in each of the  two sheaves  $\fLk$ and  $\fLkperp$ of  our construction. Each of the sheaves   $\fLk$ and $\fLkperp$  is composed of   parallel segments  on the straight part on the stadium in the direction of $\be_1$ and $\be_2$ respectively, and nearly parallel on the round parts.

\subsection {The sheaf $\fLk$ of  the  $k^2$   stadium shaped curves $\fL^k_{j, q}$}
  The curves $\fL^k_{j, q}$, $j, q=1, \ldots, k$ composing the sheave $\fLk$, are modeled on  a 
 standard stadium $\mathbb L_0 \subset \R^2$,  centered at the origin $O=(0, 0)$   that we present first.

 \begin{figure}[h]
\centering
\includegraphics[height=5cm]{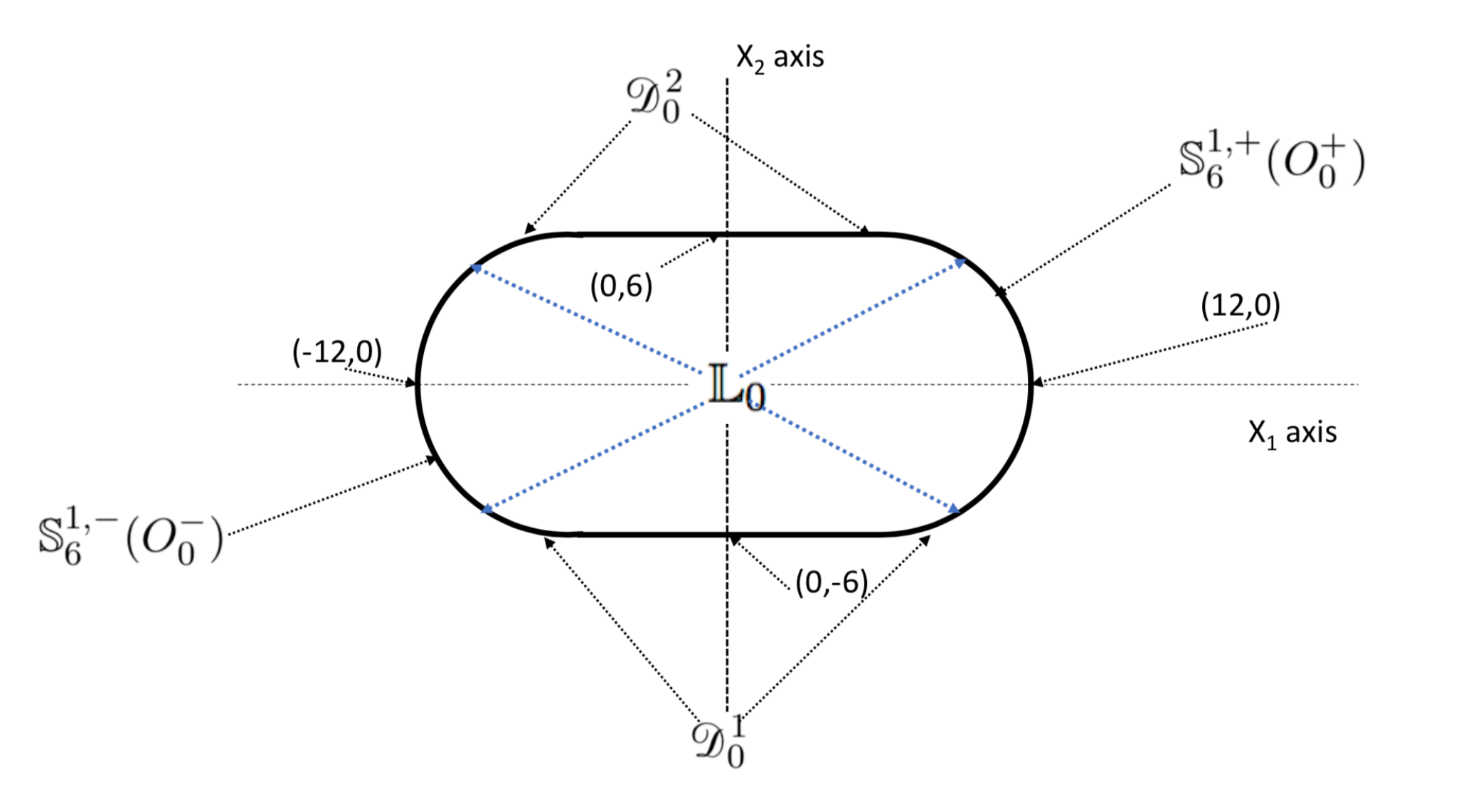}
\caption{  {\it The reference stadium $\mathbb L_0.$}}
\label{staderef0}
\end{figure}

 \medskip
 \noindent
 {\it Construction of the reference stadium $\mathbb L_0$ in the plane $\R^2$}.   Working here on the plane $\R^2$,   we consider first two  segments   $\mathcal D_0^1$  and $\mathcal D_0^2$  parallel to $ \be_1=(1, 0)$  each of  length $12$,  given by
  \begin{equation*}
 \mathcal D_{0}^1=[-6,6] \times \{ -6\}  \ {\rm \ and \ }    \mathcal D_0^2=[-6,6] \times \{ 6\}. 
  \end{equation*}
 We obtain a   stadium  $\mathbb L_0$  included in the plane $\R^2$,   supplementing  the two parallel segment with   two half circles  so that (see Figure \ref{staderef0})
$$
\mathbb L_0=\mathcal D_0^1 \cup   \mathcal D_0^2\cup \S_6^{1,+}(O_0^+) \cup \S_6^{1,-}(O_0^-)\subset \R^2,
 $$
 where $\S_6^{1,+}(O_0^+)$ and  $\S_6^{1,-}(O_0^-)$ are   half circles of  radius $r=6$    in  the plane $\R^2$  of centers $O_0^+\equiv ((6, 0)$ and  $O_0^-\equiv(-6,0)$ respectively. Here we have set, for given  $r>0$ and $A=(a_1, a_2) \in \R^2$ 
   \begin{equation*}
   \left\{
   \begin{aligned}
 & \S_r^{1, +}(A)= \{(x_1, x_2)\in \R^2, (x_1-a_1)^2+(x_2-a_2)^2=r^2, x_1\geq a_1\} \\
   &\S_r^{1,- }(A)=\{ (x_1, x_2) \in \R^2, (x_1-a_1)^2+(x_2-a_2)^2=r^2, x_1\leq a_1\}.
   \end{aligned}
   \right.
   \end{equation*}
As a result of the construction,  
$\mathbb L_0\subset [-12, 12]\times [6, 6].$

\medskip
\noindent
{\it Construction  of  concentric stadia $\mathbb L^k_{j, 0}$}.   Given $k \in \N^*$, we  construct a family of concentric stadia  which are deduced from the reference stadium by homothety as  (see Figure \ref{quadristade11})
$$
\mathbb L^k_{j, 0}= \left(1+\frac{h(k-j)}{6}\right)\, \mathbb L_0 ,  {\rm \ for \  } \ell=0, \ldots, k,  {\rm where \ } h=k^{-1}.
$$
It follows from   this definition that  $\mathbb L^k_{k, 0}=\mathbb L_0$, that $\mathbb L^k_{0,0}=\frac 76\mathbb L_0$,  and that the domains of $\R^2$ bounded by the curves $\mathbb L^k_{\ell, 0}$ are decreasing as $\ell$ increases. Moreover, one may verify that
\begin{equation*}
{\rm dist  } \left(\mathbb L^k_{j, 0}, \mathbb L^k_{j+1, 0} \right)=h, {\rm \ for \ }  j=0, \ldots, k-1.
\end{equation*}
We set 
$$\mathbb L^k=\underset{j=1} {\overset {k} \cup} \mathbb L^k_{j, 0}.$$
The straight parts  of $\mathbb L^k_{j, 0}$ are  parallel to $ \be_1$, of lengths  between $12$ and $14$. We have
\begin{equation}
\label{carun}
\left(\underset{j=1} {\overset {k} \cup} \mathbb L^k_{j, 0}\right) \cap  \left( [-6, 6]\times \R\right)=[-6, 6]\times 
\left[ \{-7+hj, j=1, \ldots, k\}\cup\{7-hj, j=1, \ldots, k\}
\right]
\end{equation}
 and 
$\displaystyle{\mathbb L_{j, 0}^k\subset [-14, 14]\times [7, 7]}$, for any $j=0, \ldots, k.$
 \begin{figure}[h]
\centering
\includegraphics[height=8.5cm]{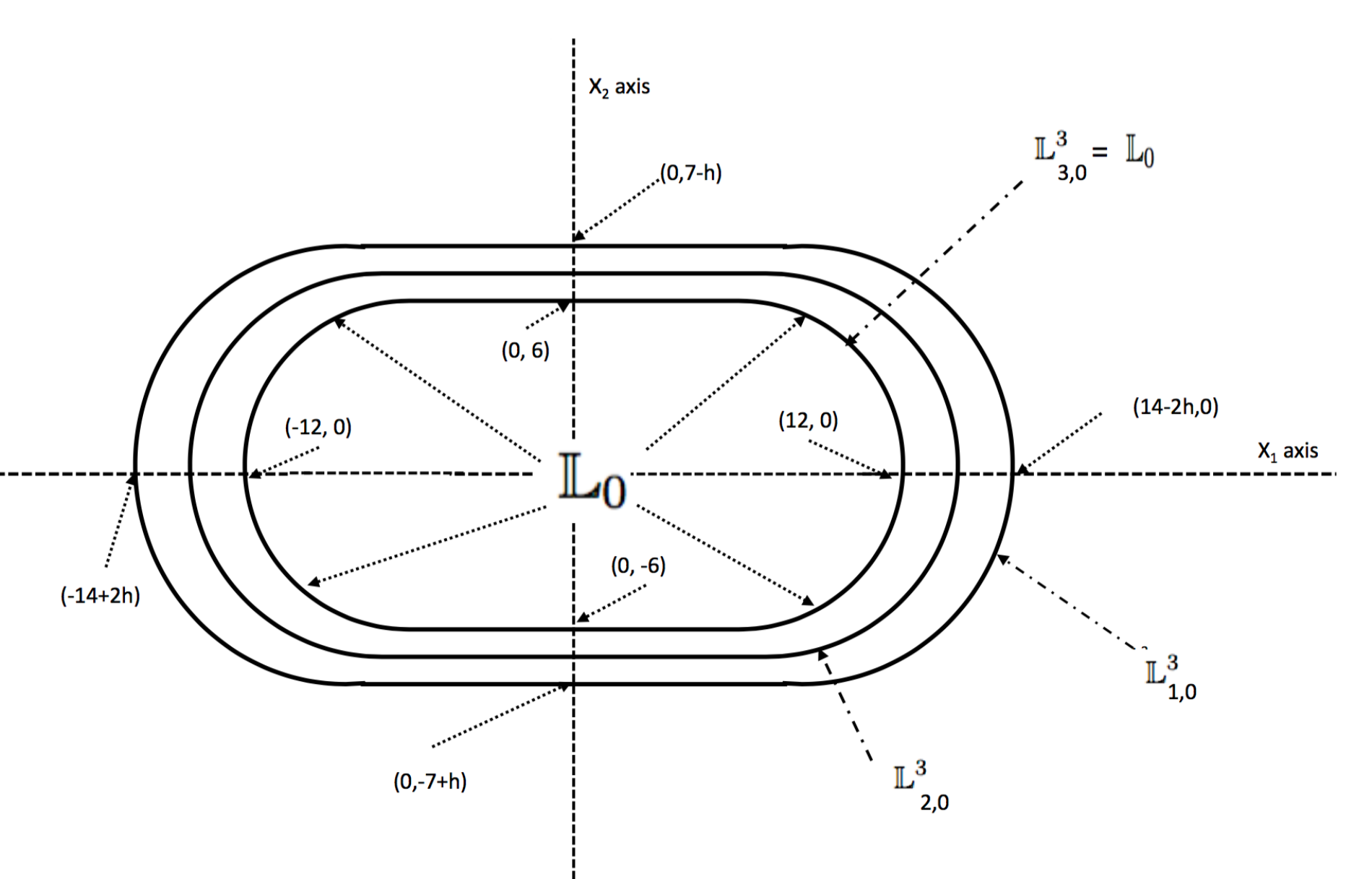}
\caption{  {\it The 3 stadia  $\mathbb L^3_{j,  0}$, $j=1,2,3$ and the reference stadium $\mathbb L_0.$}}
\label{quadristade11}
\end{figure}

\begin{figure}[h]
\centering
\includegraphics[height=9cm]{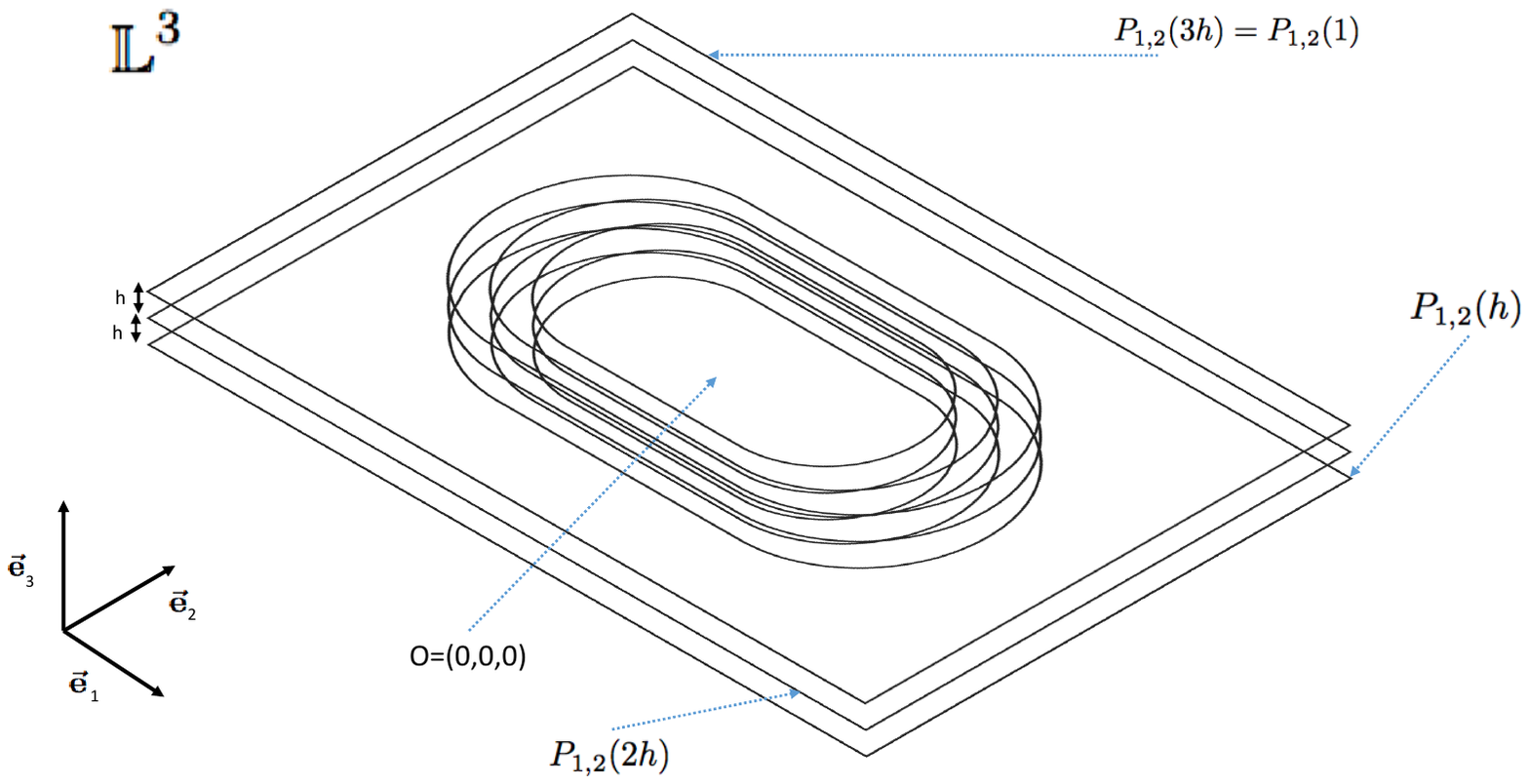}
\caption{  {\it The set  $\mathbb L^3$ and the three planes $P_{1, 2}(h)$, $P_{1, 2}(2h)$ and $P_{1, 2}(3h)=P_{1, 2}(1)$  containing each three connected curves of translates of the sheave $\mathbb L^3$}.}
\label{quadristade}
\end{figure}

\bigskip
\noindent
{\it Construction of the stadia $\mathbb L^k_{j, q}$}. For $q=1, \ldots, k$,  we consider the $k$ parallel planes  $P_{1, 2}(hk)$: Identifying these planes with $\R^2$,  we construct  in each of them the  curves $\mathbb L^k_{j,q}$ corresponding to the stadia 
$\mathbb L^k_{j, 0}$  
setting, for $j=0, \ldots, k$,
$$
\mathbb L^k_{j,q}= \mathbb L_{j, 0}^k+ q h \be_3,  {\rm \ for \  }  q=1,  \ldots, k.
$$
Curves with distinct sets of indices do not intersect. We finally  consider the union of the  $k^2$  curves $\mathbb L^k_{j, q}$, each contained  in    planes  orthogonal to $\be_3$, yielding the sheave  (see Figure \ref{quadristade11})
\begin{equation}
\label{mathbbLk}
 \mathbb L^k=\underset{j, q=1}{\overset {k}\cup}  \mathbb L^k_{j,q} \subset [-14, 14]\times [-7, 7]\times [0, 1]. 
\end{equation}

\noindent
\bigskip
{\it Construction of the curves $\fLk_{j, q}$ and of $\fLk$}.  They are deduced from $\mathbb L^k_{j, q}$ and $\mathbb L^k$ by  a simple \emph{translation} in the direction of $\be_2$. We set, for $j=1, \ldots, k$ and $q=1, \ldots, k$,  
\begin{equation}
\label{mathfrakLk}
\fL^k_{j, q}=\mathbb L^k_{j,q} +   7\,  \be_2  {\rm \ and \ } \fLk= \mathbb L^k+ 7 \be_2,
\end{equation}
 so that  
 $$\fL^k= \underset {j, q} { \overset {k} \cup}  \fL^k_{j, q}  \subset [-14, 14]\times [h, 14]\times [0, 1]\subset \R^2\times [0, 1].$$
 Property \eqref{guidon} presented in the introduction then  follows from \eqref{carun}  and the above constructions (see in particular figure \ref{above0}).  
   The mutual distant between the individual spaghetti is bounded  below by
 \begin{equation}
 \label{mutual}
  {\rm dist} (\fL^k_{j, q}, \fL^k_{j', q'}) \geq  h= \frac {1}{k},  {\rm \ for \ } (j, q)\not = (j', q'). 
 \end{equation}
 Going back to \eqref{tubular} we  may observe also that, at least for large $k$ enough,  we have
 \begin{equation}
 \label{rc}
 \varrho_0 (\fL^k) \geq \frac{1}{3k}. 
 \end{equation}
 The total linking number of $\fL^k$ is equal to zero. In order to produce topology, we  define  a second sheaf.
\subsection{The sheaf $\fLkperp$}
 \begin{figure}[h]
\centering
\includegraphics[height=9cm]{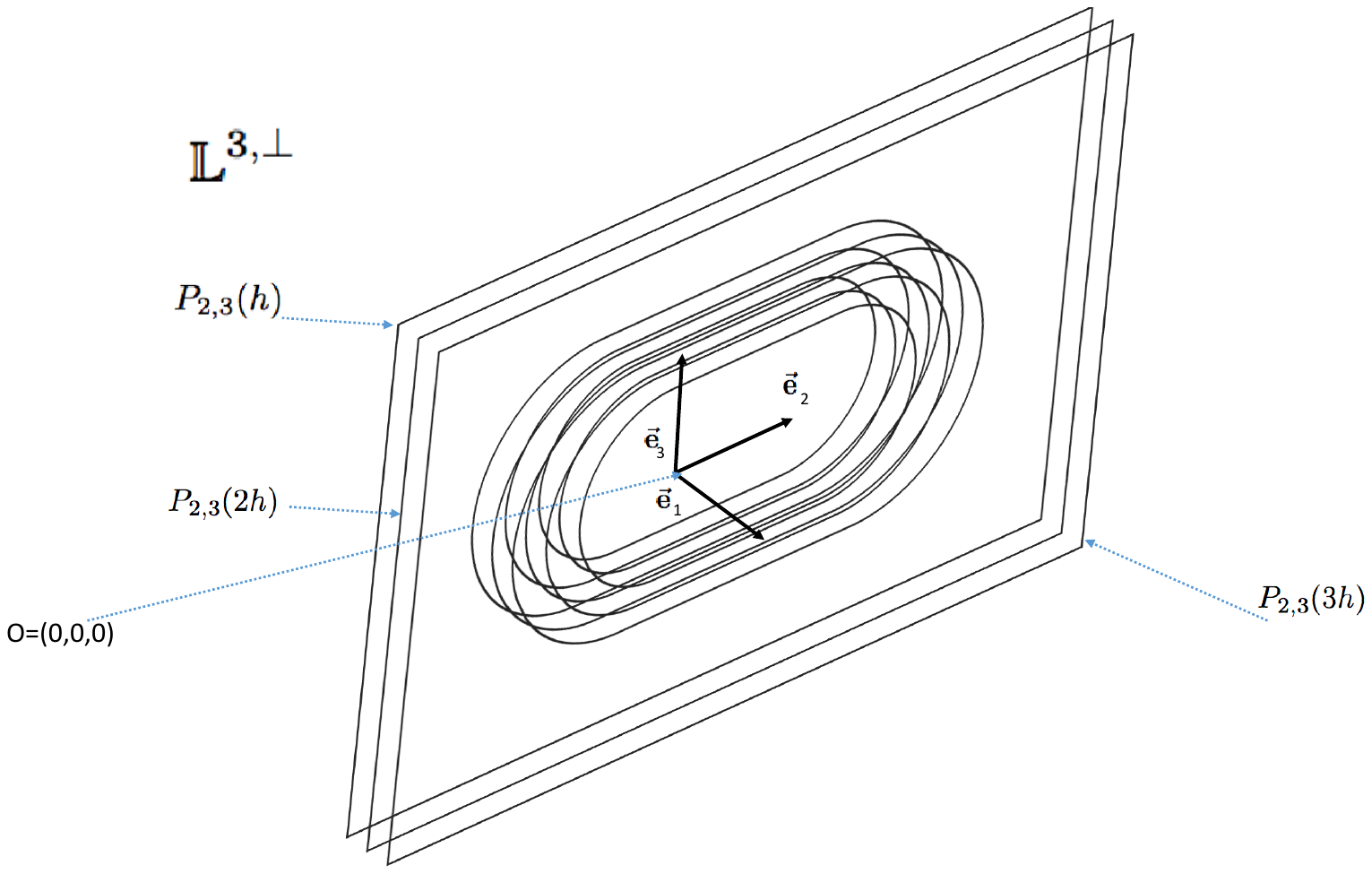}
\caption{  {\it The  set $\mathbb L^{3, \perp}$, $k=3$ and the  hyperplanes $P_{2,3}(h), P_{2,3}(2h)$ and $P_{2,3}(3h)$ with $h=\frac1 3$.}}
\label{spaghettonperp}
\end{figure}
 
We first  construct as above a sheave $\mathbb L^{k, \perp}$ deduced from     the sheave $\mathbb L^k$                                                                                                                                                                                                                                                                                                                                                                 as  
 $$ 
 \mathbb L^{k, \perp}=\left\{(x_1, x_2, x_3)\in \R^3 {\rm  \ s.t. \ } \ ( x_2,  x_3 ,x_1 ) \in \mathbb L^k\right\}.
 $$
Alternatively, we may define  $\mathbb L^{k, \perp}$ as the image of $\mathbb L^k$ by the rotations $\mathfrak R_0$ of $\R^3$ which sends 
$\be_1$ onto $\be_2$, $\be_2$ onto $\be_3$ and $\be_3$ onto $\be_1$ (see Figure \ref{spaghettonperp}). We decompose $ \mathbb L^{k, \perp}$  as
\begin{equation}
\label{chevre}
\left\{
\begin{aligned}
 \mathbb L^{k, \perp}&=\underset{i, q=1}{\overset {k} \cup} \mathbb L^{k, \perp}_{i, q}, \, {\rm \ where \ } \\
 \mathbb L^{k, \perp}_{i, q}&=\mathcal R_0 (\mathbb L^{k}_{k-q+1,i})=\left\{(x_1, x_2, x_3)\in \R^3 {\rm s.t. } \ ( x_2,  x_3 ,x_1 ) \in \mathbb L^k_{q-k+1, i}\right\}.
\end{aligned}
\right.
\end{equation}
 For $q=1, \ldots, k$,   the connected curves $ \mathbb L^{k, \perp}_{i, q}$ are included in the same affine  plane $P_{2, 3}(ih)$.  Notice an \emph{important difference} between the way we label the curves $\mathbb L_{i, q}^{k, \perp}$ and the way we label the curves  $\mathbb L_{i, q}^{j}$: The domain included in the plane  $P_{2, 3}(ih)$ bounded by the curves $\mathbb L_{i, q}^{k, \perp}$ are increasing with  $q$, for fixed $i$. As a matter of fact, we may also write
 \begin{equation}
 \label{chevre2}
  \mathbb L^{k, \perp}_{i, q}= \left(1+\frac{h q}{6}\right)\, \left(\mathbb L_0^\perp+ ih \be_1\right), 
  {\rm \ for \  } i, q=0, \ldots, k,  {\rm where \ } \mathbb L_0^\perp=\mathcal R_0 (\mathbb L_0).
 \end{equation}
 We notice that 
  \begin{equation}
  \label{gameof}
     \mathbb L^{k, \perp} \subset\left( [0, 1]\times [-14, 14] \times [-7,7] \right),
 \end{equation}
  and that $\mathbb L^{k, \perp}$ is composed of segments in the direction $\be_2$ in its central part, of lengths between $12$ and $14$. More precisely, we have
  $$
   \mathbb L^{k, \perp} \cap \left( \R \times [-6, 6] \times \R\right)
   = h\mathbb I_k \times [-6, 6] \times \left[
    \left(h\mathbb I_k-\{7\}\right) \cup \left( h\mathbb I_k+\{6-h\}
   \right) \right],
  $$
  where we have set $\mathbb I_k=\{1, \ldots, k\}$.

 The set $\fLkperp$ is deduced from the set $\mathbb L^{k, \perp}$ by a translation in the direction of $\be_2$. We set
 \begin{equation*}
 \fLkperp= \mathbb L^{k, \perp} - 3 \, \be_2  {\rm \ and \ }
 \mathfrak L^{k, \perp}_{i, q}= \mathbb L^{k, \perp}_{i, q}-3 \, \be_2 {\rm \ for \ } i, q=1, \ldots, k.
 \end{equation*}
Inclusion \eqref{gameof} then yields 
\begin{equation}
\label{thrones}
\left\{
\begin{aligned}
&\fLkperp \subset [0, 1]\times [-17, 11] \times [-7,7] \\
&\fLkperp \cap \left(\R \times[-2,2] \times \R\right)= h\mathbb I_k \times [-2, 2] \times \left[
    \left(h\mathbb I_k-\{7\} \cup h\mathbb I_k+\{6-h\}\right)
    \right].
\end{aligned}
\right.
\end{equation}

\begin{remark}
\label{chouchou}
{\rm The	labelling  \eqref{chevre} and \eqref{chevre2} of the curves $\fLkperp_{i, q}$ which is different from the labeling of the curves $\fLk_{j, q}$  is motivated by the fact that 
	\begin{equation}
\label{chouperche}
\fLkperp_{i, q}\cap \R\times [-2,2]\times \R^+= \{ih\} \times [-2,2]\times \{6+qh\}
	\end{equation}
 so that that the $q$ index labels the  upper straight part of the fibers with increasing height $x_3$.		
		}
\end{remark}
\subsection{ First properties of the sheaves $\fLk$ and $\fLkperp$}
  \begin{figure}[h]
\centering
\includegraphics[height=8cm]{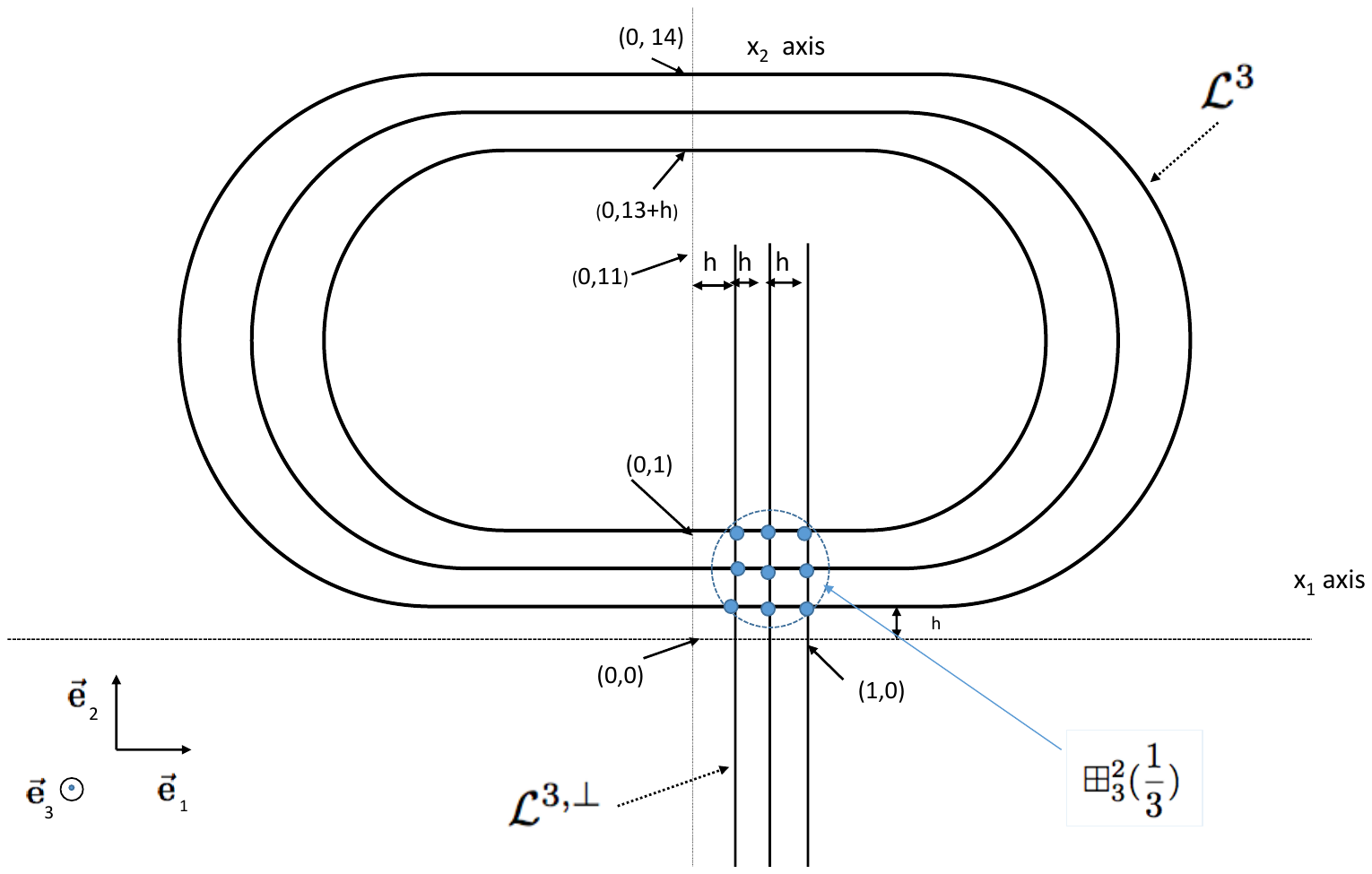}
\caption{  {\it The  set $\mathcal S^3$ seen from above.
The intersection of the orthogonal projection  onto $P_{1,2}$ of $\mathfrak L^{3, \perp}$ with
$\mathfrak L^3$ is the grid $\displaystyle{ \boxplus^2_3 (\frac 13)}$.}}
\label{above0}
\end{figure}
  \begin{figure}[h]
\centering
\includegraphics[height=10cm]{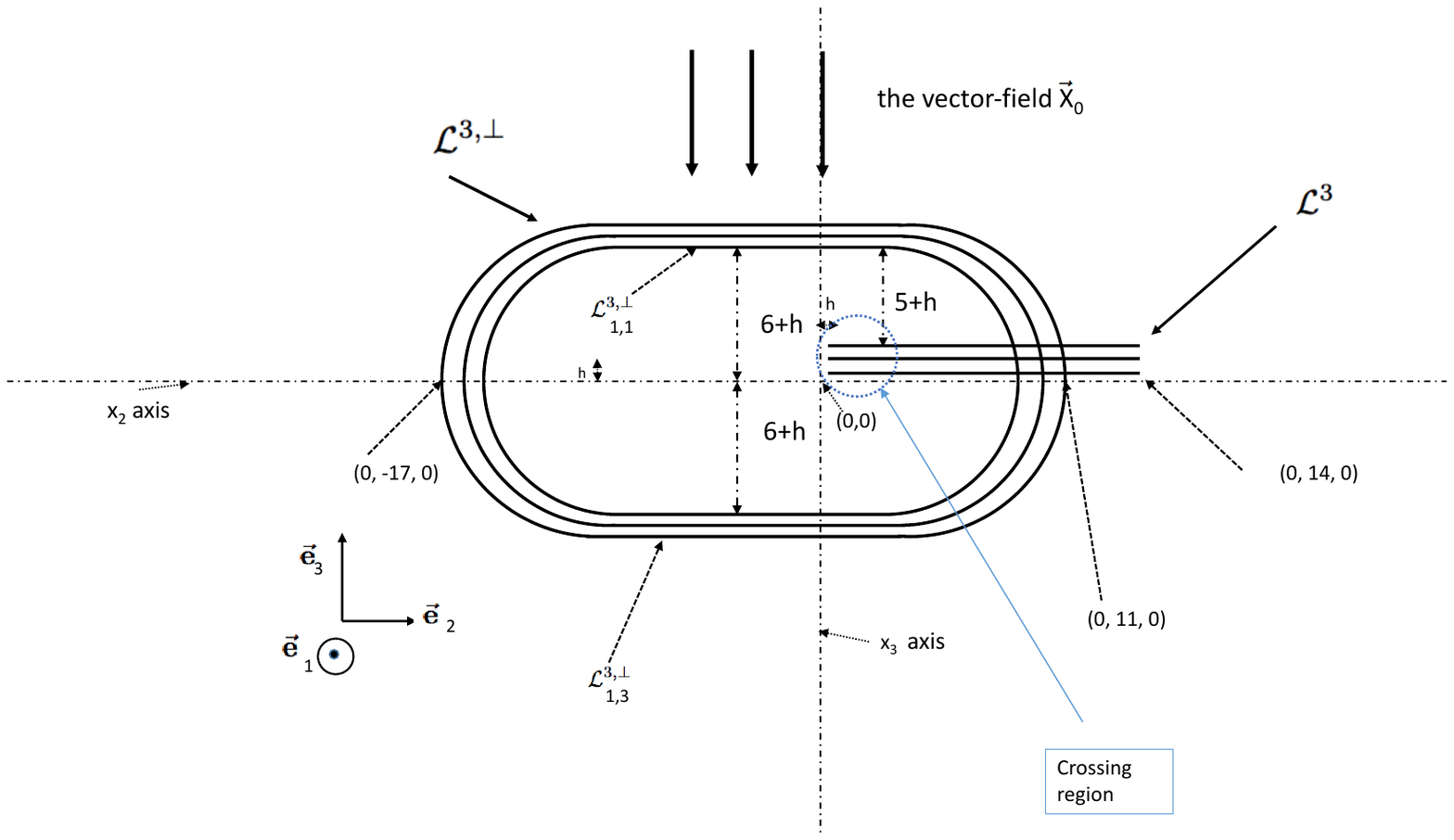}
\caption{  {\it The  set $\mathcal S^3$ seen from the $\be_1$ direction. 
The intersection of the orthogonal projection  onto $P_{1,2}$ of $\mathfrak L^{3, \perp}$ with
$\mathfrak L^3$ is the grid $\displaystyle{ \boxplus^2_3 (\frac 13)}$.  The vector field $\vec X_0$ pushes $\mathfrak L^{3, \perp}$ onto $\mathfrak L^3$ until they meet.}}
\label{side0}
\end{figure}

 Notice first (see figure \ref{above0} and \ref{side0}) that the intersection of  the  two sheaves $\fL^k$ and $\fLkperp$ is empty and that moreover
 \begin{equation}
 \label{rc2}
 {\rm dist} (\fLk, \fLkperp)=2. 
 \end{equation}  
 Since  each of  the  curves $\fL_{j, q}^k$ and $\fL_{i, q}^{k, \perp}$ are   planar curves which are either included in  affine planes  parallel to  $P_{1,2}$ or to  $P_{2,3}$ we may frame them with  the reference frames $\frameref$ which have  been defined in  Subsection \ref{deframe}. This yields,  as  we have already seen,  a natural orientation of the curves. For instance, the curves $\fL_{i, j}^k$ 
 are oriented counter-clockwise with respect to the frame $(\be_1, \be_2)$  and similarly the curves   
 $\fL_{i, q}^{k, \perp}$ are oriented counter-clockwise with respect to the frame $(\be_2, \be_3)$. 
 
 Concerning \emph{topological properties}, each  curve  $\fLk_{i_0, j_0}$ is linked to the $k^2$ curves  $\fLkperp_{i, q}$  with linking number $1$, but is linked with none of the curves of its own sheaf. Similarly,  each curve 
  $\fLkperp_{i_0, q_0}$ is linked 
   with the $k^2$ curves  $\fLk_{i, j}$ with linking number $1$, but is linked with none of the  curves of its own sheaf (see Figure \ref{quadristadelink}).   Hence, we obtain for the total linking number:
\begin{lemma}
\label{linkspagh} We have 
$\mathfrak m (\fLk, \fLkperp)=k^4$ for any $k \in \N$. 
\end{lemma}
\begin{proof}  We have 
$$
\mathfrak m (\fLk, \fLkperp)=\mathfrak m \left( \underset {j, q} \cup  \fL^k_{j, q}, \underset {i, q'} \cup  \fL^{k, \perp}_{i, q'}\right)
=\underset{j, q} \sum\underset{j, q'} \sum \mathfrak m \left( \fL_{j, q}^k, \fL_{i, q'}^{k,\perp}\right)
=k^4  \mathfrak m \left( \fL_0,  \fL_0^\perp \right), $$
and the conclusion follows from  the identities \eqref{lepuc} and \eqref{linksom}.
\end{proof}
\begin{figure}[h]
\centering
\includegraphics[height=8cm]{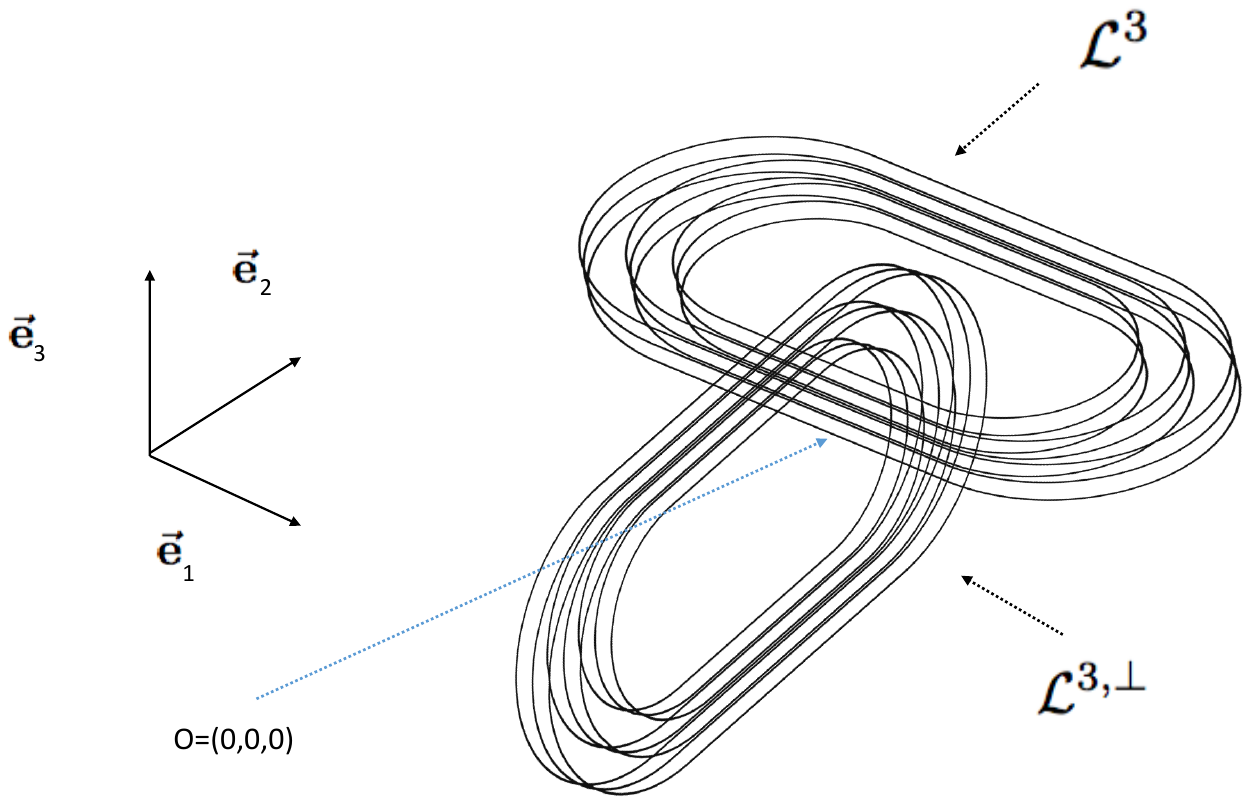}
\caption{  {\it The sets $\fLk$ and $\fLkperp$ are linked.}}
\label{quadristadelink}
\end{figure}

 Finally, as mentioned in the introduction,   we consider the one-dimensional set
  \begin{equation}
  \mathcal S^k=\fL^k \cup \fLkperp.
  \end{equation}
It follows from \eqref{rc} 	and  \eqref{rc2}  that
\begin{equation}
\label{rc3}
\varrho_0 (\mathcal S^k) \geq \frac{1}{3k}  {\rm \ and \ } \mathcal S^k \subset \B(17).
\end{equation}

\subsection {The $k$-spaghetton map $\Spagk$ and its properties}
Choosing $\varrho_k=10^{-3}\varrho_0 (\mathcal S^k)$,  we define the $k$-spaghetton map $\mathfrak S_k$ as
\begin{equation}
\label{dececco2}
\Spagk=\pont_{\varrho_k}[\mathcal S^k, \frameref].
\end{equation}
Some of its most  relevant properties are  summarized in Proposition \ref{grammage} of the introduction, that we prove next.

\begin{proof}[Proof of Proposition \ref{grammage}]
  The first assertion concerning the support of $\Spagk$  follows from the inclusion given  in \eqref{rc3},  whereas the second, the bound on the gradient, is an immediate consequence of \eqref{blablanabla}.  Since all fibers have  the same shape, which does not depend on $k$, the  constant $\rCspag$ involved in the gradient estimate  does not depend on $k$ either. Turning to  the third assertion, the computation of $\rH(\Spagk)$ follows the same lines  as the proof of Lemma \ref{hopfion2},   considering the pre-images 
  $$L_{\nP}=\mathcal S^k=\fLk \cup  \fLkperp  {\rm \  and  \ }   L_M=(\fL^k+ g^{-1}(0) \varrho_k \be_3) \cup  (\fL^{k,  \perp}+g^{-1}(0) \varrho_k \be_1 )$$
   of the North pole $\nP$ and  the point $M=(1,0,0)$ of the equator respectively. Arguing as in \eqref{frais}, we obtain
\begin{equation*}
\begin{aligned}
\mathfrak m (L_{\nP}, L_M)&=\mathfrak m 
( \fL_k \cup  \fL^{k,  \perp} , \fLk+ g^{-1}(0) \varrho \be_3) \cup  (\fL^{k,  \perp}+g^{-1}(0) \varrho \be_1) \\
&=2 \mathfrak m (\fLk, \fLkperp)=2k^4,
\end{aligned}
\end{equation*}
where the last identity follows from Lemma \ref{linkspagh}. For the
  estimate  on the energy in    the statement of Proposition \ref{grammage}, we observe  that, since  the support of $\vert \nabla \Spagk \vert $  is included in the ball $\B(17)$ independently of $k$, it suffices to integrate  the uniform bound of the gradient, which is of order $k$  to obtain  the result for the 3-energy.  In particular, we may choose the constant as
  $\displaystyle{\rKspag=17^3 \rCspag^3}$.\end{proof}
\section{Toolbox for the construction of   the gordian cut $\Gordk$}
\label{unfolding}
The proof of Proposition \ref{deform} is somewhat technical, its completion will be given in Section \ref{laforme}. The heuristic idea is however rather simple: We  push down along the $x_3$-axis the sheaf $\fLkperp$, keeping its shape essentially unchanged, whereas the sheaf $\fLk$ does not move. This presents no major  difficulty, pushing along a constant vector-field,   \emph{as long as the sheaf $\fLkperp$ does not encounter the sheaf  $\fLk$}. When the two sheafs touch, we take advantage of  the fact that we are working in a Sobolev class were singularities are allowed: Using such singularities, the sheaf $\fLkperp$ is enabled to follow his way down and to pass through  the fibers of $\fL^k$, creating on the way point singularities. 

\smallskip
In  order to provide a  sound  mathematical meaning to the previous ideas, in particular the crossing of fibers,  we    single out  some  elementary constructions   which are used extensively and iteratively in  the proof of Proposition \ref{deform}. In a first subsection, entitled   \emph {Sobolev deformation and surgery toolbox}, we present rather general  tools.  Then, we focus (see Subsection \ref{creahopf}) on a more specific construction accounting  for the creation of topological singularities of $\Gordk$.    
\subsection{Sobolev  deformation   and surgery toolbox}
\label{sobol}
\subsubsection {Gluing maps }  
\label{glue}
 This is the most elementary operation. Assume  that we are given two subdomains $\Omega_1$ and  $\Omega_2$  of a domain $\Omega_0$ of  $\R^3$ such that $\bar \Omega_1 \cap \bar \Omega_2=\emptyset$ and let $v_1$ and $v_2$ be two functions with values into $\S^2$  such that $v_1(x)=\sP$ for  $x\in \Omega_0\setminus \Omega_1$ and $v_2 (x)=\sP$ for  $ x \in \Omega_0\setminus \Omega_2$. Then we   define the function $v_1 \wedgetrois v_2$ on $\Omega_0$ by
 \begin{equation}
 \label{riritou}
 \left\{
 \begin{aligned}
& v_1 \wedgetrois v_2 (x) =v_1(x), {\rm \ for \ } x \in \Omega_1, \  
 v_1 \wedgetrois v_2 (x) =v_2(x), {\rm \ for \ } x \in \Omega_2,\\
&v_1 \wedgetrois v_2 (x) =\sP, {\rm \ for \ } x \in \Omega_0\setminus (\Omega_1\cup \Omega_2). 
 \end{aligned}
 \right.
 \end{equation}
 An alternative and even simpler  definition is
 $$v_1 \wedgetrois v_2 (x)=v_1(x)+v_2(x)-\sP, \  {\rm \ for \ any  \ } x \in \R^3.$$
  In the case both $v_1$ and $v_2$ have bounded $\rE_3$ energy, then the same holds for $v_1\wedgetrois v_2$ with
 \begin{equation}
 \label{somenergy}
 \rE_3(v_1 \wedgetrois v_2) =\rE_3 (v_1)+\rE_3(v_2).
\end{equation}
A related situation is encountered in the case $\Omega_2=\Omega_0\setminus \Omega_1$, when both $v_1$ and $v_2$ have bounded $\rE_3$ energy. If the domains are sufficiently smooth, then one may define  thanks to the trace Theorem the restrictions ${v_i}_{\vert_{\partial \Omega_i}}$  for $i=1, 2$. If moreover 
$$
v_1(x)=v_2 (x),  {\rm \ for  \ } x \in  \partial \Omega_1 \subset \partial \Omega_2=\partial \Omega_1 \cup \partial \Omega_0,
 $$
then we may define again $v_1 \wedgetrois v_2$ according to \eqref{riritou} and relation \eqref{somenergy} still holds.
Given two  disjoint oriented compact  framed curves in $\R^3$, we have,  provided   $\varrho>0$  is sufficiently small
$$ 
\pont_\varrho [(\mC_1, \mathfrak e^\perp) \cup  ( \mC_2, \mathfrak e^\perp)] =\pont_\varrho [\mC_1, \mathfrak e^\perp] \wedgetrois \pont_\varrho[\mC_2, \mathfrak e^\perp].
$$
  so that in particular
$$\Spagk=\left(\underset{j, q=1} {\overset {k} \wedgetrois}\pont_{\varrho_k} [\fLk_{j, q}, \frameref] \right) \,  \wedgetrois \,\left( \underset{i, q=1} {\overset {k} \wedgetrois}\pont_{\varrho_k} [\fLkperp_{i, q}, \frameref ]\right)
$$
Finally, we refer to similar gluing  in $\R^4$,  replacing the symbol $\wedgetrois$ by  the symbol $\wedgequatre$. 
\subsubsection{Deformations of the domain}
\label{domain}
 {\bf Definitions}. We consider  deformations of \emph{maps}   generated by deformations of the \emph{domain  $\R^3$} induced by the integration of a vector field.
  Given a smooth vector field $\vec X$ on $\R^3$,  we consider the flow $\Phi$ generated by the vector field   $\vec X$ defined by   
 \begin{equation}
 \label{flow}
 \frac {d} {dt} \Phi (\cdot, t)= \vec X [\Phi(\cdot, t) ], {\rm \, \ with \ } \ \Phi (\cdot,0)= {\rm Id}_{\R^3},
 \end{equation}
 so that, for each  fixed time $t \geq 0$,  the map $\Phi(\cdot, t):\R^3 \to \R^3$ is  a diffeomorphism of  $\R^3$. We denote by $\Phi^{-1} (\cdot, t)$ its inverse at time $t$,  so that   $\Phi^{-1} (\cdot, t)=\Phi (\cdot, -t)$. 
  The deformation of the domain  gives rise  to   corresponding deformations of general  functions: To each function $v$ defined  on $\R^3$ and given $t \geq 0$, we associate a function $v_t (\cdot)$  defined by 
  $$v_t(x)=v\left(\Phi^{-1} \left(x, t\right)\right), {\rm  \ for \ } x \in \R^3. $$
   The curve $t\mapsto  v_t$ is now a continuous deformation of the initial function $v$, since $v_0=v$. 
 We  will also   consider  the transportation of subsets of $\R^3$ by the flow $\Phi$. We  set accordingly for a subset $A \subset \R^3$ and $t \geq 0$
\begin{equation}
\label{pushset}
 \Phi(A, t)=\{ x \in \R^3, \Phi^{-1}(x, t) \in A \}.  
\end{equation}
We will apply this construction to  the curves composing the sheaves $\fLk$ and $\fLkperp$.
  Notice that, if $\mathcal C$ is a framed closed  curve of $\R^3$, then in general 
  \begin{equation}
  \label{pushcte}
   \left(\pont_\varrho\left[\mathcal C, \frameref\right] \right)_t(x)\not = \pont_\varrho \left[\Phi (\mathcal C, t), \frameref\right] (x), {\rm \ for \ } x \in \R^3 {\rm \ and \ } t \geq 0, 
 \end{equation}
where the frame has been transported accordingly.  However equality holds in case $\vec X$  is a constant function, since in that case $\Phi (\mathcal C, t)$ is a translate of $\mathcal C$.  This observation leads us to introduce a variant of the Pontryagin construction for non-constant vector fields.

   \medskip
   \noindent
   {\bf  Vertical  vector  fields}.
  We implement the previous construction with a very specific choice of vector fields $\vec X$. Since our aim is to push the sheaf
  $\fLkperp$  down according to  $x_3$-direction, 
  we   restrict ourselves to vector fields $\vec X$ of the form
  \begin{equation}
  \label{filido}
  \vec X (x_1, x_2, x_3)=-\zeta (x_1, x_2, x_3)\, \be_3,
  \end{equation}
 where $\zeta: \R^3 \to \R$ is a given non-negative  function on $\R^3$.
The related flow $\Phi$ can then be integrated as 
\begin{equation}
\label{porcului}
\Phi(x_1, x_2, x_3, t)=(x_1,x_2, \Psi(x_1,x_2,x_3, t)),
\end{equation}
where the scalar function $\Psi$ solves the ODE with respect to $t$
\begin{equation}
\label{porculuibus}
\frac{d}{dt} \Psi(x_1, x_2, x_3, t)=-\zeta(x_1,x_2,\Psi(x_1,x_2, x_3,t)).
\end{equation}  
 It follows directly  from \eqref{flow}  that
\begin{equation}
\label{differentio0}
\left \vert \frac{\partial \Phi} {\partial t} \right \vert \leq \Vert \vec X \Vert_{L^\infty (\R^3)}= \Vert \zeta \Vert_{L^\infty (\R^3)}.
 \end{equation}
 Differentiating \eqref{flow} with respect  to the variables  $x_i$, for $i=1,2,3$ we are led to 
 \begin{equation}
 \label{differentio}
 \left \vert \frac{\partial } {\partial t} [\vert   \nablatrois \Phi \vert^2 ]   \right  \vert \leq  C\Vert \nablatrois \zeta_3 \Vert_{L^\infty (\R^3)} 
 \left( \vert   \nablatrois \Phi \vert^2+1\right),  
 \end{equation}
 where $C>0$ denotes some universal constant, and where $\nablatrois$ represents the gradient with respect to the first three spatial coordinates, that is 
 \begin{equation}
 \nablatrois \Phi =(\frac{\partial \Phi}{\partial x_1}, \frac{\partial \Phi}{\partial x_2}, \frac{\partial \Phi}{\partial x_3}), {\rm \ whereas \ }
  \nabla_{_4}  \Phi =(\frac{\partial \Phi}{\partial x_1}, \frac{\partial \Phi}{\partial x_2}, \frac{\partial \Phi}{\partial x_3}, \frac{\partial \Phi}{\partial t}). 
 \end{equation}
 Integrating the differential inequality \eqref{differentio},   we obtain  the exponential bound 
   \begin{equation}
 \label{differentio2}
  \vert   \nablatrois \Phi \vert^2 (\cdot, t)    \leq  C \exp  \left (C\Vert \nablatrois \zeta \Vert_{L^\infty (\R^3)} t \right),
  {\rm \ for \ }  t\geq 0.
  \end{equation}
    In the case $\zeta$ does not depend on $x_3$,    the integration of the vector field  $\vec X$ given by \eqref{filido} is straightforward and yields
 \begin{equation}
 \label{simplet}
  \Phi (x_1,x_2, x_3,  t)=(x_1, x_2, x_3-\zeta(x_1, x_2) t). 
\end{equation}   

\bigskip
We  introduce a deformation  operator $\mathcal P_\zeta$ which relates to 
 an  arbitrary  map $v: \R^3 \to \R^\ell$ and $t\geq 0$ the map $\mathcal P_\zeta (v)(t)$ defined on $\R^3$ by the formula, for   
 $(x_1, x_2, x_3) \in \R^3$ and $t \in \R$
  \begin{equation}
 \label{defazeta}
 \mathcal P_\zeta (v)(t) (x_1, x_2, x_3)= v(\Phi^{-1} (x_1, x_2,x_3,  t))= v(\Phi (x_1, x_2,x_3, - t)).
 \end{equation}
In some places, we  use the simpler notation 
 $v_t(\cdot)=\mathcal P_\zeta (v)(t)(\cdot),$
 when this is not ambiguous. 
  In the  special case  the function $\zeta$ does not depend on $x_3$,  we   have, in view of \eqref{simplet},  
  $$v_t=v(x+\zeta (x_1, x_2) t \, \be_3).$$ 
 As a  direct consequence of the chain rule and estimates \eqref{differentio0} and   \eqref{differentio2}, we obtain: 
 
\begin{lemma} 
\label{difference}
 Assume that $v$ and $\zeta$ are differentiable.   Then we have for $x \in \R^3$ and $t \geq 0$
  \begin{equation*}
  \label{zrc}
  \left\{
  \begin{aligned}
 \vert \frac{\partial}{\partial t}  \mathcal P_\zeta (v)(t) (x)\vert 
&\leq
  C\Vert \nablatrois  v   \Vert_{L^\infty (\R^3)}   \Vert \zeta \Vert_{L^\infty (\R^3)}  \\
 \vert \nablatrois \mathcal P_\zeta (v)(t)(x) \vert 
 & \leq C\Vert \nablatrois  v   \Vert_{L^\infty (\R^3)} \exp  \left (C\Vert \nablatrois \zeta \Vert_{L^\infty (\R^3)} t \right),
 \end{aligned}
 \right.
  \end{equation*}
  where $C>0$ is some  universal constant.
 \end{lemma}


 We will next be even more specific and describe  the two different kinds of vertical vector fields which are used in the construction of  the Gordian cut.
  
  \smallskip
  \noindent
  {\bf  Constant vertical vector-fields}. We consider the constant vector field  $\vec X_0=-\be_3$ related by \eqref{filido} to 
\begin{equation}
\label{zeta0}
\zeta_{_0}(x_1, x_2, x_3)=1,\ \,  \forall (x_1, x_2, x_3) \in \R^3.
  \end{equation}
  Since $\vert \nabla \zeta_{_0} \vert=0$, we obtain, if $\Phi_0$ is the flow related to $\vec X_0$, 
  $\displaystyle{\Phi_0(x_1, x_2, x_3,t)=(x_1, x_2, x_3-t)}$ so that
  $$
 \vert  \partial_t \Phi_0  \vert \leq 1{\rm \ and \ }   \vert \nablatrois \Phi_0(\cdot, t) \vert \leq {\rm C}_{0}, 
  $$
   where ${\rm C}_{0}$ is some constant. In this case, the map $\mathcal P_{\zeta_0} (v)(t)$ has a simple form, since 
   $$\mathcal P_{\zeta_0} (v)(t)(x)= v(x_1, x_2, x_3+t), \forall x=(x_1, x_2, x_3)\in \R^3.$$
   It follows from Lemma \ref{difference} or computing directly using the chain rule  that, for any $t\geq 0$,
   \begin{equation}
   \label{differo2}
  \Vert \frac{\partial}{\partial t}\mathcal P_{\zeta_0} (v)(t)  \Vert_{L^\infty(\R^3)}+   
  \Vert \nablatrois  \mathcal P_{\zeta_0} (v)(t) \Vert_{L^\infty(\R^3)} 
  \leq {\rm K}_{0}\Vert \nablatrois  v   \Vert_{L^\infty (\R^3)}. 
\end{equation}
where ${\rm K}_{0}$ is some  absolute constant.  We will in some places also consider the vector field $\frac 12   \vec X_0$.  Obviously, all estimates remain valid for this  vector field. 

\begin{remark}
\label{constantspag}
{\rm   In the course of the proof of Proposition \ref{deform}, we will be led to transport the Pontryagin maps of the fibers $\fLkperp_{i, q}$. We have, 
for $t\geq 0$ and $x \in \R^3$,
$$ \mathcal P_{\zeta_0}\left(\pont_\varrho \left[\fLkperp_{i, q}, \frameref\right] \right)(t)(x)=\pont_\varrho [\fLkperp_{i, q}-t \be_3, \frameref](x)  {\rm \ and \ hence \ }
$$
\begin{equation}
\label{constantgord}
\Vert \frac{\partial}{\partial s} \pont_\varrho \left[\fLkperp_{i, q}-s \be_3, \frameref \right]  \Vert_{_{L^\infty(\R^3)}}+   
  \Vert \nablatrois \pont_\varrho \left[\fLkperp_{i, q}-s \be_3, \frameref \right] \Vert_{_{L^\infty(\R^3)}} 
  \leq {\rm C}_{\rm flow}^0 k, 
\end{equation}
where ${\rm C}_{\rm flow}^0$ is some universal constant. In particular, considering the map  on $\R^4$ defined by 
 $(x, s)\mapsto \pont_\varrho \left[\fLkperp_{i, q}-s\be_3\right](x)$, we obtain the gradient estimate 
\begin{equation}
\label{constantgord2}
\vert \nabla_{_4} \left(\pont_\varrho\left[\fLkperp_{i, q}-s\be_3\right] (x)\right)\vert \leq  {\rm C}_{\rm flow}^0 k , \,  {\rm \ for \ } (x, s)   \in \R^4.
\end{equation}

}
\end{remark}

\medskip
\noindent
{\bf The  "rounding"  vector-field $\vec X_1^k$}. We will use this vector-field  when we want to push down the fibers of $\fLkperp$ and avoid at the same time that they cross  the fibers of $\fLk$ which are on their way if the vector field were kept constant. To model the points to be avoided,  we consider, for  $k \in \N^*$,   the numbers 
$x_2^\ell$ defined by 
$$
x_2^\ell=\ell h {\rm \ \ for \ } \ell=1, \ldots, k, {\rm \ so \ that \ } 0<x_2^1=h <\ldots <x_2^k=1,  {\rm \ with \ }  h=k^{-1}.
$$
 The collection of  segments $[-6, 6]\times \{x_2^\ell\}$, $\ell=1, \ldots, k$ corresponds indeed   to projections onto the plane $P_{1,2}$ of  the straight segments  of the sheave $\fLk$  which lie below $\fLkperp$ (see Figures \ref{above0} and \ref{side0}). 
 We construct a vector field $\vec X_1^k$ associated by \eqref{filido} to a push function $\zeta_1^k$  depending only on the last two variables, that is 
  $\zeta_1^k(x_1, x_2, x_3)=\zeta_1^k(x_2, x_3)$, $\forall (x_1, x_2, x_3)\in \R^3$.
     We  impose the following conditions on $\zeta_1^k$:
   \begin{equation}
   \label{raffinerie}
\left\{
\begin{aligned}
&\zeta_1^k\left (x_2,  x_3\right)=\frac 1 4  {\rm \ for  \  } x_2 \in \underset{\ell=1} {\overset {k} \cup}[x_2^\ell-\frac{h}{8}, x_2^\ell+\frac{h}{8}]    {\rm \ and  \ for \  } x_3 \geq 0,   \\
&  \zeta_1^k  ( x_2, x_3)=1
 {\rm \ for \ }  \,   x_2  \not \in \underset {\ell=1} {\overset {k}\cup} [x_2^\ell-\frac{h}{4}, x_2^\ell+\frac{h}{4}]  {\rm \ and \ } x_3\geq 0,\\
 &  \zeta_1^k  ( x_2, x_3)=1
 {\rm \ for \ }  \,  x_3 \leq -\frac{h}{2}, \\
 &  \frac 14 \leq  \zeta_1^k  ( x_2, x_3) \leq 1 {\rm \ for \  }  x_3\geq 0  {\rm \ and \   } \Vert \nablatrois   {\zeta_1^k} \Vert_\infty \leq   16h^{-1} . 
\end{aligned}
\right.
\end{equation}
It follows in particular  from  conditions  \eqref{raffinerie}  that 
\begin{equation}
\label{oh}
\zeta_1^k(x)=1, \  \  {\rm except \ possibly \ if \ } x  \in   \mathcal O_h \equiv \R\times [0, 1+\frac h 4]\times [-\frac h 2, +\infty).
\end{equation}
To construct the function $\zeta_1^k$, we  proceed as follows: We   choose $\zeta_1^k$ of the form
\begin{equation}
\label{defzeta1}
\zeta_1^k (x_2, x_3)=1 -\mathfrak f^k(x_2)g_3(x_3) {\rm \  with \ }
\mathfrak f^k(x_2)\equiv\underset{\ell=1} {\overset {k}\sum} g_2 \left(k\left(x_2-x_2^\ell\right)\right), 
\end{equation}
where  $g_2: \R \to \R$ denotes a given smooth non-negative function on $\R$  such that 
\begin{equation}
\label{defh2}
g_2(s)=0 {\rm \ for  \ } s \in \R\setminus  [-\frac 14, \frac 14], \, 
g_2(s) =\frac 3 4 {\rm \ for  \ } s \in [-\frac 1 8, \frac 1 8], \,   0\leq  g_2(s) \leq \frac 3 4\  {\rm \ otherwise, }  
\end{equation}
 and   where the function $g_{3}: \R \to \R$ denotes a  smooth non-negative function   such that  $0\leq g_{3} \leq 1$ and 
 \begin{equation}
 \label{adebattre}
 g_{3}(s)=1 {\rm  \ for \ }  s\geq 0 {\rm \ and \ }  g_{3}(s)=0 {\rm  \ for \ } s \leq -\frac 12. 
 \end{equation}
We may assume additionally that $\Vert g_2' \Vert_\infty \leq 10 h^{-1}$ and $\Vert g_3' \Vert_\infty \leq 4$. In view of \eqref{defh2}, we have
\begin{equation*}
\left\{
\begin{aligned}
 \mathfrak f^k(s)&=\underset{\ell=1} {\overset {k}\sum} g_2 \left(k\left(s-x_2^\ell\right)\right)=0,  {\rm  \ for \  } s \in 
\R\setminus      \underset {\ell=1} {\overset {k}\cup} [x_2^\ell-\frac h4, x_2^\ell+\frac h4], \\
 \mathfrak f^k(s)&=\frac 34, {\rm \ for \ } s \in \underset {\ell=1} {\overset {k}\cup} [x_2^\ell-\frac h8, x_2^\ell+\frac h8], 
\end{aligned}
\right.
\end{equation*}
 so that the conclusion \eqref{raffinerie} follows. Definition \eqref{defzeta1} yields the estimate 
\begin{equation*}
\label{lepou}
\Vert \zeta_1^k \Vert _\infty \leq C {\rm \ and \ } \Vert \nabla \zeta_1^k \Vert_\infty \leq Ck,  
\end{equation*}
 for some universal constant $C>0$. We have therefore
   $$  \exp  \left (  s \,\Vert \nabla \zeta_1^k \Vert_{L^\infty (\R^3)}  \right) \leq K,  {\rm \ for  \  }  s\in [0, h], $$ 
where  $K>0$ is some universal constant. It  follows  from  Lemma \ref{difference} that 
 \begin{equation}
  \label{zrcz}
  \vert \frac{\partial}{\partial t}  \mathcal P_{\zeta_1^k} (v)(s) (x)\vert +  \vert \nablatrois \mathcal P_{\zeta_1^k} (v)(s) (x)  \vert 
     \leq {\rm K_1} \Vert \nabla  v  \Vert_{L^\infty (\R^3)}, {\rm \ for \ } x \in \R^3 {\rm \ and \ } s \in [0, h], 
      \end{equation}
where ${\rm K}_1\geq {\rm K_0}>0$ is some universal constant. In view of  the simple form \eqref{filido}-\eqref{defzeta1} of the vector field $\vec X_1^k$, its integration reduced to the integration of the scalar differential equation \eqref{porculuibus},  which can be solved by  \emph{separation of variables}.  Going back to \eqref{porcului} and writing  $\Psi(x_1,x_2,x_3,s)=\Psi_1^k(x_2,x_3, s)$ for our specific choice \eqref{raffinerie} of vector-field, 
we verify that the function $\Psi_1^k$ is given as the solution of the integral equation 
$$\int_{\Psi_1^k(x_2, x_3,s)}^{x_3} \frac{{\rm d}u}{1-\mathfrak f^k_1(x_2)g_3(u)}=s.$$
It follows from this formula that $\Psi_k(x_2, x_3) \leq x_3$, and that 
\begin{equation}
\label{integrozeta2}
\Psi_1^k (x_2, x_3, s)=x_3-\left[s- s\, \mathfrak f^k_1(x_2)\right],  {\rm \ provided \ }
  0\leq s \leq x_3. 
\end{equation}

\medskip
\noindent
{\it Transportation of curves by  the flow of $\vec X_1^k$}. We   next  look  at the fate of a  curve when transported by the flow $\Phi_1^k$ of the vector field $\vec X_1^k$. Of special interest is the  fate of the fibers of the sheaf 
$\fLkperp$.  In view \eqref{oh}, all part of the  fibers which are not in $\mathcal O_h$ (defined in  \eqref{oh}) are transported downwards along  $\be_3$ with constant speed $1$.  Since the restrictions of the fibers of $\fLkperp$ to $\mathcal O_h$ are segments parallel to $\be_2$, we  first consider  a line of the form  $D=M+\R\be_2$, where $M=(m_1, 0, m_3)$ is given, with $m_3\geq 0$. Thanks to \eqref{porcului} and \eqref{integrozeta2}, we obtain
\begin{equation}
\label{phika}
\Phi_1^k (D, s)= \{(m_1, x_2, m_3-\left[s- s\, \mathfrak f^k_1(x_2)\right]) {\rm \ for \ } x_2\in \R\}, {\rm \ for \ } 0\leq s\leq m_3.
\end{equation}
Assuming  $m_1=0$ and    restricting ourselves to the plane $P_{2,3}$, the curve $\Phi_1^k (D, s)$ corresponds to the graph of the function $x_2 \mapsto m_3-\left[s- s\, \mathfrak f^k_1(x_2)\right]$ (see figure \ref{oscillo0}).   In the course of the proof of Proposition \ref{deform}, we will use formula \eqref{phika} for the special choice $\displaystyle{m_3\geq s=h}$, so that 
$$
\Phi_1^k (D, h)= \{(m_1, x_2, m_3-\left[h-  h \, \mathfrak f^k_1(x_2)\right] ) {\rm \ for \ } x_2\in \R\}, 
$$
and hence 
\begin{equation}
\left\{
\begin{aligned}
\Phi_1^k (m_1, x_2, m_3,  h )&=(m_1, x_2, m_3-\frac{h}{4}) {\rm \ for \  } x_2\in \underset{\ell=1} {\overset{k} \cup }
 [x_2^\ell-\frac h 8, x_2^\ell+\frac h 8],  \\
\Phi_1^k (m_1, x_2, m_3,  h )&=(m_1, x_2, m_3-h)  {\rm \ for \  } x_2\in  \R \setminus  
\underset{\ell=1} {\overset{k} \cup }[x_2^\ell-\frac h 4, x_2^\ell+\frac h 4].
\end{aligned}
\right.
\end{equation}
\begin{figure}[h]
\centering
\includegraphics[height=9cm]{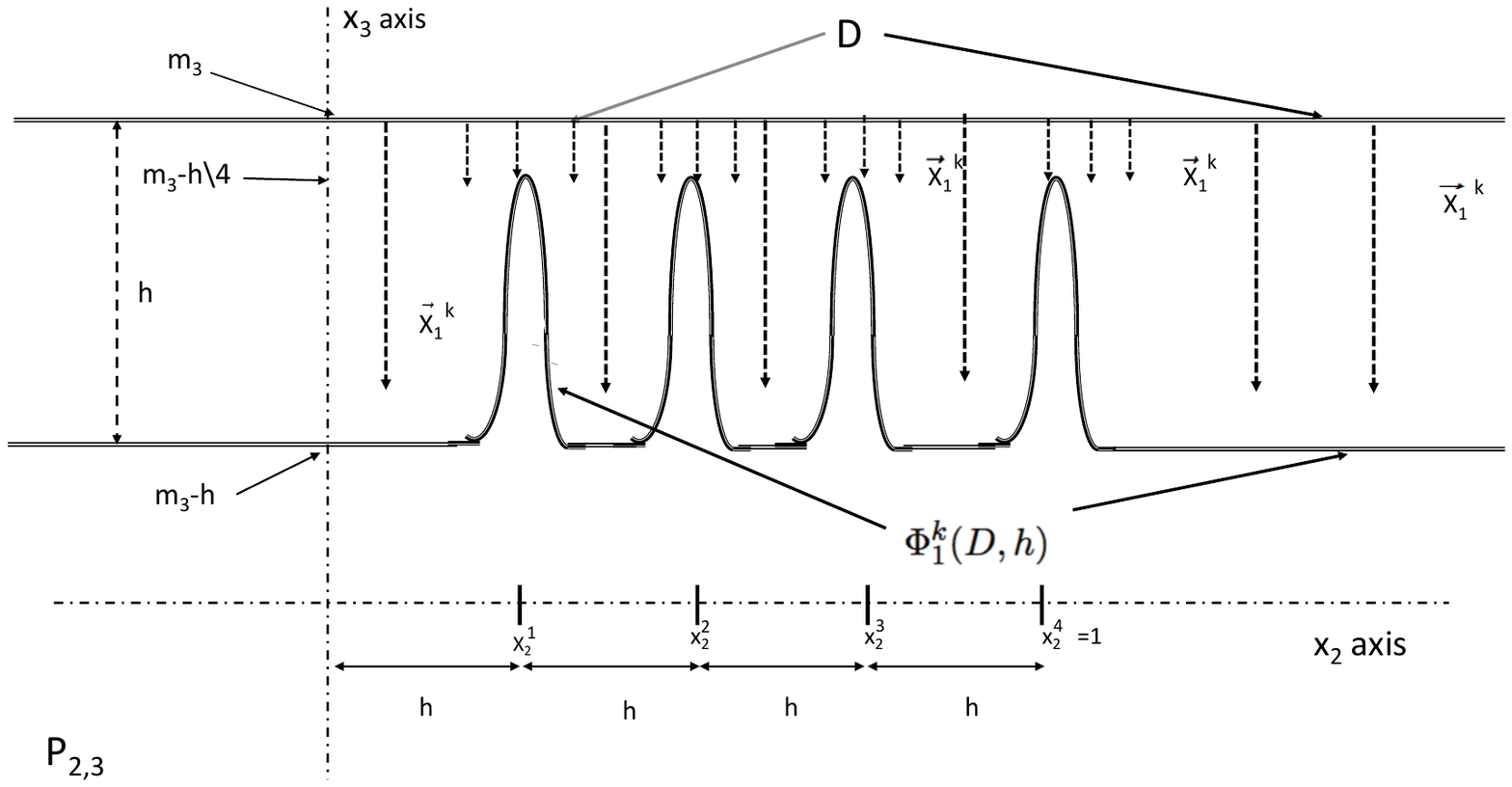}
\caption{  {\it The deformation of a line parallel to the $x_2$ axis by the flow generated by $X_1^k$, $k=4$ at time $h$.
 }}
\label{oscillo0}
\end{figure}

\smallskip
\noindent
{\it Transportation of translates of stadia $\fLkperp_{i, q}$  by  the flow of $\vec X_1^k$}.   As mentioned, the vector field $\vec X_1^k$ will be used to transport vertical translates of the stadia $\fLkperp_{i, q}$, so that, for arbitrary $i, q=1, \ldots, k$, and ${\rm c}>0$, we consider the  curve $\fLkperp_{i, q}-{\rm c} \, \be_3$  and  its deformation  $\DefLkperp_{i, q} (c, s)$ by the flow $\Phi_1^k$, given by
\begin{equation}
\label{deformlkperp}
\DefLkperp_{i, q} (c, s) \equiv \Phi_1^k (\fLkperp_{i, q}-{\rm c} \, \be_3, s), {\rm \ for \ } s\in [0, h].
\end{equation}
 \begin{figure}[h]
\centering
\includegraphics[height=8.5cm]{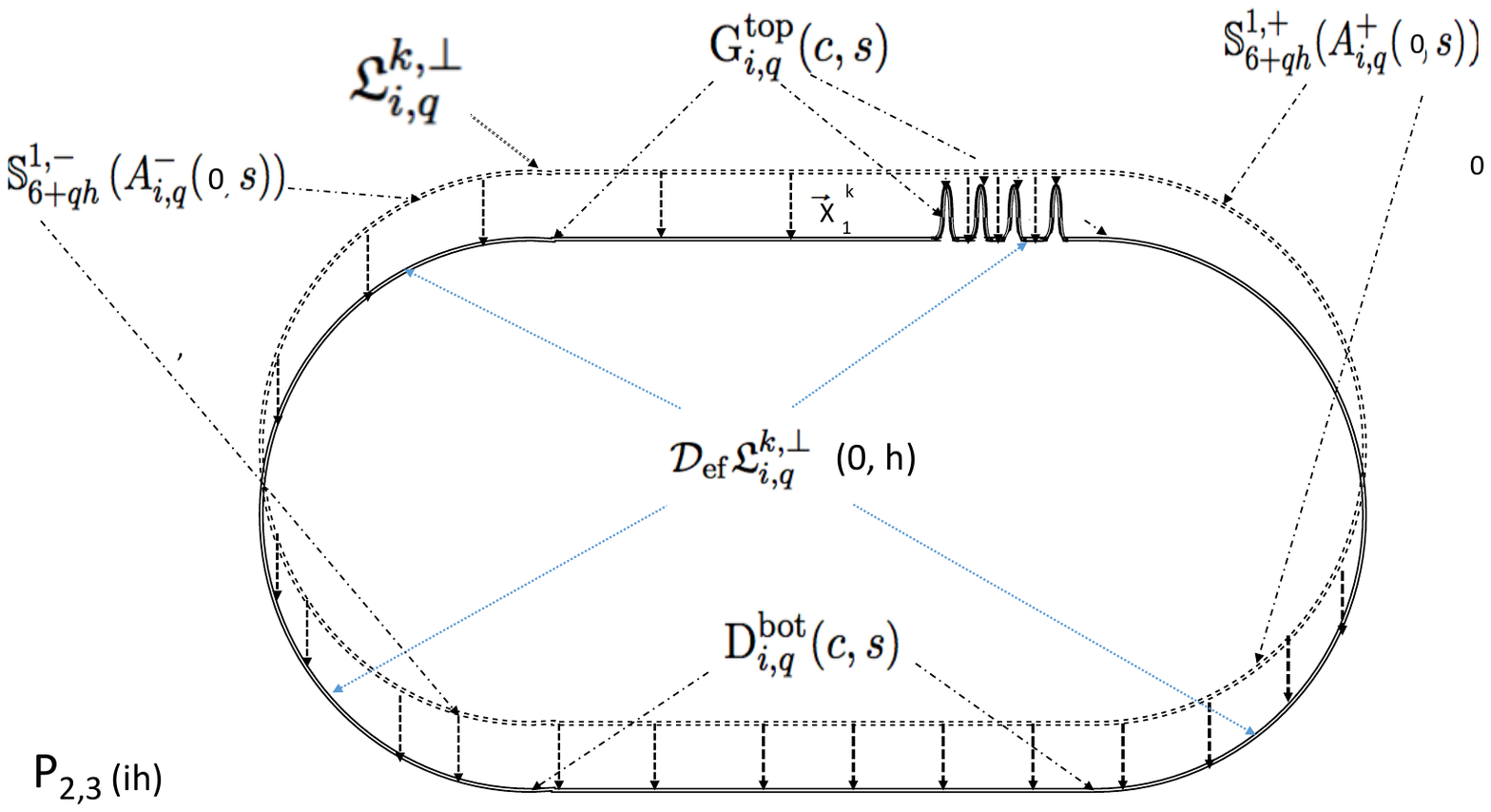}
\caption{  {\it The deformation of the   curve $\fLkperp_{i, q}$  by the flow $\Phi_1^k$, $k=4$ at time $h$.}}
\label{deflperp}
\end{figure}
 We are only be interested in the case $0\leq  c \leq 6+ qh$.  The shape of the curves $\fLkperp_{i, q}$ and its deformation  $\DefLkperp_{i, q} (0, h)$  is represented in figure \ref{deflperp} (i.e. for $c=0$).   An analytical description is provided  by  the decomposition in the plane $P_{2,3}(ih)$
 \begin{equation}
 \label{detritus}
\DefLkperp_{i, q} (c, s)= \S_{6+qh}^{1, -} \, (A_{i, q}^-(c, s))\cup  \S_{6+qh}^{1, +}(A_{i, q}^+(c, s))
 \cup {\rm G}_{i, q}^{\rm top}(c, s)\cup {\rm D}_{i, q}^{\rm bot}(c, s), 
 \end{equation}
  where $A_{i, q}^-(c, s)=(ih,-3+qh,-c-sh)$, $A_{i, q}^-(c,s)=(ih, 9+qh,-c-sh)$,  where ${\rm D}_{i, q}^{\rm bot}$ denotes the segment at the bottom of  $\DefLkperp_{i, q} (c, s)$ parallel to $\be_2$, that is  
  \begin{equation*}
  \left\{
  \begin{aligned}
  {\rm D}_{i, q}^{\rm bot}(c,s)&=ih\be_1+[(-3+qh)\be_2, (9+qh)\be_2]-(6-qh+c+s)\be_3\\
  &=[B_{i, q}^-(c, s), B_{i, q}^+(c, s)],
  {\rm \ with \ } \\
  B_{i, q}^-(c, s)&= (ih, -3+qh,-6+qh-c-s)) 
  {\rm \ and \ }
  B_{i, q}^+(c, s)=(ih, 9+qh, -6+qh-c-s)), 
  \end{aligned}
  \right.
  \end{equation*}
  and where the set  ${\rm G}_{i, q}^{\rm top}(c, s)$ has the form of a graph  in the plane $P_{2, 3}(h)$, namely
  \begin{equation}
  \label{phikaton2}
   {\rm G}_{i, q}^{\rm top}(c, s)=\left\{\left(ih, x_2, 6+qh-c-\left[s- s\,  \mathfrak f^k_1(x_2)\right] \right), x_2 \in [-3+qh, 9+qh ] \right\}.
  \end{equation}

\begin{remark}
\label{proche}
{\rm  The set $\DefLkperp_{i, q} (c, s)$ is  a "perturbation" of the set   $\fLkperp_{i, q}-(c+s)\be_3$ in view of the relation 
\begin{equation*}
 \DefLkperp_{i, q} (c, s)\setminus \left( \fLkperp_{i, q}-(c+s)\be_3\right) 
 \subset
 \left \{\left(ih, x_2, 6+qh-c-s+ s\,  \mathfrak f^k_1(x_2) \right), x_2 \in [0, 1+\frac h 4]\right\}.
\end{equation*}
We have hence
\begin{equation}
\label{ise}
\DefLkperp_{i, q} (c, s)\subset \left[ \fLkperp_{i, q}-(c+s)\be_3\right] \cup V(c, s),  {\rm \ for \ } 0\leq s \leq h,
\end{equation}
 where $V(c, s)$ denotes, for $0\leq s \leq h$,  the parallelepipedic region 
 \begin{equation}
 \label{vcs}
 \begin{aligned}
 V(c, s)&=-(c+s)\be_3+\frac h2 \be_2+\left([0, 1]^2 \times [0, \frac{3s}{4}]\right) \\
 &\subset V(c, h)=
 -(c+h)\be_3+\frac h2 \be_2+\left([0, 1]^2 \times [0, \frac{3h}{4}]\right). 
\end{aligned}
 \end{equation}
}
\end{remark}

 \begin{remark}
 \label{disfibre}
 {\rm As  direct consequence of the inclusions \eqref{ise} and \eqref{vcs}, we  check that, given two distinct sets of indices ${i, q}\not = (i', q')$,  we have
 \begin{equation}
 \label{disfibre2}
 \left\{
 \begin{aligned}
&\left(  \DefLkperp_{i, q} (c, s)+\B_{h\slash  8} ^3(0)\right)  \cap \DefLkperp_{i', q'} (c, s)=\emptyset  {\rm \  and \  }  \\
&\left(  \DefLkperp_{i, q} (c, s)+\B_{h\slash 8}^3(0)\right)  \cap\left( \fLkperp_{i', q'}-(c+h) \be_3\right)=\emptyset,
{\rm \ for \ any \ } 
 s\in [0, h]  {\rm \ and \ }  c>0.
\end{aligned}
\right.
 \end{equation}
Moreover, if condition \eqref{ondiment}, which will be introduced later,  is satisfied, then we have 
\begin{equation}
 \label{disfibre20}
\left(  \DefLkperp_{i, q} (c, s)+\B_{h\slash  32} ^3(0)\right)  \cap \fLk_{i', q'} (c, s)=\emptyset,
{\rm \ for \ any \ }  s\in [0, h].
 \end{equation}
 }
 \end{remark}

\subsubsection{A variant of the Pontryagin construction}
\label{variante}
  Whereas deformations of the domain act both on curves and functions,  we have seen in \eqref{pushcte}   that they  not "commute" in general with the Pontryagin construction, that is the  Pontryagin map of a deformed curve is not in general the deformation of the initial Pontryagin map. Concerning the curves   of interest for us, namely the curves $\DefLkperp_{i, q} (c, s)$, we taylor a specific variant of the Pontryagin  construction for  later use.  Given $\varrho>0$,  $i, q=1, \ldots, k$, $0\leq  c \leq 6+ qh$ and $s>0$, our variant $ \tpont_\varrho  [ \DefLkperp_{i, q} (c, s)]: \R^3 \to \S^2$ will  by different from the Pontryagin map only in a neighborhood of the top  part  $    {\rm G}_{i, q}^{\rm top}(c, s)$. We introduce therefore the set 
  $$
   {\rm U}_{i, q}^{\rm top}(c, s, \varrho)=
   \underset {a \in {\rm G}_{i, q}^{\rm top}(c, s)} 
    \cup \D_{1,3}^2 (\varrho, a)
  $$
  where, for $a=(a_1, a_2, a_3)\in \R^3$,   $\D_{1,3}^2 (\varrho, a)$ denotes the disk in the plane $P_{1, 3}(a_2)$ of radius $\varrho$ centered  at $a$, namely
\begin{equation*}
 \D_{1,3}^2 (\varrho, a)=\{(x_1, a_2, x_3)\in \R^3, (x_1-a_1)^2+(x_2-a_2)^2 <\varrho^2\} \subset P_{1,3}(a_2).
 \end{equation*}
  We  define the variant   $\tpont_\varrho  [ \DefLkperp_{i, q} (c, s), \, \frameref]$ of the Pontryagin map as follows: We set  
  \begin{equation}
   \label{varpont}
 \tpont_\varrho  [ \DefLkperp_{i, q} (c, s)] (x)=\pont_\varrho  [ \DefLkperp_{i, q} (c, s)] (x),  {\rm \ for \ } x \not \in  {\rm U}_{i, q}^{\rm top}(c, s, \varrho) 
 \end{equation}
 Otherwise, if  $x \in \D_{1,3}^2 (\varrho, a)$   for some  	$a =(a_1, a_2, a_3)\in  {\rm G}_{i, q}^{\rm top}(c, s)$,  we set 
\begin{equation}
\label{varpont1}
\tpont_\varrho  [ \DefLkperp_{i, q} (c, s)] (x_1, a_2, x_3)=\chi_\varrho \left(x_1-a_1, x_3-a_3\right),
 \end{equation}
 where $\chi_\varrho$ is defined in \eqref{defchi1} and \eqref{blablanabla}. 
In other words,  in the construction of  $\tpont_\varrho  [ \DefLkperp_{i, q} (c, s)]$, we replace the  plane  orthogonal  to the curve by the plane parallel to $P_{1, 3}$\footnote{These two planes coincide if $\mg'(x_2)=0$.} on the part   ${\rm G}_{i, q}^{\rm top}(c, s)$.  We are going to rely on the following:

\begin{lemma}  We have,  provided  $\DefLkperp_{i, q} (c, s) \subset \R^2 \times \R^+$
$$ \tpont_\varrho  [ \DefLkperp_{i, q} (c, s)]=\mathcal P_{\zeta_1^k}( \pont_\varrho  [ \fLkperp_{i, q} -c\be_3, \frameref])(s), {\rm \ for \ } s \in [0, h].$$
\end{lemma}
\noindent
{\it Sketch of the proof}.
The proof follows from the observation that the flows $\Phi_1^k$ transforms disks  in planes parallel to $P_{1,3}$ into disks of the same radius, namely
$$\Phi_1^k(\D_{1,3}^2 (\varrho, a),s)=\D_{1,3}^2 (\Phi_1^k(a,s), \varrho), $$
 a consequence of the fact that $\zeta_1^k$ depends only on the variable $x_2$ in the region considered and the fact that we consider only disks in planes orthogonal to $\be_2$.

\medskip
\noindent
Concerning gradient estimates, we have:

\begin{lemma} 
\label{gradur}
 We have,  for some constant ${\rm C}_{\rm def}>0$ and  for any $c\in \R^+$ and $s\in [0, h]$
$$
\vert \nabla_3  \tpont \left [\DefLkperp_{i, q} (c, s)\right] \vert + 
\vert \frac{\partial}{\partial s}\tpont\left [ \DefLkperp_{i, q} (c, s)\right] \vert \leq {\rm C}_{\rm def} k.
$$
In particular, for fixed $c \geq 0$, considering the map $(x, s)\mapsto \DefLkperp_{i, q} (c, s)(x)$, we obtain 
\begin{equation}
\label{gradur1}
\vert \nabla_4\tpont\left [ \DefLkperp_{i, q}\right] (c, s)(x)\vert \leq  {\rm C}_{\rm def} k,  {\rm \ for \ } (x, s)   \in \R^3 \times [0, h].
\end{equation}
\end{lemma}
One may deduce these estimates from \eqref{zrcz} or one  might proven  them directly.

\begin{remark}
\label{toutcomme}
{\rm It follows from \eqref{disfibre2} in Remark \ref{disfibre} that the respective   sets of points in $\R^3$ where  the maps $\tpont \left [\DefLkperp_{i, q} (c, s)\right]$,  $\tpont \left [\DefLkperp_{i', q'} (c, s)\right]$  and $\pont[\fLk]$   are different from $\sP$  are disjoint sets  for any $s\in [0, h]$, provided condition  \eqref{ondiment} is fulfilled. The same holds for the maps $\tpont \left [\DefLkperp_{i', q'} (c, s)\right]$, $\pont [\fLk_{i, q'}-(c+s)\be_3]$ and  $\pont[\fLk]$, provided condition \eqref{ondiment} is fulfilled.

} \end{remark}

\subsubsection{Cubic extensions}
\label{cubic}
Whereas the previous  construction  works for quite general classes of maps and is hence not specific to the Sobolev framework, the extension method presented here induces singularities and hence is specially appropriate in  the Sobolev setting.   

\medskip
\noindent
{\it The cube  $\rQ_r^4(\rba)$ and its boundary}. We consider the $\infty$-norm on $\R^4$  given by
$$ \vert \rbx \vert_\infty= \underset{i=1, \ldots, 4}{ \overset {}\sup} \vert x_1 \vert,  {\rm \ for \ } \rbx=(x_1, x_2, x_3, x_4) \in \R^4,  
$$
 and   the corresponding $\infty$-ball  ${\rm  Q}^4_r$ of radius  $r>0$ defined by
 $$
 \rQ_r^4=\rQ_r^4(0),  {\rm \ where \ more \ generally \  } \rQ_r^4(\rba)\equiv \{ \rbx \in \R^4, \vert x-\rba  \vert_\infty < r \} {\rm \ for \ } \rba \in \R^4, 
 $$
 so that  actually $\rQ_r^4$ corresponds  the hypercube $\rQ_r^4=[-r, r]^4$. 
Given a 4-dimensional hypercube $\rQ_r^4(\rba)$,  its boundary $\partial \rQ_r^4(\rba)$ is the union  of the closure of  $8$  distinct three-dimensional cubes $\mathfrak Q_{p}^{3,\pm} (r, \rba)$ of size $r$ defined, for  $\rba=(a_1, a_2, a_3, a_4)$ and $p=1, \ldots, 4$, by
\begin{equation}
\label{scarface}
\mathfrak Q^{3,\pm}_p(r, \rba)=\{\rbx=(x_1, x_2, x_3, x_4) \in \R^4, x_p=a_p\pm r, \,  \underset {i\neq  p} \sup\,  \vert x_i-a_i \vert < r\}.
\end{equation}
 The sets $\mathfrak Q^{3,\pm}_p(r, \rba)$ are therefore included in a 3-dimensional hyperplane of $\R^4$ orthogonal to the vecteur $\be_p$. We have  
\begin{equation}
\label{delpotro}
\partial \rQ_r^4(\rba)=\underset {p=1} {\overset {4}\cup} \left( \overline{ \mathfrak Q^{3,+}_p(r, \rba)} \cup 
\overline{\mathfrak Q^{3,-}_p(r, \rba)} \right).
\end{equation}
     \medskip
   \noindent
  {\it Construction of the extension operator.}  
Given a map $v:\partial \rQ_r^4( a)\to \R^\ell$ defined on   the boundary $\partial \rQ_r^4(\rba)$ of a cube $\rQ_r^4(\rba)$ we consider its cubic-radial extension   $\Ext_{r, \rba} (v)$ defined on the full cube $\rQ_{r}^4(\rba)$  for   $v: \partial \rQ_r^4( \rba)\to \R^\ell$  by 
 \begin{equation}
  \Ext_{r, \rba} (v)(\rbx)=v \left(\rba+ r\frac{\rbx-\rba\, \, \, \, }{\vert \rbx-\rba \vert_{\infty} } \right), {\rm \ for \ } \rbx \in \rQ_{r}^4 (\rba), 
 \end{equation}
 so that 
 $ \displaystyle{
 \Ext_{r, \rba} (v)=v {\rm \ on  \  the \  boundary  \ } \partial \rQ_r^4(\rba)}.
 $
If $v$ is Lipschitz, then $\Ext_{r, \rba} (v)$ is locally Lipschitz on $\rQ_r^4(\rba)\setminus \{\rba\}$, but possesses a singularity at the point  $\rba$, except if $v$ is constant. However,  if the map $v$ has finite energy $\rE_3$  on the three-dimensional set  $\partial \rQ_r^4(\rba)$, then  the same assertion holds for its extension  
 $\Ext_{r, \rba}(v)$ on the cube $\rQ_r^4(\rba)$ with the estimate
 \begin{equation} 
 \label{enextension}
 \rE_3 \left(\Ext_{r, \rba}(v), \,  \rQ_r^4(\rba)\right) \leq  {\rm K _{\rm ext}} \, r \, \rE_3(v, \partial \rQ_r^4(\rba)),
\end{equation}
where $ {\rm K _{\rm ext}}$  denotes some universal constant. 
\subsubsection {Deforming topologically  trivial maps  to constant maps }
\label{trivial}
  We  assume  that we are given a map $w\in {\rm Lip} \cap \left( W^{1, 3}(\R^3, \S^2)+ \{\sP\}\right) $ such that we have 
  $$  w(x)=\sP {\rm \ for \ }  x \in \R^3 \setminus [-R, R]^3,  {\rm \ for \ some \ } R>0. $$
  Hence, we may  define the $\rH(w)$.   We have (see e.g.  \cite {BeZ, Be2} for related constructions):
  
  \begin{proposition}
  \label{trivialextend}
   Let $w$ and $R>0$ be as above and assume that $\rH (w)=0$.   There exists a map 
  $W \in  C^0\cap W^{1, 3}([-R, R]^4,  \S^2)$ such that the following holds: 
  \begin{itemize}
  \item  $W(x, -R)=w(x)$, for $x\in [-R, R]^3$,
  \item  $W(x, R)=\sP$, for $x\in [-R, R]^3$,
  \item $W(x, s)=\sP$, for $x\in \partial([-R, R]^3)$ and $s\in [-R, R]$,
  \item  $ \rE_3(W, [-R, R]^4)  \leq 2{\rm C _{\rm ext}} \, R \, \rE_3(v, [-R, R]^3).$
  \end{itemize}
  \end{proposition}
\begin{proof}   We consider the   continuous map $\tilde w$ from  the boundary $\partial ([-R, R]^4)$ to $\S^2$ defined by
\begin{equation*}
\left\{
\begin{aligned}
&\tilde w(x, -R)=w(x), {\rm \ for \  } x \in [-R, R]^3  {\rm \ and \ }   \\
&\tilde w(\rbx)=\sP, {\rm \ for \ }  \rbx\in \partial ([-R, R]^4)\setminus [-R, R]^3 \times \{-R \}, 
\end{aligned}
\right.
\end{equation*}
so that $\tilde w$ is Lipschitz and  the homotopy class of $\tilde w$ is trivial. There exists  therefore a Lipschitz map  
$\varphi:  [-R, R]^4 \to  \S^2$ such that 
$$ \varphi(\rbx)=\tilde w(\rbx), {\rm \ for \ } \rbx \in \partial ([-R, R]^4).$$
  Since $\varphi$ is Lipschitz, we have
  $$ I_1\equiv  \int_{[-R, R]^4} \vert \nabla_{_4} \varphi \vert^3 <+\infty.$$ 
 Let $0<\rho \leq R$ be such that 
$\rho I_1 \leq {\rm C _{\rm ext}} \, R^2 \rE_3(v, [-R, R]^3)$.  We define $W$ as 
\begin{equation*}
\left\{
\begin{aligned}
 W(\rbx)&=\tilde w\left(\frac{  R\, \rbx}{\ \, \vert  x \vert_\infty}\right) {\rm  \ if \ }  \vert \rbx \vert_\infty \geq \rho  \\
 W(\rbx)&=\varphi \left(\frac{R\rbx}{\rho}\right) {\rm  \ if \ }  \vert \rbx \vert_\infty \leq \rho, 
 \end{aligned}
\right.
\end{equation*}
so that $W$ satisfies the three first condition in Proposition \ref{trivialextend}. For the energy estimate, we observe that, by \eqref{enextension}, we have
$$\int_{ \vert \rbx \vert_\infty \geq \rho} \vert \nablaq W \vert^3 \leq {\rm C _{\rm ext}} \, R \rE_3(w, [-R, R]^3).$$
On the other hand, by scaling we have 
$\displaystyle{\rE_3(W, [-\rho, \rho]^4)=\frac{\rho}{R} \rE_3(\varphi, [-R, R]^4)=\frac{\rho I_1}{R}.
}$
The conclusion follows combining the previous estimates.
 \end{proof}
 
\subsubsection{Adjusting boundary values}
The next result, which is of  technical nature, will be used only  in Subsection  \ref{okdac}, for the proof of Proposition \ref{unpoco}. The proof strongly relies on ideas introduced in \cite{SU} and \cite{HL}.

\begin{lemma}
\label{epsilonerie} 
 Let $v \in C^\infty (\overline{\B^4}, \S^2)$ be such that 
\begin{equation}
\label{petitude}
\int_{\B^4\setminus \B^4(3\slash4)} \vert v(\rbx-\sP) \vert^3 \rd \rbx \cdot \rE_3(v, \B^4) \leq \epsilon_0, 
\end{equation} 
where $\epsilon_0>0$ is some absolute constant. Then,  there exists some map $w\in  C^\infty (\overline{\B^4}, \S^2)$  such that $w=v$ on $\B^4(3\slash 4)$,  $w(\rbx)=\sP$ for  $\rbx \in \partial \B^4$, and  such that, for some absolute constant   ${\rm K}_{\rm dir } \geq 1$, we have the estimates 
\begin{equation}
\label{hautpuech}
\left\{
\begin{aligned}
&\rE_3(w,  \B^4) \leq {\rm K}_{\rm dir} \left( \rE_3(v, \B^4)+ \int_{\B^4\setminus \B^4(3\slash4)} \vert v(\rbx-\sP) \vert^3 \right)  {\rm \ and \  }   \\
& \int_{\B^4\setminus \B^4(3\slash4)} \vert w(\rbx)-\sP \vert^3 \rd \rbx  \leq {\rm K}_{\rm dir}   
 \int_{\B^4\setminus \B^4(3\slash4)} \vert v(\rbx)-\sP \vert^3.
 \end{aligned}
\right.
\end{equation} 
\end{lemma}

\begin{proof} We proceed  in two steps. In the first step, we construct a map $  \tilde w$  which is smooth, except at two point singularities, and which merely satisfies condition \eqref{hautpuech}.  We then remove these singularities using concentration along a line joining the singularities. 

\smallskip
\noindent
{\it Step 1: Construction of $\tilde w$}.  By a  standard  mean-value argument, we may find some radius $r_0 \in [3\slash4, 7\slash 8]$ such that, on the sphere $\S^3(r_0)$, we have 
\begin{equation}
\label{minus}
\left\{
\begin{aligned}
\int_{\partial \B^4(r_0)}  \vert \nabla v\vert ^3 (x) \rd x &\leq  16\, \rE_3(v)  {\rm \ and  \ }  \\
\int_{\partial \B^4(r_0)}  \vert  v-\sP \vert ^3 (x) \rd x \leq  &16  \Vert v-\sP \Vert_{L^3(\mathcal C^4(3\slash 4))}^3 \leq 16 \epsilon_0,
\end{aligned}
\right. 
\end{equation} 
where $\mathcal  C^4(r)=\B^4(7\slash 8)\setminus \B^4(r)$, for $0 \leq r\leq 7\slash 8.$  In order to reduce once more the dimension,  we consider for $h\in [0, 1\slash 4]$ the two-dimensional sphere $\mathcal  S^2(h)$ defined by 
$$\mathcal  S^2(h)=\partial \B^4(r_0) \cap P_{1,2}(h)=\{\rbx=(x_1, x_2, x_3, h), {\rm \ s.t. \  } x_1^2+x_2^2+ x_3^3=r_0^2-h^2  \geq  1 \slash 2 \}.$$
We decompose the sphere $\S^3(r_0)$ as $\S^3(r_0)=\S^{3, +} (r_0, h) \cup \S^{3, +} (r_0, h))$,  and the annulus $\mathcal  C^4(r_0)$ as 
$\mathcal C^4(r_0)=\mathcal C^{+}(h)\cup \mathcal C^-(h)$, where  we have set
\begin{equation*}
\left\{
\begin{aligned}
&\S^{3, +}(r_0, h)=\left\{ \rbx=(x_1, x_2, x_3, x_4 ) \in \S^3(r_0), x_4 \geq h \right\}, \S^{3, -}(r_0, h)=\left\{ \rbx \in \S^3(r_0), x_4 \leq h \right\}, \\
&\mathcal C^{ +} (h)=\left\{\rbx \in \mathcal C^4(r_0), \frac{r_0\rbx}{\vert \rbx \vert} \in \S^{3, +} (r_0, h)\right\}  {\rm \ and \ } 
\mathcal C^{-}(h)=\left\{\rbx \in \mathcal C^4(r_0),  \frac{r_0\rbx}{\vert \rbx \vert} \in \S^{3, +} (r_0, h)\right\}. 
\end{aligned}
\right.
\end{equation*}
We deduce, once more by a mean-value argument,  that there  exists $h_0 \in [0, 1\slash4]$ such that 
\begin{equation}
\label{minus2}
\left\{
\begin{aligned}
\int_{\mathcal S^2(h_0))}  \vert \nabla v \vert ^3 (\upsigma) \rd \upsigma &\leq 8\int_{ \S^3(r_0)}  \vert \nabla v \vert ^3 (x) \rd x  \leq 128\, \rE_3(v)  {\rm \ and  \ }  \\
\int_{\mathcal S^2(h_0) }  \vert  v(\upsigma)-\sP \vert ^3 \rd \upsigma & 
\leq 8 \int_{\S^3(r_0)}  \vert  v-\sP \vert ^3 
 \leq 128  \Vert v-\sP \Vert_{L^3(\mathcal C^4(\frac 3 4))}^3. \\
\end{aligned}
\right. 
\end{equation} 
Since $\mathcal S^2(h_0)$ is two-dimensional, a Gagliardo-Nirenberg interpolation inequality shows that, for any number $1\slash 3 <\alpha <1$, we have, for some constant 
${\rm C}_{GN}$ depending on $\alpha$, 
\begin{equation*}
\Vert v-\sP\Vert_{L^\infty( \mathcal S^2(h_0)) }^3 \leq {\rm C}_{GN}\left( \int_{\mathcal S^2(h_0))}  \vert \nabla v \vert ^3 (\upsigma) \rd \upsigma\right)^\alpha  \cdot
 \left( \int_{\mathcal S^2(h_0) }  \vert  v(\upsigma)-\sP \vert ^3 \rd \upsigma\right)^{1-\alpha}.  
\end{equation*}
Choosing $\alpha =1 \slash 2$, we deduce from the above  inequality combined with \eqref{minus2} that, if $v$ satisfies assumption \eqref{petitude}, then we have 
$\displaystyle{\Vert v-\sP\Vert_{L^\infty( \mathcal S^2(h_0)) } \leq 128^{1\slash 3} {\rm C}_{GN} \epsilon_0^{2\slash3}. }$
We fixe now the value of $\eps_0$  so small that $128^{1\slash 3} {\rm C}_{GN} \epsilon_0^{2\slash3}\leq 1 \slash 4$: This yields 
\begin{equation}
\label{prochitude2}
\Vert v-\sP\Vert_{L^\infty( \mathcal S^2(h_0)) } \leq \frac{1}{4}. 
\end{equation}
We consider next the interpolation $\tilde v$ of $v$ and $\sP$ on the set $\mathcal C^4(r_0)$, defined by   
$$ \tilde v(\rbx)=\sP+\frac{7\slash 8-\vert \rbx\vert }{7 \slash 8-r_0} \, \left[ v\left ( \frac{r_0\rbx}{\vert\rbx \vert}\right)-\sP \right],  {\rm \ for \ } \rbx \in \mathcal C^4(r_0), $$
so that $\tilde v(\rbx)=v(\rbx)$ if $\rbx \in \S^3(r_0)$ and $\tilde v(\rbx)=\sP$ if $\rbx \in \S^3(7 \slash 8)$. It follows from \eqref{prochitude2} that, for 
$\rbx \in V(r_0, h_0)$, we have 
\begin{equation}
\label{prochitude}
\vert \tilde v(\rbx)-\sP \vert \leq \frac{1}{4}, 
{\rm \ where \ }
V(r_0, h_0)= \left\{\rbx \in \mathcal C^4(r_0), \frac{r_0\rbx}{\vert \rbx \vert} \in \mathcal S^2 (r_0, h_0)\right\}. 
\end{equation}
We are now in position to define  the value $\tilde w$ on the boundaries $\partial \mathcal C^\pm (h_0)$, which may be  decomposed as
$\displaystyle{\partial \mathcal C^\pm (h_0)= \S^{3, \pm}(r_0, h_0) \cup  V(r_0, h_0) \cup \left(\partial \mathcal C^\pm (h_0) \cap \S^3(7\slash8)\right)}$.
We set
\begin{equation}
\left\{
\begin{aligned}
&\tilde w(\rbx)=v(\rbx) {\rm \ for \ } \rbx \in \S^3(r_0), \, \tilde w(\rbx)=\sP, {\rm \ for \ } \rbx \in \S^3(7\slash 8), {\rm \ and \ }  \\ 
&\tilde w(\rbx)= \frac{ \tilde v(\rbx)}{\vert  \tilde v(\rbx)\vert},  {\rm \ for \ } \rbx \in V(r_0, h_0),  
\end{aligned}
\right. 
\end{equation}
 the last definition being well defined thanks to \eqref{prochitude}. One verifies that $\tilde w$ is Lipschitz on $\partial \mathcal C^\pm (h_0)$ and, in view of \eqref{minus} and \eqref{minus2}, that, for some constant $C>0$, we have
 \begin{equation}
 \label{goalitude}
 \left\{
 \begin{aligned}
& \int_{\partial \mathcal C^\pm (h_0)}  \vert  \nabla  \tilde w\vert ^3   \leq  C \left(\rE_3(v) + \Vert v-\sP \Vert_{L^3(\mathcal C^4(3\slash 4))}^3 \right) \  {\rm \ and  \ } \\
& \int_{\partial \mathcal C^\pm (h_0)}  \vert \tilde w-\sP\vert ^3  \leq C \Vert v-\sP \Vert_{L^3(\mathcal C^4(3\slash 4))}^3.
 \end{aligned}
 \right.
 \end{equation}
 We extend $\tilde w$ to the interior of $\mathcal C^\pm (h_0)$ taking advantage that there is a Lipschitz homeomorphism from $\mathcal C^\pm (h_0)$ to the standard cube: Hence, we may use (or adapt) the cubic extension  presented in Subsection \ref{cubic}. This 
 allows us to define $w$ on  $\mathcal C^\pm (h_0)$ in such a way that the restriction of  $\tilde w$ is Lipschitz on every compact subset of $\overline {\mathcal C^4 (h_0)}\setminus \{A^+, A^-\}$, where 
 $A^\pm=(0, 0, 0, \pm \frac{13}{16})$, and such that, for some universal constant $C>0$, we have 
  \begin{equation}
  \label{goalitude2}
 \left\{
 \begin{aligned}
& \int_{ \mathcal C^4 (r_0)}  \vert  \nabla  \tilde w\vert ^3   \leq  C \left(\rE_3(v) + \Vert v-\sP \Vert_{L^3(\mathcal C^4(3\slash 4))}^3 \right) \  {\rm \ and  \ } \\
& \int_{\mathcal C^4 (r_0)}  \vert \tilde w-\sP\vert ^3  \leq C \Vert v-\sP \Vert_{L^3(\mathcal C^4(3\slash 4))}^3.
 \end{aligned}
 \right.
 \end{equation}
Finally, we complete the construction of $\tilde w$ setting  
\begin{equation}
\label{goalitude3}
 \tilde w(\rbx)=v(\rbx), {\rm \ for \ } \rbx \in \B^4(r_0), {\rm \ and \ }  \tilde w(\rbx)=\sP, {\rm \ for \ } \rbx \in \B^4 \setminus \B^4(\frac 78),
\end{equation}
so that $\tilde w$  is defined on the whole $\B^4$, Lipschitz on every compact subset of $\overline {\B^4}\setminus \{A^+, A^-\}$.  Moreover,  estimates \eqref{goalitude2} remains valid with a different constant $C$,  if we integrate on $\B^4$ instead of $\mathcal C^4(r_0)$.

\medskip
\noindent
{\it Step 2: Construction of $ w$ completed}. To complete the proof, it remains to remove the two singularities. This can be achieved concentrating "bubbles" along the segment $[A^+, A^-]$,  as in Proposition \ref{segment}. This yields 
a sequence of Lipschitz maps $\tilde w_n \in {\rm Lip }(\B^4, \S^2)$ such that $\tilde w_n \rightharpoonup  \tilde w$ in $W^{1, 3}(\partial \B^4)$,  such that $\tilde w_n=\sP$   on $\B^4 \setminus \B^4(r_0)$,  and such that 
\begin{equation}
\label{fifinal}
\underset { n \to + \infty} \lim \rE_3(\tilde w_n, \B^4) = \rE(\tilde w, \B^4)+ \upnu_{\S^2} (d)\vert A^+ -A^-\vert,
\end{equation}
 where  $d$ represents the Hopf invariant at the singularity $A^+$.  Using a standard "mollifying-reprojection" method, we may moreover assume that $\tilde w_n$ is smooth. The number $d$ corresponds also to the Hopf invariant of the map $\tilde w$ restricted to  $\partial \mathcal C^+(h_0)$.  We therefore have,
 \begin{equation}
 \label{fifinal2}
 \upnu_{\S^2} (d)  \leq C \int_{\partial \mathcal C^+ (h_0)}  \vert  \nabla  \tilde w\vert ^3   \leq  C \left(\rE_3(v) + \Vert v-\sP \Vert_{L^3(\mathcal C^4(3\slash 4))}^3 \right), 
 \end{equation}
  for some  $C>0$, where we used \eqref{goalitude} for the last inequality. Combining \eqref{fifinal}, \eqref{fifinal2} and \eqref{goalitude2} (with  the domain of integration replaced by $\B^4$), we obtain
 \begin{equation}
\label{fifinal3}
\underset { n \to + \infty} \lim \rE_3(\tilde w_n, \B^4)  \leq C \left(\rE_3(v) + \Vert v-\sP \Vert_{L^3(\mathcal C^4(3\slash 4))}^3 \right).
\end{equation}
By compact embedding, $\Vert \tilde w_n-\tilde w\Vert_{L^3(\B^4)} \to 0$ as $n \to + \infty$, so that, combining with \eqref{goalitude2} (with  the domain of integration replaced by $\B^4$), we are led to 
\begin{equation}
\label{fifinal4}
\underset { n \to + \infty} \lim \int_{\B^4}  \vert \tilde w_n-\sP\vert ^3  \leq C \Vert v-\sP \Vert_{L^3(\mathcal C^4(3\slash 4))}^3
\end{equation}
We finally choose $w=w_{n_0}$, for $n_0$ sufficiently large. The proof of \eqref{hautpuech}  then follows from \eqref{fifinal3} and \eqref{fifinal4}.
\end{proof}

\subsection{Creating  singularities through  the crossing of lines}
\label{creahopf}
We analyze in this subsection   a situation which   accounts for the creation of singularities in the  construction of the Gordian cut $\Gordk$. Roughly speaking, the construction we present allows to transform the  "rounded" part of  a deformation  of $\fLkperp_{j, q}$ into a straight segment, the "rounded" part being here represented by the curve $C_{\perp, h}^-(a)$,  whereas  the straight segment is represented by   ${\rm D}_{\perp, h}^+(a)$, both  sets being defined in \eqref{matches} and \eqref{defsegments} below  respectively (see Figures \ref{cube11} and \ref{cube11bis}) .  However, in order to perform this transformation, we need to cross a part of $\fLk$, represented in this subsection by the straight  segment $D_{0, h}(a)$, defined in \eqref{defsegments} (see Figure \ref{DversusC}).  The deformation  for the corresponding maps  is allowed thanks to the extension operator, at  the cost of a topological singularity.

We work  on cubes of radius $r=h/2$, the $x_4$ coordinate will be used for our further purposes as a deformation variable.   
 The singularities are created applying  the extension operator  to  Lipschitz  maps  $\Upsilon_{\rba}^h$ defined on the boundary of four-dimensional cubes\footnote{As a general rule roman bold characters as $\rba$ correspond to points in $\R^4$ whereas  symbols as $a$ refer to points in $\R^3$.}
 \begin{equation}
 \label{upgam}
   \Upsilon_{\rba}^h: \partial \rQ_{h/2}^4(\rba)  \to \S^2, {\rm \ where \ }  \rba=(a_1, a_2, a_3, a_4)\equiv(a, a_4)\in \R^4, h\geq 0, 
   \end{equation}
which we are going to define next,    using  the Pontryagin construction or its variant.   These construction are build on   the relevant curves  mentioned above, embedded  on the top and bottom  faces  
 $$
 \mathfrak Q^{3,\pm}_4(h/2,  \rba)=\rQ^3_{h/2}(a)\times \{a_4\pm\frac h2\}.  
 $$ 
 We  start the   precise description working in the  \emph{three-dimensional} reference cube  
 $$\rQ_{h/2}^3(a)=\{(x_1, x_2, x_3)\in \R^3,  \vert x_i-a_i  \vert \leq h/2,\,  i=1,2,3 \} \subset \R^3, {\rm \ for \  } \ a=(a_1, a_2, a_3)\in \R^3.  $$
 
  
 \subsubsection{Some relevant curves in $\rQ^3_{h/2}(a)$}
 \label{relevant}
 Let $a \in \R^3$.    We   consider the two  segments of $\rQ_{h/2}^3(a)$ given by  
 \begin{equation}
 \label{defsegments} 
 \left\{
 \begin{aligned}
 {\rm D}_{0, h}(a)&=\{a\} + [-\frac{h}{2}\be_1, +\frac{h}{2}\be_1] \subset P_{1,3}(a_3) \\
   {\rm  and \  \ \  \ } & \\
 {\rm D}_{\perp, h}^+(a)&= \{a-\frac{3h}{8}\be_3\}+ [-\frac{h}{2}\be_2, \frac{h}{2}\be_2] \subset P_{2,3}(a_1).
 \end{aligned}
 \right.
 \end{equation}
The two segments  are parallel to $\be_1$ and $\be_2$ respectively, have hence orthogonal directions,  each of them    joining  opposite faces of the cube (see Figure \ref{cube11}).

\smallskip
We consider  also   the  smooth curve $\mathcal C_{\perp, h}^-(a)$  given as the following graph in the plane $P_{2,3}(a_1)$
\begin{equation}  
\label{matches}
  \mathcal C_{\perp, h}^-(a)=\left\{\left( a_1,x_2, a_3-\frac{3h}{8} + hg_2 \,(\frac{x_2-a_2}{h})\right), \,  x_2\in  [a_2-h/2, a_2+h/2]  \right\}, 
  \end{equation}
 where  the function $g_2$ is defined in \eqref{defh2}.  This  definition  is consistent with \eqref{phika} and \eqref{phikaton2}:  Indeed, ${\rm D}_{0, h}(a)$ on one hand and  $\mathcal C_{\perp, h}^-(a)$  and ${\rm D}_{\perp, h}^+(a)$  on the other are aimed to model suitable subsets of fibers $\fLk_{i, j}$ and $\DefLkperp_{i, q}$ respectively,  as we will discuss in Subsection \ref{pes16} below.
   Both ${\rm D}_{\perp, h}^+(a)$ and   $\mathcal C_{\perp, h}^-(a)$   are included  in    the same affine  plane $P_{2,3}(a_1)$ and intersect along two segments parallel to $\be_2$, namely we have
   \begin{equation}
 \label{important}
 {\rm D}_{\perp, h}^+(a)\cap   \mathcal C_{\perp, h}^-(a) \supset
  \{a-\frac{3h}{8}\be_3\}+\left(  [-\frac{h}{2}\be_2,- \frac{h}{4}\be_2] +
 [\frac{h}{2}\be_2, \frac{h}{4}\be_2]\right).
   \end{equation} 
  In particular, their respective intersection with a suitably small neighborhood of the boundary coincide, see Figure \ref{DversusC} and \ref{cube11bis}.
  \begin{figure}[h]
\centering
\includegraphics[height=8.5cm]{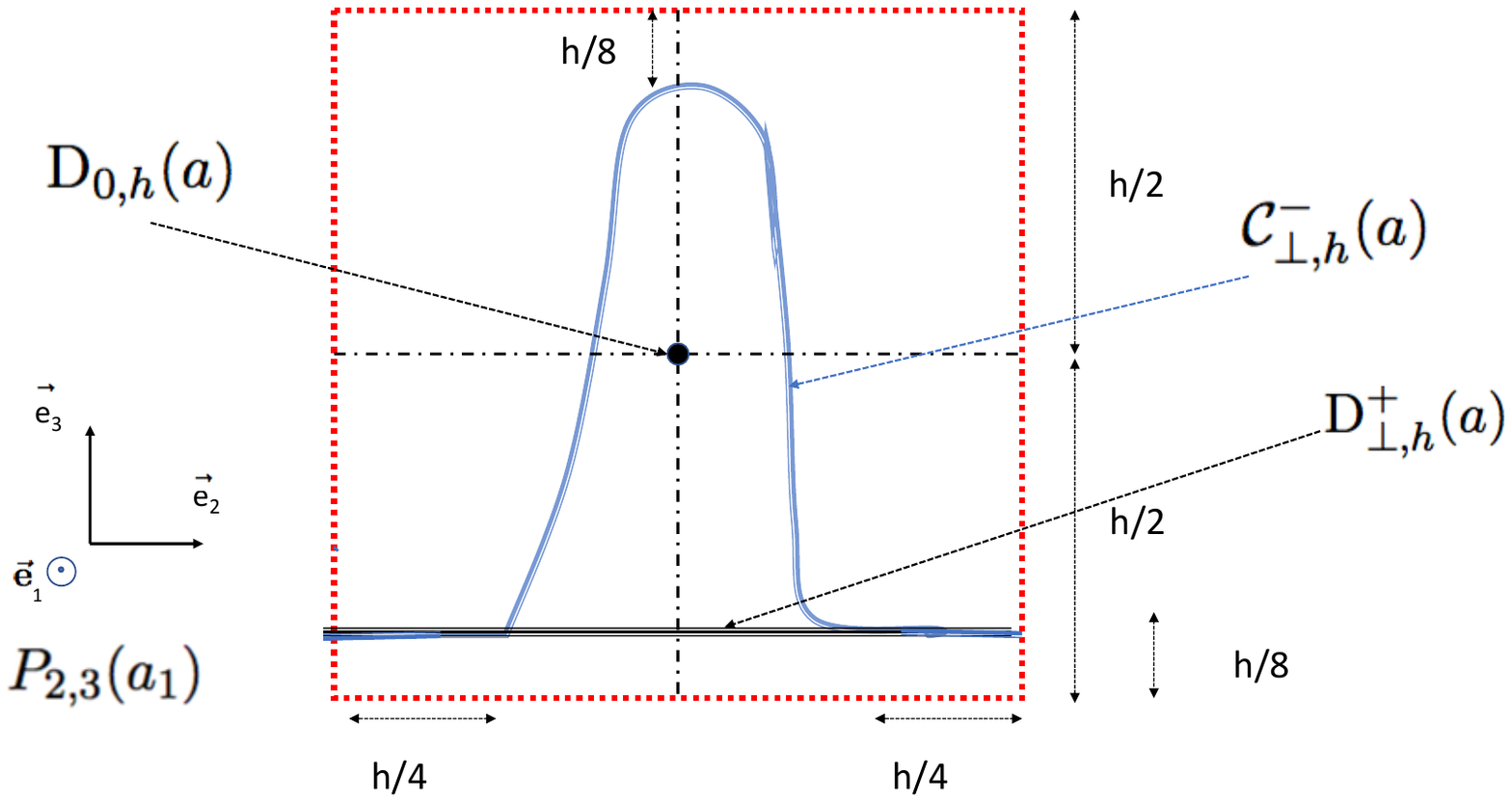}
\caption{  {
\it  The segment $ {\rm D}_{\perp, h}^+(a)$ and the curve  $\mathcal C_{\perp, h}^{-}(a)$  are in the same affine plane, the plane $P_{2,3}(a_1)$. Their union bounds in this plane a region containing the point $a$.  This region is crossed orthogonaly by the segment ${\rm D}_{0, h}(a)$. 
} }
\label{DversusC}
\end{figure}

  \begin{figure}[h]
\centering
\includegraphics[height=8.5cm]{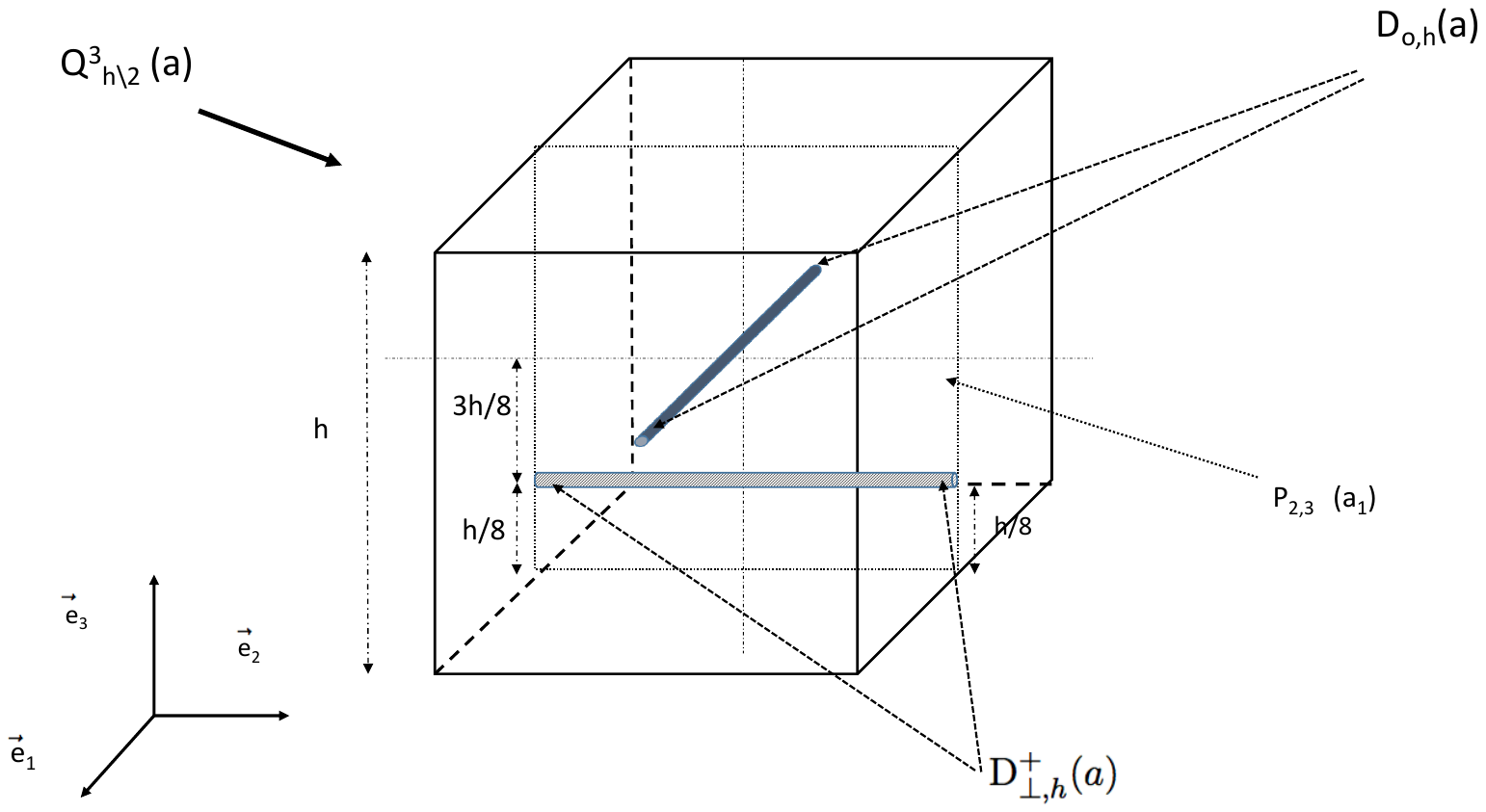}
\caption{  {
\it  The segment $ {\rm D}_{\perp, h}^+(a)$ and the segment ${\rm D}_{0, h}(a)$. These curves yield  by Pontryagin's  construction and its variant the map $\upgamma_{a}^{h, +}$ on   $\mathfrak Q^{3,+}_4(h/2,  \rba)$.
} }
\label{cube11}
\end{figure}

    \begin{figure}[h]
\centering
\includegraphics[height=9cm]{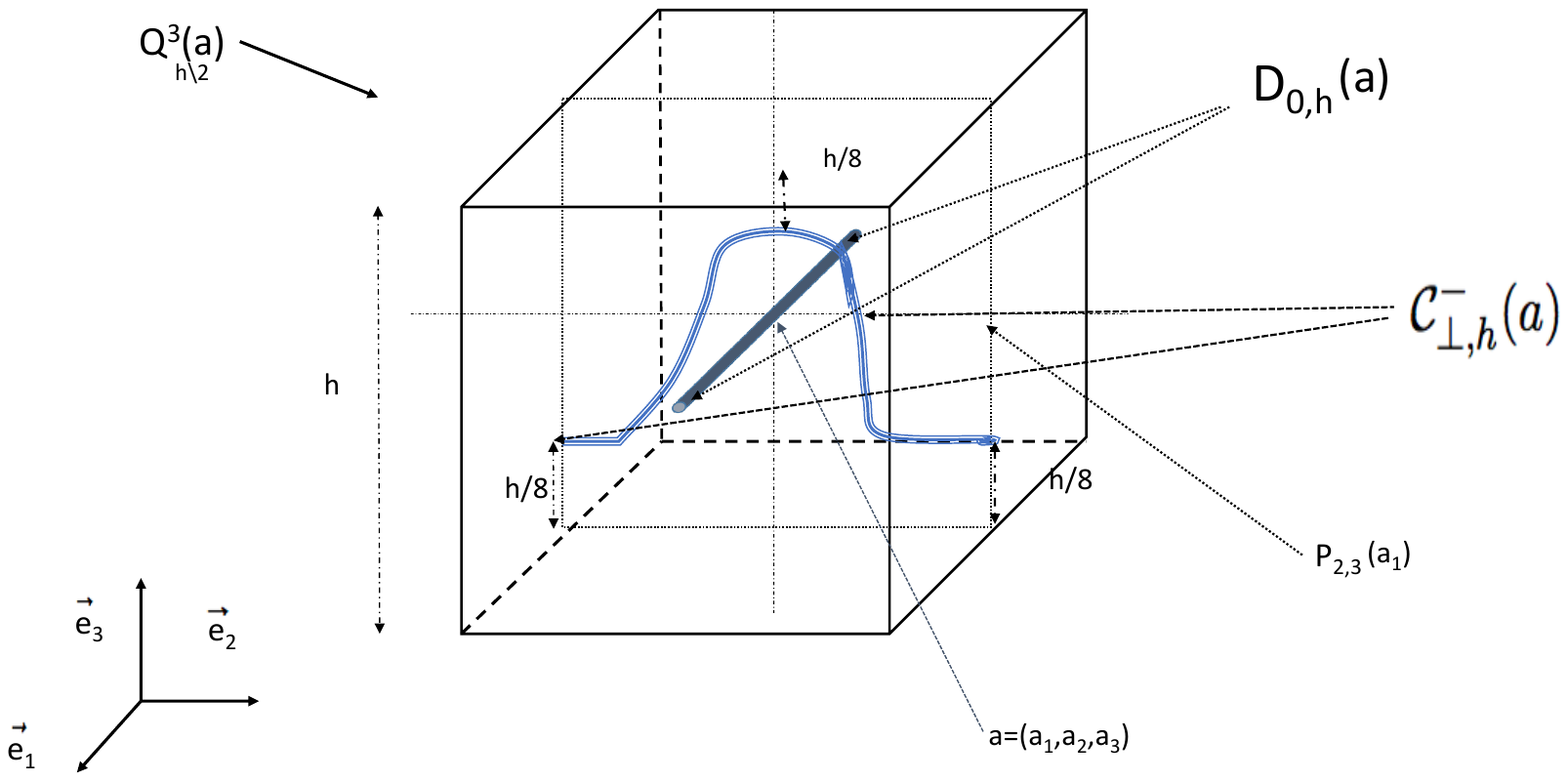}
\caption{  {
\it  The curve $ \mathcal C_{\perp, h}^-(a)$ and the segment ${\rm D}_{0, h}(a)$. These curves yield  by Pontryagin's  construction and its variant the map $\upgamma_{a}^{h,-}$ on   $\mathfrak Q^{3,-}_4(h/2,  \rba)$.
} }
\label{cube11bis}
\end{figure}

 \subsubsection{ $\S^2$-valued maps  on  $\rQ^3_{h/2}(a)$}
  Let  $ \varrho = 10^{-3}h$.  We relate to the curves constructed above  $\S^2$-valued maps using the Pontryagin construction or its  variant.  In order to have orientations consistent with  the  constructions in subsection \ref{laforme},  in particular the framings on the sheaves, we choose on ${\rm D}_{0, h}(a)$ the framing $\mathfrak e_{0}^\perp=( \be_3,- \be_2)$, whereas on ${\rm D}_{\perp, h}^+(a)$ we set $\mathfrak e_0^\perp=(\be_1, -\be_3)$.  We first consider the map $\upgamma_{a} ^{h,-}$ defined  on $\rQ^3_{h\slash2}(a)$ by (see Figure \ref{DversusC})
 \begin{equation}
  \label{alcatel}
  \upgamma_{a} ^{h,-}=\pont[ ({\rm D}_{0, h}(a),  \mathfrak e_0^\perp)]\wedgetrois \tpont_\varrho[ {\mathcal C}_{\perp, h}^-(a)].
 \end{equation}
  The notation  $\tpont_\varrho[ {\mathcal C}_{\perp, h}^+]$ which appears in \eqref{alcatel} refers to the \emph{ variant of the Pontragyin construction} defined in Paragraph \ref{variante}, for which the  plane  orthogonal  to the curve is replaced by a plane parallel  
  \footnote{The corresponding  framing  would correspond then to the framing on ${\rm D}_{\perp, r}^-$ that is $(\be_1, -\be_3)$.}
   to $P_{1, 3}$. More explicitely, it is  defined on $\rQ^3_{h/2}(a)$ by 
\begin{equation}
\label{important2}
 \tpont_\varrho[ {\mathcal C}_{\perp, h}^-(a)](x_1, x_2, x_3)=
 \chi_\varrho \left(x_1-a_1, -\left[x_3-a_3 +\frac {3h}{8}-hg_2(\frac{x_2-a_2}{h})\right] \right), 
 \end{equation}
  where  $\chi_\varrho$ is defined in  \eqref{defchi1} and  \eqref{blablanabla}.
 We define  on $\rQ^3_{h/2}(a)$ another  map $ \upgamma_{a} ^{h,+}$  (see Figure \ref{cube11bis}) as 
\begin{equation}
\label{alcatel0}
 \upgamma_{a} ^{h,+}=\pont_\varrho[ ({\rm D}_{0, h}(a),  \mathfrak e_0^\perp)]\wedgetrois \pont_\varrho[ {\rm D}_{\perp, h}^+(a),  \mathfrak e_0^\perp)]. 
\end{equation}
  An important  consequence of  the definitions \eqref{alcatel0} and \eqref{alcatel} 
  as well as of  \eqref{important} 
  is that 
\begin{equation}
\label{equalbor}
 \upgamma_{a} ^{h,-}(x)= \upgamma_{a} ^{h,+}(x) {\rm \ for \ } x \in \partial \rQ_{h/2}^3(a), 
\end{equation}
  In analogy with \eqref{scarface},  we set,  for $p=1,2,3$,
\begin{equation}
\label{pinkmec}
\mathfrak Q^{2,\pm}_p(r, a)=\{x=(x_1, x_2, x_3) \in \R^3, x_p=a_p\pm r, \,  \underset {i\neq  p} \sup\,  \vert x_i-a_i \vert < r\}.
\end{equation}
 so that $\partial \rQ_{h/2}^3(a)=\cup\,  \overline{\mathfrak Q^{2,\pm}_p(r, a)}$. We observe  that,  in view of 
\eqref{important} and \eqref{important2}, we have
\begin{equation}
\label{equalbor2}
\left\{
\begin{aligned}
& \upgamma_{a} ^{h,+}(x)= \upgamma_{a} ^{h,-} (x)=\sP, {\rm \  for   \ } x \in \mathfrak Q^{2,\pm}_3(h \slash 2, a),   \\
 & \upgamma_{a} ^{h,+}(x)= \upgamma_{a} ^{h,-} (x)=\chi_\varrho (x_1-a_1, a_3-x_3+\frac{3h}{8}),  {\rm \ for   \ } 
 x  \in \mathfrak Q^{2,\pm}_2(h \slash 2, a), \\
& \upgamma_{a} ^{h,+}(x)= \upgamma_{a} ^{h,-} (x)=\chi_\varrho \left((x_3-a_3),- (x_2-a_2) \right),
{\rm \ for  \ }  x \in  \mathfrak Q^{2,\pm}_1(h \slash 2, a).  
\end{aligned}
\right.
\end{equation}
  We notice also the symmetry properties on the boundary, for $p=1,2,3$
\begin{equation}
\gamma_{a}^{h, +}(x)=\gamma_{a}^{h,- }(x)=\gamma_{a}^{h, +}(x\mp h\be_p)=\gamma_{a}^{h, +}(x\mp h\be_p),  {\rm \ for \ } x \in
  \mathfrak Q^{2,\pm}_p(h \slash 2, a).
  	\end{equation}
  Moreover, the maps $\upgamma_{a} ^{h,\pm}$ are Lipschitz on $\rQ^3_{h/2}(\rba)$ and  we have the gradient bound
  \begin{equation}
  \label{gradupgamma}
  \vert \nablatrois \upgamma_{a} ^{h,\pm}(x)  \vert  \leq  C     h^{-1},   \forall h\geq 0, \forall x\in \rQ^3_{h/2}(\rba), 
  \end{equation}
  where $C>0$  is some universal constant. 
  
\subsubsection {Construction of   $\Upsilon^h_{\rba}: \partial \rQ^4_{h/2}(\rba) \to \S^2$}  Let $\rba=(a, a_4)=(a_1, a_2, a_3, a_4)$. We   take  advantage of \eqref{equalbor}  and \eqref{equalbor2} to define on 
 $\partial \rQ^4_{h/2}(\rba)$ an $\S^2$-valued map $\Upsilon^h_{\rba}$  whose restriction to the top  face of the boundary is $ \upgamma_{a} ^{h,+}$  and   whose restriction to the bottom  face is $\upgamma_{a} ^{h,-}$.  We define it as follows:
 \begin{equation}
 \label{definupsilon}
 \left\{
 \begin{aligned}
&  \Upsilon_{\rba}^h (x, a_4+\frac h2)= \upgamma_{a} ^{h,+}(x),{\rm \ for \ } x\in \rQ^3_{h/2}(a) {\rm \ i.e.} \
\rbx=(x, a_4+\frac h2) \in  \mathfrak Q^{3,+}_4(h/2,  \rba),\\
&  \Upsilon_{ \rba}^h (x, a_4-\frac h2)= \upgamma_{a} ^{h,-}(x), {\rm \ for \ } x\in \rQ^3_{h/2}(a)  
 {\rm \ i.e.} \
\rbx=(x, a_4-\frac h2) \in  \mathfrak Q^{3,-}_4(h/2,  \rba), \\
&   \Upsilon_{\rba}^h (x, x_4)= \upgamma_{a} ^{h,+}(x)= \upgamma_{a} ^{h,-}(x) 
   {\rm \ for \ } x\in \partial (\rQ^3_{h/2}(a))  {\rm \ and \ } x_4\in [a_4-\frac h 2, a_4+\frac h2]. 
  \end{aligned}
 \right.
 \end{equation}
  It follows  from \eqref{equalbor}  that $\Upsilon_{\rba}^h$ is a Lipschitz  $\S^2$-valued map on  $\partial \rQ_{h/2}^4(\rba)$, a set  which has the topology of  $\S^3$. We deduce  from  \eqref{gradupgamma} that we have a similar gradient bound for $\Upsilon_{\rba}^h$,
  \begin{equation}
  \label{gradupsilon}
  \vert \nablatrois   \Upsilon_{\rba}^h (\rbx) \vert  \leq   C h^{-1},   \forall h\geq 0, \forall \rbx \in \partial \rQ^4_{h/2}(\rba),
  \end{equation} 
  where $C>0$  is some universal constant. By integration on  $\partial \rQ^4_{h/2}(a)$, whose measure is $8h^{3}$,  we   see  that $\rE_3$-energy of $\Upsilon_{\rba}^h$ is bounded by a constant independent of $h$, namely
\begin{equation}
\label{lilou}
\rE_3\left(\Upsilon_{\rba}^h,\partial \rQ^4_{h/2}(a)\right) \leq  {\rm K}_\upgamma, {\rm \ for \ } h>0, 
\end{equation}
where $ {\rm K}_\upgamma,$ is some universal constant.

\subsubsection{The map  $\Upsilon_{\rba}^h$ and the Pontryagin construction}
  The map  $\Upsilon_{\rba}^h$ can be defined alternatively on $\partial \rQ_{h/2}^4(\rba)$ using Pontryagin's constructions for a curve we define next.
  Consider  the two-dimensions square 
  $${\rm B}(\rba)=\rba+  [-h/2, h/2] \times \{(0,0)\} \times [-h/2, h/2]\subset   \rQ_{h/2}^4(\rba)\subset \R^4$$
 and  its boundary, the curve
 \begin{equation}
 \label{defL0a}
 \mathcal L_{0,\rba}=\partial \rm B(\rba).
 \end{equation}
 Both ${\rm B}(\rba)$ and its boundary  $ \mathcal L_{0,\rba}$ are    included in the two-dimensional  subspace $\rbP_{1,4}(a_2, a_3)$ of $\R^4$ given by   $\rbP_{1,4}(a_2, a_3)=\{ \rbx \in \R^4, x_2=a_2  {\rm \ and \   }x_3=a_3\}$ (see Figure
 \eqref{deformL0}, where  $\mathcal L_{0,0}= \mathcal L_{0}(0)$).  The curve $ \mathcal L_{0,\rba}$ is equipped with the framing 
 \begin{equation}
 \label{framatome0}
 \mathfrak e_{0}^\perp=( \be_3,- \be_2).
 \end{equation}
 We consider also the curve $\mathcal L_{\perp,\rba}$,   included in the hyperspace $x_1=a_1$  and
   defined by
 \begin{equation}
 \label{defLperpa}
 \mathcal L_{\perp,\rba}=\left( \mathcal C_{\perp, h}^-(a)\times \{a_4-\frac h2\} \right)
 \cup  
 \left({\rm D}_{\perp, h}^+(a)\times \{a_4+\frac h2\} \right)
 \cup F_{\perp,1} \cup F_{\perp,2}, 
 \end{equation}
 where $F_{\perp,1}$ and $F_{\perp,2}$ denote the  segments parallel to $\be_4$ given by
 \begin{equation}
\label{grasz}
 \left\{
 \begin{aligned}
 &F_{\perp,1}\equiv [N_{\rm L, \rba}^{h, +}, N_{\rm L, \rba}^{h, -}]=\{\rba\} +\{(0, -\frac h2, -\frac {3h}{8})\} \times  [-\frac h2, +\frac h 2]  {\rm \ and \ }  \\ 
&F_{\perp,2}\equiv[N_{\rm R, \rba}^{h, +}, N_{\rm R, \rba}^{h, -}]=\{\rba\} + \{(0, +\frac h2, -\frac {3h}{8})\} \times  [-\frac h2, +\frac h 2],
\end{aligned}
\right.
\end{equation}
where $\displaystyle{N_{\rm L}^{h, \pm}(\rba)= \rba+(0,  -\frac h2, -\frac {3h}{8}, \pm \frac h2)}$ and  
$\displaystyle{N_{\rm R}^{h, \pm}(\rba)= \rba+(0,  \frac h2, -\frac {3h}{8}, +\frac h2)}$.
  The curve $\mathcal L_{\perp, \rba}$ is equipped with the framing $\mathfrak e_0^\perp=( \be_1,\be_3)$
  (see Figure
 \ref{deformLperp}, where  $\mathcal L_{\perp,0}= \mathcal L_{\perp}(0)$). The main result of this subsection is:

\begin{lemma} 
\label{grimm}
The sets $\mathcal L_{0, \rba}$ and $\mathcal L_{\perp, \rba}$ are  connected curves which don't intersect. Set $\mathcal L_\rba=\mL_{0,\rba} \cup \mL_{\perp,\rba}.$
We have $\mathcal L_\rba \subset \partial(\rQ_{h/2}^4(\rba))$
 and  
\begin{equation}
\label{hermanos}
 \begin{aligned}
  \Upsilon_{\rba}^h&= \tpont_\varrho  [  \mathcal L_\rba, \mathfrak e_0^\perp],  \\
\end{aligned}
 \end{equation} 
\end{lemma}
\begin{proof}  For the first assertion, we observe
 that $\mathcal L_{0,\rba}$ is composed of four segments of length $h$, two of them parallel to $\be_1$, the two others to $\be_4$.  The vertices $M_{\rm L}^\pm(\rba), M_{\rm R }^\pm(\rba)$ are given by
  $$
   M_{\rm L, \rba}^\pm=\rba +(-\frac h 2, 0, 0,\pm \frac h2)  {\rm \ and \ }
   M_{\rm R, \rba }^\pm=\rba +(\frac h 2, 0, 0,\pm \frac h2),
   $$
   so that 
   $$
   \mathcal L_{0,\rba}=\left[M_{\rm L, \rba}^+,M_{\rm R, \rba}^+\right] \cup
  \left[M_{\rm L, \rba}^-,M_{\rm R, \rba}^-\right]  \cup
  \left [M_{\rm L, \rba}^-,M_{\rm L, \rba}^+\right] \cup
  \left  [M_{\rm R, \rba}^-,M_{\rm R, \rba}^+\right].
   $$ 
   
   Notice that
$
\left[M_{\rm L, \rba}^\pm,M_{\rm R, \rba}^\pm\right]={\rm D}_{0,h}(a)\times \{a_4\pm \frac h 2\}
$ and that the two other segments are parallel to $\be_4$.  We have the inclusions
\begin{equation}
\label{carreaux}
\left\{
\begin{aligned}
& F_{0, {\rm top}}\equiv{\rm D}_{0,h}(a)\times \{a_4+ h \slash 2\}\subset \mathfrak Q^{3,+}_4(h/2,  \rba), \\
& F_{0, {\rm bot}}\equiv{\rm D}_{0,h}(a)\times \{a_4- h \slash 2\}\subset \mathfrak Q^{3,-}_4(h/2,  \rba), \\
& F_{0,\rm L} \equiv  \left [M_{\rm L, \rba}^-,M_{\rm L, \rba}^+\right] \subset 
	\mathfrak Q^{3,-}_1(h/2,  \rba)  {\rm \ and \ } \\
&F_{0,\rm R}\equiv  \left [M_{\rm R, \rba}^-, M_{\rm R, \rba}^+\right] \subset 
 \mathfrak Q^{3,+}_1(h/2,\rba).
 \end{aligned}
 \right.
 \end{equation}
 
  Concerning   $\mathcal L_{\perp,\rba}$, we have
  \begin{equation*}
  \left\{
  \begin{aligned}
  & {\rm D}_{\perp, h}^+(a)\times \{a_4 + h \slash 2\}\subset \mathfrak Q^{3,+}_4(h/2,  \rba),\\
  & \mathcal C_{\perp, h}^-(a)\times \{a_4 - h \slash 2\}\subset \mathfrak Q^{3,-}_4(h/2,  \rba),  \\
  & {F_{\perp,1} \subset 
  	\mathfrak Q^{3,-}_2(h/2,  \rba) } {\rm \ and \ } F_{\perp,2} \subset
  \mathfrak Q^{3,+}_2(h/2,  \rba), 
  \end{aligned}
  \right.
  \end{equation*}
   so that $\mathcal L_{\perp,\rba}\subset  \partial(\rQ_{h/2}^4(\rba))$ as well.    One may then check that the curves are connected and that their intersection is empty (see also the proof of Lemma \ref{commea}). Identity \eqref{hermanos} is a direct consequence of  identities \eqref{alcatel} and \eqref{alcatel0} for the maps $\upgamma_{a, h}^-$ and 
 $\upgamma_{a, h}^+$ respectively. 
\end{proof}

\subsubsection{Relationship  with  the curves $\fLk_{i, j} $ and $\DefLkperp_{i, q}$}
 \label{pes16}
 
 The sets   ${\rm D}_{\perp, h}^+(a)$,  $\mathcal C_{\perp, h}^-(a)$ and ${\rm D}_{0, h}(a)$  are designed  to represent  specific parts  of  the fibers $\fLk_{i, j}$ and $\DefLkperp_{i, q}$.  In the proof of Proposition \ref{deform}, we are led to consider points of the form 
\begin{equation}
\label{share}
a=\ra_{i, j, q}^k\equiv h(i, j, q),{\rm \ for \ some \ integers \ } i, j, q=1, \ldots, k.  
\end{equation}
These points  belong to the cube $[0, 1]^3$. 
 We   verify that $\ra_{i, j, q}^k\in \fLk_{j, q}$, so that    
the segment  ${\rm D}_{0, h}(\ra_{i, j, q}^k)$  represents  hence  the portion of the fiber $\fLk_{j, q}$ which intersects  the cube  $\rQ^3_{h/2}(\ra_{i, j, q}^k )$. Concerning the "orthogonal" sets $\mathcal C_{\perp, h}^-(\ra_{i, j, q}^k)$ and ${\rm D}_{\perp, h}^+(\ra_{i, j, q}^k)$,  we notice that, given any number $c\geq 0$ and any numbers $i', q' \in \{1, \ldots, k\}$,  we have (see e.g. Figure \ref{side0} as well as  identity \eqref{chouperche} in Remark \ref{chouchou})
\begin{equation}
\label{rabia}
\left\{
\begin{aligned}
&\ra_{i, j, q}^k\in \fLkperp_{i', q'}-c\,  \be_3,  {\rm \ provided \ } i=i' {\rm \ and \ } 6+q' h-c=qh, {\rm \ i.e. \ } (q-q')=\frac{6-c}{h}, \\
&\ra_{i, j, q}^k-\frac{3h}{8}\in \fLkperp_{i', q'}-c\,  \be_3, {\rm \ provided \ } i=i' {\rm \ and \ } 6+\frac {3h}{8}+ q' h-c=qh. 
\end{aligned}
\right.
\end{equation}
We deduce:

\begin{lemma} 
\label{round-up}
Assume that  the condition 
\begin{equation}
 \label{ondiment}
c=6-\frac {5h}{8}+(q -q')h, 
 \end{equation}
  involving only\footnote{So that only the "vertical coordinates" are involved in condition \eqref{ondiment}.} the numbers $c, q$ and $q'$,  \emph{but not}   $i$ and $j$, is fulfilled. 
Then, we  have:
  \begin{equation}
  \label{sel}
  \left \{
  \begin{aligned}
 \mathcal C_{\perp, h}^-(\ra_{i, j, q}^k)&= \DefLkperp_{i, q'}(c, h) \cap \rQ^3_{h/2}(\ra_{i, j, q}^k),    \\
{\rm D}_{\perp, h}^+(\ra_{i, j, q}^k)&=\left (\fLkperp_{i, q'}-\left(c +h\right)\be_3 \right)\cap \rQ^3_{h/2}(\ra_{i, j, q}^k)
 {\rm  \ and }\\
{\rm D}_{0, h}(\ra_{i, j, q}^k) &=\rQ^3_{h/2}(\ra_{i, j, q}^k)\cap \fLk=\rQ^3_{h/2}(a_{i, j, q}^k)\cap \fLk_{j, q}. 
 \end{aligned}
 \right. 
  \end{equation}
 \end{lemma}
 Turning to the map $\Upsilon_{\rba}^h$,  we notice that, for   a point  $\rba$ of the form
 $\rba=(\ra_{i, j, q}^k, a_4)$ where $\ra_{i, j, q}^k$ is of the form given by \eqref{share}, we have
  \begin{equation}
 \label{konami}
 \Upsilon_{\rba}^h(\rbx)=\pont\left[\fLk\cup(\fLkperp_{j, p}-(c+h)\be_3), \frameref\right](\rbx) {\rm \ for  \ }  \rbx \in 
 \partial \rQ^3 (\ra_{i, j, q}^k)\times [a_4-\frac h2, a_4 +\frac h2 ], 
\end{equation}
 provided the numbers $c, q$ and $q'$ satisfy  relation \eqref{ondiment}.  Since
 $$
 \partial \rQ^3 (\ra_{i, j, q}^k)\times [a_4-\frac h2, a_4 +\frac h2 ]=\partial \rQ^4_{h\slash 2}((\ra_{i, j, q}^k), a_4)
 \setminus \left( \mathfrak Q^{3,+}_4(h/2,  \rba) \cup  \mathfrak Q^{3,-}_4(h/2,  \rba)\right), 
 $$
   it follows  from \eqref{konami} and \eqref{definupsilon}  that 
  \begin{equation}
 \label{konami2}
 \Upsilon_{\rba}^h(\rbx)=\pont\left[\fLk\cup(\fLkperp-(c+h)\be_3), \frameref\right](\rbx), {\rm \ for  \ }  \rbx \in
 \partial \rQ^4_{h\slash 2}(\rba)
 \setminus  \mathfrak Q^{3,-}_4(h/2,  \rba) 
\end{equation}
  provided there exists some $q'\in \{1, \ldots, k\}$ such that \eqref{ondiment} holds. Notice that the r.h.s of \eqref{konami} does no longer depend on $i, j, q$.  As  a consequence we have the periodicity property
  \begin{equation}
  \label{periodicity}
   \Upsilon_{\rba}^h(x,s)= \Upsilon_{\rba\mp h\be_\ell}^h(x, s),  {\rm \ if \ } 
   x  \in  \mathfrak Q^{2,\pm}_\ell(h/2,  \rba), \ell=1,2
   {\rm \ and \ } s \in  [a_4-\frac h2, a_4 +\frac h2 ].
  \end{equation}

 \subsubsection{Topological Properties of the map $\Upsilon_{\rba}^h$} 
 We establish in this subsection the following  property:
  \begin{lemma}
 \label{topo}
  We  have $\displaystyle{ \rH (\Upsilon_{\rba}^h)=2.}$
 \end{lemma}

The proof relies on the following result:

  \begin{lemma}  
 \label{commea}
 	The two curves $\mathcal L_{0,\rba}$ and $\mL_{\perp,\rba}$ are linked in $\partial \rQ^4_{h \slash 2}(\rba)$ and 
   \begin{equation} 
   \label{lamaison}
   \mathfrak m\left(\mL_{0,\rba}, \mL_{\perp,\rba}\right)=1.
   \end{equation}
 \end{lemma}

 \begin{proof}  Although the  curves $\mathcal L_{0,\rba}$ and $\mL_{\perp,\rba}$ have  simple shapes, the main difficulty, in order  to prove   \eqref{lamaison},  is that we  have to work on the boundary of a 
 four-dimensional cube, a situation which is  rather hard to visualize. In order to overcome this difficulty, 
 we deform the two curves in a continuous way so to obtain a simpler geometry,  allowing us eventually to visualize the geometry  in the usual three-dimensional space, where the properties become more intuitive. 
 In order to simplify  notation,  we  assume without loss of generally that $\rba=0$ and  $h=1$,  and set $\mL_0=\mL_{0,0}$ and $\mL_\perp=\mL_{\perp,0}$. Our strategy is to construct  families of sets $\{\mathcal L_0(t)\}_{0\leq t\leq 1}$ and  $\{\mathcal L_{\perp}(s)\}_{0\leq s \leq 1}$ which are continuous deformations of $\mathcal L_0$ and $\mathcal L_\perp$ respectively, and which stay inside  the set $\partial \rQ^4_{1 \slash 2}(0)$.  Once the deformations are  performed,  $\mathcal L(1)$ and $\mathcal L_\perp(1)$ are included in  the same three-dimensional subspace of $\R^4$.
 
 \smallskip 
  For the  deformation of   $\mathcal  L_0$, we take advantage of the fact that $\mathcal  L_0$ is included in a two-dimensional space spanned by the vectors $\be_1$ and $\be_4$, that is $\mathcal L_0\subset {\rm \bf P}_{1, 4}(0)$, 
  where, 
  \begin{equation}
  \label{defhyperpl}
   {\rm \bf P}_{i, j}(\rba)= {\rm Vect }(\be_i, \be_j)+ \{\rba\}  {\rm \ for \ } i\not=j  {\rm \ in \ } \{1, 2, 3, 4\} {\rm \ and \ } \rba \in \R^4.
  \end{equation}
   This allows us to construct the deformation inside  the three dimensional subspace $\mathbbmss V_2$ spanned by the vectors $\be_1, \be_3$ and $\be_4$, and even  inside the half-space   $\mathbbmss V_2^+$ defined by
  \begin{equation*}
  \mathbbmss V_2^+=\{\rbx=(x_1, x_2, x_3, x_4) \in \R^4, x_2=0, x_3 \geq 0\} \subset \mathbbmss V_2 \equiv {\rm Vect} (\be_1, \be_3, \be_4).
  \end{equation*}
  The idea is to "rotate" the rectangle $\mathcal L_0$ around one of   the segments parallel to $\be_4$, deforming it so  to stay inside  the boundary of the three-dimensional cube $\rQ^4_{1 \slash 2}(0)\cap \mathbbmss V_2$.
Ultimately,   we will check that    $\mathcal L_0(t)$ is a  connected curve and satisfies   the following  conditions 
 \begin{equation}
 \label{diffus}
 \left\{
 \begin{aligned}
& \mathcal L_0(0)=\mathcal L_0, \\
&\mathcal L_0(1)\subset    \mathfrak Q^{3,-}_4(1/2,  0)\cap  \mathbbmss V_2^+ \subset 
{\rm  \bf P}_{1, 3}(\rbA_0) {\rm \ with  \ } \rbA_0\equiv (0,0,0, -1\slash2), \\
  & \mathcal L_0(t) \subset  \partial \rQ^4_{1 \slash 2}(0)\cap \mathbbmss V_2^+ {\rm \ for \ } t \in [0, 1] {\rm \ and  \ } \\
  & \mathcal L_0(t) \cap \mathcal L_{\perp} =\emptyset, {\rm \ for \ } t \in [0, 1]. \\
 \end{aligned}
 \right. 
 \end{equation}
Since the deformation takes places in the three-dimensional space $\mathbbmss V_2^+$, we are in position to propose a graphical representation  of the deformation, see  Figure \ref{deformL0}.

\smallskip 
The deformation of $\mathcal L_\perp$ is  more involved, since 
$\mathcal L_\perp \subset \mathbbmss V_1=(\R\be_1)^\perp={\rm Vect} (\be_2, \be_3, \be_4)$, but is not included in a lower dimensional subspace. The idea is to "rotate" the part  of    $\mathcal L_\perp$  (see definitions  \eqref{defLperpa}  and \eqref{grasz})  composed of three segments   around a line parallel to $\be_2$,  so  to stay inside  the boundary of the three-dimensional cube $\rQ^4_{1 \slash 2}(0)\cap {\mathbbmss V_1^-}$, where
\begin{equation*}
  \mathbbmss V_1^-=\{\rbx=(x_1, x_2, x_3, x_4) \in \R^4, x_1=0, x_3 < 0\} \subset \mathbbmss V_1 \equiv {\rm Vect} (\be_2, \be_3, \be_4)=(\R\be_1)^\perp.
  \end{equation*}
The continuous deformation of $\mathcal L_\perp$ will transform it into a connected curve which is included into a two dimensional affine subspace. More precisely, we require $\mathcal L_\perp(s)$ to be connected and to satisfy the following conditions:
\begin{equation}
 \label{diffus2}
 \left\{
 \begin{aligned}
& \mathcal L_\perp(0)=\mathcal L_\perp, \\
& \mathcal L_\perp(1)\subset    \mathfrak Q^{3,-}_4(1/2,  0)\cap  \mathbbmss V_1 \subset 
{\rm  \bf P}_{2, 3}(\rbA_0) {\rm \ with  \ } \rbA_0\equiv (0,0,0, -1\slash2), \\
  & \mathcal L_\perp(s) \subset  \partial \rQ^4_{1 \slash 2}(0)\cap \mathbbmss V_1{\rm \ for \ } t \in [0, 1] {\rm \ and  \ } \\
  & \mathcal L_\perp(s) \cap \mathcal L_{0}(1) =\emptyset, {\rm \ for \ } t \in [0, 1]. \\
 \end{aligned}
 \right. 
 \end{equation}
At the end of the deformations the curves $\mathcal L_0(1)$ and $\mathcal L_\perp(1)$ belong to the same three-dimensional affine space, namely the  space $\mathbbmss V_4 (\rbA_0)$ given by 
\begin{equation}
\label{defv4}
\mathbbmss V_4 (\rbA_0) \equiv \mathbbmss V_4 + \{\rbA_0\}  \supset {\rm  \bf P}_{1, 3}(\rbA_0)\cup {\rm  \bf P}_{2, 3}(\rbA_0)
{\rm \ where \ } \mathbbmss V_4 \equiv {\rm Vect }(\be_1, \be_2, \be_3)=\R^3\times\{0\}. 
\end{equation}
We  next provide the  analytical details of the deformations.

 \medskip
 \noindent
   {\it Construction of $\mL_0(t)$, $0\leq t \leq 1$}.  Recall that $\mL_0$ is a square (see \eqref{defL0a}): We  construct $\mathcal L_0(t)$ is such a way that it remains a rectangle.  
    For $0\leq t \leq  1$, we denote  by $M_{\rm L}^\pm(t)$ and $M_{\rm R}^\pm(t)$ the vertices  the rectangle. Two of the vertices (the ones with the 
    "-"  superscript) are not moved by the deformation and stay fixed, namely we set, for $0 \leq t \leq 1$, 
    \begin{equation*}
  M_{\rm L}^-(t)=M_{\rm L}^-(0)=(-\frac12,0,0,-\frac12)  {\rm \ and  \ } M_{\rm R}^-(t)=M_{\rm R}^-(0)=(\frac12,0,0,-\frac12).
  \end{equation*}
   For the two other vertices, we decompose the deformation in two steps (see Figure \ref{deformL0}). First,  for $0\leq t \leq 1\slash 2$, we define 
      \begin{equation*}
      M_{\rm L}^+(t)=M_{\rm L}^+(0)+t \be_3=(-\frac12, 0, t,\frac12) {\rm \ and \ }
       M_{\rm R}^+(t)=M_{\rm R}^+(0)+t \be_3=(\frac12, 0, t,\frac12).
       \end{equation*}
       whereas for $ 1\slash 2 \leq t \leq 1$, we set 
       \begin{equation*}
       \left\{
       \begin{aligned}
       &M_{\rm L}^+(t)=M_{\rm L}^+(0)+\frac 12 \be_3-
       (2t-1) \be_4=(-\frac12,0,\frac12,\frac32-2t)  {\rm \ and \ } \\
       &M_{\rm R}^+(t)=M_{\rm L}^+(0)+\frac 12 \be_3 -
       (2t-1) \be_4=(\frac12,0,\frac12,\frac32-2t).
       \end{aligned}
       \right.
     \end{equation*}
   The functions $t\mapsto M_{\rm L}^\pm(t)$ and $t\mapsto M_{\rm L}^\pm(t)$ are hence  continuous on $[0,1]$ and notice  that 
   \begin{equation}
 \label{notrice}
   \left\{
   \begin{aligned}
& \{M_{\rm L}^\pm(t), M_{\rm R}^\pm (t)\} \subset   \partial \rQ^4_{1 \slash 2}(0) \cap  \mathbbmss V_2^+ {\rm \ for \ }  t \in [0, 1]\\
&  \{M_{\rm L}^\pm(1), M_{\rm R}^\pm (1)\} \subset    \mathfrak Q_4^{3, -} (1 \slash 2, 0) \cap  \mathbbmss V_2^+.
\end{aligned}
\right. 
 \end{equation} 
      We define, for $0\leq t \leq 1$, the curve    $\mL_0(t)$ as the rectangle  
     \begin{equation}
     \label{defmL0t}
     \mL_0( t)=  [M_{\rm L}^-(t), M_{\rm R}^-(t)] \cap [M_{\rm R}^-(t), M_{\rm R}^+(t)]\cap [M_{\rm R}^+(t), M_{\rm L}^+(t)]\cap  [M_{\rm L}^+(t), M_{\rm L}^-(t)].
  \end{equation}
  It follows that  $\mL_0=\mL_0(0)$ and that  $t\mapsto \mL_0(t)$ is a continuous deformation of $\mL_0$.   It remains to verify that conditions \eqref{diffus} are satisfied.  The first and second conditions in \eqref{diffus}  are an immediate consequence of the definition of the points $M_{\rm L}^\pm(t)$ and $M_{\rm R}^\pm(t)$ and  of \eqref{notrice}. For the third  condition, we verify that 
   \begin{equation}
   \label{vospapier}
   \left\{
   \begin{aligned}
   &[M_{\rm R}^+(t), M_{\rm L}^+(t)] \subset \mathfrak Q_4^{3,+}(1\slash2, 0)  \cap \mathbbmss V_2^+, {\rm \ for \ } t\in [0, 1\slash 2]  \\
   & [M_{\rm R}^+(t), M_{\rm L}^+(t)] \subset \mathfrak Q_3^{3,+}(1\slash2, 0) \cap \mathbbmss V_2^+,{\rm \ for \ } t\in [ 1\slash 2, 1], \\
   &  [M_{\rm L}^+(t), M_{\rm L}^-(t)] \subset \mathfrak Q_1^{3,-}(1\slash2, 0) \cap \mathbbmss V_2^+, {\rm \ for \ } t\in [0, 1] {\rm  \ and \ }  \\
   &[M_{\rm R}^+(t), M_{\rm R}^-(t)]  \subset   \mathfrak Q_1^{3,+}(1\slash2, 0) \cap \mathbbmss V_2^+, {\rm \ for \ } t\in [0, 1]. 
  \end{aligned} 
   \right. 
   \end{equation}
Since $\mathfrak Q_i^{3, \pm} (1\slash2, 0)\subset \partial \rQ^4_{1\slash2}(0)$ for $i=1, \ldots, 4$,   the third condition in \eqref{diffus}  is checked. We finally turn to the fourth condition in \eqref{diffus}. To that aim,  we observe that 
    \begin{equation}
    \label{simic}
    \mathcal L_\perp \cap \mathbbmss V_2^+=\{\rbA_1\},  {\rm \ where  \ } \rbA_1=(0, 0, \frac{3}{8}, -\frac 12 ).
    \end{equation}
    Since, for any $0\leq t \leq 1$, we have  $\mathcal L_0(t)\subset \mathbbmss V_2^+$, it suffices therefore to verify that $\rbA_1 \not \in \mathcal L_0(t)$.   To check this latest fact, assume by contradiction that there is some $t_0\in [0, 1]$ such that $\rbA_1\in \mathcal L(t_0)$. Since the fourth coordinate of $\rbA_1$ is $-1\slash2$ and since the third coordinate is non zero, that one has necessarily $t_0=1$ and one checks that $\rbA_1 \not \in \mathcal L_0(1)$
    (see Figure \ref{deformL0}). This establishes the third assertion in \eqref{diffus}, and hence finishes the proof of \eqref{diffus}.

  \begin{figure}[h]
\centering
\includegraphics[height=7.5cm]{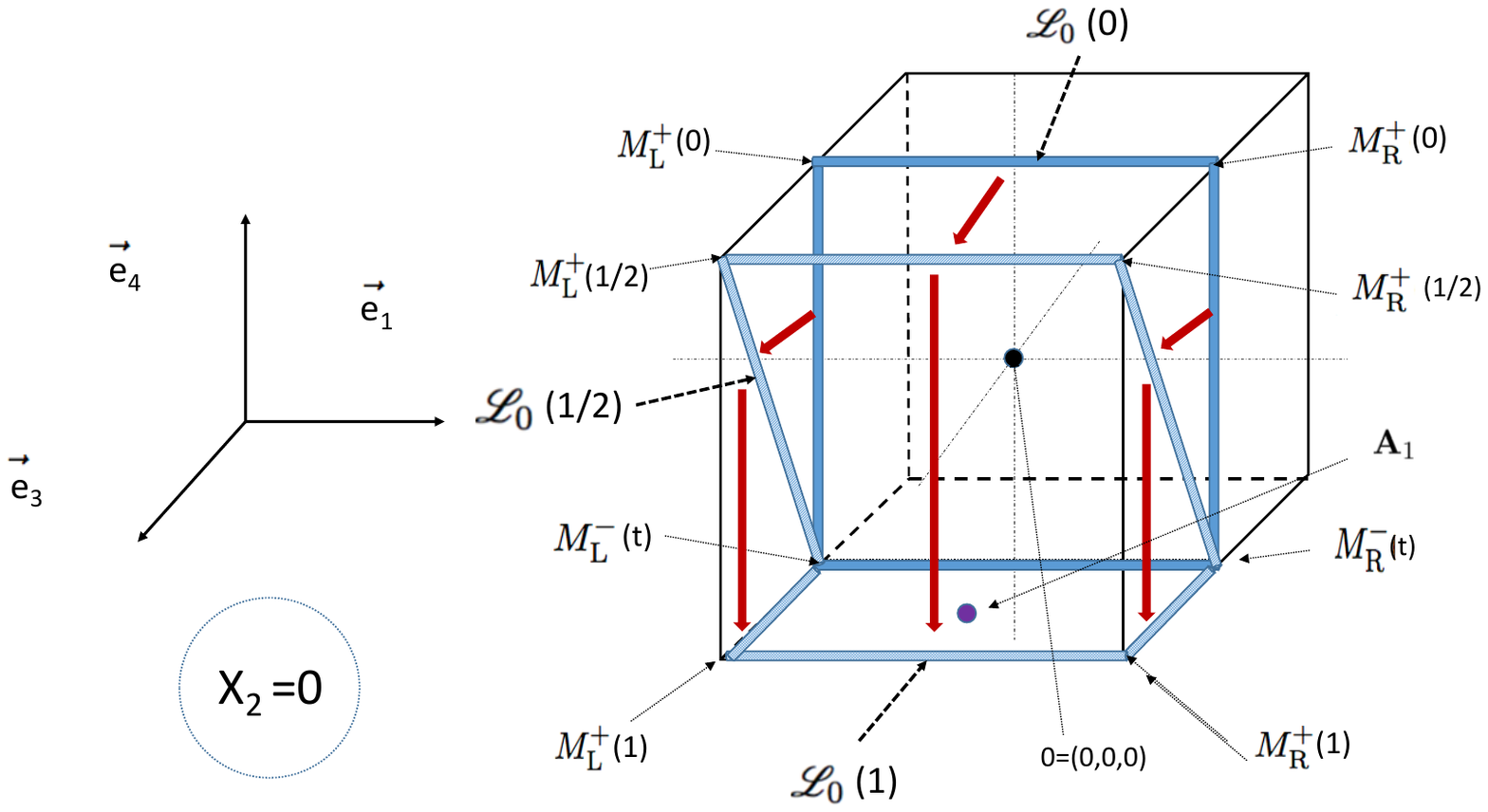}
\caption{  {\it The  deformation of the curve $\mL_0$ at  $t=0, 1\slash2$ and $t=1$ in the space $\mathbbmss V_2=(\R\be_2)^\perp.$}}
\label{deformL0}
\end{figure}

 \medskip
 \noindent
 {\it Step 2. Construction of $\mathcal L_\perp(s)$, $0\leq s \leq 1$}.  
  We perform a similar deformation for  the curve $\mL_\perp$, a  deformation which takes places in the three-dimensional space $(\R\be_1)^\perp$ ((see Figure \ref{deformLperp}).  In this deformation, the part of $\mathcal L_\perp$  provided by the  curve 
  $\displaystyle{\mathcal C_{\perp, h=1}^-(0)\times \{-\frac 12\}}$ (see \eqref{defLperpa}) remains fixed, whereas the three segments which are parts of a rectangle, are moved in a way similar to $\mathcal L_0$. We may hence decompose, for $0\leq s \leq 1$, the set   $\mathcal L_\perp(s)$ as  
  \begin{equation}
  \label{defLperpa1}
 \mathcal L_{\perp}(s)\equiv  \left(\mathcal C_{\perp, h}^-(0)\times \{-\frac 12\} \right) \cup   
[N_{\rm L}^-(s),N_{\rm L}^+(t)] \cup [N_{\rm L}^+(t),N_{\rm R}^+(t)]\cup [N_{\rm R}^+(t),N_{\rm R}^-(t)] .
 \end{equation}
 The vertices $N_{\rm L}^{\pm}(t)$ and  $N_{\rm R}^{\pm}(s)$ are defined as follows. We  first  set, for $0\leq s \leq 1$
 \begin{equation*}
 N_{\rm L}^-(s)=N_{\rm L}^-(0)=(0,-\frac12,-\frac38,-\frac12)  {\rm \ and  \ }
 N_{\rm R}^-(s)=N_{\rm R}^-(0)=(0,\frac12,-\frac38,-\frac12), 
 \end{equation*}
 so that these points are \emph{not moved}. 
   For the two other points, they are moved, for   $0\leq s \leq 1\slash 2$, in the direction $-\be_3$ according to the formulae
 \begin{equation}
 \label{defLperpa2}
 \left\{
 \begin{aligned}
 &N_{\rm L}^+(s)\equiv N_{\rm L}^-(0) -\frac{s}{4} \be_3+\be_4=\left(0, -\frac12, -\left(\frac 38+\frac{s}{4}\right),\frac12\right) {\rm \ and \ }  \\
 &N_{\rm R}^+(s)\equiv N_{\rm R}^-(0) -\frac{s}{4} \be_3+\be_4=\left(0,\frac12, -\left(\frac 38+\frac{s}{4}\right),\frac12\right).
 \end{aligned}
 \right.
  \end{equation}
Hence, we have  $\displaystyle{N_{\rm L}^+(\frac12)=(0, -\frac12, -\frac 12,\frac12)}$ and 
$\displaystyle{N_{\rm R}^+(\frac 12)=(0, \frac12, -\frac 12,\frac12)}$.
For $ 1\slash 2 \leq s \leq 1$, we are moved in the direction $-\be_4$ according to  
\begin{equation}
 \label{defLperpa3}
\left\{
\begin{aligned}
&N_{\rm L}^+(s)=N_{\rm L}^+(\frac 12)-
(2s-1) \be_4=\left(-\frac12,0,\frac12,\frac32-2s\right)  {\rm \ and \ } \\
&N_{\rm R}^+(s)=N_{\rm R}^+ (\frac 12) -
 (2s-1) \be_4=\left(\frac12,0,\frac12,\frac32-2s\right).
 \end{aligned}
 \right.
 \end{equation}
We next verify that \eqref{diffus2} is satisfied. First since $N_{\rm L}^\pm(0)=N_{\rm L, 0}^{1, \pm}$ and $N_{\rm R}^\pm(0)=N_{\rm R,0}^{1, \pm}$
(see \eqref{grasz}) we deduce that $\mathcal L_\perp(0)=\mathcal L_{\perp}$. On the other hand, since 
$\displaystyle{ N_{\rm L}^+(1)=(0, -\frac12, -\frac 12,-\frac12)}$ and 
$\displaystyle{N_{\rm R}^+(1)=(0, \frac12, -\frac 12,-\frac12), }$
 the points $N_{\rm L}^\pm(1), N_{\rm R}^\pm(1)$  belong the face  $\mathfrak Q^{3,-}_4(1/2,  0) \subset \R^3 \times \{-1\slash 2\}$, so  that the first and second assertions in \eqref{diffus2} are established.  The third  is an immediate consequence of definitions \eqref{defLperpa1},\eqref{defLperpa2}, and \eqref{defLperpa3}.  Finally for the fourth assertion in \eqref{diffus2},  we notice that, for any $s\in [0, 1]$, we have, similar to  \eqref{simic} 
\begin{equation}
\mathcal L_\perp (s) \cap \mathbbmss V_2^+ =\{\rbA_1\},  {\rm \ where  \ } \rbA_1=(0, 0, \frac{3}{8}, -\frac 12 ).
\end{equation}
Since $\mathcal L(1) \subset \mathbbmss V_2^+$ and since $\rbA_1 \not \in \mathcal L (1)$, we deduce the  third assertion, so that the proof of \eqref{diffus2} is complete.

  \begin{figure}[h]
\centering
\includegraphics[height=7.5cm]{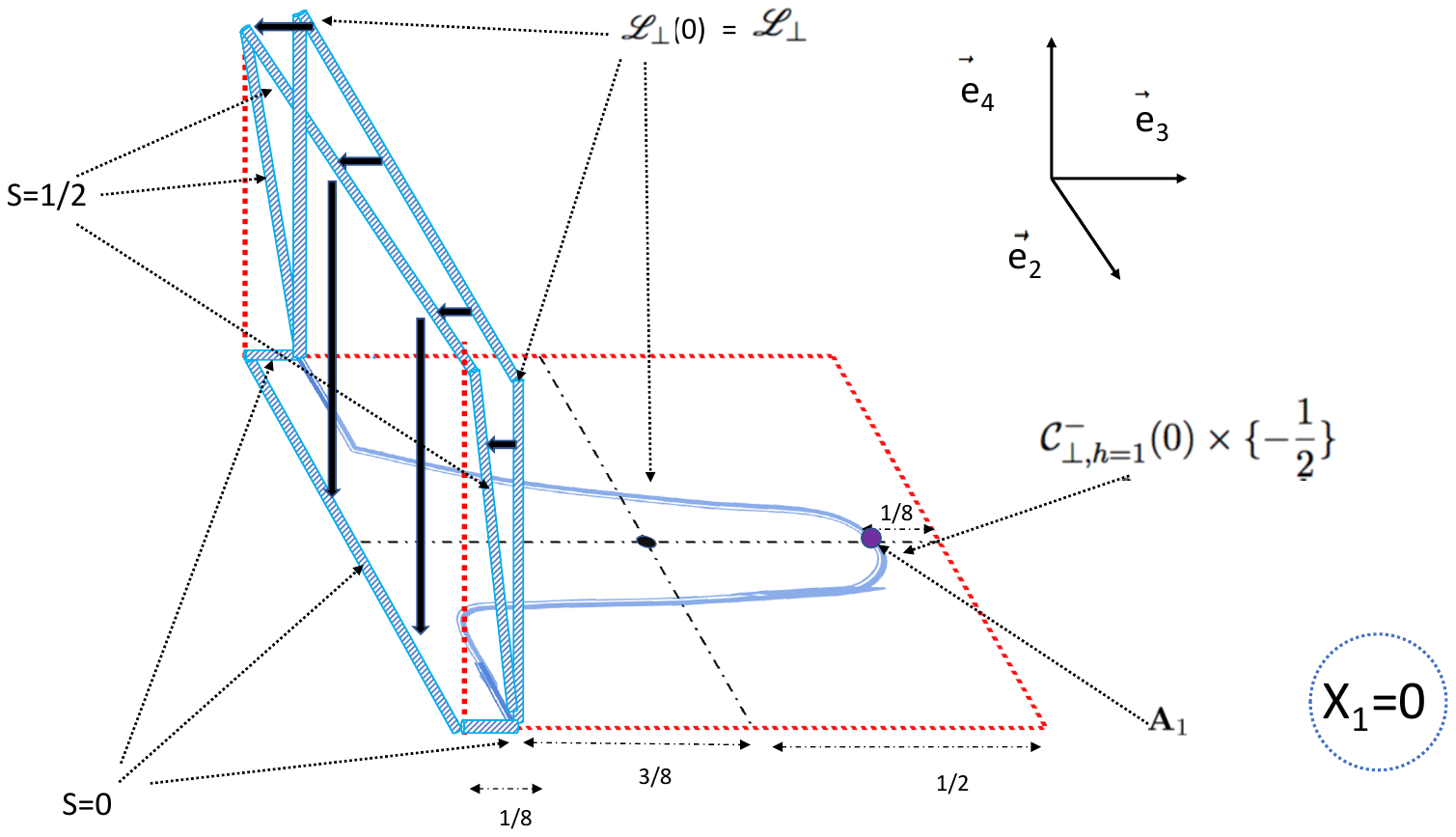}
\caption{  {\it The  deformation of the curve $\mL_\perp$ at  $t=0, 1\slash2$ and $t=1$ in the space $\mathbbmss V_1=(\R\be_1)^\perp$.}}
\label{deformLperp}
\end{figure}

  \begin{figure}[h]
\centering
\includegraphics[height=7.5cm]{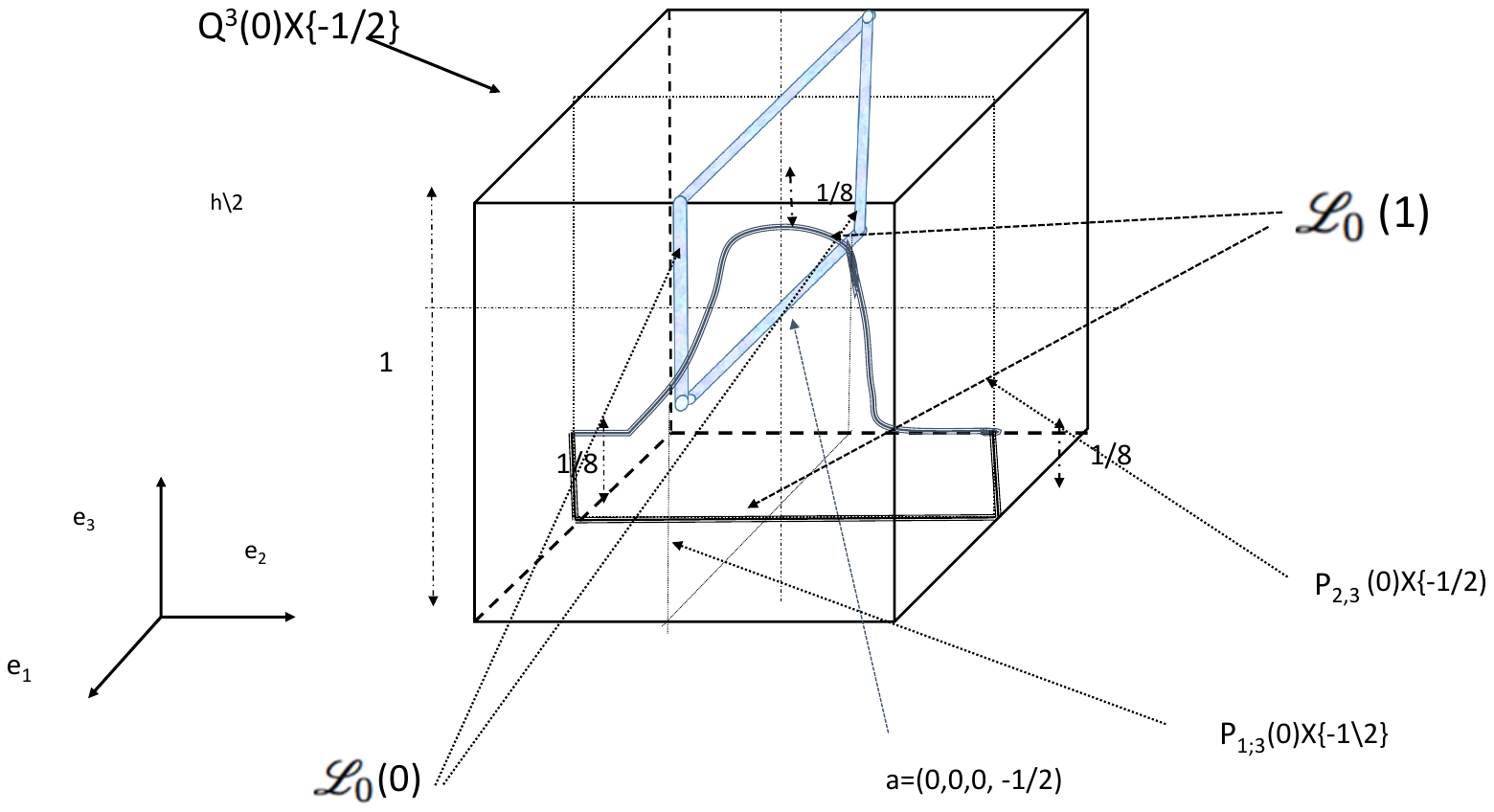}
\caption{  {\it The  linking of the  curves $\mL_0(1)$ and $\mL_\perp(1)$.}}
\label{anneau}
\end{figure}

 \medskip
 \noindent 
 {\it Proof of Lemma \ref{commea} completed}. 
By continuity of the linking number  and of the deformations, we have
\begin{equation*}
\begin{aligned}
\mathfrak m \left( \mL_0, \mL_\perp\right)&=\mathfrak m \left( \mL_0 (0), \mL_\perp\right)   
= \mathfrak m \left( \mL_0 (t), \mL_\perp\right),  {\rm \ for \ any \ } t \in [0, 1], \\
&=\mathfrak m \left( \mL_0 (1), \mL_\perp\right).
\end{aligned}
\end{equation*} 
Similarity, we have 
\begin{equation*}
\begin{aligned}
\mathfrak m \left( \mL_0(1), \mL_\perp\right)&=\mathfrak m \left( \mL_0 (1), \mL_\perp(0)\right)  
= \mathfrak m \left( \mL_0 (1), \mL_\perp(s)\right), {\rm \ for \ any \ } s \in [0, 1] \\
&=\mathfrak m \left( \mL_0 (1), \mL_\perp(1)\right),
\end{aligned}
\end{equation*} 
so that we obtain 
 \begin{equation}
 \label{laboue}
  \mathfrak m\left(\mL_0, \mL_\perp\right)= \mathfrak m\left(\mL_0(1), \mL_\perp(1)\right).
 \end{equation}
We may now take advantage, in view of \eqref{defv4},   that the two curves are planar curves  in the three dimensional affine  space $\mathbbmss V_4(\rbA_0)= \R^3 \times \{-1\slash 2\}$, included in planes which are not parallel, see Figure \ref{anneau}. Using  the method of   crossing numbers, we may then show that  
  $$
   \mathfrak m\left(\mL_0(1), \mL_\perp(1)\right)=1.
  $$
  Combining this identity with identity \eqref{laboue}, we deduce the desired result.  
\end{proof}

\begin{proof}[Proof of Lemma \ref{topo}]   Since  $\partial \rQ^4_{h/2}(\rba)$  is not \emph{a smooth manifold}, and since the framing might appear as discontinuous, we rely instance on the definition  of the Hopf invariant \eqref{linkhopf} based on  the linking number of two pre-images, and argue as  in the proof of Lemma \ref{hopfion2}.  As there, we set, for $Q \in \S^2$, $\displaystyle{L_Q={\Upsilon_{\rba}^h}^{-1}(Q)}$, and consider the points $M=(1, 0, 0)$ and $\nP$  on $\S^2$.  We have, in view of the definition \eqref{hermanos} of $\Upsilon_{\rba}^h$
\begin{equation}
\left\{
\begin{aligned}
&L_{{}_{\nP}}=\mL_{0,\rba} \cup \mL_{\perp,\rba} {\rm \ and \ } \\
&L_M=\left( \mL_{0,\rba}+g^{-1}(0)\varrho_k \be_3 \right) \cup \left(\mL_{\perp,\rba}+ g^{-1}(0)\varrho_k \be_1\right).
\end{aligned}
\right.
\end{equation}
The proof of Lemma \ref{topo} is  then completed adapting identity \eqref{frais} of Lemma \ref{hopfion2}. 
\end{proof}
 \subsubsection{Extension to the cube  $\rQ^4_{h/2}(\rba)$ and energy estimates}
 Finally, we introduce   the extension $\boxbox_{h, \rba}$ of  $\Upsilon_{\rba}^h$ to the cube $\rQ^4_{h/2}(\rba)$ given by 
\begin{equation}
\label{etoo}
\boxbox_{h, \rba} (\rbx)=\Ext_{h/2, \rba} \left(\Upsilon_{\rba}^h\right)(\rbx), {\rm \   for  \  }  \rbx \in \rQ^4_{h/2}(\rba).  
\end{equation}
\begin{lemma}
\label{extension}  The map $\boxbox_{h, \rba}$ is locally Lipschitz on the cube  $\rQ^4_r(\rba)\setminus \{a\}$, where  it possesses a point singularity  of Hopf invariant equal to $2$.  We  have the energy identity
\begin{equation}
\label{engamma}
\rE_3(\boxbox_{h, \rba}) \leq  {\rm \bf K}_{\rm box} h, {\rm \ for \ any \ } h>0, 
\end{equation}
where $ {\rm \bf K}_{\rm box}$ is a universal constant. 
\end{lemma}
\begin{proof}  The first assertions  follow from the results in subsection \ref{cubic}. Concerning  the energy estimate \eqref{engamma}, it  is  a direct consequence  of \eqref{enextension}  and \eqref{lilou},  with 
${\rm \bf K}_{\rm box}= {\rm K _{\rm ext}}{\rm K}_\upgamma$.
  \end{proof}


\section{Proof of Proposition \ref{deform} }
\label{laforme}

\subsection {General outline  of the proof}
\subsubsection{ Crossings of sheaves} 
\label{dreamteam}
We provide in  this section the detailed  construction   the map  $\Gordk$, which is defined on the strip 
$\Lambda=\R^3\times [0, 50]\subset \R^4$,   with values into  $\S^2$, 
 and  corresponds to a  deformation of the Spaghetton map $\Spagk$ to a constant map.  The main part of the construction consist of deforming  
 $\Spagk$ to a  map of trivial homotopy class, the later being deformed to a constant maps thanks to Proposition \ref{trivialextend}.   In the deformation process, the fourth space variable $x_4$ stands for a deformation or \emph{time} parameter. The construction relies on corresponding deformations of the fibers of the sheaf  $\fLkperp$, the sheaf $\fLk$ remaining  unchanged.   The value of $\Gordk$ is  then obtained thanks to the Pontryagin construction or its variant.  
  The guiding idea   for deforming $\fLkperp$  is to "push down" along the $x_3$-axis the set $\fLkperp$,  while keeping  the set  $\fLk$ fixed.  Singularities are  created when two fibers meet and cross.    When the sheaf  $\fLkperp$ has been pushed down sufficiently, then the two sheaves are no longer linked, so that we obtain for the corresponding three-dimensional "time" slices  a map with trivial homotopy class. 

  In order to present the strategy, let us start with a naive approach.  If we would  push   down   the sheaf $\fLkperp$  in the direction $ \be_3$ with \emph{constant speed} $1\slash2$, then  the transported sheaf   at time $s\geq 0$ is  given by  $\displaystyle{\fLkperp -s\slash2 \, \be_3}$.  It  would cross the fixed sheaf $\fLk$ at the \emph{collision times}
    $$\displaystyle{s_\ell=10+ 2\ell h}, {\rm \  where \ } \ell=1, \ldots, 2k-1.$$
     For $s>s_{2k-1}$, the two sheaves would  no longer  be linked.    Collisions occur in the straight parts\footnote{They consist  of segments parallel to $\be_2$ and $\be_3$ respectively.} of the sheaves $\fLk$ and $\fLkperp-  s_\ell \slash2 \, \be_3$.  The points   where the fibers would cross  belong to the cube $[0, 1]^3$ and are given by
  \begin{equation}
 \label{singr3}
 {\rm a}_{i, j, q}^k=(\frac{i}{k}, \frac{j}{k}, \frac{q}{k})=(i\, h, j \, h, q\, h )=h(i, j, q),  {\rm \ for \ }  i, j, q=1, \ldots, k 
  {\rm \ with \ } h=\frac 1 k, 
 \end{equation}
 so that the set of crossing points in $\R^3$ is given by 
 \begin{equation}
 \label{rbAstar}
   \rbA_\star^k \equiv \{{\rm a}_{i, j, q}^k\}_{i, j, q=1, \ldots, k}=(h\mathbb I_k)^3=\boxplus_k^3(h)\subset \fLk \cap [0,1]^3.
   \end{equation}
In view of the first relation in \eqref{rabia},  we also  have 
 \begin{equation}
 \label{singloton}
 {\rm a}_{i, j, q}^k  \in \fLkperp -c\be_3, {\rm \ provided \ } q-\frac{6-c}{h} \in \{1, \ldots, k\}.
 \end{equation}
 Since our aim is to construct $\S^2$-valued maps based on the Pontryagin construction,  we would have to face, if \emph{we would use} constant speed  transportation of $\fLkperp$ as above, the presence of crossing points at the times $s_\ell=10+2\ell h$, $\ell=1, \ldots, 2k-1$. These singularities  are  however incompatible with the Pontryagin construction, so that an adaptation of the  strategy  is required.   
 The main tool which allows  the crossing of fibers, in the  Sobolev  context,   has been  developed so far in Subsection \ref{creahopf}.  It is tailored in such a way  that we may use it here as a "black box". It requires as a preliminary step   the "rounding" of curves near singularities, for which we use the tools developed in Subsection \ref{domain} as another "black box". 
 \subsubsection{Adapting the strategy, dividing into steps}
 The idea we are going to develop  is still to rely \emph{deformations of the sheaves}, so that we denote, for given time 
 $s>0$, by  $\fLk_{j, q} (s)$ and $\fLkperp_{i, q}(s)$ the deformations at time $s$ of the curves  $\fLk_{j, q} $ and $\fLkperp_{i, q}$ respectively, imposing the initial  value conditions  
 $$ 
 \fLk_{j, q} (0)=\fLk_{j, q}  {\rm \ and \ } \fLkperp_{i, q}(0)=\fLkperp_{i, q}.
 $$
  As mentioned, the sheaf $\fLk$ will not move, so that we have throughout
\begin{equation}
\label{notemoved}
\fLk_{j, q} (s)=\fLk_{j, q}  {\rm \ for \ every \ }  s \geq 0.
\end{equation}
In order to define  $\fLkperp(\cdot)$, we  introduce the discrete set of  times $\{\tk_n, n=1, \ldots, 4k\}$ given  by 
 \begin{equation}
 \label{timestep}
 \tk_1=10+\frac{3h}{4} {\rm \ and \ } \tk_{n+1}=\tk_n+ h, {\rm \ so \ that\ } \tk_n= nh +\tau_h,   {\rm \ for \ }  n \in \{1, \ldots, 4k-1\}, 
  \end{equation}
where $\displaystyle{\tau_h=10-\frac{h}{4}}$, $\displaystyle{h=k^{-1}}$, the final time $\tk_{4k}$ being given by  $\tk_{4k}=\tk_{4k-1}+40$.  We  divide the construction into   $4k$  distinct  steps.   Each step  $n$  corresponds to a specific  time interval $[\tk_{n-1}, \tk_{n}]$, with $\tk_0=0$. 
   All these intervals, except the first and the last one,   have hence size $h$. At  each step $n\in \{ 1, \ldots, 4k\}$, we  define the  restriction of  the  map $\Gordk$  to  the corresponding strips in the $\R^4$ space, namely the strip  $\Lambda^k_n$  given by
   \begin{equation}
   \label{meryl}
    \Lambda^k_n=\R^3 \times  [\tk_{n-1}, \tk_{n}], 
    \end{equation}
   taking care that the construction yields  \emph{the same value  on their interfaces}, that is on the time slices\footnote {Notice indeed that $\Lambda^k_n \cap \Lambda^k_{n+1}=\R^3 \times \{\tk_n\}$.}   $\R^3\times \{ \tk_n\}$. 

   \subsubsection{Description of  the deformation off the crossing region}
   We  focus in most parts  of the construction of $\Gordk$ on  the region where fibers  cross, that is on appropriate  neighborhoods of the  cube $[h, 1]^3$ of $\R^3$: Such neighborhoods will be called  \emph{crossing regions}. In these regions,  which contain  the \emph{crossing points} introduced  in \eqref{singr3},    the sheaf $\fLkperp(\cdot)$ undergoes the most delicate transformations. However, \emph{off the crossing regions}, the dynamics has a simple behavior.  For instance, choosing the cube $[0, 1+h]^3$ as a neighborhood of the unit cube,  our construction (see \eqref{poucepied}  and \eqref{dejavu} below) imposes  that 
   \begin{equation}
  \label{presque}
  \fLkperp (s)\setminus [0, 1+h]^3=\left( \fLkperp-c_h(s) \be_3  \right) \setminus[0, 1+h]^3,  {\rm \ for \ } s \in [0, \tk_{4k-1}]. 
  \end{equation}
   The  function $s \mapsto c_h(s)$ represents   a piecewise affine function  defined  on $[0, \tk_{4k-1}]$ by
  \begin{equation}
  \label{defchuisse}
  \left\{
  \begin{aligned} 
  c_h(s)&=\frac{s}{2},{\rm \ for \ } s\in [0, \tk_1], \\
  c_h(s)&=\frac{\tk_{2\ell}}{2},  {\rm \ for \ } s\in [\tk_{2\ell}, \tk_{2\ell+1}]   {\rm \ and \ } \ell=1, \ldots, 2k-1,\\
  c_h(s)&=s-\frac{\tk_{2\ell-1}}{2}, {\rm \ for \ } s\in [\tk_{2\ell-1}, \tk_{2\ell}] {\rm \ and \ } \ell=1, \ldots, 2k-1.  
  \end{aligned}
  \right. 
  \end{equation}
  We
    construct $\Gordk$ accordingly, for $
    (x, x_4)\in 
    \left( \R^3\setminus [0, 1+h]^3 \right) \times  [0, \tk_{4k-1}]$, as
    \begin{equation}
  \label{gordpresque} 
   \Gordk (x, x_4)=\pont\left[\left( \fLkperp-c_h(x_4) \be_3  \right)\cup \fLk\right](x).  
 \end{equation} 
   The frame on the fibers of $\fLk(s)$ and $\fLkperp(s)$  are taken to be the reference frames, so that we omit to mention them in the operator $\pont$.
It follows from \eqref{presque} and \eqref{defchuisse} that outside $[0, 1+h]^3$, the sheaf $\fLkperp$ is translated in the $\be_3$ direction with speeds either equal to $-1$, for the even steps,  or $0$, for the odd steps, except the first one, for which the speed is $1\slash2$,   so that the "average speed" is $-1\slash 2$. 

The precise construction provided in the next subsections involves \emph{two distinct}  crossing regions. For the \emph{odd steps}, we will work on the region $\Ocrosstar$ defined by  
\begin{equation}
\label{hermanstraut}
\Ocrosstar \equiv [0, 1]^3+ \frac{h}{2}(\be_1+\be_2+\be_3)= \underset{i, j, q=1}{\overset k \cup } \overline{\rQ_{h\slash2}^3(\ra_{i, j, q})}, 
\end{equation} 
and  define $\Gordk$ on each of the elementary cubes 
$\rQ_{h\slash2}^3(\ra_{i, j, q})$ using the results  of Subsection \ref{creahopf}. For the construction on the even steps, we use  instead the set $\tOcrosstar$ given  by
 \begin{equation}
\label{hermanstraut2}
\begin{aligned}
\tOcrosstar &\equiv (\frac h 8, 1+\frac{h}{8}]^3+ \frac{h}{2}(\be_1+\be_2+\be_3)\\
&= 
\left(\Ocrosstar + \frac h 8 \be_3\right)\setminus [0, 1+\frac{h}{2}]^2 \times \{\frac {5h} {8}\}.  
\end{aligned}
\end{equation} 
Both $\Ocrosstar$ and $\tOcrosstar$ are subsets of $[0, 1+h]^3$.
\subsubsection{The  preliminary step}
It corresponds to our   first step, aimed to define  the map $\Gordk$ on $\Lambda_1^k=\R^3 \times [0, \tk_1]$.  
We move, along the ideas  exposed in  Subsection \ref{dreamteam}, the sheaf $\fLkperp$	 along   the constant vector field $\displaystyle{1 \slash 2 \vec X_0=- 1\slash2\be_3}$ while keeping $\fLk$ fixed  throughout. The purpose is  to bring $\fLkperp$ sufficiently close to $\fLk$, so to enter the crossing region $\tOcrosstar$.   
For  $0\leq s \leq \tk_1$,  the sheave   $\fLkperp(s)$ is hence translated with constant speed, without touching $\fLk$. It is hence given by
 \begin{equation}
 \label{poucepied}
 \fLkperp (s)\equiv \fLkperp-\frac{s}{2}\be_3, \,  {\rm \  for \ }  s \in [0, \tk_1]=[0, 10+\frac{3h}{4}].
\end{equation}
 This is consistent with \eqref{presque}, definition \eqref{presque} being simply extended to the whole sheaf $\fLkperp(s)$. As there, we  define the map $\Gordk$ on the strip $\Lambda^k_1\equiv \R^3\times [0,\tk_1]$ accordingly as 
\begin{equation}
\label{stepzero}
\Gordk (x, s) =\pont_\varrho\left [(\fLkperp -\frac{s}{2} \be_3) \cup \fLk\right](x), \forall x \in \R^3  {\rm  \ and \ } s \in [0, \tk_1].
\end{equation}
We have:

\begin{proposition} 
\label{restriction1}
The map  $\Gordk$  defined in \eqref{stepzero} is Lipschitz on the strip $\Lambda^k_1$. We have 
\begin{equation}
\label{restriction11}
\vert  \nabla_{_4} \Gordk(\rbx)  \vert \leq {\rm C}^0_{\rm flow } k,  {\rm \ for \ }   \rbx\in \Lambda_1^k
\end{equation}
and, setting  ${\rm K}^0_{\rm flow}= 11 (60)^3({{\rm C}^0_{\rm flow }})^3$,
\begin{equation}
\label{energie1} 
\int_{\Lambda^k_1} \vert \nabla_{_4} \Gordk\vert^3 \leq 11 (60)^3({{\rm C}^0_{\rm flow }})^3  k^3\leq  {\rm K}^0_{\rm flow}k^3.
\end{equation}
\end{proposition}
\begin{proof}Inequality \eqref{restriction11} is a direct consequence of inequalites \eqref{constantgord}   and \eqref{constantgord2} of   Remark \ref{constantspag}. For inequality   \eqref{energie1}, we notice that   $ \Gordk$ is constant  on the complement of $[-30, 30]^3\times [0, \tk_{4k-1}]$, so that its gradient vanishes there. Therefore,  integrating  on the set $[-30, 30]^3\times [0, \tk_{1}]$, we obtain the desired result, observing that  $0<\tk_1\leq 11$. 
\end{proof}

 \medskip
We verify that  $\Gordk$ is indeed  a deformation of $\Spagk$, since
  \begin{equation}
\label{pointzero}
\Gordk(\cdot, 0)=\Spagk(\cdot)   {\rm \ on \ } \R^3. 
\end{equation}
We notice  also
 that $\fLkperp(\tk_1)$ has come close to $\fLk$. Indeed, we have at time $\tk_1$, setting 
 $\displaystyle{\fL_{\star, 1}^{k, \perp}=   \underset{i=1}{\overset {k}\cup} \fLkperp_{i, 1}}$, the relations 
\begin{equation*}
\left\{
\begin{aligned}
\fLkperp(\tk_1) \cap \tOcrosstar &=\left(  \fLkperp_{\star, 1}-\frac{\tk_1}{2}\be_3\right)\cap \tOcrosstar   = 
\underset{i=1}{\overset {k}\cup}\{ih\} \times  
[\frac{h}{2}, (1+\frac h2 ]  \times \{1+\frac{5h}{8}\},  \\
\fLkperp(\tk_1) \cap \tOcrosstar&=\emptyset {\rm \ for \ } s \in [0, \tk_1), 
\end{aligned}
\right. 
\end{equation*} 
so that $\tk_1$ represents the entrance time into $\tOcrosstar$. The lowest fibers, with respect  to the $x_3$ coordinate, 
 of $\fLkperp (\tk_1)$, i.e.  the curves $\fLkperp_{i, 1}(\tk_1)$  for $i=1, \ldots,  k$, are  at distance $5h/8$ of  the upper fibers of $\fLk$, i.e. the curves $\fLk_{j, k}$ for $j=1, \ldots,  k $.  

  \begin{figure}[h]
\centering
\includegraphics[height=10cm]{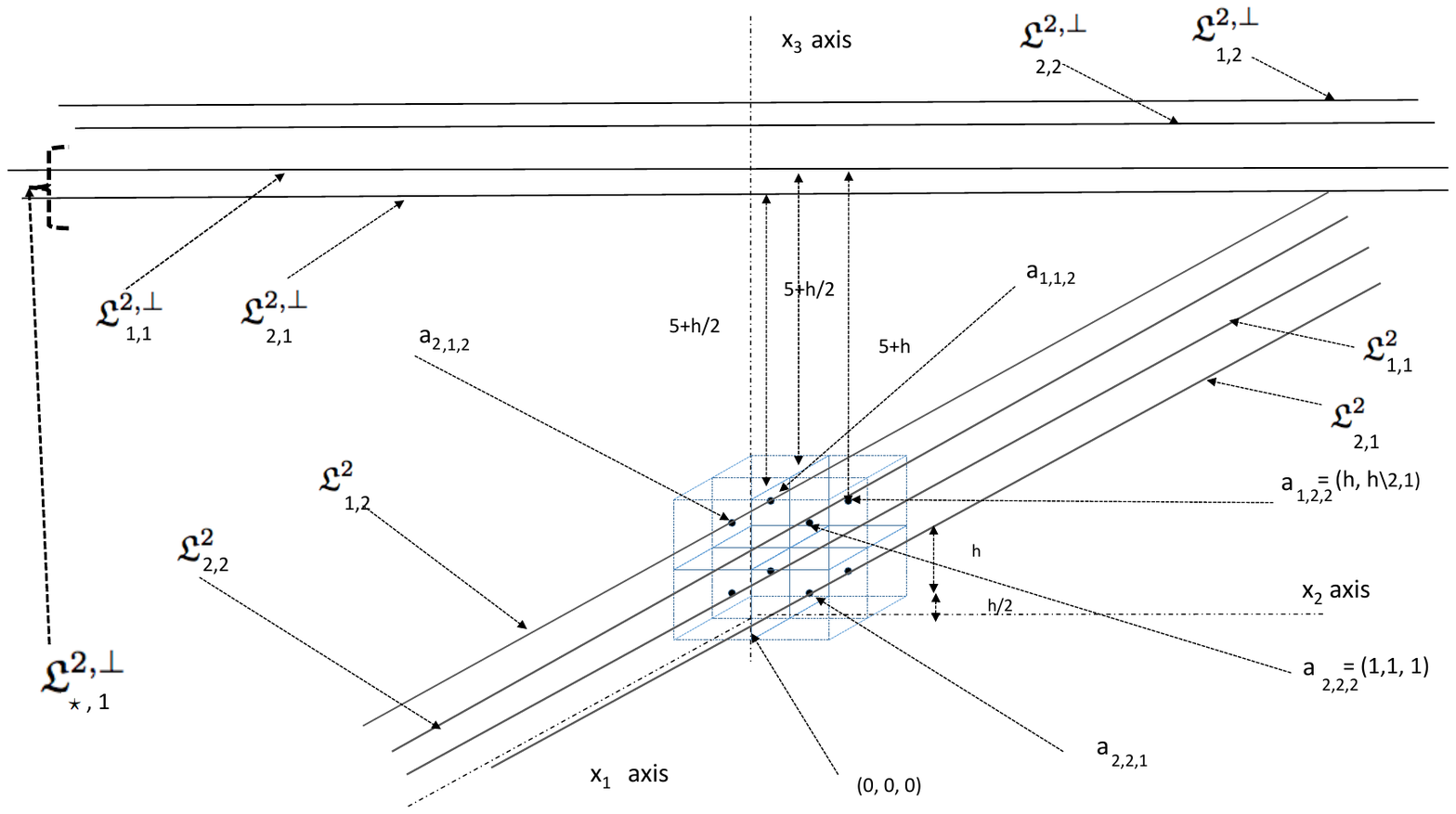}
\caption{  {\it  A zoom of  the crossing $\Ocrosstar$ area at time $0$, $k=2$.}}
\label{timetk2}
\end{figure}

\subsubsection{The iterative argument}
The main part of the rest of the construction is to define the map $\Gordk$  on the set $\R^3 \times [\tk_1, \tk_{4k-1}]$. As mentioned, we proceed by steps,  constructing at step $n$ the   map $\Gordk$ on the  strip $\Lambda_n^k$. Steps go by \emph{pairs of even and off steps}, these pairs of steps being indexed by an integer $\ell\in \{1, \ldots, 2k-1\}$:  

\smallskip
- On steps with \emph{even index}  $n=2\ell$, we deform   the fibers   $\fLkperp_{i, q}(\tk_{2\ell-1})$  intersecting   $\tOcrosstar$ using the flow generated by the vector field $\vec X_1^k$, whereas the fibers  $\fLkperp_{i, q}(\tk_{2\ell-1})$  which do not intersect $\tOcrosstar$  are moved by the constant flow $\vec X_0$.  As a resulting effect, the fibers intersecting  $\tOcrosstar$ remain smooth, but are  \emph{curved} near the singularities defined in \eqref{singr3} which are on their way. This procedure allows them to \emph{circumvent  these crossing points}, and is  a preliminary step for the \emph{next odd step}, where crossing  takes place.   We  define $\Gordk$ accordingly, using the Pontryagin's construction or its variant.  The restriction of the map $\Gordk$  to the corresponding sets $\Lambda_{2\ell}$ is hence \emph{Lipschitz}.  
 
 \smallskip
 - On steps with \emph{odd index} $n=2\ell+1\geq 3$, we use the extension operator and the results in Subsection \ref{creahopf}, in particular the map $\Upsilon_{\rba}$,  to allow "fibers to cross".  The map will then have space-singularities of the form $(\ra, (\tk_{2\ell}+\tk_{2\ell+1})\slash2)$, where
$\ra$ is of the form \eqref{singr3}.

\smallskip
  An  salient feature of this part of the construction is that, at the end of the  odd step $2\ell-1$, with $\ell=1, \ldots, 2k$,  we obtain the following  identities, stated as property\footnote{It corresponds to imposing a  "boundary value" on the time slice  $\R^3 \times \tk_{2\ell-1}$} $\mathcal {BV}(\ell)$
\begin{equation}
\label{mochouille}
\mathcal {BV}(\ell)
\left\{
\begin{aligned}
\fLkperp (\tk_{2\ell-1})  &=\fLkperp-\frac{\tk_{2\ell-1}}{2} \be_3=\fLkperp-c_h\left(\frac{\tk_{2\ell-1}}{2}\right) \be_3,  \\
\Gordk(x, \tk_{2\ell-1})&= \pont \left[( \fLkperp-\frac{\tk_{2\ell-1}}{2} \be_3) \cup \fLk \right](x),  {\rm \ for \ } x\in \R^3.
\end{aligned}
\right.
\end{equation}
  The r.h.s of the second identity is well defined in 
 view of Remark \ref{toutcomme}, noticing that condition \eqref{ondiment}  holds for $c=\tk_{2\ell-1}\slash 2$.
 The resulting effect is that, after a pair of even and odd steps,  the sheaf $\fLkperp$ is simply  translated in the $\be_3$ direction, \emph{as if it were moved downwards with constant speed $-1\slash2$}. Identity $\mathcal {BV}(\ell)$ is  consistent  with \eqref{presque} and the second property in \eqref{defchuisse}. In view of\eqref{stepzero}, Property $\mathcal {BV}(\ell)$, that is identities  \eqref{mochouille},  is   satisfied  for $\ell=1$. The   construction  has hence as an iterative character: We start assuming $\mathcal {BV}(\ell)$, then construct $\Gordk$   on $\R^3 \times[\tk_{2\ell-1}, \tk_{2\ell+ 1}]$   so that $\mathcal {BV}(\ell+1)$ holds,  until we finally reach $\ell={2k-1}$. 
 
 We next provide a complete description of the construction.
\subsection{Description of the  even steps: Rounding fibers near singularities}
\subsubsection{Definition  of $\Gordk$ on $\Lambda_{2\ell}^k$}
Let us  assume     that  we are given $\ell \in \{1, \ldots, 2k-1\} $, that such the odd step $2\ell-1$ has been completed and such  that \eqref{mochouille} holds at time $\tk_{2\ell-1}$. This is for instance the case for $\ell=1$, as seen above.   Let us explain next how, starting from  identity \eqref{mochouille} at time $\tk_{2\ell-1}$, we construct $\Gordk$  for the even step $n=2\ell$. We  introduce the condition  for $i,q \in \{1, \ldots, k\}$, 
\begin{equation}
\label{ahaut}
 \fLkperp_{i, q}(\tk_{2\ell-1}) \cap \tOcrosstar= (\fLkperp_{i, q}- \frac{\tk_{2\ell-1}}{2}\be_3) \cap \tOcrosstar \not =\emptyset, 
\end{equation}
  which defines entrance at time $\tk_{2\ell-1}$ in the crossing region  $\tOcrosstar$ of the fiber $\fLkperp_{i, q}(\tk_{2\ell-1})$. We  define $\fLkperp_{i, q}(s)$  for $i, q=1, \ldots, k$ and  $s \in [\tk_{2\ell-1}, \tk_{2\ell}]$ by the formulas  
\begin{equation}
\label{definepaire}
\left\{
\begin{aligned}
\fLkperp_{i, q}(s)&=\DefLkperp_{i, q}(\tk_{2\ell-1}, s-\tk_{2\ell-1}),   {\rm   \ if    \ }  \eqref{ahaut}{\rm \ holds, \ and  \ } \\
\fLkperp_{i, q}(s)&=\fLkperp_{i, q}-\left(s-\frac{\tk_{2\ell-1}}{2}\right) \be_3, \,   {\rm   \ if \  }  
\eqref{ahaut}  {\rm \ does \ not \ hold}.
\end{aligned}
\right. 
\end{equation}
In other words, fibers  $\fLkperp_{i, q}(\tk_{2\ell-1})$ such that \eqref{ahaut} holds are rounded near the singularities, whereas the others are translated down with constant speed equal to $1$.
This  again is consistent with \eqref{presque},  \eqref{defchuisse} and \eqref{mochouille}, and yields the same value on the sets where $\Gordk$ has already been defined.  Moreover, since 
$\DefLkperp_{i, q}(\tk_{2\ell-1},0)=\fLkperp_{i, q} (\tk_{2\ell-1})$, it yields at time $s=\tk_{2\ell-1}$ the same result as \eqref{mochouille}.
On the other hand, condition \eqref{ondiment}  holds for $c=\tk_{2\ell-1}\slash 2$, so that, 
in view of Remark \ref{toutcomme}, we may then define $\Gordk$  using the variant of the Pontryagin operator, namely
\begin{equation}
\label{flashgordon}
\Gordk(x, s)=\tpont \left  [\fLkperp(s) \cup \fLk\right ](x), {\rm \ for \ }  (x, s) \in \Lambda_{2\ell}^k. 
\end{equation}
The  results in Subsection \ref{variante}, in particular Lemma \ref{gradur} yield:
 
 \begin{proposition}
 \label{flahut} Let $\ell$ be in $\{1, \ldots, 2k-1\}$, and  let $\Gordk$ be defined by \eqref{flashgordon} on $\Lambda_{2\ell}^k$.   Then $\Gordk$ is Lipschitz on $\Lambda_{2\ell}^k$,  Property $\mathcal {BV}(\ell)$ holds, and   we have the gradient estimate
 \begin{equation}
 \label{flahut0}
 \vert \nabla_4 \Gordk(x, s) \vert \leq {\rm C}_{\rm def} k, {\rm \ for \ } (x, s) \in \Lambda_{2\ell},  
   \end{equation}
 where ${\rm C}_{\rm def }>0$ is  given in Lemma  \ref{gradur}. We obtain, by integration, with ${\rm K}_{\rm def}=(60)^3{\rm C}_{\rm def}^3$, 
 \begin{equation}
 \label{flahut1}
 \int_{\Lambda_{2\ell} } \vert \nabla_4 \Gordk(\rbx ) \vert^3 {\rm d} \rbx  \leq   {\rm K}_{\rm def} k^2. 
 \end{equation}
 
 \end{proposition}
\begin{proof}  Estimate \eqref{flahut0}  follows immediately applying on one hand estimate \eqref{gradur1} of  Lemma \ref{gradur} for the fibers $\fLkperp_{i, q}(s)$ such that \eqref{ahaut} holds, and using  inequalities \eqref{constantgord}   and \eqref{constantgord2} of   Remark \ref{constantspag}  for the fibers $\fLkperp_{i, q}(s)$ such that \eqref{ahaut} does not hold, as well as for the fibers of $\fLk$.  Assuming moreover that ${\rm C}_{\rm flow}^0 \leq {\rm C}_{\rm def}$, the conclusion follows. 

\smallskip
For inequality 
   \eqref{flahut1}, we notice that   $ \Gordk$ is constant  on the complement of $[-30, 30]^3\times [\tk_{2\ell-1}, \tk_{2\ell}]$, so that its gradient vanishes there. It suffices therefore to integrate  on the set $[-30, 30]^3\times [\tk_{2\ell-1}, \tk_{2\ell}]$, whose measure is 
   $\displaystyle{(60)^3h=(60)^3 k^{-1}}$. The conclusion follows. 
    \end{proof}
\subsubsection{Definition \eqref{definepaire} and condition \eqref{ahaut} revisited} 
\label{defnore}
Condition  \eqref{ahaut} involves  only the index $q$ and not the index $i$. It  can be rephrased as 
$$6+qh-\frac{\tk_{2\ell-1}}{2}\in (\frac{5h}{8}, 1+\frac{5h}{8}],$$
 which yields, in view of the definition \eqref{timestep} of $\tk_{2\ell-1}$, 
 the relation $\displaystyle{k+q-\ell \in (0, k]}$.  
Taking into account  the fact that  we deal with integers, \eqref{ahaut} is equivalent  to  the simple relation 
\begin{equation*}
\ell-q \in \{0, \ldots, k-1\}.  
\end{equation*}
Introducing  the set of indices 
\begin{equation}
\label{encorelui}
\Gamma(\ell)=\{q \in \{1, \ldots, k\}, {\rm \ \  such \ that \ } \eqref{ahaut} {\rm \ holds \ } \}, 
\end{equation}
 we obtain hence a more explicit definition of $\Gamma(\ell)$, namely
\begin{equation}
\label{setindique}
\left\{
\begin{aligned}
\Gamma(\ell)&=\{1, \ldots, \ell\},   {\rm \ if \ } \ell \leq k,  \\
\Gamma(\ell)&=\{\ell-k+1, \ldots, k \},                     {\rm \ if \ }  k+1 \leq \ell  \leq 2k-1. 
\end{aligned}
\right.
\end{equation}
In view of definitions \eqref{definepaire}  and \eqref{flashgordon}, we introduce the notation 
\begin{equation}
\label{derrick}
\left\{
\begin{aligned}
&\fL_{\star, \ell}^{k, \perp}(s)= { \underset {q\in \Gamma(\ell)} \cup}\left( \underset{i=1}{\overset {k}\cup} \fL^{k, \perp}_{i, q}(s)\right) {\rm \ and \ }  \\
&\DefLkperp_{\star, \ell}(\tk_{2\ell-1}, s-\tk_{2\ell-1})= \underset {q\in \Gamma(\ell)}\cup \left(   \underset{i=1}{\overset {k}\cup}\DefLkperp_{i, q}(\tk_{2\ell-1}, s-\tk_{2\ell-1}) \right).
\end{aligned}
\right.
\end{equation}
We recast \eqref{flashgordon} as, for $(x, s) \in \Lambda_{2\ell}$, with $\fL_{\star, \ell}^{k, \perp}=\fL_{\star, \ell}^{k, \perp}(0) $
\begin{equation}
\label{flashouille}
\begin{aligned}
\Gordk(x, s)=
  & \pont\left[
 \underset {q  \not \in \Gamma(\ell)} \cup \left( \underset{i=1}{\overset {k}\cup} \fLkperp_{i, q}\right)-\left(s-\frac{\tk_{2\ell-1}}{2}\right) \be_3
   \right]  
  \wedgetrois \pont[ \fLk]  \\
  &\wedgetrois\tpont \left  [\DefLkperp_{\star, \ell}(\tk_{2\ell-1}, s-\tk_{2\ell-1})
 \right](x). \\
\end{aligned}
\end{equation}

\begin{remark} {\rm Relations \eqref{setindique} may be interpreted as follows:

\smallskip
 - If $1\leq \ell \leq k$,  the  layer of fibers  $\underset{i=1}{\overset {k}\cup} \fL^{k, \perp}_{i, \ell}(.)$ enters the crossing region $\tOcrosstar$ at the  very  end of step  $2\ell-1$, i.e. at time $\tk_{2\ell-1}$,   but enters only later, at time $\tk_{2\ell-1} +h\slash8$  
 the crossing region $\Ocrosstar$.
 The layers   $\underset{i=1}{\overset {k}\cup} \fL^{k, \perp}_{i, q}(.)$, for $q=1, \dots,  \ell-1$ are curved  down, but remain in the crossing region  $\Ocrosstar$.  The crossing region  $\Ocrosstar$ therefore intersects $\ell$ such layers.

\smallskip
- If  $k<\ell \leq 2k-1$,  the layer    $\underset{i=1}{\overset {k}\cup} \fL^{k, \perp}_{i, \ell-k}(\cdot)$ exits  the crossing region  $\tOcrosstar$  at the very end
 \footnote {It remains however until time $\tk_{2\ell-1}+h\slash 8$ inside $\Ocrosstar$.}  of step $2\ell-1$  whereas the layers   $\underset{i=1}{\overset {k}\cup} \fL^{k, \perp}_{i, q}(.)$, for $q=\ell-k+1, \dots,  k$ are curved  down, but remain in the crossing region.  The crossing region $\Ocrosstar$ intersects   $2k-\ell$ such layers.
}
\end{remark}
\subsubsection{The shape of the fibers at the end of the even steps}
We provide a few comments concerning the shape of the fibers at  time $\tk_{2 \ell}$, see also Figure \ref{timetk2}. We first observe that,
 in view of \eqref{ise} and \eqref{vcs}  of Remark \ref{proche}, we have for $s\in [\tk_{2\ell-1}, \tk_{2\ell}]$
\begin{equation*}
\left\{
\begin{aligned}
\fLkperp_{i, q}(s)&= \fLkperp_{i, q}-(s-\frac{\tk_{2\ell-1}}{2})\, \be_3,  {\rm \ if \ } q \not \in \Gamma(\ell),  \\
\fLkperp_{i, q}(s)&\subset 
  \left( \fLkperp_{i, q}-(s-\frac{\tk_{2\ell-1}}{2})\, \be_3\right)+\left(\frac h2 \be_2+[0, 1]^2 \times  \left  [0, \frac{3h}{4}\right]\right),
  {\rm \ if \ }  q \in \Gamma(\ell). 
\end{aligned}
\right.
\end{equation*}
At time $s$, all fibers of $\fLkperp$ have  hence  been translated by $\displaystyle{ -(s-\frac{\tk_{2\ell-1}}{2})\be_3}$, except the fibers  in $\fLkperp_{\star, q}$, for 
$q \in \Gamma(\ell)$,  which have been \emph{rounded near  points in $\rbA_\star^k$}, defined in \eqref{rbAstar},  in order to avoid collision with $\fLk$, 
 which they would otherwise have crossed.  As a a matter of fact, the subset of $\rbA_\star^k$ of points near which $\fLkperp_{\star, q}(\cdot)$ are rounded is provided by the set
 \begin{equation}
\rbA_\ell^k=\underset{ q \in \Gamma(\ell)} \cup \left ( \underset {i, j=1} {\overset k \cup} \{\ra_{i, j, k+q-\ell}^k\}\right) =\underset{ q' \in \tilde {\Gamma}(\ell) } \cup \left ( \underset {i, j=1} {\overset k \cup} \{\ra_{i, j, q'}^k\}\right), 
\end{equation}
where we have set
\begin{equation}
\label{tildegamma}
\tilde {\Gamma}(\ell)=\left\{q'=k+q-\ell, \,  q \in \Gamma (\ell)\right\}=\{1, \ldots, k\}\cap (k-\ell, 2k-1-\ell].
\end{equation}
It follows from \eqref{tildegamma}  that $ \tilde {\Gamma}(\ell)=\{k-\ell+1, \ldots, k \}$ if $\ell \leq k$ and  
$\tilde{\Gamma}(\ell)=\{1, \dots, 2k-\ell \}$ if $k<\ell \leq 2k-1$\footnote{For instance, $\tilde{\Gamma}(1)=\{k\}$, $\tilde{\Gamma}(2)=\{k-1, k\}$, \ldots,   $\tilde{\Gamma}(k)=\{1, \ldots, k\}$, $\tilde{\Gamma}(k+1)=\{1, \ldots, k-1\}$,   and $\tilde{\Gamma}(2k-1)(1)=\{1\}$.}.
We introduce    the corresponding  subset $\mathcal O_{\rm cross, \ell}^h$ of $\Ocrosstar$ given by 
\begin{equation}
\label{crossell1}
\mathcal O_{\rm cross, \ell}^h=\underset {a \in \rbA_\ell^k} \cup \overline{\rQ^3_{h/2}(a)}=
\frac h2 (\be_1+\be_2+\be_3)+[0, 1]^2\times  
 \underset{q' \in \tilde{\Gamma}(\ell)} \cup [q'h, (q'+1)h], 
\end{equation}
so that 
\begin{equation}
\label{crossell2}
\left\{
\begin{aligned}
\mathcal O_{\rm cross, \ell}^h&=\frac h2 (\be_1+\be_2+\be_3)+[0, 1]^2\times [1-\ell h, 1],  {\rm \ if \ } \ell \leq k,  {\rm \ and \ \ }  \\
\mathcal O_{\rm cross, \ell}^h &=\frac h2 (\be_1+\be_2+\be_3)+[0, 1]^2\times [0, (2k-\ell)h],  {\rm \ if \ } k<\ell \leq 2 k-1. 
\end{aligned}
\right.
\end{equation}
 The "rounding" process mentioned above  is  obtained  thanks to   the transformation  of $\fL_{\star, \ell}^{k, \perp}-(\tk_{2\ell-1}\slash 2) \be_3$ into the curve $\displaystyle{\DefLkperp_{\star, \ell} (\tk_{2\ell-1},h)}$ at time $s=\tk_{2\ell}$.  We have indeed, since 
 $$\tk_{2 \ell}-\frac{\tk_{2\ell-1}}{2}=\frac{\tk_{2\ell+1}}{2}, $$ 
the relations
 \begin{equation}
 \label{laborieux}
\left\{
\begin{aligned}
\fLkperp_{i, q}(\tk_{2\ell})&= \fLkperp_{i, q}-\frac{\tk_{2\ell+1}}{2}\be_3 \ {\rm \ and \ }
\fLkperp_{i, q}(\tk_{2\ell}) \cap  \mathcal O_{\rm cross, \ell}^h=\emptyset,  {\rm \ if \ } q \not \in \Gamma(\ell),  \\
\fLkperp_{i, q}(\tk_{2\ell})&=\DefLkperp_{i, q} (\tk_{2\ell-1},h) \subset \mathcal O_{\rm cross, \ell}^h,  {\rm \ if \ }  q \in \Gamma(\ell). \\
\end{aligned}
\right.
\end{equation}
It follows
\begin{equation}
\label{laborieux2}
\fLkperp_{i, q}(\tk_{2\ell})\subset 
  \left( \fLkperp_{i, q}-\frac{\tk_{2\ell+1}}{2}\be_3\right)+\left(\frac h2 \be_2+[0, 1]^2 \times  \left  [0, \frac{3h}{4}\right]\right).
   \end{equation}
 We next focus on the shape  of $\fLkperp(\tk_{2\ell})$ near the points of $\rbA_\ell^k$  (see Figure \ref{timetk2}).   We have:
 
 \begin{proposition} 
 \label{roundupe}
 Let $q\in \Gamma(\ell)$ and let $q'=k+q-\ell$. Then, for $i, j \in \{1, \ldots, k\}$,  we have:
 \begin{equation}
 \label{roustanvt}
 \left\{
 \begin{aligned}
 \fLkperp(\tk_{2\ell})\cap\rQ^3_{h\slash 2} (\ra_{i, j, q'}^k)&=\mathcal C_{\perp, h}^-(\ra_{i, j, q'}^k)  {\rm \ and \ } \\
  \fLk(\tk_{2\ell})\cap\rQ^3_{h\slash 2} (\ra_{i, j, q'}^k)&={\rm D}_{0, h}(\ra_{i, j, q'}^k), 
 \end{aligned}
 \right.
 \end{equation}
where, for $a \in \R^3$,  the curves  $\mathcal C_{\perp, h}^-(a)$ and ${\rm D}_{0, h}(a)$ are defined in Subsection \ref{relevant} (see also Figure \ref{cube11bis}).  We have $q' \in \tilde {\Gamma}(\ell)$  and 
 \begin{equation}
 \label{rondupe2}
 \Gordk(x, \tk_{2\ell})=\upgamma_{\ra_{i, j, q'}}^{h,- }(x),  {\rm \  for \ } x \in \rQ^3_{h\slash 2} (\ra_{i, j, q'}^k), 
 \end{equation}
 where  the map $\upgamma_{a}^{h,- }:\rQ^3_{h\slash 2} (a)\to \S^2$ is defined in \eqref{alcatel}. 
 \end{proposition}
\begin{proof} Since $q\in \Gamma(\ell)$, It follows from definition \eqref{setindique}  that $q'\in \{1, \ldots, \ell\}.$  Moreover, since  
$$\tk_{2\ell-1}\slash2=5+ \ell h-5h\slash 8,$$
 Going back to Lemma \ref{round-up}, we verify, since $kh=1$,  that  relation  \eqref{ondiment} is satisfied with the choice  $c=\tk_{2\ell-1}\slash2, q$ and $q'=k+q-\ell$. Identity \eqref{roustantv} are then a direct consequence of the corresponding identities \eqref{sel} in Lemma \ref{round-up}.
   \end{proof}

   \begin{figure}[h]
\centering
\includegraphics[height=9cm]{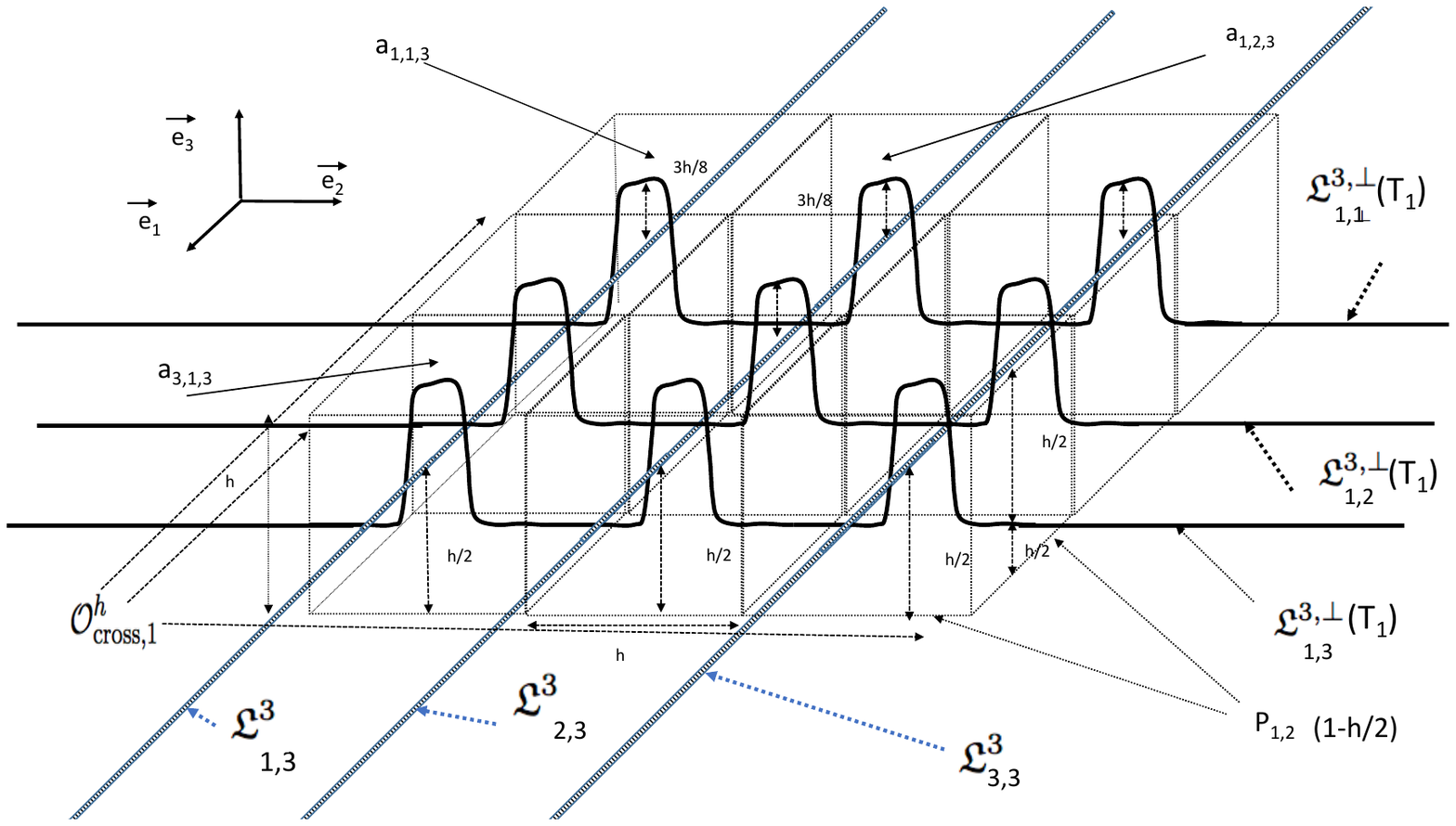}
\caption{  {\it The curves  $\fLkperp_{i, 1}(\tk_2)$ and $\fLk_{j, k}$ at time $\tk_2$, $k=3$ corresponding to the end of the even step in iteration $\ell=1$.  All fibers $\fLkperp_{\star, 1}$ have entered the  set
$\mathcal O_{\rm cross, 1}^h$ formed of $k^2=9$ cubes  $\overline{\rQ_{ h\slash 2} (\ra)}$,   of edge $h/2$.}}
\label{timetk2}
\end{figure}
\subsection {Description of the odd steps: Crossings of fibers thanks to singularities}
\label{odieux}
In order to set up an iterating procedure,  our aim is to recover, at time $\tk_{2\ell+1}$,  Property $\mathcal {BV}(\ell+1)$, that is, we wish to have identities  \eqref{mochouille} with $\ell$ replaced by $\ell+1$, i.e.  to obtain
\begin{equation}
\label{mochouillebis}
\mathcal {BV}(\ell+1) \left\{
\begin{aligned}
&\fLkperp (\tk_{2\ell+1})  =\fLkperp-\frac{\tk_{2\ell+1}}{2} \be_3  {\rm  \ and \ }  \\
&\Gordk(x, \tk_{2\ell+1})= \pont \left[( \fLkperp-\frac{\tk_{2\ell+1}}{2} \be_3) \cup \fLk \right](x),  {\rm \ for \ } x\in \R^3.
\end{aligned}
\right.
\end{equation}
We begin with the construction off the crossing region $\Ocrosstar$.
\subsubsection{Construction of $\Gordk$ on $\Lambda_{2\ell+1}\setminus\left( \mathcal O^h_{{\rm cross},\ell}\times [\tk_{2\ell}, \tk_{2\ell+1}]\right)$}
\label{porcstar}

We notice that, at time $\tk_{2\ell}$,   conditions $\mathcal {BV}(\ell)$ are already satisfied on a part of $\R^3$, since we have, in view of  
 \eqref{laborieux}
 \begin{equation}
 \label{dejafait}
 \Gordk(x, \tk_{2\ell})= \pont \left[( \fLkperp-\frac{\tk_{2\ell+1}}{2} \be_3) \cup \fLk \right](x),  {\rm \ for \ } x\in \R^3\setminus \mathcal O^h_{{\rm cross},\ell}. 
 \end{equation}
 Hence, we choose, since, by definition,   $c_h(s)=\tk_{2\ell+1}\slash 2$  for $s \in  [\tk_{2\ell}, \tk_{2\ell+1}]$ (see \eqref{defchuisse}), to define 
 $\Gordk(x, s) $ , for   $(x, s) \in \left(\R^3\setminus \mathcal O^h_{{\rm cross},\ell} \right) \times  
  [\tk_{2\ell}, \tk_{2\ell+1}]$,  as
 \begin{equation}
 \label{dejavu}
 \Gordk(x, s)= \pont \left[( \fLkperp-\frac{\tk_{2\ell+1}}{2} \be_3) \cup \fLk \right](x).
  \end{equation}
 In other words, nothing has moved on the set $\R^3\setminus \mathcal O^h_{{\rm cross},\ell}$. Definition  \eqref{dejavu} is   consistent with \eqref{presque},  \eqref{defchuisse}  and yields the same value on the sets where $\Gordk$ is already defined.  Moreover, we have as in the previous subsections, in particular Proposition \ref{flahut}, the gradient estimate \begin{equation}
 \label{restriction1bis}
 \left   \vert \nabla_4 \Gordk(\rbx) \right \vert  \leq   {\rm C}_{\rm flow}^0 k  {\rm \ for \ } \rbx \in \left(\R^3\setminus \mathcal O_{\rm cross, \ell}^{h}\right)\times [\tk_{2\ell}, \tk_{2\ell+1}].
 \end{equation}
 Arguing as in  the proof of  Proposition \ref{flahut}, we are led to the integral estimate
 \begin{equation}
 \label{integral}
 \rE_3\left( (\Gordk, \left(\R^3\setminus \mathcal O^h_{{\rm cross},\ell} \right) \times  
  [\tk_{2\ell}, \tk_{2\ell+1}] \right) \leq {\rm K}_{\rm flow} k^2.
 \end{equation}

 \begin{remark} 
 \label{also}
   {\rm  It follows from Proposition  \ref{roundupe} that \eqref{dejafait} is also satisfied on the boundaries of the cubes $\rQ_{h\slash2}^3(a)$, for $a\in \rbA_\ell^k$, that is  
    \begin{equation}
 \label{vitefait}
 \Gordk(x, \tk_{2\ell})= \pont \left[( \fLkperp-\frac{\tk_{2\ell+1}}{2} \be_3) \cup \fLk \right](x), {\rm \ for \ } 
 x \underset {a\in \rbA_\ell} \cup  \partial\rQ_{h\slash2}^3(a) . 
 \end{equation}
  }
 \end{remark}
\subsubsection{Defining  $\Gordk$ on $ \mathcal O^h_{{\rm cross},\ell}\times [\tk_{2\ell}, \tk_{2\ell+1}]$} 
\label{devint}
{\it Decomposition of the  crossing set into cubes}. It remains  to define   $\Gordk$ on the {space-time crossing region}
\begin{equation*}
\begin{aligned}
\Theta_{\rm cross, \ell}^h &\equiv \mathcal O_{\rm cross, \ell}^h \times [\tk_{2\ell}, \tk_{2\ell}+h=\tk_{2\ell+1}]  \subset \R^4 \\
&=\underset{ \ra\in  {\bf  \rbA}_{\ell}^k }\cup \left(\overline{{\rm Q}_{ h\slash 2}^3( \ra)}\times [\tk_{2\ell}, \tk_{2\ell+1}]\right), 
\end{aligned}
\end{equation*}
where we used   the decomposition \eqref{crossell1}  of $\mathcal O_{\rm cross, \ell}^h $ for the last identity. It  yields a related decomposition  of  $\Theta_{\rm cross, \ell}^h$  into four-dimensional cubes of edges of length $h$. Indeed, setting 
$$
\tk_{2\ell+\frac12}=\frac 12 \left(\tk_{2\ell} +\tk_{2\ell+1} \right) =\tk_{2\ell}+\frac{h}{2}=\tk_{2\ell+1}-\frac{h}{2}, 
$$
we have, for $\ra \in \rbA_\ell^k$,
 \begin{equation*}
 \begin{aligned}
 {\rm Q}_{ h\slash 2}^3( \ra)\times [\tk_{2\ell}, \tk_{2\ell+1}]&={\rm Q}_{ h\slash 2}^3( \ra)\times [\tk_{2\ell+\frac12}-\frac h2, \tk_{2\ell+\frac12}+\frac h2] \\
 &={\rm Q}_{ h\slash 2}^4\left( (\ra,\tk_{2\ell+\frac12} )\right). 
 \end{aligned}
 \end{equation*}
This leads us to   introduce, for   $i, j, q=1, \ldots, k$   and $\ell=1, \ldots, 2k-1$,  the points in $\R^4$ given by
  \begin{equation}
  \label{textswaba}
\rba^k_{i, j, q, \ell}\equiv(\ra_{i, j, q}^k, \tk_{2\ell+\frac 12})=(\ra_{i, j, q}^k, 10+ (2\ell +\frac 14) h) \in \Lambda^k_{2\ell+1}, \\
\end{equation}
and to consider, for $\ell\in \{1, \ldots, 2k-1\}$, the collection  of points $\mathbb A_{\ell}^k$, i.e. 
\begin{equation}
 \label{grognon}
 \begin{aligned}
 {\bf \mathbb A}_{\ell}^k &=\underset {i, j=1}  {\overset {k} \cup}\underset {q' \in \tilde {\Gamma}(\ell)}  { \cup}
 \{\rba_{i, j,q', \ell}\}=\rbA_{\ell}^k\times \{ \tk_{2\ell+\frac 12}\}, \\
 &=\boxplus^2_k(h) \times h  \,  \tilde {\Gamma}(\ell) \times  \{ \tk_{2\ell+\frac 12}\} \\
 &=\left(\left(10\, + \frac h4\right) \be_4\right)+h\left[  \{1, \ldots, k\}^2 \times  \tilde {\Gamma}(\ell)\times \{2\ell \}\right].
 \end{aligned}
\end{equation}
In view of this notation, we may now write  
 \begin{equation}
\label{thetacross}
\Theta_{\rm cross, \ell}^h=\underset{ \rba\in  {\bf  \mathbb A}_{\ell}^k }\cup \overline{{\rm Q}_{h\slash2}^4( \rba)}
=\underset {q\in \Gamma (\ell)}  {\overset {k} \cup}  \left(
\underset {i, j=1}  {\overset {k} \cup}
\overline{ {\rm Q}_{h \slash2}^4( \rba^k_{i, j,k+q-\ell, \ell})}  \right).
\end{equation}
 The intersection of two distinct cubes composing $\Theta_{\rm cross, \ell}^h$ is  empty whereas  their closure might  intersect on a common face, that is, $\displaystyle{\mathfrak Q^{2,\pm}_p({h}\slash{2}, \ra)}$,  the two-dimensional square  defined  in \eqref{pinkmec}, 
 \begin{equation}
 \label{intermarche}
  \left\{
  \begin{aligned}
 {\rm Q}_{ h\slash 2}^4( \rba) \cap  {\rm Q}_{ h\slash 2}^4( \rba') &=\emptyset  {\rm \ for \ } \rba =(\ra,\tk_{2\ell+\frac12}) \not =\rba'
 =(\ra',\tk_{2\ell+\frac12} ) {\rm \ in \ } \mathbb A_{\ell}^k, \\
\overline{{\rm Q}_{ h\slash 2}^4( \rba)} \cap \overline{ {\rm Q}_{ h\slash 2}^4( \rba') }&=\overline{\mathfrak Q^{2,\pm}_p(\frac{h}{2}, \ra)}
\times   [\tk_{2\ell}, \tk_{2\ell+1}] {\rm  \  iff  \ } \rba'-\rba=\pm h\be_p.
 \end{aligned}
\right. 
 \end{equation}

 \bigskip
 \noindent
 {\it Defining the value of the $\Gordk$ on each cube $\overline{{\rm Q}_{h\slash2}^4( \rba)}$, $\rba\in  \mathbb A_{\ell}^k$}.
Recall that we have seen, in identity \eqref{rondupe2} of Proposition \ref{roundupe},  that, for $\rba=(\ra, \tk_{2\ell+\frac 12}) \in \mathbb A_\ell^k$,  with 
$\ra \in \rbA_\ell^k$,  we have
\begin{equation}
\label{finito}
\Gordk (x, \tk_{2\ell})=\upgamma_{\ra}^{h,- }(x)=\Upsilon_{\rba}(x,\tk_{2\ell}), {\rm \ for \ }   x\in  \overline{ \rQ^3_{h\slash 2}(\ra)}. 
\end{equation}
We notice  that the set 
$$\displaystyle{\rQ^3_{h\slash 2}(\ra) \times  \{\tk_{2\ell}\}= \mathfrak Q^{3,-}_4(h/2,  \rba) \subset \partial \rQ^4_{h\slash 2}(\rba)}$$
  is the "lower" (with respect to the $x_4$-variable) face of the cube  $\rQ^4_{h\slash 2}(\rba)$. Similarly,    the set  
$$\displaystyle{\rQ^3_{h\slash 2}(\ra) \times  \{\tk_{2\ell+1}\}= \mathfrak Q^{3,+}_4(h/2,  \rba) \subset \partial \rQ^4_{h\slash 2}(\rba)}$$
 corresponds to the "upper" face of the cube $\rQ^4_{h\slash 2}(\rba)$. Arguing as in the proof of Proposition \ref{roundupe}, we have, in view of the second identity in \eqref{sel} of Lemma \ref{round-up}, 
 \begin{equation}
 \label{roustantv}
\left(  \fLkperp-\frac{\tk_{2\ell+1}}{2} \be_3\right)\cap \rQ^3_{h\slash 2}(\ra) =
\left (\fLkperp_{i, q}-\left(\frac{\tk_{2\ell-1}}{2} +h\right)\be_3 \right)\cap \rQ^3_{h/2}(\ra)
 ={\rm D}_{\perp, h}^+(\ra), 
 \end{equation}
 where, $i, j ,q'$ are such that $\ra=\ra_{i, j, q'}^k$  and $q$ is defined by $q'=k+q-\ell$. Combining with the second identity in \eqref{roustanvt}, we notice that, if \eqref{mochouillebis} holds, then, we need to have  
 \begin{equation}
 \label{mochete}
 \Gordk(x, \tk_{2\ell+1})=\upgamma_{\ra}^{h,+ }(x)=\Upsilon_{\rba}(x,\tk_{2\ell+1}),  {\rm \ for \ }   x\in  \overline{ \rQ^3_{h\slash 2}(\ra)}.
  \end{equation}
 Combining \eqref{finito}, which is a consequence of our construction on step $2\ell$,  with \eqref{mochete}, which is \emph{the goal of the construction at step $2\ell+1$}, we are led to impose  the value of $\Gordk$ on the boundaries on the cubes ${\rm Q}_{\frac h 2}^4( \rba)$ as
 \begin{equation}
 \label{gordcubeas}
\Gordk (\rbx)=\Upsilon_{\rba}^h (\rbx) {\rm \ on \ } \partial ({\rm Q}_{\frac h 2}^4( \rba)),   {\rm \ for \ } \rbx  \in \partial \rQ_{h/2}^4(\rba),\  \rba=(\ra_{i, j, q'}^k,   \tk_{2\ell+\frac 12})\in \mathbb A_{\ell}^k. 
\end{equation}
This definition is also consistent with Remark \ref{also}.     We extend the value given by \eqref{gordcubeas}  inside  each cube using the  \emph{cubic extension}, which yields 
 \begin{equation}
 \label{pointeell}
\Gordk (\rbx)=\boxbox_{h, \rba }(\rbx)=\Ext_{h/2, \rba} \left(\Upsilon_{\rba}\right)(\rbx),   \ {\rm \   for  \  }  
\rbx \in  \overline { \rQ^4_{h/2}(\rba)},  \rba \in \mathbb A_{\ell}^k.  
\end{equation}
Being defined on each cube $\overline{\rQ^4_{h/2}(\rba)}$, the map is hence  defined on the whole of $\mathcal O^h_{{\rm cross},\ell}$, in view of the decomposition \eqref{thetacross}. However, since the cubes may possibly intersect, in view of \eqref{intermarche}, on their boundaries, we need to check that the respective  definitions agree on the intersections of the domains. This will be done in the next subsection. 

\medskip
Restricting ourselves  for the moment to a  single cube $\overline {\rQ^4_{h/2}(\rba)}$ , $\rba  \in \mathbb A_\ell^k$,  definition \eqref{pointeell} of $\Gordk$  on $\overline {\rQ^4_{h/2}(\rba)}$  and the results in Lemma \ref{extension}  show that 
  the map $\Gordk$ restricted to $\overline{\rQ^4_{h/2}(\rba)}$ is  Lipschitz on  every compact subset of    $\overline{\rQ^4_r(\rba)}\setminus \{\rba\}$. At the point ${\rba}$,  it possesses a singularity  of Hopf invariant equal to $2$.  We  have   moreover the energy inequality
\begin{equation}
\label{engamma}
\rE_3\left(\Gordk,\overline{\rQ^4_{h/2}(\rba)} \right) \leq  {\rm \bf K}_{\rm box} h, {\rm \ for \ any \ } h>0, 
\end{equation}
where $ {\rm \bf K}_{\rm box}$ is the  universal constant provided in Lemma \ref{extension}. 

 \subsubsection{Properties of the construction of $\Gordk$ on $ \Theta_{\rm cross, \ell}^h$} 
 \label{proptheta}

 We first show that   \eqref{pointeell} defines $\Gordk$ unambiguously on  $\Theta_{\rm cross, \ell}$ defined in 
 \eqref{thetacross}. 
 We have:
 
 \begin{lemma}
  \label{tony}
 Let $\Gordk$  be  defined by \eqref{pointeell}  on each cube $\overline{\rQ^4_{h/2}(\rba)}$, $\rba \in \mathbb A_\ell^k$, composing $\Theta_{\rm cross, \ell}$ (see \eqref{thetacross}). Then,  definition \eqref{pointeell}  yields the same value on the intersections of these cubes and  defines  the map $\Gordk$ on the whole of $\Theta_{\rm cross, \ell}$. The    restriction of    $\Gordk$ 
  to $\Theta_{\rm cross, \ell}^h$ defined  in  that way  belongs on  $C^0(\Theta_{\rm cross, \ell}^h \setminus \mathbb A_\ell^k, \S^2).$
It is Lipschitz on every compact subset of  $\Theta_{\rm cross, \ell}^h \setminus \mathbb A_\ell^k$. Every singularity in  $\mathbb A_\ell^k$ has Hopf invariant equal to $2$. 
  \end{lemma}

\begin{proof}
The only point to check is that definition \eqref{pointeell} yields the same value on the intersection of distinct cubes. As  noticed,   the intersection of two distinct cubes   $\rQ^4_{h/2}(\rba)$ and $\rQ^4_{h/2}(\rba')$, with $\rba, \rba'$ in $\mathbb A_\ell^k$ is included in the intersection of their boundaries, see \eqref{intermarche}. 
Therefore, we have only to check that 
\begin{equation}
\label{coincide}
\Upsilon_{\rba}^h(\rbx) =\Upsilon_{\rba'}^h(\rbx),   {\rm \ for  \ } 
\rbx \in \partial \rQ^4_{h/2}(\rba)  \cap \partial \rQ^4_{h/2}(\rba'). 
\end{equation}
   To that aim, we recall that, thanks to   \eqref{konami2}, we have, for  
  $\rba \in  \mathbb A_\ell^k$
 \begin{equation}
 \label{coincide2}
 \Upsilon_{\rba}^h(\rbx) =\pont_\varrho \left[(\fLkperp-\frac{\tk_{2\ell+1}}{2}\be_3) \cup\fL^k\right](x), {\rm \ for \ } \rbx=(x, s) \in  \partial \rQ^4_{h/2}(\rba)\setminus  \mathfrak Q^{3,-}_4(h/2,  \rba).
 \end{equation}
 On the other hand, we have $\mathfrak Q^{3,-}_4(h/2,  \rba) =\rQ_{h\slash2}(\ra) \times \{\tk_{2\ell}\}$, where we have written $\rba=(\ra, \tk_{2\ell+1\slash2})$, with $\ra \in \rbA_\ell^k$.  It follows therefore, in view of \eqref{flashouille} that 
 \begin{equation}
 \label{coincide3}
  \Upsilon_{\rba}^h(\rbx)= \tpont \left  [\DefLkperp_{\star, \ell}(\tk_{2\ell-1}, h)\right] (x), {\rm \ for \ }
  \rbx=(x, \tk_{2\ell})\in \mathfrak Q^{3,-}_4(h/2,  \rba). \end{equation}
  We observe that  the r.h.s of  \eqref{coincide2} and   \eqref{coincide3} do not depend on the choice of point $\rba\in \mathbb A_\ell^k$. Hence the formula lead the same value for points $\rbx$ in the intersection $\partial \rQ^4_{h/2}(\rba)  \cap \partial \rQ^4_{h/2}(\rba')$. This establishes the claim \eqref{coincide} and completes the proof of the Lemma.
 \end{proof}

 With the same arguments, we obtain:
 
 \begin{lemma} 
 \label{grouinstar}
  The   value  of $\Gordk$ on   $\partial \Theta_{\rm cross, \ell}^h$, as defined by \eqref{pointeell},  is given by
 \begin{equation*}
 \label{bound cross}
 \left\{
 \begin{aligned}
 \Gordk(\rbx)&=\pont  \left[(\fLkperp-\frac{\tk_{2\ell+1}}{2}\be_3) \cup\fL^k\right](x), {\rm \ for \ } \rbx=(x, s) \in  \partial \Theta_{\rm cross, \ell}^h \setminus 
 \mathcal O_{\rm cross, \ell}^h \times \{\tk_{2\ell}\}, \\
 \Gordk(\rbx)&=\tpont \left  [\DefLkperp_{\star, \ell}(\tk_{2\ell-1}, h)\right] (x),  {\rm \ for \ }  \rbx=(x, \tk_{2\ell}) \in  \mathcal O_{\rm cross, \ell}^h \times \{\tk_{2\ell}\}.
 \end{aligned}
 \right.
 \end{equation*}
  \end{lemma}

The proof is an immediate  consequence of identities \eqref{coincide2} and  \eqref{coincide3}.
  Concerning energy estimates, we have: 
\begin{lemma}  
\label{enertheta}
We have the energy bound
\begin{equation}
\label{enertheta2}
\rE_3\left(\Gordk ,\Theta_{\rm cross, \ell}^h \right) \leq  {\rm \bf K}_{\rm box} \sharp \left( \Gamma (\ell)\right) k.
\end{equation}
\end{lemma}
 \begin{proof} It follows from \eqref{thetacross}, \eqref{intermarche}, the fact that  $\sharp (\mathbb A_\ell^k)=\sharp (\rbA_\ell^k)=k^2\sharp \left(\tilde{ \Gamma }(\ell)\right)=k^2\sharp \left({ \Gamma }(\ell)\right)$ and the identity $hk=1$. Indeed, we have  
 \begin{equation*}
 \rE_3(\Gordk ,\Theta_{\rm cross, \ell}^h )=\underset{\rba \in \mathbb A_\ell^k} \sum \rE_3\left(\Gordk,\overline{\rQ^4_{h/2}(\rba)} \right) 
 \leq \sharp (\mathbb A_\ell^k) \, {\rm \bf K}_{\rm box} h =
 \sharp \left( \Gamma (\ell)\right) {\rm \bf K}_{\rm box}k.
 \end{equation*}
 \end{proof}
  \subsubsection{Properties of the construction of $\Gordk$ on the strip $ \Lambda_{2\ell+1}^h$} 
 Combining the results in subsections  \ref{porcstar}, \ref{devint} and \ref{proptheta}, we are led to:
 
 \begin{proposition} 
 \label{goretstar}
  Let $\Gordk$ be defined by \eqref{dejavu} and  \eqref{pointeell} on  the strip $\Lambda_{2\ell+1}^h=\R^3 \times [\tk_{2\ell}, \tk_{2\ell+1}]$. Then the restriction of $\Gordk$ to  every compact subset of  
  $\Lambda_{2\ell+1}^k \setminus \mathbb A_\ell^k$ is Lipschitz. Every singularity in  $\mathbb A_\ell^k$ has Hopf invariant equal to $2$. A time $\tk_{2\ell+1}$  
  we have
$$
 \Gordk(x, \tk_{2\ell+1})= \pont \left[( \fLkperp-\frac{\tk_{2\ell+1}}{2} \be_3) \cup \fLk \right](x),  {\rm \ for \ } x\in \R^3, 
$$
 so that relations \eqref{mochouillebis}  and condition $\mathcal {BV}(\ell+1)$ holds. We have the gradient estimate
 \begin{equation}
 \label{porkistar}
 \rE_3(\Gordk ,\Lambda_{2\ell+1}^h)  \leq {\rm \bf K}_{\rm box} \sharp\left( \Gamma \left(\ell\right)\right) k + {\rm K}_{\rm flow} k^2.
 \end{equation}
\end{proposition}

\begin{proof}  We have first to check that definitions \eqref{dejavu} and  \eqref{pointeell}  coincide on their interface, that is the set 
$\displaystyle{\partial \Theta_{\rm cross, \ell}^h \setminus 
 \left( \mathcal O_{\rm cross, \ell}^h \times \{\tk_{2\ell}, \tk_{2\ell+1}\}\right)}$:  This fact is an immediate consequence of \eqref{dejavu} and Lemma \ref{grouinstar}.  The first two statement are then  a direct consequence of Lemma \ref{tony} and \eqref{restriction1bis}. Relations \eqref{mochouillebis}
are consequences of definition \eqref{dejavu} and Lemma \ref{grouinstar}. Finally, the energy estimate \eqref{porkistar} follows combining \eqref{integral} and \eqref{enertheta2}. 
\end{proof}

\subsection{The iteration process}
 The construction of $\Gordk$ may be considered as an iterative process, involving at each iteration $\ell$  a pair of steps  composed of an  even step, the step $2\ell$,   and an odd step, the step $2\ell+1$. 
 \subsubsection{Pairing even and odd steps}
 For $\ell \in \{1, \ldots, 2k-1\}$, we set 
 \begin{equation}
 \label{upsigma}
 \Upsigma_\ell^h=\Lambda_{2\ell}^k \cup \Lambda_{2\ell+1}^k=\R^3\times [\tk_{2\ell-1}, \tk_{2\ell+1}].
 \end{equation}
 The building block  for the iterative construction of $\Gordk$ is provided by the following proposition, which summarizes the main properties of the construction of $\Gordk$ on the strip $\Upsigma_\ell^h$:
 
 \begin{proposition} 
 \label{building}
 Let $\ell$ be in $\{1, \ldots, 2k-1\}$ and let $\Gordk$ be defined by \eqref{flashgordon}, \eqref{dejavu} and  \eqref{pointeell} on  the strip 
 $\Upsigma_\ell^h=\R^3\times [\tk_{2\ell-1}, \tk_{2\ell+1}]$.Then,  the restriction of $\Gordk$ to  every compact subset of  
 $\Upsigma_\ell^h \setminus \mathbb A_\ell^k$ is Lipschitz.
 Every singularity in  $\mathbb A_\ell^k$ has Hopf invariant equal to $2$.  At time  $\tk_{2\ell-1}$, relations \eqref{mochouille} are satisfied
 so that condition $\mathcal {BV}(\ell)$ holds. 
 At time  $\tk_{2\ell+1}$, relations \eqref{mochouille} are satisfied
 so that condition $\mathcal {BV}(\ell+1)$ holds. 
  Moreover, we have the integral  estimate
 \begin{equation}
 \label{porkistar2}
 \rE_3(\Gordk, \Upsigma_\ell^h)  \leq {\rm \bf K}_{\rm box} \sharp\left(\Gamma (\ell)\right) k + {\rm K}_{\rm flow} k^2+ {\rm K}_{\rm def} k^2.
 \end{equation}
\end{proposition}
  \begin{proof}  The result in Proposition \ref{building} follow directly combining the corresponding results on the strip $\Lambda_{2\ell}$ and  the strip $\Lambda_{2\ell+1}$, that is from the results in Proposition \ref{flahut} and in Proposition \ref{goretstar}. The only new point which has to be checked is that   definition \eqref{flashgordon}  of $\Gordk$ on $\Lambda_{2\ell}$ on one hand, and definition \eqref{flashgordon}  of $\Gordk$ on $\Lambda_{2\ell+1}$
coincide on their intersection, namely on 
$$
\Lambda_{2\ell} \cap\Lambda_{2\ell+1}=\R^3  \times \{\tk_{2\ell}\}.
$$
This latest fact is a consequence of Lemma \ref{grouinstar}, definition \eqref{flashouille}, see also \eqref{laborieux}.  Concerning the energy estimate 
\eqref{porkistar2}, it suffices to combine estimates \eqref{porkistar} and \eqref{flahut1}.
  \end{proof}

 \begin{figure}[h]
\centering
 \includegraphics[height=9cm]{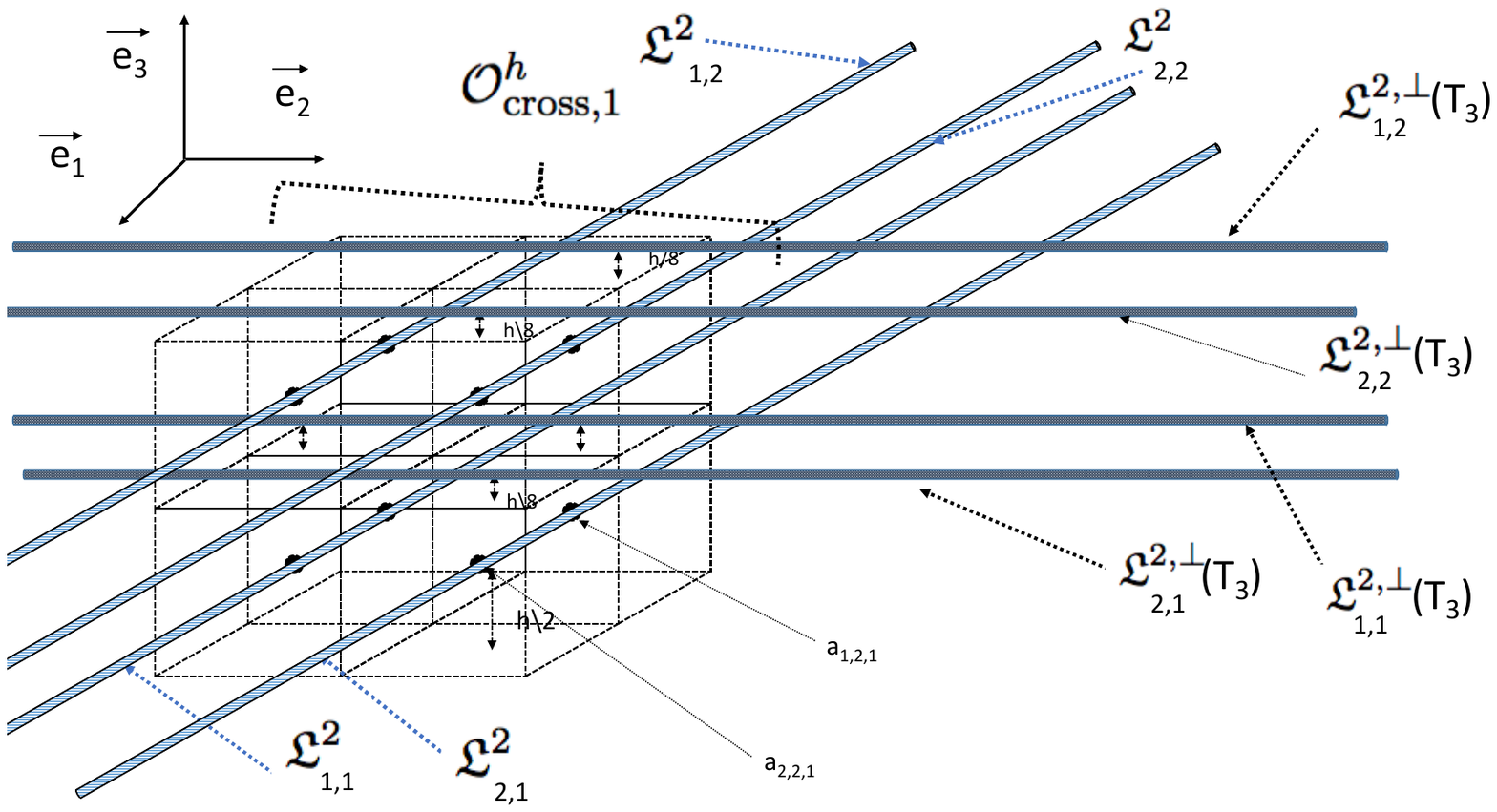}
\caption{  {\it  The shape of the fibers $\fLkperp(\tk_3)$ at time $\tk_3$ for $k=2$, corresponding to  the end of the even step in iteration  $\ell=1$.  The set of fibers $\fLkperp_{\star, 1}(\tk_3)$ have crossed the upper layer of fibers on $\fLk$, but the fibers  $\fLkperp_{j, 2}(\tk_3)$ have entered  the crossing region $\tOcrosstar$, but not yet entered the crossing region $\Ocrosstar$. }}
\label{timetk3}
\end{figure}

\subsubsection {Constructing $\Gordk$  on $\R^3 \times [0, \tk_{4k-1}]$ gluing pairs of even and odd steps}
We decompose the set $\R^3 \times [0, \tk_{4k-1}]$  as 
\begin{equation}
\R^3 \times [0, \tk_{4k-1}]=\Lambda_1\cup \left(   \underset {\ell=1} {\overset{2k-1} \cup} \Upsigma_\ell^h\right), 
{\rm \ with \ } \Upsigma_\ell^h=\Lambda_{2\ell}^k \cup \Lambda_{2\ell+1}^k, 
\end{equation}
and consider the subset $ \mathbb A_\star^k$ of  $\R^3\times  (\tk_1, \tk_{4k-1})$  given by 
\begin{equation}
 \label{mathastar}
 \mathbb A_\star^k=\underset{\ell=1}{\overset{2k-1}\cup } \mathbb A_\ell^k  
 =\left(\left(10\, + \frac h4\right) \be_4\right)+h\left[  \{1, \ldots, k\}^2 \times  
 \underset{\ell=1}{\overset{2k-1}\cup }  \left( \tilde{\Gamma}(\ell)\times \{2\ell \}\right)\right], 
 \end{equation}
see Figure \ref{singlet}. We have constructed so far the $\Gordk$ on the strip $\Lambda_1$ as well as on all the  strips $\Upsigma_\ell^h$. Gluing the definitions on the different strips we introduced so far, we are led to:

\begin{proposition}
\label{enfinouf} Let $\Gordk$ be defined by \eqref{stepzero} on $\Lambda_1^k$, and by \eqref{flashgordon}, \eqref{dejavu} and  \eqref{pointeell} on each of the strips 
 $\Upsigma_\ell^h$, for $\ell=\{1, \ldots, 2k-1\}$.  Then, the following properties hold:
 \begin{itemize}
 \item The restriction of $\Gordk$ to  every compact subset of  
 $\left( \R^3 \times [0, \tk_{4k-1}] \right) \setminus \mathbb A_\star^k$ is Lipschitz.
 Every singularity in  $\mathbb A_\star^k$ has Hopf invariant equal to $2$. 
 \item   We have at time $t=0$, $\Gordk(x, 0)=\Spagk(x)$ for any $x\in \R^3$.
 \item At time  $\tk_{2\ell-1}$, relations \eqref{mochouille} are satisfied  for $\ell=\{1, \ldots, 2k\}$. 
 \item $\Gordk (x, s)=\sP$, For any $x\in \R^3, \vert x \vert \geq 20$ and $s \in [0, \tk_{4k-1}]$.
 \end{itemize}
 Moreover, we have the integral  estimate
 \begin{equation}
 \label{porkistar4}
 \rE_3(\Gordk, \R^3 \times [0, \tk_{4k-1}])  \leq  \left(  {\rm \bf K}_{\rm box}  + 2{\rm K}_{\rm flow} +2 {\rm K}_{\rm def}\right) k^3.
 \end{equation}
\end{proposition}

\begin{proof} The results are direct consequences of the results in Proposition \ref{building} and  Proposition \ref{restriction1}. Besides the energy estimate \eqref{porkistar4}, the only point we have to check is that the definitions provided on different strips yield the same result on their intersections, which are given by 
\begin{equation}
\left\{
\begin{aligned}
\Upsigma_\ell \cap \Upsigma_{\ell+1}&=\R^3 \times \{\tk_{2\ell+1}\}  {\rm  \  and \  } \\
\Lambda_1\cap \Upsigma_1&=\R^3 \times \{\tk_{1}\}.
\end{aligned}
\right. 
\end{equation} 
It follows from the second and third statements in Proposition \ref{building} that the value of $\Gordk$ at time $\tk_{2\ell+1}$ is given by \eqref{mochouillebis} on each of the strip $\Upsigma _\ell$ and $\Upsigma_{\ell+1}$. Hence  the  definitions coincide, yielding continuity on $\R^3\times [0, \tk_{4k-1}]$. Concerning the energy estimate,  adding estimates \eqref{porkistar2} for $\ell=1$ to $2k-1$ together with estimate \eqref{flahut1}, we obtain 
\begin{equation}
 \label{porkistar5}
 \begin{aligned} 
 \rE_3\left(\Gordk, \R^3 \times [0, \tk_{4k-1}] \right)  &\leq
{\rm \bf K}_{\rm box} \underset{\ell=1}{\overset {2k-1} \sum} \sharp\left( \Gamma (\ell)\right) k \\
&+(2k-1) \left[ {\rm K}_{\rm flow} k^2+ {\rm K}_{\rm def} k^2\right]+ {\rm K}_{\rm def} k^2\ .
\end{aligned}
 \end{equation}
For the first term on the r.h.s of inequality \eqref{porkistar5}, we have, in view of identities \eqref{setindique},  
$$
\sharp\left( \Gamma (\ell)\right)=\ell, { \rm \ for \ } \ell \in \{1, \ldots, k\},    {\rm \ and \ }  \sharp\left( \Gamma (\ell)\right)=2k-\ell, { \rm \ for \ } \ell \in \{1, \ldots, k\} .
$$
Hence, we obtain 
\begin{equation}
\label{rover}
 \underset{\ell=1}{\overset {2k-1} \sum} \sharp\left( \Gamma (\ell)\right)= \underset{\ell=1}{\overset {k} \sum}\ell  +  \underset{\ell=k+1}{\overset {2k-1} \sum}(2k-\ell) =\frac{k(k+1)}{2}+\frac{k(k-1)}{2}=k^2. 
\end{equation}
Combining \eqref{rover} and \eqref{porkistar5}, we obtain the energy estimate \eqref{porkistar4}.
\end{proof}

 \begin{figure}[h]
\centering
 \includegraphics[height=9cm]{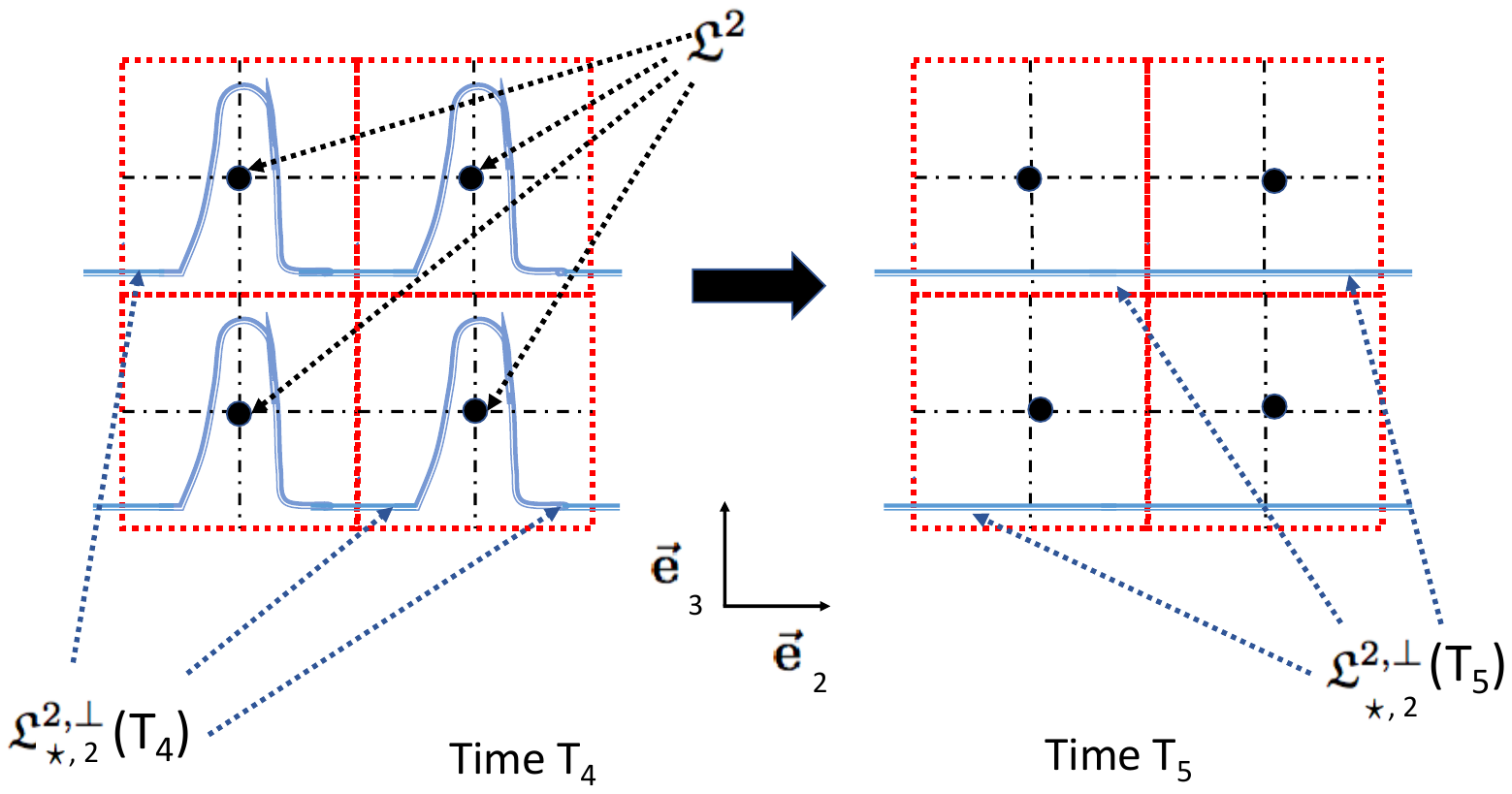}
\caption{  {\it  Sideview of the shape of the fibers $\fLkperp(\tk_4)$ and $\fLkperp(\tk_5)$ at times $\tk_4$ and $\tk_5$, respectively,  for $k=2$, corresponding to iteration $\ell=2$.  The restriction to the  cube $\overline{\rQ_{h\slash2}(a)}$ composing $\Ocrosstar$  corresponds to Figure \ref{DversusC}. At time $\tk_4$,  the set of fibers $\fLkperp_{i, 1}(\tk_4)$ have crossed the  layer of  fibers $\fLk_{j, 2}$, but not yet the fibers  $\fLk_{j, 1}$. At time $\tk_5$,  they have crossed the whole sheaf $\fLk$ and property $\mathcal {BV}(3)$ holds. }}
\label{timetk4V2}
\end{figure}

\begin{figure}[h]
\centering
\includegraphics[height=9cm]{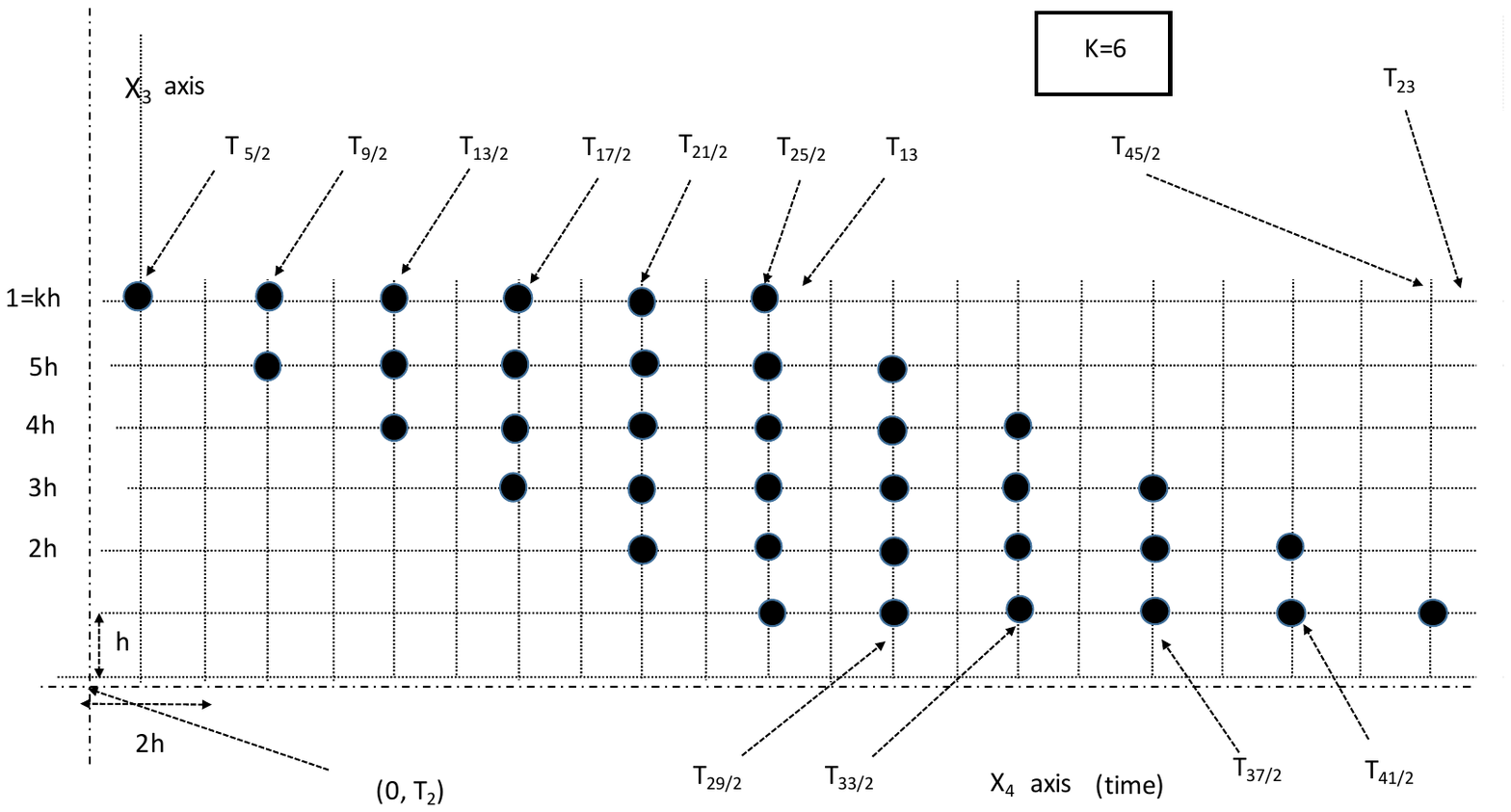}
\caption{  {\it  The singular set $\mathbb A_{\star}^k$   for $k=6$ projected onto the $(\be_3, \be_4)$ plane.
}}
\label{singlet}
\end{figure}

\begin{remark}
\label{bvl}
{\rm    Going back to Subsection \ref{defnore}, property $\mathcal {BV}(\ell+1)$, show that at time $\tk_{2\ell+1}$,  that the fibers $\fLkperp_{i, q}(\tk_{2\ell+1})$  have undergone a downwards vertical translation, so to cross the fibers of $\fLk$. More precisely, for $1\leq \ell \leq k$, the fibers  $\fLkperp_{i, 1}(\tk_{2\ell+1})$  have crossed $\ell$ layers of fibers of  $\fLk$. In particular, for $\ell=k$, the $\fLkperp_{i, 1}(\tk_{2\ell+1})$  are no longer linked to the sheave $\fLk$. For $\ell=k+1$, the same property holds for the fibers $\fLkperp_{i, 2}(\tk_{2\ell+3})$.  More generally, for  $\ell=k+r$, 
 $\fLkperp_{i, r}(\tk_{2\ell+2r+1})$ are no longer linked to $\fLk$, until the two sheaves are unlinked. 
}
\end{remark}

\subsubsection{Properties of the  map $\Gordk$ at time $\tk_{4k-1}$}
At time $\tk_{4k-1}=14-5h\slash 4$,   
 all fibers  of $\fLkperp(\tk_{4k-1})$  have  left the crossing region $\tOcrosstar$  and the two sheaves   $\fLkperp(\tk_{4k+1})$  and $\fL^k$ are not longer linked. Indeed, we have 
 \begin{equation*} 
 \left\{
 \begin{aligned}
&\fL^{k, \perp}(\tk_{4k-1})=\fLkperp-(7-5h\slash 8)\be_3  \subset  \R^2 \times (-\infty,\frac{5h}{8}] {\rm \ and \ } \\
&\fLk (\tk_{4k-1})=\fLk\subset \R^2\times [h, +\infty),
\end{aligned}
\right.
 \end{equation*}
 so that the two sets are separated by a hyperplane. It follows that
 \begin{equation}
 \label{lompa}
 \rH\left( \Gordk (\cdot, \tk_{4k-1})\right)=0.  
  \end{equation}
  The  map $\Gordk (\cdot, \tk_{4k-1})$ is constant, equal to $\sP$ on $\R^3 \setminus [-20, 20]^3$. Concerning the energy, it follows from 
 $\mathcal{BV}(2k)$ that $\Gordk (\cdot, \tk_{4k-1})$ has  the  energy of the Spaghetton map, i.e. 
 \begin{equation}
 \label{lompa2}
 \rE_3\left( \Gordk (\cdot, \tk_{4k-1}), \R^3)\right) \leq {\rm K}_{\rm spag} k^3.  
  \end{equation}
\subsection{  Step   $4k$: deforming to a constant map.}
In this last step, we define the map $\Gordk$ and the strip 
$\Lambda_{\rm fin}^{4k}=[\tk_{4k-1},\tk_{4k}]\equiv\tk_{4k-1}+40 ]$. We deform for  that purpose the map $\Gordk (\cdot, \tk_{4k-1})$ which is Lipschitz and  has trivial homotopy class to a constant map  in a Lipschitz way invoking Proposition \ref{trivialextend}. We apply Proposition \ref{trivialextend} with $R=20$, take as map $w$   the restriction of $\Gordk (\cdot, \tk_{4k-1})$ to the  cube $[-20, 20]^3$: This yields a \emph{Lipschitz map} $W=W^k$ defined on $[-20, 20]^4$ satisfying  the four properties listed in Proposition \ref{trivialextend}.   We set 
\begin{equation}
\label{fenelon}
 \Gordk (x, \tk_{4k-1} +s)=W^k (x, s-20),  {\rm \ for \ } x \in \R^3 {\rm \ and \ } s \in [0, 40],  
 \end{equation}
  so that  the value of $\Gordk(\cdot, \tk_{4k-1})$ given by \eqref{fenelon} and \eqref{gordcubeas} for $\ell=2k-1$ coincide,  and $\Gordk (\cdot, \tk_{4k})=\sP$.
Moreover,  $\Gordk$ is \emph{Lipschitz} on $\Lambda^{4k}$, continuous near $\R^3\times \tk_{4k-1}$ and  
  \begin{equation}
  \label{lesud}
  \Gordk (x, s)=\sP  {\rm  \  on  \ } (\R^3 \setminus [-20, 20]^3) \times [\tk_{4k-1}, \tk_{4k}] \cup \R^3  \times \{\tk_{4k}\}.
  \end{equation}
   The fourth property in Proposition  \ref{trivialextend} yields the energy estimate
   \begin{equation}
   \label{finalenergy}
  \rE_3(\Gordk, \Lambda^{4k}) \leq 40 {\rm C}_{\rm ext} \rE_3(\Gordk(\cdot, \tk_{4k-1}, \R^3) 
  \leq 40 {\rm C}_{\rm ext}{\rm K}_{\rm spg} k^3.
   \end{equation}
   \subsection{Completion of the construction of $\Gordk$}
  Finally,  we notice that 
  $\tk_{4k}=\tk_{4k-1}+40= 44-5h\slash 4< 50, $
  so that we complete the construction of $\Gordk$ setting 
\begin{equation}
\label{lusurp}
\Gordk (\rbx)\equiv \sP, {\rm \ for \ } x\in \Lambda_{\rm 4k+1}^k=\R^3\times[ \tk_{4k}, \tk_{4k+1}=50].
\end{equation}
This definition yields a continuous map on a open neighborhood of $\Lambda_{\rm 4k+1}^k$. We have:

\begin{proposition}
\label{ultimatum}
Let $\Gordk$ be defined by \eqref{stepzero} on $\Lambda_1^k$,  by \eqref{flashgordon}, \eqref{dejavu} and  \eqref{pointeell} on each of the strips 
 $\Upsigma_\ell^h$, for $\ell=\{1, \ldots, 2k-1\}$, by \eqref{fenelon} on $\Lambda_{4k}$, and \eqref{lusurp} on $\Lambda_{4k+1}$.   Then, $\Gordk$ is defined on $\Lambda$ with values into $\S^2$ and satisfies Property \eqref{kontiki}. Moreover, the  restriction of $\Gordk$ to  every compact subset of  
 $\left( \R^3 \times [0, 50] \right) \setminus \mathbb A_\star^k$ is Lipschitz.
 Every singularity in  $\mathbb A_\star^k$ has Hopf invariant equal to $2$. 
  We have the integral  estimate
 \begin{equation}
 \label{porkistar5}
 \rE_3(\Gordk, \Lambda])  \leq  \left(  {\rm \bf K}_{\rm box}  + 2{\rm K}_{\rm flow} +2 {\rm K}_{\rm def}+ 
 40 {\rm C}_{\rm ext}{\rm K}_{\rm spg} \right) k^3.
 \end{equation}
\end{proposition}
\begin{proof}  The fact that $\Gordk$ belongs to $\mathcal R(\Lambda, \S^2)$ is an immediate consequence of Proposition \ref{enfinouf}, combined with  the fact that $W^k$ is Lipschitz, and  the fact that  the definitions  yield the same result on their intersections. The second statement in \eqref{kontiki} follows from the second statement in Proposition \ref{enfinouf}, whereas the third is a consequence of \eqref{lusurp}. The last condition in \eqref{kontiki} follows from the fourth statement in Proposition \ref{enfinouf}, combined with \eqref{lesud} as well  as \eqref{lusurp}. The energy estimate \eqref{porkistar5} is obtained combining \eqref{porkistar4} with \eqref{finalenergy}.
\end{proof}

   \subsection{Proof of Proposition \ref{deform} completed.}
In view of Proposition \ref{ultimatum} we have already established property \eqref{kontiki} for $\Gordk$, and the energy estimate \eqref{porkistar5} yields \eqref{averell}, with the choice of constant 
$$
{\rm K}_{\rm gord}={\rm \bf K}_{\rm box}  + 2{\rm K}_{\rm flow} +2 {\rm K}_{\rm def}+ 
 40 {\rm C}_{\rm ext}{\rm K}_{\rm spg}.
$$  
 It remains to investigate the properties of the singular set $\mathbb A_\star^k$, which is defined in \eqref{mathastar},  and in particular to establish identity \eqref{joedalton}. To that aim, 
 we observe that 
 \begin{equation}
 \label{ramonstar}
 \begin{aligned}
\underset{\ell=1}{\overset {2k-1}\cup}\tilde{ \Gamma}(\ell)\times \{2\ell \}&= \underset {p=0} {\overset {k-1} \cup} \underset {q=1} {\overset {k} \cup}  (k-p,  2(q+p)) 
= \underset {p'=1} {\overset {k} \cup} \underset {q'=1} {\overset {k} \cup}  (p',  2(k+q-p')) \\
&= (0, 2k) + \underset {p'=1} {\overset {k} \cup} \underset {q=1} {\overset {k} \cup}  (p', 2q-2p')= (0, 2k)+\mathbbmss  T(\{ 1, \dots, k\}^2), 
\end{aligned}
 \end{equation}
 where $\mathbbmss T:\R^2 \to \R^2$ denotes the linear mapping defined by
 $\mathbbmss T(a, b)=(a, -2a+2b)$, for $(a, b)\in \R^2$.
Indeed, for the first identity in \eqref{ramonstar},  recall  that  a positive  integer  $q$ belongs to $\Gamma (\ell)$ if and only if there exists $p \in \{0,\ldots,  k-1\}$ such that $\ell=p+q$.  On the other hand, in view of definition \eqref{encorelui}, elements in $\tilde{ \Gamma }(\ell)$ are of the form $k+q-\ell$, hence of the form $k-p$,  with  $p \in \{0,\ldots,  k-1\}$: Hence, we may sum on $p$ instead of   $\ell$. The first  identity in \eqref{ramonstar} then follows.  For the second identity in \eqref{ramonstar}, we introduce the new variables $p'=k-p$ and $q'=k+1-q$.

Combining \eqref{ramonstar} with the definition \eqref{mathastar} of 
$\mathbb A_\star^k$, we obtain \eqref{joedalton}, so that the proof of Proposition \ref{deform} is complete.
 \section{Proof of the main results}
 \label{proofmain}
 \subsection {Proof of Proposition \ref{deformons}}
 \label{proxim}
 \subsubsection {Constructing the sequence  $(\mathfrak v_k)_{k \in \N}$}

 The maps $\mv_k$ are   deduced from the maps $\Gordk$  performing  a few  elementary transformations,  in order  to transform the set of singularities given by \eqref{bboxplus}, which are the nodes of a distorted grid,  into the nodes  of a four  dimensional orthonormal regular grid. 
 
 \medskip
 \noindent
  {\it Transforming singularities into an orthonormal regular  grid: The map $\tGordk$}. The map $\Gordk$ is only defined on the  strip $\Lambda$ defined in 
  \eqref{kontiki}.  Given an integer $k \in \N^\star$, we first extend  the map  $\Gordk$  to the whole space $\R^3 \times  \R $ setting 
  \begin{equation}
  \label{wholespace}
  \left\{
  \begin{aligned}
   \Gordk (x, s)=\sP, {\rm \ for \ }    x \in \R^3 {\rm  \ and \ } s \geq 50, \\
   \Gordk (x, s)=\Spagk(x),  {\rm \ for \ }    x \in \R^3 {\rm  \ and \ } s \leq 0.
   \end{aligned}
   \right.
\end{equation}
It follows from this definition, setting   ${\mathcal   V} \equiv 
   \{ (x, s)\in \R^3 \times \R {\rm \ s.t. \  }  \vert x \vert \leq 40 {\rm \ and \ }  s \leq 50\}$, that 
\begin{equation}
\label{suditude}
   \Gordk (\rbx)=\sP,  {\rm \ for \ } \rbx \in \R^4\setminus \mathcal V.
\end{equation}
In view of the results in Proposition \ref{deform} and Proposition \ref{grammage}, we have the energy estimate
\begin{equation}
\label{energitude}
\rE_3(\Gordk, \R^3\times [-a, 0] )\leq \left({\rm \bf  K}_{\rm Gord} +a {\rm \bf K}_{\rm spg}\right)\, k^3, {\rm \ for \ any  \ }  a\geq 0.
\end{equation}
Let $ \Upphi_k: \R^4 \to \R^4$  be the affine map defined in \eqref{bboxplus},  given, for $\rbx=(x_1, x_2, x_3, x_4)$, by  
$ \Upphi_k(\rbx)=(x_1, x_2, x_3, 12+h\slash 4 -2x_3+2x_4)$. It is  onto and   such that $ \mathbb A_{\star}^k=\Upphi_k\left( \boxplus^4_k (h)\right)$. Consider next the map $\tGordk$ defined on $\R^3\times \R^+$ by 
$$ 
\tGordk (\rbx)=\Gordk \circ \Upphi_k(\rbx)= \Gordk \left( \Upphi_k(\rbx)\right), {\rm \ for \ } \rbx \in \R^3\times [0, +\infty),
$$
so that $\tGordk \in C^0(\R^3\times \R^+\setminus \boxplus_k^4(h), \S^2)$, each of the $k^4$ singularities  in $\boxplus_k^4(h)$ having Hopf invariant equal to $+2$.   We claim that 
\begin{equation}
\label{claque}
\left\{
\begin{aligned}
&\tGordk (\rbx)=\sP, {\rm  \ for \  }  \vert \rbx \vert \geq 100, \ \rbx   \in \R^3\times \R^+ {\rm \ and \ }\\
&\rE_3(\tGordk, \R^3\times [0, +\infty) )\leq 54  \left({\rm \bf  K}_{\rm Gord} +144{\rm \bf K}_{\rm spg}\right)\, k^3.
\end{aligned}
\right.
\end{equation}
 In order to prove   the first assertion of  \eqref{claque}, we take advantage of \eqref{suditude}, and show that, for any $\rbx \in \R^3\times \R^+$, $\Upphi_k(\rbx) \in \mathcal V$ implies $\rbx \in \B^4(100)$. Let   
$\Omega_k \equiv \Upphi_k(\R^3 \times \R^+)=\left\{ (x_1, x_2, x_3, -2x_3+2x_4+12+h\slash4) {\rm  \ with  \ }  x_4\geq 0\right\}$: We have   to show that 
\begin{equation}
\label{claquemur}
\Upphi_k^{-1}(\Omega_k \cap \mathcal V) \subset\B^4 (100),
 \end{equation}
  the inverse  $\Upphi_k^{-1}$ of $\Upphi$ being given by
 $\displaystyle{ 
\Upphi_k^{-1}(\rby)=(y_1, y_2, y_3, y_3+ \frac 12y_4-6-\frac h 8)}$, 
  for any $\rby=(y_1, y_2, y_3, y_4)\in \R^4$. Let  $\rby\in \Omega_k\cap \mathcal V$ and           
  $\rbx=\Upphi_k^{-1}(\rby)$, so that in particular $\rby \in \mathcal V$ and $\rbx   \in \R^3\times \R^+$. It follows that 
  $\vert y_i \vert=\vert x_i \vert  \leq 40$, for $i=1, 2, 3$,  that $y_4 \leq 50$ and  that  $0 \leq x_4 $. In view of the  form of $\Upphi_k^{-1}$  we have therefore 
    $$0\leq x_4=y_3+ \frac 12y_4-6-\frac h 8 \leq 40+\frac{50}{2}-6-\frac h 8\leq 40+25-6=59.  $$ 
It follows that $\vert \rbx \vert \leq \sqrt{3\times 40^2+59^2}= 91<100,$  so that \eqref{claquemur} and hence the first assertion in \eqref{claque} is established. For the second assertion in \eqref{claque}, we have, by the chain rule, 
$ \displaystyle{\partial_{x_i} \tGordk (\rbx)=\partial_{x_i} \Gordk (\Upphi(\rbx))}$  for $ i=1, 2$, whereas
$$
 \partial_{x_3} \tGordk (\rbx)=\partial_{x_3} \Gordk (\Upphi(\rbx))-2\partial_{x_4} \Gordk (\Upphi(\rbx) ) {\rm \ and \ }
  \partial_{x_4} \tGordk (\rbx)=2\partial_{x_4} \Gordk (\Upphi(\rbx)).
  $$ 
If follows that we have 
$$
\vert \nablaq \tGordk (\rbx) \vert^2 \leq 9\vert \nablaq \Gordk (\Upphi_k(\rbx))\vert^2, { \rm \ for \ }    
\rbx \in \R^3 \times \R^+.
    $$
Integrating on $\R^3\times \R^+$, we may restrict  ourselves  to  the set  $\B^4(100) \cap (\R^3 \times \R^+)$, in view of the first assertion of \eqref{claque}, so that  
\begin{equation*}
\begin{aligned}
\rE_3(\tGordk, \R^3\times [0, +\infty) )&\leq \int_{\B^4(100)} \vert \nablaq \tGordk (\rbx) \vert^3 {\rm d} \rbx \leq 
27 \int_{\B^4(100)} \vert \nablaq \Gordk (\Upphi(\rbx)) \vert^3 {\rm d} \rbx  \\ 
&\leq 54 \int_{\Upphi^{-1}(\B^4(100))} \vert \nablaq \Gordk (\rby)) \vert^3 {\rm d} \rby \\
& \leq 
54 \int_{\R^3\times [-144, +\infty)  } \vert \nablaq \Gordk (\rby)) \vert^3 {\rm d} \rby, 
\end{aligned}
\end{equation*} 
 where we used the fact that the Jacobian determinant of $\Upphi_k$ is equal to $2$ for the  second line, and the fact that 
 $\Upphi^{-1}(\B^4(100)) \subset \R^3 \times [-144, +\infty)$ for the last line. The second assertion in \eqref{claquemur} follows combined with \eqref{energitude}.

\medskip
 \noindent
 {\it Extending $\tGordk$ by symmetry}.  So far, the map $\tGordk$ is only defined on $\R^3 \times \R^+$.  We   extend  the map $\tGordk$ by symmetry to the whole on $\R^4$, setting
 \begin{equation}
  \tGordk(x, s) =\tGordk(x, -s),  {\rm \  for \ }  x \in \R^3 {\rm \ and \ } s\leq 0. 
 \end{equation}
 It follows from  the trace theorem that this extension belongs to  $W_{\rm loc} ^{1, 3}(\R^4, \S^2)\cap C^0(\R^4 \setminus {\tilde {\mathbb A}}_{\star}^k, \S^2)$, where the set  ${\tilde {\mathbb A}}_{\star}^k$ is given by
  $\displaystyle{
   {\tilde {\mathbb A}}_{\star}^k= \boxplus_k^4(h)\cup \mathbbmss S_{\rm sym} \left(\boxplus_k^4(h)\right), 
  }$
   with  $\mathbbmss S_{\rm sym}$ corresponding  to the symmetry defined in \eqref{bbmsym}. In view of  \eqref{claque}, we have    \begin{equation}
\label{claque2}
\left\{
\begin{aligned}
&\tGordk (\rbx)=\sP, {\rm  \ for \  }  \vert \rbx \vert \geq 100 {\rm \ and \ }\\
&\rE_3(\tGordk, \R^4 )\leq 108 \left({\rm \bf  K}_{\rm Gord} +144 {\rm \bf K}_{\rm spg}\right)\, k^3.
\end{aligned}
\right.
\end{equation}

  \medskip
 \noindent
 {\it Rescaling  $\tGordk$}. We now are in position to define the map $\mathfrak v_k$ as 
 \begin{equation*}
 \mathfrak v_k(\rbx)=\tGordk(100\, \rbx), {\rm \ for \ } \rbx \in \R^4. 
 \end{equation*}
 It follows then from \eqref{claque2} and scaling laws that 
 \begin{equation}
  \label{claque3}
\left\{
\begin{aligned}
&\mathfrak v_k (\rbx)=\sP, {\rm  \ for \  }  \vert \rbx \vert \geq 1{\rm \ and \ }\\
&\rE_3(\mathfrak v_k, \R^4 )\leq \frac{108}{100} \left({\rm \bf  K}_{\rm Gord} +144 {\rm \bf K}_{\rm spg}\right)\, k^3.
\end{aligned}
\right.
\end{equation}
  Moreover $\mv_k \in C^0(\R^4\setminus  \Sigma_{\rm sing}, \S^2)$ where  $\Sigma_{\rm sing}$ is described in \eqref{sigmasing}. In order  to prove \eqref{debranche}, we will  rely on some additional notion related to branched transportation which are exposed in Appendix A, in particular the branched connection to the boundary $\Lbra$, with the exponent $\upalpha$ equal to the critical exponent in dimension 4, namely $\upalpha_4=1-\frac 14=\frac {3}{4}$.   As a direct is a direct consequence of Proposition \ref{droppy} of the Appendix, we have: 
   
  \begin{proposition}
  \label{castelmoche}
  We have  the lower bound, for some universal constant $C>0$
  $$
  \Lbrafour (\boxplus_k^4(\hscal), \partial ([0, \frac{1}{100}]^4)  \geq C  k^3 \,  \log k,  {\rm   \  \ for \ any \ }  k \in \N^*,  
  {\rm where  \ }  \hscal=\frac{h}{100}.
  $$
  \end{proposition} 
On the other hand, we notice that the functional $\Lbr$ defined in \eqref{Lbranch} corresponds precisely to the functional $\Lbrafour$. More precisely, we have 
\begin{equation}
\label{Lbrafouille}
\Lbr (\mathfrak v_k)=\Lbrafour ( \Sigma_{\rm sing}, \partial \B^4).
\end{equation} 
\subsubsection{Proof of Proposition \ref{deformons} completed}
The only part of Proposition \ref{deformons} which has to be established is \eqref{debranche}. For that purpose, we invoke  
 Lemma \ref{toufou}, with $\fP=\{1\}$ and  $\Omega_1=[0, 1\slash 100]^4 $. Since all singularities in $\boxplus_k^4(\hscal)$ have  positive charge equal to $+2$, the conclusion of Lemma 
\ref{toufou}
applies, yielding  the inequality
\begin{equation}
\label{medieval}
 \Lbrafour ( \Sigma_{\rm sing}, \partial \B^4) \geq \Lbrafour (\boxplus_k^4(\hscal), [0, \partial ([0, \frac{1}{100}]^4 ])\geq  C  k^3 \,  \log k, 
\end{equation}
where we have also invoked  the   result of Proposition \ref{castelmoche} for the last inequality. Combining with \eqref{Lbrafouille}, 
 the lower bound  \eqref{debranche}  follows. The proof of Proposition \ref{deformons} is hence complete. 
\subsection{Proof of Proposition  \ref{unpoco}}
\label{okdac}
As mentioned, Proposition \ref{unpoco} in a consequence of    \cite{HR2}, Theorem  1.1 (see  also Theorem 6.1 and 7.2 in \cite{HR1}).We start recalling these results,  stated  for maps defined on the whole space $\R^4$, so that  a direct application  requires \emph{some adaptations}  to our context. 
 \subsubsection{The Hardt-Rivi\`ere results}  
 \label{hardtriviere}
 Let $u \in \mathcal R_{\rm ct}(\B^4, \S^2)$   be such that $u=\sP$ on  $\partial \B^4$, $\rbP$ denote the set of singularities  of $u$  with positive Hopf invariant, possibly repeated with multiplicity, and $\rbQ$ denote the set of singularities with negative Hopf invariant, repeated with multiplicity as well. Since  $u$ restricted to  the boundary is constant, there are  as many positive singularities as negative ones.  We extend $u$ to $\R^4$ setting $u(\rbx)=\sP$ on $\R^4\setminus \B^4$, so that $\nabla u \in L^3(\R^4)$ and $\rE_3(u, \R^4)=\rE_3(u, \B^4)$.  Theorem 1.1 of \cite{HR2}, 
 inequality (1.2) applied to $u$, shows that, given any sequence of  smooth maps $(w_n)_{n \in \N}$  from $\R^3 \to \S^2$ such that $ w_n$ converges weakly to $u$ in $W_{\rm loc}^{1,3}$, one has 
\begin{equation}
\label{saturne}
\underset{n \to + \infty} \liminf \int_{\R^4} \vert \nabla w_n(\rbx)\vert^3 {\rm d} \rbx \geq  {\rm K}_{\rm HR}  \Lbr (\rbP, \rbQ), 
\end{equation}
where ${\rm K}_{\rm HR}>0$ denotes some universal constant. In \eqref{saturne},  $\Lbr (\rbP, \rbQ)$ denotes  a quantity defined in a way similar to the quantity defined in \eqref{Lbranch}, except that singularities may not be connected to points on the boundary. More precisely, it is defined by  
\begin{equation}
\label{debranch}
\Lbr( \rbP, \rbQ)= \inf \{\rbW(G), G \in \mathcal G (\rbP, \rbQ)\}, 
\end{equation}
where $ G \in \mathcal G (\rbP, \rbQ)$ if and only $G$ satisfies \eqref{balance} and $V(G) \cap \partial \B^4=\emptyset$.  It follows from this definition that 
$\displaystyle{\Lbr( \rbP, \rbQ) \geq \Lbr( \rbP, \rbQ, \partial \B^4)}$. Hence,  going back to \eqref{saturne}, we obtain 
\begin{equation}
\label{saturne2}
\underset{n \to + \infty} \liminf \int_{\R^4} \vert \nabla w_n(\rbx)\vert^3 {\rm d} \rbx \geq  {\rm K}_{\rm HR}  \Lbr (\rbP, \rbQ, \partial \B^4). 
\end{equation}

\subsubsection{Modifying the sequence $(u_n)_{n\in \N}$} 
   \begin{lemma} 
 \label{romanee}
 Let $u \in W^{1, 3} (\B^4, \S^2)$ be such that  $u(\rbx)=\sP$ for    $3\slash 4 \leq \vert \rbx\vert \leq  1$, and $\displaystyle{(u_n)_{n\in \N}}$ be a sequence of smooth maps from $\B^4$ to $\S^2$ such that  
$u_n \underset {n \to \infty} \rightharpoonup u   {\rm \ weakly \ in  \  } W^{1, 3} (\B^4, \S^2)$. Then,  there exists a sequence $(w_n)_{n \in \N}$ of maps in $C^\infty (\B^4, \S^2)$ such that, for $n \in \N$ 
\begin{equation}
\label{moissac}
\left\{
\begin{aligned}
&w_n(\rbx)=\sP  {\rm \  for \    } x \in \partial \B^4,
w_n  \underset {n \to \infty}  \rightharpoonup u   {\rm \ weakly \ in  \  } W^{1, 3} (\B^4, \S^2), {\rm \ and \ } \\
&\underset {n \to \infty} { \liminf }\, \rE_3(w_n, \B^4) \leq {\rm K}_{\rm dir}\,   \underset {n \to \infty}\liminf \, \rE_3(u_n, \B^4) \,  ({\rm with \ } {\rm K}_{\rm dir} \ {\rm  introduced   \ in  \ Lemma \  \ref{epsilonerie}}). 
\end{aligned}
\right.
\end{equation}
 \end{lemma}
\begin{proof} By Banach-steinhaus theorem,  $\rE_3(u_n)$ is bounded independently of $n$, so that  we may assume, passing possibly to a subsequence,  that $\rE_3(u_n)$ convergences to some number $\upeta$, as $n \to + \infty$.   By compact embedding, we deduce that $u_n \to u$ strongly in $W^{1, 3}(\B^4, \S^2)$, so that 
\begin{equation}
\label{montrouge}
\gamma_n\equiv \int_{\B^4\setminus \B^3(3\slash4)} \vert u_n(\rbx)-\sP\vert^3 \rd \rbx \underset{n \to +\infty} \to 0 {\rm \ and \ }
\beta_n\equiv \rE(u_n)  \underset{n \to +\infty} \to \upeta.  
\end{equation}
It follows that, for $n$ sufficiently large, $\gamma_n \cdot \beta_n \leq \epsilon_0$,  so that condition \eqref{petitude} is satisfied for $v=u_n$: We are hence in position to apply Lemma \ref{epsilonerie} to the map $v=u_n$. This yields a map $w=w_n\in C^\infty(\B^4, \S^2)$  such that $w_n=u_n$ on $ \B^3(3\slash4)$ and $w_n=\sP$ on $\partial \B^4$. Moreover, combining \eqref{montrouge} with \eqref{hautpuech}, we derive the third statement of \eqref{moissac}. It remains  to show that $w_n \rightharpoonup u$, weakly in $W^{1, 3}(\B^4, \S^2)$ as $n \to \infty$. We may assume  without loss of generality, passing possibly to a further subsequence, that  $w_n \rightharpoonup w$, weakly in $W^{1, 3}(\B^4, \S^2)$ as $n \to \infty$, for some map $w\in W^{1, 3}(\B^4, \S^2)$. It suffices therefore   to show that $w=u$. In view of the first convergence in \eqref{montrouge} and inequality  \eqref{hautpuech} for $u=u_n$ and $w=w_n$,  we  see that $w_n$ converges to $\sP$ in 
$L^3(\B^4\setminus \B^4(3\slash4))$, so  that  $w=\sP=u$ on $\B^4\setminus \B^4(3\slash4)$.  On the other hand, we have $w_n =u_n$ on $\B^3(3\slash 4)$. Since, by compact embedding,  $u_n$ convergences  strongly to $u$  in $L^3$ as $n \to \infty$, it follows that $w=u$ on $\B^3(3\slash4)$, so that the proof is complete. 
\end{proof}

\subsubsection{Proof of Proposition \ref{unpoco}  completed} 

Let $u$ and  $(u_n)_{n \in \N}$ be as in Proposition \ref{unpoco}, and let $w_n$ be the sequence provided by Lemma \ref{romanee}, so that \eqref{moissac} holds. We extend $w_n$ by $w_n(\rbx)=\sP$ for $x \in \R^4\setminus \B^4$, so  that $w_n$ is now defined on  $\R^4$.  Applying inequality \eqref{saturne} to the sequence $(w_n)_{n \in \N}$, we are led to
\begin{equation}
\label{jupi}
\underset{n \to + \infty} \liminf \int_{\R^4} \vert \nabla w_n(\rbx)\vert^3 {\rm d} \rbx= \underset{n \to + \infty} \liminf \int_{\B^4} \vert \nabla w_n(\rbx)\vert^3 {\rm d} \rbx\geq  {\rm K}_{\rm HR}  \Lbr (\rbP, \rbQ, \partial \B^4), 
\end{equation}
where $\rbP$ (resp. $\rbQ$) denotes the set of positive (resp. negative) singularities of $u$ possibly repeated with  multiplicity. On the other hand, we have, by the last inequality in \eqref{moissac}
 \begin{equation}
 \label{jupi2}
 \underset {n \to \infty} { \liminf }\, \rE_3(v_n, \B^4) \leq {\rm K}_{\rm dir}\,   \underset {n \to \infty}\liminf \, \rE_3(u_n, \B^4),
 \end{equation}
so that, combining \eqref{jupi} and \eqref{jupi2} we obtain \eqref{leprecieux} with ${\rm C}_{\rm conv}= {\rm K}_{\rm dir}^{-1} {\rm K}_{\rm HR}$. 
\subsection{Proof of Theorem \ref{maintheo}}
\subsubsection{Sequences of radii and multiplicities}
\label{radinitude}
The following elementary observation will be used in our proof: 
\begin{lemma}  There exists
  a sequence of  radii $ (\mathfrak r_{\mathfrak i})_{\mi\in \N}$  and a sequence of integers $(\rk_\mi)_{\mi \in \N}$  such that the following properties are satisfied
  \begin{equation}
  \label{sequencitude}
 \underset{\mi \in \N} \sum \mr_\mi = \frac 1 8, \     \ 
 \underset{\mi \in \N} \sum \mathfrak r_\mi \rk_\mi^3  < + \infty, \,
   \underset{\mi \in \N} \sum \mr_\mi \   \rk_\mi^3 \,\log    \rk_\mi \,=+ \infty, 
    {\rm \ and \ } 3 \mathfrak r_{\mi+1} \geq  \mathfrak r_\mi.
  \end{equation}
\end{lemma}
 \begin{proof}  Consider   sequences $ (\tilde \mr_{\mathfrak i})_{i\in \N}$  and  $ (\tilde\rk_\mi)_{\mi \in \N}$ such that,   for $i \in \N \setminus\{0,1\} $, we have  
 $\displaystyle{\tilde \mr_i= \frac {1} {\mi^4 (\log \mi)^2}}$ and  $\displaystyle{\rk_i= {\mi}}$. We then have 
$$
\underset{\mi=2} {\overset {+\infty} \sum}\, \tilde \mr_i=\underset{\mi=2} {\overset {+\infty} \sum}  \,  \frac {1}{\mi^4(\log \mi)^2} < +\infty,  \ \, 
\underset{\mi=2} {\overset {+\infty} \sum}\, \tilde \mr_i\,  \rk_\mi^3=\underset{\mi=2} {\overset {+\infty} \sum}  \,
  \frac {1}{\mi \, (\log \mi)^2} < +\infty,  {\rm \ whereas \ }
$$
$$
\underset{\mi=2} {\overset {+\infty} \sum} \,\tilde \mr_i\, \rk_\mi^3 \log ( \rk_\mi)=\underset{\mi=2} {\overset {+\infty} \sum}  \,
  \frac {1}{\mi \, (\log \mi)}   =+\infty. 
$$
We  choose arbitrary values for $i=0$ and $i=1$ and finally  set $ \mr_i = {\rm c}\tilde { \mr}_i$, where the positive constant ${\rm c}$ is defined so that the first condition holds, that is satisfies $\displaystyle{{\rm c}^{-1}=8\underset{\mi=0} {\overset {+\infty} \sum} \tilde \mr_i.}$
 \end{proof}
\subsubsection{Defining $\mathcal U$ gluing copies of the $\mathfrak v_k$'s}
 We introduce the set of  points $\{\rm M_\mi\}_{\mi \in \N}$ in $\R^4$ defined by
\begin{equation}
\label{toujoursitude} 
{\rm M}_\mi= 4 \left (   \underset {j=0} {\overset {\mi} \sum}   \mr_\mj \right) \be_1,   {\rm \ for \ } \mi \in \N, {\rm \ where \ }  \be_1=(1,0,0,0).
\end{equation}
 The points $\fM_\mi$ belong the segment joining $0$  to the point 
 $\displaystyle{\fM_\star=\frac 12 \be_1= (\frac 12, 0, \ldots, 0)}$,  converging thanks to the first identity in \eqref{sequencitude},  to the point $\fM_\star$ as $\mi \to + \infty$. Moreover, we have 
 $\vert M_{\mi+1}-M_{\mi}\vert =4\mr_{\mi+1}$. We consider the  collection of disjoint balls 
 ${(B_\mi)}_{\mi\in \N}$ defined by 
$$B_\mi \equiv \B^4(\fM_\mi, \mr_\mi)  {\rm \ for \  } \mi \in \N,  \, {\rm\  so \  that  \ } {\rm  dist } (B_\mi,  B_{\mj}) \geq  4\mr_{\mi+1}-(\mr_{\mi+1}+\mr_{\mi})=3\mr_{\mi+1}-\mr_{\mi}>0,
 {\rm \ if \ } \mi <  \mathfrak j,  
$$
  the last assertion being  deduced from the third assertion in \eqref{sequencitude}. In particular, we have $B_\mi\cap B_\mj=\emptyset, $ if $\mi\not= \mj$.  We   then define  the map $\mathcal U$ on $\B^4$  as 
  \begin{equation}
  \label{mathcaluitude}
\mathcal U(\rbx)=\mathfrak v_{\rk_\mi}\left (\frac {2(\rbx-\fM_\mi)}{\mr_\mi}\right) \ {\rm\  if \ } \rbx \in B_\mi, \ \ 
  \mathcal U(\rbx)=\sP  \  {\rm \  if \ } \rbx \in \B^4\setminus \underset  {\mi \in \N}  \cup B_\mi.
  \end{equation} 
   We notice that, for $\mi \in \N$, the restriction of $\mathcal U$ to $B_i$ is  Lipschitz on every compact subset of $B_\mi\setminus \Sing^\mi$, where the set of singularities $\Sing^\mi=\Singp^\mi \cup \Singm^\mi$ is given by
   \begin{equation}
   \label{singuloi}
   \Singm^\mi\equiv    \boxplus^4_{\rk_\mi} (\frac {\mr_\mi  \rh_\mi } {200})+{\rm M}_\mi  {\rm \ and \ } 
   \Singm^\mi \equiv  \mathbbmss S_{\rm sym} \left(  \boxplus^4_{\rk_\mi}
    ( \frac {\mr_\mj  \rh_\mj } {200})\right)+{\rm M}_\mi, 
   \end{equation}
 $\Singp^\mi$ (resp.$\Singm^\mi$) being  the set of singularities of degree $+2$ (resp. $-2$). Moreover, we have 
   \begin{equation}
   \label{noticitude}
  \mathcal U(\rbx)=\sP {\rm \  for \ }  \rbx  \in C_\mi\equiv B_\mi\setminus \B^4(\fM_\mi, \frac{\mr_\mi}{2})  {\rm \ and \ }  
   \rE_3(\mathcal U, B_\mi)=\frac{\mr_\mi}{2} \rE_3(\mv_{\rk_{\mi}})  \leq \frac{\rC_1}{2}  \mr_\mi \rk_\mi^3.
   \end{equation} 
   We have: 
   \begin{lemma}
  \label{glauditude}
   The map $\mathcal U$ belongs to $W_{\rm ct}^{1,3} (\B^4, \S^2)$  with $\mathcal U(\rbx)=\sP$ for $\rbx \in \partial \B^4$. Its restriction to every compact subset of 
   $\overline{\B^4}\setminus \Sing$ is Lipschitz, where $\displaystyle{\Sing=\underset {\mi \in \N} \cup \Sing^\mi}$.
  \end{lemma}
 
 \begin{proof}  The first identity in \eqref{noticitude} ensures that the gluing procedure yields a map which is Lipschitz  on every compact subset of 
   $\overline{\B^4}\setminus \Sing$. Concerning the Sobolev property, we have 
   \begin{equation}
   \rE_3(\mathcal U,\B^4)=\underset{\mi \in \N} \sum \rE_3(\mathcal U, B_\mi)=\underset{\mi \in \N} \sum \frac{\mr_\mi}{2} \rE_3(\mv_{\rk_{\mi}})  \leq \frac{\rC_1}{2}  \underset{\mi \in \N} \sum \mr_\mi \rk_\mi^3<+\infty, 
   \end{equation}
   which yields the conclusion. 
 \end{proof} 

\subsubsection{Proof of theorem \ref{maintheo} completed}
 We argue by contradiction and assume that there exists a sequence  $(v_n)_{n \in \N}$ of maps in $C^\infty (\B^4, \S^2)$ 
 such that 
$\displaystyle{ v_n \underset {n \to +\infty}\rightharpoonup \mathcal U}$
   weakly  in  
   $W^{1,3}(\B^4, \S^2)$.
 The Banach-Steinhaus Theorem yields
  \begin{equation}
  \label{lejuste}
  \gamma \equiv \underset{n \to +\infty} \limsup \,\rE_3\, (v_n, \B^4) <+\infty.
  \end{equation}
    For $\mi \in \N$, let $v_n^\mi\in C^\infty (B_\mi,\S^2)$ be the restriction  of $v_n$ to the ball $B_\mi$, so that $v_n^\mi \rightharpoonup \mathcal U^\mi$ weakly in $W^{1, 3}(B_\mi, \S^2)$ as $n\to + \infty$, where 
    $\mathcal U^\mi$  is the restriction of $\mathcal U$ to $B_\mi$, that is, in view of  \eqref{mathcaluitude}, 
    $$\displaystyle{\mathcal U^\mi(\rbx)=\mathfrak v_{\rk_\mi}\left (\frac {2(\rbx-\fM_\mi)}{\mr_\mi}\right), {\rm \ for \ } \rbx \in B_\mi}.$$
     Since, for $k \in \N$, $\mv_k$ is in $\mathcal R(\B^4, \S^2)$, and such that $\mv_k(\rbx)=\sP$, for $\rbx\in \B^4$, we are in position to apply Proposition \ref{unpoco}, which yields, combined with a scaling argument
    \begin{equation}
    \label{grenouille}
\underset {n \to \infty}   \liminf \,   \rE_3(v_n, B_\mi)= \underset {n \to \infty} \liminf  \, \rE_3(v_n^\mi, B_\mi)  \geq \frac{\mr_\mi}{2} \Lbr(\mv_{\rk_\mi})  \geq \frac{C_2}{2} \mr_\mi\rk_\mi^3 \, \log \rk_\mi.
    \end{equation}
Since $B_\mi \cap B_{\mj}=\emptyset$ if $\mi\not = \mj$, we may sum   relations \eqref{grenouille}, so to obtain 
   \begin{equation}
    \label{grenouilles}
\underset {n \to \infty}   \liminf \,   \rE_3(v_n, \B^4)= \underset {n \to \infty} \liminf  \underset{\mi \in \N} \sum  \rE_3(v_n^\mi, B_\mi)  \geq \frac{C_2}{2}\underset{\mi \in \N} \sum \mr_\mi\rk_\mi^3 \, \log \rk_\mi= +\infty.
    \end{equation}
This yields a contradiction with \eqref{lejuste} and establishes the theorem. 
\subsection{Proof of Theorem \ref{bis} }
  The main additional  arguments involved in  the proof of Theorem \ref{bis} are not specific to the sphere $\S^2$, so that  we  consider a  general  a compact manifold $\mN$. We will invoke the following:   

\begin{proposition}
\label{corobalkitude}Let $m_0\in \N^*$, and  assume that there exists a map  $u$ in $W_{\rm ct}^{1,p}(\B^{m_0}, \mN)$ which is not the weak limit of  smooth maps between $\B^{m_0}$ and $\mN$. That given any integer $m \geq m_0$ there exists a map $v$ in $W_{\rm ct}^{1,p}(\B^{m}, \mN)$ which is not the weak limit of smooth maps between $\B^{m}$ and $\mN$. 
\end{proposition}
  The proof relies on two constructions we present next. 

\subsubsection{Adding dimensions}
\label{additude}
 Let $m \in \N^*$ and consider a map $u: \R^m \to  \R^\ell$  such that $u$ is constant, equal to some value ${\rm c}_0$,  outside  the unit   ball $\B^m$.  We construct a  map  $ \mathbbmtt I_{\rm cyl}^{m+1}  (u)$ from $\R^{m+1} \to  \R^\ell$ constant equal to ${\rm c}_0$ outside    the unit ball $\B^{m+1}$  as follows.  First, we consider the  map $u_{_A}$ defined on $\R^{m}$ by
 $$  u_{_A} (x)=u(x-A) {\rm  \ where \ } A=(2, 0,\ldots, 0), {\rm \  for \ } x \in \R^m.
 $$
  Hence,  $u_{_A}$ is equal to ${\rm c}_0$  outside the ball $\B^{m}_1(A)$, in particular in the region $\{x_1\leq 1\}$.  We then  introduce the map $T^{m+1}(u): \R^{m+1}\to \R^\ell$ defined  for   $(x_1, x_2, \ldots, x_m, x_{m+1})\in \R^{m+1}$ by 
  \begin{equation*}
T^{m+1} (u)(x_1, x_2,\ldots,  x_m, x_{m+1})= u_{_A}(x_1, x_2,\ldots x_{m-1},\sqrt{x_m^2+x_{m+1}^2})).
  \end{equation*}
The map   $T^{m+1} (u)$ possesses cylindrical  symmetry around the $m-1$ hypersurface $x_{m}=x_{m+1}=0$. Moreover, $T^{m+1}(u)$ is equal to ${\rm c}_0$  outside  the ball $\B_3^{m+1}(0)$ and  also  in region 
$\{{x_m^2+x_{m+1}^2} \leq 1  \}$, that is on 
$\displaystyle{\R^{m-1}\times \B^2}$. In order to obtain  maps  which are  constant outside the unit ball  $\B^{m+1}$,   we normalize $T^{m+1}(u)$  and consider  the map $\Icyl^{m+1}(u)$ given by 
\begin{equation}
\Icyl^{m+1}(u)(\rbx)=T^{m+1}(u) (3\rbx), {\rm \ for \ } \rbx \in \R^{m+1}.
\end{equation}
It follows from this definition that  $\Icyl^{m+1}(u)$ equals ${\rm c_0}$ outside  $\B^{m+1} $, that 
    $\Icyl^{m+1}(u)(\rbx) ={\rm c}_0$ for  $\rbx \in   \mathfrak A^{m+1}\equiv  \R^{m-1}\times  \B_{1\slash3}^2 $  and that, for $x=(x_1,\ldots x_m) \in \R^{m}$, 
   \begin{equation}
   \label{bleudegex}
   \Icyl^{m+1}(x_1,\ldots x_m, 0)=u_{_A}(3x)=u(3x-A).
   \end{equation}
\subsubsection{Restrictions to lower  dimensional hyperplanes}
For $\theta \in \R$, we consider the $m$-dimensional hyperplane $\mathcal {P}_\theta^m$ of $\R^{m+1}$ defined by
$$\displaystyle{\mathcal {P}_\theta^m\equiv {\rm Vect} \left \{ \be_1,   \be_{2}, \ldots,\,  \be_{m-1}, \cos \theta  \, \be_m+ \sin \theta \,  \be_{m+1} \right\},}$$
 and the corresponding half-hyperplane $\mathcal P_\theta^{m, +}$ defined by
 \begin{equation}
 \label{peplum}
 \mathcal P_\theta^{m, +}=\{ v \in \mathcal P_\theta^{m}, v. (\cos \theta\, \be_m+ \sin \theta  \,   \be_{m+1})\geq 0\}.
 \end{equation}
Let $1<p<+\infty$ and consider  a map $v \in W^{1,p} (\B^{m+1} ,  \R^\ell)$.
Its restriction to  $\mathcal P_\theta^{m, + } \cap \B^{m+1}$ is, in view of the trace theorem, a map in $W^{1-\frac 1p, p}(\mathcal P_\theta^{m,+} \cap \B^{m+1})$. It yields a map $ \Tr(v)$ defined  on  the $m$-dimensional half-ball $\B^{m, +}=\B^m\cap \{x_1\geq 0\}$,  setting for $(x_1, \ldots, x_m) \in \B^{m, +}$,
\begin{equation}
\label{duxbellorum}
\Tr(v)(x_1, \ldots, x_m) =v( x_1\,  x_2, \ldots, x_{m-1},  \cos\theta \,  x_m,  \sin \theta   \, x_m).
\end{equation}

\begin{proposition}
\label{blaise}
  Let $\rm c_0 \in \R^\ell$ be given and let  $U$ be given in $W_{\rm ct}^{1,p}(\B^{m+1}, \R^\ell)$ such that  $U=\rm c_0$ on  
  $\mathfrak A^{m+1}\equiv \R^{m-1} \times  \B_{1\slash3}^2$. Let $(W_n)_{n \in \N}$ be a sequence converging weakly to $U$ in $W^{1, p}(\B^{m+1}, \R^\ell)$.
Then, there exists a subsequence $\left(w_{\sigma(n)}\right)_{n \in \N}$ and a sequence of angles $(\theta_n)_{n \in \N}$ converging to some limit $\theta_\star$
 such that 
 $$
 \mathbbmtt T^m_{\mathbbmtt r, {\theta_n}} (W_{\sigma(n)})(\cdot)  \rightharpoonup  \mathbbmtt T^m_{\mathbbmtt r, {\theta_\star}} (U)   \ 
  {\rm weakly \ in  \ } W^{1,p} (\B^{m, +},  \R^\ell) {\rm \  as \ } n \to + \infty, 
$$ 
\end{proposition}
\begin{proof}  Since $U={\rm c}_0$  on $\mathfrak A^{m+1}$,  and since $(w_n)_{n \in \N}$ is bounded in  $W^{1,p} (\B^{m+1},  \R^\ell)$,  we have, for some  $C>0$ independent of $n$, setting
$ x_\theta=x\cdot(\cos \theta \be_m+\sin \theta \be_{m+1})$
\begin{equation*}
\begin{aligned}
C \geq \int_{\B^{m+1}} \vert \nabla W_n \vert^p {\rm d}x&=
\int_0^{2\pi} \left(\int_{\mathcal P_{ \theta}^{m,+}\cap \B^{m+1}}\vert \nabla W_n   \vert^p \vert x_{\theta} \vert  {\rm d} x \right) {\rm d} \theta \\
&\geq \frac{1}{3} \int_0^{2\pi}\left(\int_{\mathcal P_{ \theta}^{m,+}\cap \B^{m+1}}\vert \nabla W_n   \vert^p {\rm d} x \right) {\rm d}\theta.  
\end{aligned}
\end{equation*}
 By a mean-value argument,   given $n\in \N$,  there exists some angle $\tilde \theta_n\in \R$ such that 
$$
\int_{\mathcal P_{\tilde \theta_n}^{m,+}\cap \B^{m+1}} \vert \nabla W_n \vert^p  {\rm d} x\leq 3C, 
$$
  so that the sequence $( \mathbbmtt T^m_{\mathbbmtt r, {\tilde{\theta}_n}}  (W_n))_{n \in \N^*}$ is bounded in $W_{\rm ct}^{1,p}(\B^m, \R^\ell)$. By sequential weak compactness, we may extract a subsequence $(\sigma(n))_{n \in \N}$  such that  $\theta_n\equiv\tilde \theta_{\sigma(n)}$ converges to some limit $\theta_\star$ and such that 
$ \mathbbmtt T^m_{\mathbbmtt r, {\theta_n}}  ( W_{\sigma(n)})$ converges weakly to some map $v$ in $W^{1,p} (\B^m)$. By the trace theorem,  we  already know that  the sequence converges in the trace space to  $\mathbbmtt T^m_{\mathbbmtt r, {\theta_\star}}U$, so that the conclusion follows.
\end{proof}

Using \eqref{bleudegex}, we observe that, for any $\theta \in \R$, the operators   $\Tr$ and $\Icyl^{m+1}$  are related, for any for $v$ with compact support in $\B^m$, by 
\begin{equation}
\label{ngolo}
\Tr \circ  \Icyl^{m+1} (v)(x)=v_{_A} (3x)=v(3x-A),  \forall x \in \B^{m, +}.
\end{equation}

\subsection{Proof of Proposition \ref{corobalkitude}}
For $m\in \N^*$, we define property $\mathcal P(m)$ as: 
$$
\mathcal P(m): {\it  There \,  exists \,   \,  } u_m {\it \ in \,  }W_{\rm ct}^{1,p}(\B^{m}, \mN)
{\it \  which \,  is \, not \, the \,  weak \, limit \,  of  \,  maps \,  in  \ } C^\infty( \B^{m}, \mN).
$$
We argue by induction and assume that $\mathcal P(m)$ holds.   We claim that if $\mathcal P(m)$ holds, then 
\begin{equation}
\label{icyl}
\Icyl^{m+1}(u_m) {\rm \ is \ not \ the  \  weak \ limit \ in \ } W^{1, p}(\B^{m+1}, \mN)  {\rm \ of \ maps \ in \ } C^\infty (\B^{m+1}, \mN).
\end{equation}
In order to proof the claim \eqref{icyl}, we argue by contradiction on assume that there exists a sequence of maps $(W_n)_{n \in \N}$  in 
$C^\infty (\B^{m+1}, \mN)$ converging weakly to $U\equiv \Icyl^{m+1}(u_m)$.We apply Proposition \ref{blaise}  to the map $U$ and the sequence 
$(W_n)_{n \in \N}$ so that, for some subsequence 
$$
 \mathbbmtt T^m_{\mathbbmtt r, {\theta_{\sigma(n)}}}(W_{\sigma(n)})(\cdot)\underset{n \to +\infty}  \rightharpoonup  \mathbbmtt T^m_{\mathbbmtt r, {\theta_\star}} (U) 
 =\mathbbmtt T^m_{\mathbbmtt r, {\theta_\star}}   \circ  \Icyl^{m+1} (u_m)=u_m(3\cdot-A)  \ 
 $$
  weakly in  $W^{1,p} (\B^{m+1 },  \R^\ell)$,  where we have invoked    \eqref{ngolo}. It follow that the map $v=u_m(3\cdot-A)$
  is the weak limit of smooth maps between  $\B^{m+1 }$ and $\mN$. Since $u_m(x)=v(\frac{x+A}{3})$ on $\B^{m}$ the same holds for $u_m$, but this contradicts our assumption and proves the claim \eqref{icyl}.
  
  It follows from \eqref{icyl} that, if $\mathcal P(m)$ holds then $\mathcal P(m+1)$ holds also, so that the proposition is proved by induction.

\subsubsection{Proof of Theorem  \ref{bis} completed}
\label{subsectionitude}
 In Theorem 3, we have constructed a map $\mathcal U$ in $W_{\rm ct}^{1,3}(\B^4, \S^2)$ which is not the weak limit of smooth maps. Applying    Proposition \ref{corobalkitude} with $m_0=4$  and 
 $\mN=\S^2$,  we deduce that for any given integer  $m \geq 4$ there exists a map $\mathcal V_m$ in $W_{\rm ct}^{1,3}(\B^m, \S^2)$ which is not the weak limit 
of maps in $C^\infty (\B^m, \S^2)$. This provides the proof  of Theorem \ref{bis} in the special case $\mM=\B^m$.

 We extend  next the result to  an arbitrary smooth manifold  $\mM$ of dimension $m$. For that purpose, we choose  an arbitrary point $A$ on $\mM$ and glue a suitably adapted copy of  $\mathcal V_m$ at the point $A$.  More precisely,  we consider for $\rho >0$ the geodesic  ball $\mathcal O_\rho (A)$ centered at $A$. If $\rho$ is chosen sufficiently small, then there exist a diffeomorphism $\Phi:\mathcal O_\rho (A) \to \B^m$ and we may define a map $\mathcal W_m: \mM \to \S^2$ setting
 $$
 \mathcal W_m(x)=\mathcal V_m \left(\Phi(x)\right), {\rm \ if \ } x \in \mathcal O_\rho (A), \ \mathcal W_m(x)=\sP {\rm \ otherwise}.
 $$
 One may then verify that  $\mathcal W_m$  belongs  $W^{1,3}(\mM, \S^2)$ and cannot be approximated weakly by maps in $C^\infty (\mM, \S^2)$, which completes the proof.
\section{The lifting problem} 
 \label{reliftage}
 \subsection{Lifting the $k$-spaghetton map}
Let $k \in \N^*$ be given and  consider  on $\R^3$ an arbitrary  lifting ${\rm U}_k$  of the spaghetton map $\Spagk$, that is a map $\rU^k : \R^3 \to  SU(2)\simeq \S^3$ such that $\Pi \circ \rU^k=\Spagk$. Although the relationship between $\rU^k$ and $\Spagk$ has a genuine nonlocal nature, as suggested by  relation \eqref{laplace}, the peculiar  geometry of the  Spaghetton map allows to recover some locality, as expressed in the following lower bound: 
 
 \begin{proposition}
 \label{lowerlift}
 Let $\rU^k$ be any  Lipschitz  lifting  of the Spaghetton map $\Spagk$, that is such that $\Spagk=\Pi \circ \rU^k$. Then we have, for every $1\leq p < + \infty$ and for some constant $C_p>0$ depending only on $p$, 
 \begin{equation}
 \label{liftitude}
  \int_{[-8,9] \times  [7, 12] \times [0, 1]} \vert \nabla  \rU^k \vert^p(x) {\rm d} x\geq C_p k^{2p}.
  \end{equation}
 \end{proposition}
 
The result is a consequence of the following:

\begin{lemma} 
\label{tricotin}
 Let $a\in [0, 1]$. We have
 \begin{equation}
 \label{tricotin1}
 \left  \vert \int_ {[0, 1]\times [8, 11] } [\Spagk]^*(\omedeux) (x_1, x_2, a))  {\rm d} x_1 {\rm d}x_2 \right\vert =4\pi k^2.
   \end{equation}
 Let  $\Omega \subset P_{1, 2}(a)\cap \{x_2\geq 0\}$  be a smooth regular convex set such that 
 $\Omega \supset  [0, 1]\times [8, 11]\times \{a\}$ and  let $\mathcal C=\partial \Omega$. If $\rU^k$ is as in Proposition \ref{lowerlift}, then we  have
\begin{equation}   
\label{tricotin2}
 \int_\mC \vert \nablatrois \rU^k (\ell)\vert {\rm d} \ell \geq  8 \pi k^2.
\end{equation}
\end{lemma}
 \begin{proof}[Proof of Lemma \ref{tricotin}]  Each of the $k^2$ fibers $\fLkperp_{i, q}$ intersects the half-plane $\R\times [0, +\infty[ \times \{a\}$ at a unique  point $M_{i, q}(a)$ (see Figures 
 \ref{above0} and \ref{side0}). The points $M_{i, q}(0)$ belong to the square $[0, 1] \times [9, 11] \times \{0\}$,  a little trigonometry shows that more generally $M_{i, q}(a)$  belongs to the rectangle $[0, 1] \times [17\slash 2, 11]\times  \{a\}$. Our assumption on $\Omega$ hence implies that a neighborhood  of the points $M_{i, q}(a)$ belongs to $\Omega$. In view of the Pontryagin construction, near each point $M_{i, j}(a)$, the restriction of the Spaghetton map $\Spagk$  to the plane $P_{1, 2} (a)$ maps a small neighborhood of $M_{i, j}(a)$ onto the sphere $\S^2$ yielding a contribution equal to   the area of $\S^2$, that is $4\pi$ to the integral in \eqref{tricotin1}. Adding the contributions of the $k^2$ points, \eqref{tricotin1} follows. 
  
  For the second assertion,  we  consider, as in subsection \ref{grouicguinec},  the $su(2)$ valued 1-form   $A^k \equiv (\rU^k)^{-1}\cdot  d\rU^k$ and its first component the real-valued 1-form $A_1^k=A^k. \upsigma_1$. The curvature equation \eqref{curvature} leads to the relation
 \begin{equation}
 \label{letsdance}
 dA_1^k =2[\Spagk]^*(\omedeux).
 \end{equation}
  Integrating on $\Omega$ we deduce from \eqref{tricotin1} and  \eqref{letsdance}  that 
  $\displaystyle{ \vert \int_\mC A_1^k \vert=\vert \int_\Omega dA_1^k\vert  = 8\pi  k^2}.$   Since  $\vert \nablatrois \rU^k \vert \geq  \vert A_1^k \vert$,  the conclusion \eqref{tricotin2} follows. 
  \end{proof}

\begin{proof}[Proof of Proposition \ref{lowerlift}]
Let $a\in [0, 1]$. We choose as sets $\Omega$ the disks 
$$\mathcal D (r, a)\equiv \D^2(r)\times \{0\} + N_0(a) {\rm \ where \ }  N_0(a)=\{(1\slash2,  19\slash 2\} \times \{a\}$$
 so that for $r\geq 2$ we have $\mathcal D(r, a) \supset  [0, 1]\times [8, 11]\times \{a\}$ and $\mathcal D(r, a)  \subset [-7,9]\times [7,12]  \times \{a\}$ for $r  \leq 8$. We may hence apply \eqref{tricotin2}  to the circles  $\mathcal C(r, a)\equiv \partial \mathcal D(r, a)$ for $2\leq r \leq 8$.  Integrating the obtain estimate with respect to the variable $r$, we are led to
 \begin{equation}
 \label{tricotin3}
 \int_{\mathcal D(a, 8)\setminus \mathcal D(a, 2)}  \vert \nablatrois \rU^k \vert {\rm d}x_1\, {\rm d}x_2 \geq 48 \pi k^2.
 \end{equation}
 We set $\mathcal F=\underset{ a \in [0, 1]} \cup \left( \mathcal D(8, a)\setminus \mathcal D( 2, a)\right) \subset [-7,9]\times [7,12]  \times [0,1]$. Integrating \eqref{tricotin3} with respect to $a$, we are led to 
 $$
 \int_{\mathcal F}  \vert \nablatrois \rU^k (x)\vert {\rm d}x_1    {\rm d}x_2  {\rm d}x_3   \geq 48 \pi k^2, 
 $$
 which yields  \eqref{liftitude} in the case $p=1$. The general case is deduced using H\"older's inequality. 
    \end{proof}
  
\begin{remark}
{\rm   In view of Proposition \ref{grammage}, we have $\vert \nabla \Spagk \vert \leq \rCspag k$, so that we have 
  \begin{equation}
  \label{fromagerape}
  \int_{\fL^k} \vert \nabla \Spagk \vert ^p  {\rm d} \ell \leq C_p k^p, {\rm \ for  \ any \ } 1\leq p< + \infty, 
\end{equation}
which has to be compared with \eqref{liftitude}, where the exponent on the r.h.s is $2p$ instead of $p$.  }
\end{remark}

  The  previous  result extends to  some  Sobolev classes:
  \begin{proposition}  
  \label{mascarpone}
  Let $2\leq p <+\infty$ and $\rU^k \in W^{1,p}_{\rm loc} (\R^3, \S^3)$ be  such that $\Spagk=\Pi \circ \rU^k$. Then \eqref{liftitude} holds. 
 \end{proposition}
\begin{proof} In the case $p\geq 3$  smooth maps are dense in $W^{1,p}_{\rm loc} (\R^3, \S^3)$  and a  standard approximation result yields the result.
In the case $2\leq p <3$ smooth maps are no longer dense, but one may prove that, since the Spaghetton map is smooth,  any $W^{1,p}$ lifting of  the Spaghetton map can be approximated by smooth maps, yielding hence a similar proof. 
\end{proof}

\begin{remark}
\label{nichlifitite}{\rm The result of Proposition \ref{mascarpone}  does not extend  to the range  $1\leq p<2$. This observation is  related to the fact that there are liftings in   $W^{1,p}_{\rm loc} (\R^3, \S^3)$  which are singular,  see  \cite{BeChi}.
}\end{remark}
  \subsection{Extension to higher dimensions}
    We add dimensions following the same scheme as in subsection \ref{additude}.   Since the Spaghetton map $\Spagk$  is constant outside  $\B^{4}(20)$,  we   renormalize it  so to obtain a  map which is constant outside the unit ball, introducing the  map
    $ \tilde  {\rm \bf  {S}}_{\rm \bf pag}^k (\cdot)= \Spagk \left( 20 \, \cdot\right), $
    and then consider the map
 \begin{equation*}
 {\rm \bf  S}_{\rm \bf pag}^{k, 5}  =\mathbbmtt I^5_{\rm cyl} ( \tilde {\rm \bf  S}_{\rm \bf pag}^k)),
 \end{equation*}
 which is a Lipschitz map   on $\R^5$, constant outside the unit  ball $\B^5$. More generally, given $m \geq 5$, we define iteratively  the map 
 $ {\rm \bf  S}_{\rm \bf pag}^{k, m}$ on the ball $\B^m$ 	as 
 \begin{equation*}
   {\rm \bf  S}_{\rm \bf pag}^{k, m}(x)  =\mathbbmtt I^m_{\rm cyl} (  {\rm \bf  S}_{\rm \bf pag}^{k, m-1}(20 x ))  {\rm \ for \ } x \in \B^m, 
 \end{equation*}
  with the convention $ {\rm \bf  S}_{\rm \bf pag}^{k, 3}=\tilde  {\rm \bf  {S}}_{\rm \bf pag}^k$.    In view of \eqref{fromagerape}, we obtain the bound 
  \begin{equation}
  \label{rapitude}
  \int_{\B^m} \vert \nabla  {\rm \bf  S}_{\rm \bf pag}^{k, m}\vert ^p {\rm d} x\leq  Ck^p.
  \end{equation}
  \begin{lemma}
   \label{spagetude}
    Let $2\leq p <+\infty$ and $\rU_k^m \in W^{1,p}_{\rm loc} (\B^m, \S^3)$ be  such that 
    $ {\rm \bf  {S}}_{\rm \bf pag}^{k, m}=\Pi \circ \rU_k^m$. Then we have
  \begin{equation}
  \label{rUmk}
  \int_{\B^m} \vert \nabla \rU^m_k \vert^p  {\rm d} x\geq {\rm C}_p^m k^{2p}. 
  \end{equation}
 \end{lemma}
\begin{proof} We establish  \eqref{rUmk}  arguing by  induction on the dimension $m$.   We first observe that the lower bound \eqref{rUmk} has already been established for $m=3$ in Lemma \ref{mascarpone} with  the choice  $ {\rm C}_p^m =C_p$, where $C_p$ refers to  the constant in inequality \eqref{liftitude}. Assume next  that inequality \eqref{rUmk} is  established  for some integer  $m\geq 3$:  We  are going to show that it then holds also in dimension $m+1$. For that purpose,  let $\rU^{m+1}_k\in W^{1,p}_{\rm loc} (\B^{m+1}, \S^3)$ be any arbitrary lifting of the map $ {\rm \bf  {S}}_{\rm \bf pag}^{k, m+1}$. 
 For $\theta \in [0,2\pi)$,  we consider  the half-hyperplane $\mathcal P_\theta^{m, +}$ defined in \eqref{peplum} and the map 
 $\Tr (\rU^{m+1}_k)$ defined on  $\B^{m, +}$ thanks to \eqref{duxbellorum}.  It  follows from  \eqref{ngolo} that 
 \begin{equation*}
 \begin{aligned}
 \Pi \circ  \Tr (\rU^{m+1}_k)&=   \Tr (\Pi \circ \rU^{m+1}_k) =\Tr ( {\rm \bf  {S}}_{\rm \bf pag}^{k, m+1})=
 \Tr(\mathbbmtt I^{m+1}_{\rm cyl} (  {\rm \bf  S}_{\rm \bf pag}^{k, m}(20 \cdot ))) \\
 &= {\rm \bf  S}_{\rm \bf pag}^{k, m}(20 (3\cdot-A ))={\rm \bf  S}_{\rm \bf pag}^{k, m}(60\cdot-20 A )), 
\end{aligned}
\end{equation*}
  so that $\Tr (\rU^{m+1}_k)$ is a lifting of ${\rm \bf  S}_{\rm \bf pag}^{k, m}(60\cdot-20 A ))$. Since by induction we assume that \eqref{rUmk} holds in dimension $m$, we are led to the lower bound
$$  
\int_{\mathcal P_\theta^{m, +}\cap  \,  \B^m(\frac 13  A, \frac {1}{60})}\vert \nabla \rU^{m+1}_k \vert^p \geq  {\rm C}_p^m \left ( \frac {1}{60} \right)^{m-p}    k^{2p}.
$$
Integrating with respect to $\theta$ on the interval $(0,2\pi)$,  we obtain 
$$
\int_{\B^{m+1}}\vert \nabla \rU^{m+1}_k \vert^{^p}  \geq   {\rm C}_p^m \frac {2\pi}{3} \left ( \frac {1} {60}\right)^{m-p}    k^{2p}, 
$$
 so that  Property \eqref{rUmk} is established for the dimension $m+1$ choosing  the constant ${\rm C}_p^{m+1}$ as 
${\rm C}_p^{m+1}=\frac{2\pi}{3}  \left ( \frac {1} {60}\right)^{m-p}  {\rm C}_p^m.$
\end{proof}
  \subsection {Proof of Theorem \ref{repulpage}}
  We first construct a map $\mathcal V=\mathcal  V_0$ in  the special  case $\mM=\B^m$, imposing moreover  the additional condition $\mathcal V_0=\sP$ on $\partial \B^m$.  \\

\subsubsection{Construction  of $\mathcal V_0$ on $\partial \B^m$ }
  \noindent
  {\it Gluing copies of the $ {\rm \bf  S}_{\rm \bf pag}^{k, m}$'s}.  We construct as in subsection \ref{radinitude}
  a sequence of radii $ (\mathfrak r_{\mi, p})_{\mi\in \N^*}$  and a sequence of integers $(\rk_{\mi, p})_{\mi \in \N}$  such that the following properties are satisfied:
  \begin{equation}
  \label{sequencitude2}
 \underset{\mi \in \N} \sum \mr_{\mi, p} = \frac 1 8, \     \ 
 \underset{\mi \in \N} \sum \mathfrak r_{\mi, p}^{m-p} \rk_{\mi, p}^p  < + \infty, \,    
   \underset{\mi \in \N} \sum \mr_{\mi, p}^{m-p} \   \rk_{\mi, p} ^{2p}   \,=+ \infty  {\rm \ and  \  }  3\mr_{\mi+1, p} \geq  \mr_{\mi, p}.
  \end{equation}
In the case $m-p > 1$,  a possible choice for these sequences is given,  for $\mi \geq 2$,   by 
$$\mr_{\mi, p}=\frac{{\rm c} }{\mi (\log \mi)^2 }   {\rm \ and \ }
\rk_{\mi, p}=\left[\mi^{\frac {m-p-1}{p}}\right],  {\rm \  where  \ }
{\rm c}=\frac {1}{8}\underset{ \mi \in \N}  \sum \frac{ 1}{\mi (\log \mi)^2 }.
$$
   In the case $0<m-p\leq 1$, we may choose instead   for $\mi \geq 2$,
$\displaystyle{\mr_i=\mi^{-\frac{3}{m-p}} }$ and  $\displaystyle{\rk_\mi=\mi^{\frac {1}{p}}}$. 
We define as above  a set of  points $\{\fM_\mi\}_{\mi \in \N}$ in $\B^m$  by
$$\fM_\mi= 4 \left (   \underset {j=0} {\overset {\mi} \sum}   \mr_j \right) \be_1, {\rm \ for \ } \mi \in \N, {\rm \  where  \ }  \be_1=(1,0,\ldots, 0).
$$
 These points converge to  $\fM_\star=\frac 12 \be_1 $ as $\mi \to + \infty$. We  consider the  collection of disjoint balls $\displaystyle{ \left(B_\mi\right)_{i\in \N}}$ defined by 
$B_\mi \equiv \B^m(\fM_i, \mr_i)$  for  $\mi \in \N$.
  We  then define  the map $\mathcal V_0$ on $\B^m$  as 
  \begin{equation}
  \label{matuiditude}
\mathcal V_0(x)= {\rm \bf  S}_{\rm \bf pag}^{k, m}\left (\frac {x-\fM_\mi}{\mr_\mi}\right), \ {\rm\  if \ } x \in B_\mi, \ \ 
  \mathcal U(x)=\sP,  \  {\rm \  if \ } x \in \B^4\setminus \underset  {\mi \in \N}  \cup B_\mi,
  \end{equation} 
   so that  $\mathcal V_0=\sP$ on  $\partial \B^m$.  Invoking the scaling properties  \eqref{scalingprop} of the $p$-energy, we are led to
 $$
\rE_p (\mathcal V_0, \B^m)=\underset{ \mi \in \N} \sum\,  \rE_p  (\mathcal V_0, B_\mi) = \underset{ \mi \in \N} \sum\,  \mr_\mi^{m-p}  \rE_p ({\rm \bf  S}_{\rm \bf pag}^{k, m})\leq
  \rC  \underset{ \mi \in \N}\sum\,  \mr_\mi^{m-p} \rk_\mi^p <+\infty, 
$$
 hence  $\mathcal V_0$ belongs to $W^{1,p} (\B^m, \S^2)$.  Next assume that there exists  a lifting $\rU_0$ of  $\mathcal V_0$ in $W^{1,p} (\B^m, \S^3)$ and consider its restriction $\rU_\mi$ to the ball $B_\mi$.  It follows from Lemma \ref{spagetude} and the scaling properties of the energy that 
 $$\rE_p  (\rU_0, B_\mi)\geq   C_p^m \mr_\mi^{m-p} \rk_\mi^{2p}
 {\rm \ and \  hence \ }  \ 
  \rE_p (\rU_0, \B^m) \geq C_p^m \underset { \mi \in \N} \sum\mr_\mi^{m-p} \rk_\mi^{2p}=+\infty,
 $$
 leading to a contradiction, which established the proof of the theorem in the special case considered in this 
  present subsection.
 
 \subsubsection{Proof of Theorem \ref{repulpage} completed  for a general manifold $\mM$}
 The argument is parallel to the argument in subsection \ref{subsectionitude}.  With the same notation, we set
 $$\mathcal V=\mathcal V_0 \left(\Phi (x)\right)  {\rm \ if \ }  x \in \mathcal O_a, \mathcal V=\sP {\rm \ otherwise},$$
  and we verify that the map $\mathcal V$ has the desired property.


\setcounter{equation}{0} \setcounter{subsection}{0}
\setcounter{lemma}{0}\setcounter{proposition}{0}\setcounter{theorem}{0}
\setcounter{corollary}{0} \setcounter{remark}{0}
\setcounter{definition}{0}
\renewcommand{\theequation}{A.\arabic{equation}}
\renewcommand{\thesubsection}{A.\arabic{subsection}}
\renewcommand{\thelemma}{A.\arabic{lem}}
\renewcommand{\thedefinition}{A.\arabic{definition}}
\section*{Appendix: related  notions on branched  transportation}
\renewcommand{\thesection}{A}
\numberwithin{equation}{section} \numberwithin{theorem}{section}
\numberwithin{lemma}{section} \numberwithin{proposition}{section}
\numberwithin{remark}{section} \numberwithin{corollary}{section}

In this Appendix,  we recall and recast  some  aspects of branched transportation,    an optimization  problem involved  in a  wide area of applications, including practical ones, for instance   leafs growth, or network  design.   We focus on questions  directly related to our main problem,  keeping   however this part \emph{completely} self-contained.

\smallskip
 Branched transportation appears when one seeks to  optimize   transportation costs when   the average  cost   decreases  with  density.       
Consider   a finite set $A$  of points belonging to   the closure of a  bounded open   domain $\Omega$ of $\R^m$: We  wish to connect (or transport) them  to the boundary  $\partial \Omega$. The total cost be   to  be  minimized  is   the sum of the length of paths  joining the given points  to the boundary multiplied by a  \emph {density function $\varphi$},  depending on the density  representing  the number of points using the same portion of paths. For minimizers, such paths are unions of segments, but possibly with varying densities.    The intuitive idea   is  that, when $\varphi$ is sublinear,  it is \emph{cheaper to  share the same path than to travel alone}, so that  \emph{high densities} are  selected by the minimization process.  This induces branching points, i.e. points where segments join to induce higher multiplicity.  The density function  appearing in our context, as well as in  a large of part of the literature,  is given  by the power law
$\displaystyle{\varphi(d)=d^\upalpha}$,  with given parameter  $0<\upalpha< 1$. Notice that $\varphi$ is sublinear, $\left( d_1+d_2\right)^\alpha <<d_1^\alpha+ d_2^\alpha$ for large numbers. 
Our aim is to describe the behavior of minimal branched transportation when the number of \emph{points increases and ultimately goes to $+ \infty$}. Special emphasis is put on the critical case $\upalpha=\upalpha_m=1-1\slash m$.
 Our presentation closely follows \cite{xia, xia2} and also  \cite{bermocas}: We perform  the necessary adaptation for connections to the boundary, which have been less considered so far. As far as we are aware of, the main result of this Appendix, presented in  Theorem \ref{droppy}, is new.

\subsection{ Directed graphs connecting  a finite set to the boundary}
\label{ukraine}
\subsubsection{Directed graphs and charges}

The  \emph{theory  of oriented  graphs} offers an appropriate framework to describe the object we have in mind\footnote {We might as well  invoke the theory of 1-dimensional integer currents, which is however far more abstract.}. Such oriented graphs involve: 
\begin{itemize}
	\item {\bf Points}.  These points    are of two kinds: The points in $A$ we wish to connect to the boundary, but also additional points,  the branching points and points on the boundary. 
	\item {\bf Oriented  segments}. These segments  join  points  described   points above. Orientation is important, as well as multiplicity which is a positive  integer.
\end{itemize}
 A   general directed graphs $G$  is defined by 
a  \emph{ finite }  set $E(G)$ of \emph{oriented segments  with endpoints belonging to $\overline{\Omega}$}: If $e$ is a segment in $E(G)$,  then we denote by $e^-$ and $e^+$ the endpoints of $e$, $e^-$ (resp $e^+$) denoting the entrance point (resp the exit point),  so that $e=[e^-, e^+]$ and $\partial e=\{e^-, e^+\}$.  We assume that for any segment $e$ in $E(G)$ the additional condition  
 
 \begin{equation}
 \label{asp2}  
 {\rm if \ } [e^-, e^+] \in E(G),  {\rm   \  then  \ } [e^+, e^-] \not \in E(G)
 \end{equation}
holds,  i.e. if an oriented segment  belongs to the graph, the segment with opposite direction does not.  Segments may be repeated with \emph{multiplicity}. If $e\in E(G)$, we denote by $d(e, G)\in \N^\star$ its multiplicity\footnote{This is of course an essential feature for branched transportation} and simply write $d(e)$ is this is not a source of confusion. 
 We denote  by  $\mathcal  G(\Omega)$ the set of graphs having the previous properties, namely
$$
\mathcal G(\Omega)=\left\{ {\rm graphs \ } G  {\rm \ such \ that \ }  \eqref{asp2} {\rm \ holds }  \right\}. 
$$
We denote by $V(G)$ be set of vertex  of the graph $G$,  i.e.
$$
V(G)=\underset {e\in E(G)}\cup \partial e=\underset {e\in E(G)}\cup \{e^-, e^+\}\subset  \overline{ \Omega}.
$$
Given a vertex $\upsigma \in V(G)$, we set 
 $$E^\pm (\upsigma, G)= \{  e\in E(G), e^\mp=\upsigma\} {\rm \ and \ }  E(\upsigma, G)=E^+(\upsigma, G) \cup E^-(\upsigma, G), $$
so that $E^+ (\upsigma, G)$ (resp $E^- (\upsigma, G)$) represents the set of segments of  $G$  having  $\upsigma $ as entrance point (resp. exit point) and $E(\upsigma, G)$ the set of segments having $\upsigma$ as endpoint.   We set 
$$
\sharp \left(E^\pm(\upsigma, G)\right)=\underset{e\in E^\pm (\upsigma, G)}\sum  d(e)  \in \N^\star, 
$$
 and introduce the notion of \emph{charge} of a point  $\upsigma \in V(G)$ as
   \begin{equation}
   \label{charge}
   \charge (\upsigma, G)=\sharp \left(E^+(\upsigma, G)\right)-\sharp \left(E^-(\upsigma, G)\right)  \in \mathbb Z .
     \end{equation} 
      We consider the subsets $V_0(G)$, $\Vcharged(G)$ and $\Vbd(G)$ of $V(G)$ defined by
  \begin{equation}
  \left\{
  \begin{aligned}
  V_0(G)&=\{\upsigma \in V(G),  \charge (\upsigma, G)=0\}\\ 
\Vcharged(G)&=\{\upsigma \in V(G), \charge (\upsigma,G)\neq 0, \upsigma \in V(G)\setminus( V_0(G)\cup  \partial  \Omega)  \} \\
\Vbd(G)&=\{\upsigma \in V(G),  \upsigma \in \partial \Omega  \}.
  \end{aligned}
  \right.
  \end{equation}
   A point $\upsigma \in  V_0(G)$ will be termed  \emph{a pure branching point}, a point in $V_{\rm chg}(G)$ a charged point or simply a \emph{charge}\footnote{Notice that a charged point may however also be a branching point}.   The set of graphs with only \emph{positive charges} plays a distinguished role in the later analysis. We set 
   \begin{equation}
   \label{bridou}
   \left\{
   \begin{aligned}
   \mathcal G^+(\Omega)&=\left\{ G\in \mathcal  G(\Omega), {\rm \ s.t.\  }\charge (\upsigma, G)\geq 0 \  \forall \upsigma \in V(G)\setminus \partial \Omega \right\}\ \\
    \mathcal G_0(\Omega)&=\left\{ G\in \mathcal  G(\Omega), {\rm \ s.t.\  }\charge (\upsigma, G)= 0 \  \forall \upsigma \in V(G)\setminus \partial \Omega \right\}.
       \end{aligned}
       \right.   
       \end{equation}   In several places, we will invoke the fact that, if $G\in \mathcal G^+(\Omega) $, then
       \begin{equation}
       \label{clearwater}
       E^+(\upsigma, G) \neq \emptyset,  {\rm \ for \ any  \ }  \upsigma \in V(G).
       \end{equation}
       Indeed, by definition  $E(\upsigma, G)$ contains at least one element, and since the charge is positive there are at least as many elements in $E^+(\upsigma, G)$ as in $E^-(\upsigma, G)$.
\subsubsection{Elementary operations on directed  graphs}
\label{mathelem}
  \noindent
  {\it Gluing graphs}.   Let
   $G_1$ and $G_2$ be two graphs in $\mathcal G (\Omega)$.  We assume furthermore that  
  \begin{equation}
  \label{enplus}
 {\rm if \ }  e_1\in E(G_1), \ e_2 \in E(G_2){\rm \  then  \  either  \ } e_1=e_2, {\rm \ or \ }  e_1\cap e_2 {\rm  \  contains \ at \ most \ one \ point.}
  \end{equation}
  If condition \eqref{enplus} is not met,  one may add new points and divide some segments in two so that the transformed graph satisfy the condition. Given  $e \in E(G)$,   we denote  $-e$  the segment with opposite orientation, i.e. if $e=[e^-, e^+]$, then $-e\equiv [e^+ , e^-]$.  We introduce  the following subsets of $E(G_1)\cup E(G_2)$
 \begin{equation*}
 \left\{
\begin{aligned}
 E_0(G_1, G_2)&\equiv \{e \in G_1 {\rm \ s.t. \ }, -e\in G_2 {\rm \ with \ } d(e, G_1)=d(-e, G_2) \}\\
E^+(G_1, G_2)&\equiv \{e \in G_1 {\rm \ s.t. \ } -e \not\in G_2  \} \cup\{e \in G_2 {\rm \ s.t. \ } -e \not\in G_1  \} \\
E^\pm(G_1, G_2)&\equiv \{e \in G_1 {\rm \ s.t. \ } -e \in G_2 {\rm \ with \ } d(e,G_1) >d(-e,G_2)\} \\
E^\mp(G_1, G_2)&\equiv \{e \in G_2 {\rm \ s.t. \ } -e \in G_1 {\rm \ with \ } d(e,G_2) >d(-e,G_1)\}
\end{aligned}
\right. 
 \end{equation*}
We   define the \emph{glued} graph
  \begin{equation}
  \label{decgraphe}
   G=G_1 \curlyvee G_2  \in \mathcal  G (\Omega), 
   \end{equation}
   given by the   set of its directed segments  
    \begin{equation}
    \label{jeankevin}
   \begin{aligned}
    E(G)&\equiv E(G_1)\cup E(G_2)\setminus E_0(G_1,G_2) \\
   &= E^+(G_1, G_2)\cup E^\pm(G_1, G_2)\cup E^\mp(G_1, G_2), 
    \end{aligned}
    \end{equation}
   with, for $e\in E(G)$,  multiplicities $d(e, G)$ given by
   \begin{equation}
   \label{longpoin}
   \left\{
  \begin{aligned}
   d(e, G)&=d(e, G_1)+d(e, G_2)  {\rm \ if \ } e \in E^+(G_1,G_2)\\
    d(e, G)&=d(e, G_1)-d(e, G_2)  {\rm \ if \ } e \in E^\pm(G_1,G_2)\\ 
    d(e, G)&=d(e, G_2)-d(e, G_1)  {\rm \ if \ } e \in E^\mp(G_1,G_2),
   \end{aligned}
   \right.
  \end{equation}
 where we have used the convention, for $i=1, 2$,  that $d(e, G_i)=0$ if $e \not \in G_i$.     The  vertex set $V(G_1\uvee G_2)$ is  then provided by the endpoints of the segments  in $E(G_1\uvee G_2)$, so that 
   $\displaystyle{V(G) \subset V(G_1) \cup V(G_2)}$.  The inclusion might be strict in the general case. We have:
 \begin{proposition} 
 \label{frumentum}
 Let $\upsigma \in V(G_1\uvee G_2)$. We have
 \begin{equation}
 \label{mbappe}
 \charge (\upsigma, G_1 \curlyvee G_2)= \charge (\upsigma, G_1)+ \charge (\upsigma, G_2),
 \end{equation}
 with the convention, for $i=1, 2$,  that $\charge(\upsigma, G_i)=0$ if $\upsigma \not \in G_i$. 
If    $G_i \in \mathcal G^+(\Omega)$ for $i=1, 2$,  then we have 
\begin{equation}
\label{jka}
 \Vcharged(G)= \Vcharged(G_1) \cup \Vcharged(G_2). 
\end{equation}
 \end{proposition}
  The result is  a direct consequence of \eqref{longpoin}.  We reader may check also that the gluing operation $\curlyvee$ enjoys classical properties as commutativity and associativity.  Finally we write 
  \begin{equation}
  \label{uvee}
  G=G_1\uvee G_2
  \end{equation}
in the case when, if  a segment $e$ belongs to $E(G_1)$, then the opposite segment does not belong to $G_2$, so that no cancellations for segments occur in the gluing process.  The set $E(G)$ is in that case the union $E(G_1)\cup E(G_2)$, the multiplicities being simply  summed.
  

  \medskip
  \noindent
 {\it Subgraphs.} Let $G_1$ and $G$ be two graphs in $\mathcal G(\Omega)$. We say that $G_1$ is a  subgraph of $G$ if $E(G_1) \subset E(G)$  and if the multiplicities satisfy the conditions
 \begin{equation}
 \label{sousmultiplicite}
 d(e, G_1)\leq d(e, G), {\rm \ for \  any \ segment \   } e \in E(G_1).
 \end{equation}
  If the  two previous conditions are  fullfilled, then we write $G_1\Subset G$.
 We introduce    the complement $G_2=G\smallsetminus G_1$ of $G_1$ with respect to $G$,  defining the set of oriented segments of $G_2$ as 
$$
E( G\smallsetminus G_1)=E(G_2) \equiv  \left[ E(G)\setminus E(G_1) \right] \cup E_{\rm comp}(G_1, G)
$$
 where $E_{\rm comp}(G_1, G)$ is defined as
\begin{equation*}
E_{\rm comp}(G_1, G)\equiv \{ e \in E(G_1), d(e, G_1)<d(e, G)\},
\end{equation*}
 and with multiplicities given by
 \begin{equation}
 \label{reglemult}
 \left\{
 \begin{aligned}
 d(e, G\smallsetminus G_1)&=d(e, G), {\rm \ if \ } e \in E(G)\setminus E(G_1),\\
 d(e, G\smallsetminus G_1)&=d(e, G)-d(e, G_1), {\rm \ if \ } e \in E_{\rm comp}(G_1, G).
 \end{aligned}
 \right.
 \end{equation}
 Notice that there are no segments in $G_1$ and  $G_2$ with opposite orientations.
 It follows from these definitions that  
$$G=G_1 \uvee G_2=G_1\uvee (G\smallsetminus G_1).
  $$
In view of Proposition \ref{frumentum} , if $G$ and $G_1$ belong to $\mathcal G^+(\Omega)$ with $G_1 \Subset G$, and if furthermore
\begin{equation}
\label{bonsous}
\charge(\upsigma, G_1)\leq \charge (\upsigma, G), 
{\rm \ for  \ any \ }
\upsigma \in \Vcharged (G),
\end{equation}
 then $G_2 \in \mathcal G^+(\Omega)$.   If $G_1$ and $G$ are two graphs in $\mathcal G(\Omega)$ such that $G_1\Subset G$ and  such that condition \eqref{bonsous} is satifies, then we write $G_1 \sousgraphe G$.


\medskip
\noindent
{\it Restrictions of graphs to subdomains}. Let   $\Omega_1 \subset \Omega$ be a subdomain of $\Omega$  and assume  for  the sake of simplicity (and also for further applications) that both   $\Omega_1$ and $\Omega$ are polytopes.  Let  $G$  be a graph in $\mathcal G(\Omega)$.  We define the restriction  $G_1=G  \rest \Omega_1$ of $G$ to   $\Omega_1$ as  the graph in $\mathcal G(\Omega_1)$ such that  its set of segments $E(G  \rest \Omega_1)$ is given by 
$$\displaystyle{E(G  \rest \Omega_1)=\{e \cap \bar \Omega_1, e \in E(G)\}}.$$
Its set of vertices is then given by
$$
V(G_1)=\left (V(G)\cap \Omega_1\right) \cup \left(\underset{ e\in E(G)}\cup  \partial \left ( \bar e \cap \bar {\Omega} _1\right )\right).
$$ 
 One may check that $G_1\in \mathcal G(\Omega_1)$ and also  $G_1\in \mathcal G(\Omega)$. If we assume moreover that $G\in \mathcal G^+(\Omega)$, then we have $G\in \mathcal G^+(\Omega_1)$, but it does not belong, in general to 
$\mathcal G^+(\Omega)$, since negative charges may be created on $\partial \Omega_1$.


\subsubsection{The single path property}
 The next property,  termed \emph{the single path property}, has been  considered in  \cite{xia, xia2, bermocas}.
\begin{definition}  Let $G \in \mathcal G(\Omega)$. We say that $G$ possesses the single path property, if for any vertex $\upsigma \in V(G)\cap \Omega$ there is at most one segment $e$ in $E(G)$, possibly repeated with multiplicity,  such that $\upsigma$ is the entrance point of $e$, that is 
$E^+(\upsigma, G)$ is a singleton or empty. 
\end{definition}

In other words, if $G$ possesses the single path property, then there might be several segments ending at the same vertex, but at most one starting from it. This property  possibly models some intuitive features, as for  instance in   river networks. We denote by $\mathcal G_{\rm sp}(\Omega)$  (resp. $\mathcal G_{\rm sp}^+(\Omega)$) the  set of all graphs in $\mathcal G(\Omega)$ (resp. 
$\mathcal G^+ (\Omega)$) which possess the single path property. Notice that if $G \in \mathcal G^+_{\rm sp}(\Omega)$, then  $E^+(\upsigma, G)$ is not  empty for $\upsigma \in \Omega$, so that it is necessarily a singleton. 

\subsubsection {Threads, loops and bridges}
 A heuristic image  of the notion of thread we describe next,  is provided by a a curve for with one end is given by a point in  $\overline{\Omega}$, reaching to the boundary $\partial \Omega$, and constructed using  only  segments in $E(G)$.   This image suggests the following definition. 
 
 \begin{definition}
 \label{fildefer}
  A  directed graph $G$ is said to be a polygonal curve in $\Omega$, in short a $\Pom$-curve,  if and only if there exists an ordered collection $B=(b_1, \ldots, b_{\rq})$ of $q$  not necessary distinct points in $\bar \Omega$  such that $G$ satisfies  $V(G)=B$, relation  \eqref{asp2} holds,  and
 \begin{equation}
 \label{filoche}
  E(G)=\{[b_i, b_{i+1}],{\rm \ with \ multiplicity \  } 1,  i=1, \ldots, \rq\}
 {\rm \ and  \ }
 b_\rq \in \partial \Omega  {\rm \ or \ } b_\rq=b_1.
 \end{equation}
 \end{definition}
  Since the $\Pom$-curve $G$ is completely determined by the orderet set $B$, we may set 
   $$G= {\rm G}_{\rm rp} (B).$$
  Even if in \eqref{filoche} each segment $[b_i, b_{i+1}]$ appears with multiplicity one, the same segment may be repeated   further in  the sequence, so that its final multiplicity might be larger that one.

\begin{definition}
\label{filocher}
 Let $G= {\rm G}_{\rm rp} (B)$ be a $\Pom$-curve. 
 We say that ${\rm G}_{\rm rp} (B)$ is 
\begin{itemize}
\item  a loop if   $b_1=b_{\rq}$.
\item a bridge if $b_1\in \partial \Omega$ and $b_\rq \in \partial \Omega$.
\item A thread emanating  from a point $p \in \Omega$ if  $p=b_1$ and $b_{\rq} \in \partial \Omega$. 
\end{itemize}
\end{definition}
 A first elementary observation is:
 
 \begin{lemma}
 \label{doublon} Let  $B=(b_1, \ldots, b_{\rq})$ of $q$  be an ordered collected of points in  $\bar \Omega$. We have 
$ \charge(b_i,  {\rm G}_{\rm rp} (B))=0$ for   $i=2, \ldots, q-1$ and for $i=1, q$ if ${\rm G}_{\rm rp} (B)$ is a loop. If ${\rm G}_{\rm rp} (B)$ is not a loop, then we have 
 \begin{equation}
 \label{doublon1}
 \charge(b_1,  {\rm G}_{\rm rp} (B))=1  
  {\rm \ and \ }  \charge(b_q,  {\rm G}_{\rm r} (B))=-1.
 \end{equation}
\end{lemma}
 
\begin{proof}    Given $b_i$ in $B$, we denote by $m(i)$ its multiplicity:    If $i\in \{2, q-1\}$, then $b_i$  is $m(i)$ times an entrance point as well as an exit point, so that the first assertion follows. The others follow similar arguments. 
\end{proof} 
 
  We denote by   $\mathcal T_{\rm hread}(p,  \Omega)$  the set of  all threads emanating from $p$. We notice that
  \begin{equation}
  \label{chargefred}
  \left\{
  \begin{aligned}
    \Vcharged(\Grp(B))&=\emptyset {\rm \ when \ } \Grp(B) {\rm \ is \ a  \ loop \ or \ a \ bridge}, \\ 
 \Vcharged(\Grp(B))&=\{a\} {\rm \ with \ } \charge(a)=1,{\rm \ if \ } \Grp(B)  {\rm \ is \ in \ } \mathcal T_{\rm hread}(p,  \Omega).
  \end{aligned}
    \right.
    \end{equation}
Loops and bridges are elements in $\mathcal G_0(\Omega)$.   We denote by $\mathcal L_{\rm oop} (\Omega)$ the set of loops.   We say that a graph $G$ has a loop if there exists a loop $L$ such that $L \Subset G$. 
 In particular a thread $G= {\rm G}_{\rm rp} (B)$ has  \emph{ a loop}  if  there exists  a subset  formed of consecutive  points in $B$  yielding  a loop. Given a point  $p \in \Omega$ we denote by   
 $$\mathbbmss T_{\rm hread}(p,  \Omega)\subset \mathcal T_{\rm hread}(p, \Omega)$$
   the set of  all threads \emph{without loops emanating} from $p$. It follows rather straightforwardly from the definitions  that the segments of a thread in $\mathbbmss T_{\rm hread}(p,  \Omega)$ have  \emph{multiplicity} one and that, if  a thread has the single path property, then it has no loops.
  One may moreover verify:
  
  \begin{lemma}
  \label{lupoon}
  Let  $\upsigma \in \Omega$ and let    $T \in \mathcal T_{\rm hread }(\upsigma, \Omega) $. There exists a finite  family   $(L_j)_{j \in J}$  of loops  such that
 \begin{equation}
 \label{lupo}
  T= T_ p { \uvee} \left(\underset { j \in J} \uvee  L_j \right),  {\rm \ with \ } T_p \in \mathbbmss T_{\rm hread}(p,  \Omega). 
  \end{equation}
 \end{lemma}
\begin{proof}  We may write $T=\Grp(B)$ where $B$ denotes an ordered set $B=\{b_1=\upsigma,  b_2, \ldots, b_{\rq}\}$, with $b_{\rq} \in \partial \Omega$.  If all points  in $B$ are distinct, then $T \in \mathbbmss T_{\rm hread}(\upsigma, \Omega)$ and there is nothing to prove. Otherwise there are two points, say $b_{i_1}$ and $b_{i_2}$ with $1\leq i_1 <i_2 <b_{\rq}$ which are identical.  Then,  we set 
$L_1=\Grp\{b_{i_1}, \ldots, b_{i_2}=b_{i_1}\}$ and ${\tilde T}_1=\Grp\{b_1, \ldots, b_{i_1}, b_{i_2+1}, \ldots, b_{\rq}\}$. We verify that 
$$  T=\tilde T_1 \uvee L_1, {\rm  \  with \ }  \tilde T_1 \in \mathcal T_{\rm hread}(p,  \Omega)  {\rm  \ and \   } L_1 {\rm \ is \ a  \ 	 loop}.$$
If $\tilde T_1$ has  no loop, then we are done. Otherwise,  we start the process again with $T$ replaced  by $\tilde T_1$. It stops in a finite number of iterations, since the number of points is finite. 
\end{proof}
\subsubsection{Maximal subcurves, subthreads and subloops}
  Consider a graph $G$ in $\mathcal G(\Omega)$ and an  ordered  set $B=(b_1, \ldots, b_{\rq})$ of elements of $V(G)$.
  
  \begin{definition}
  The $P_\Omega$-curve 	$\Grp(B)$ is said to be a \emph{maximal subcurve} of $G$ if $\Grp(B)\Subset G$,  if $b_i \in \Omega$ for $i=1, \ldots, b_{\rq-1}$ and
either $b_\rq \in \partial \Omega$ or $E^+\left(b_\rq, G \setminus \Grp(B)\right)$ is empty.
  \end{definition}
 If $\Grp(B)$ is  a \emph{maximal subcurve} of $G$ and $b_{\rq}\in \Omega$, then we have  $E^+\left(b_\rq, G \setminus \Grp(B)\right)=\emptyset$, which  implies that 
 $\charge (b_\rq, G\setminus \Grp(B))\leq 0$, and hence 
 \begin{equation}
 \label{doublon2}
\charge (b_\rq, G)=\charge (b_\rq, G\setminus \Grp(B))+ \charge (b_\rq,  \Grp(B))\leq  \charge (b_\rq,  \Grp(B)).
\end{equation}
 Our next result readily follows from the definition:
 
 \begin{lemma}
 	\label{excellence}
 	Let $G \in \mathcal G(\Omega)$ and $\upsigma \in V(G)$.  There exists an ordered set $B=(b_1, \ldots, b_{\rq})$ such that $b_1=\upsigma$ and such that  $\Grp(B)$ is a maximal subcurve of $G$.
 \end{lemma}
 \begin{proof}
 	  We construct the maximal subcurve inductively.  If $E^+(\upsigma, G)=\emptyset$,  we take $B=\{\upsigma\}$ and we are done. Otherwise, we choose some $b_2\in E^+(\upsigma, G)$, so that  $[\upsigma, b_2] \in E^+(\upsigma, G)$ and therefore $\Grp\{\upsigma, b_2\} \Subset G$. If $b_2 \in \partial \Omega$  or if $E^+(b_2, G)=\emptyset$, then $B\equiv \{\upsigma, b_2\}$ is maximal and we are done. Otherwise, since $E^+(b_2, G)$ is not empty,  there exists some point $b_3 \in V(G)$ such that $[b_2, b_3] \in E^+(b_2, G\setminus \Grp \{\upsigma, b_2\})$ and therefore $\Grp\{\upsigma, b_2, b_3\} \Subset G$. If $b_3 \in \partial \Omega$ or if  $E^+(b_2, G\setminus \Grp \{\upsigma, b_2\})$,  then $B\equiv \{\upsigma, b_2, b_3\}$ is maximal and we are done. Otherwise we go on, until we reach the boundary or have no more segments available to go on.
	  
 \end{proof}

\begin{lemma}
	\label{looping}
Let $G\in \mathcal G^+(\Omega)$.	 A maximal subcurve $\Grp(B)$  of   $G$ is either a thread emanating from  a point $\upsigma  \in \bar \Omega$ or a  loop.
\end{lemma} 
 
 \begin{proof}  If $b_{\rq} \in \partial \Omega$, then $\Grp(B)$ is a thread emanating from $\upsigma$ and the statement is proved. If   $b_{\rq} \in \Omega$, we argue by contradiction, and assume that $\Grp(B)$ is not a loop, so that $b_1\not = b_\rq$.  It then follows from Lemma \ref{doublon} that $\charge(b_\rq, \Grp(B))=-1$. Hence, invoking \eqref{doublon2}, we obtain $\charge(b_\rq, G)\leq-1$. This is a contradiction with the fact that $G \in \mathcal G^+(\Omega))$. 
    \end{proof}
 
 \begin{lemma}
 	\label{maxcourbe}
Let  $B=(b_1, \ldots, b_{\rq})$  be such that $\Grp(B)$ is a maximal subcurve of $G$. If $\charge(b_1, G)>0$, 
 then $\Grp(B)$ is a thread emanating from $b_1$. 
 \end{lemma}
\begin{proof} We argue by contradiction: If the statement is not true, it follows from Lemma \ref {excellence} that $\Grp(B)$ is a loop, so that $b_1=b_\rq$, and in view of lemma \ref{doublon}, $\charge (b_1, \Grp(B))=0$.  Invoking  \eqref{doublon2},  we derive that $ \charge (b_1, G)=\charge (b_q, G)\leq 0$,  in  contradiction with the assumption $\charge(b_1, G)>0$. \end{proof}
  
  Combining Lemmas \ref{lupoon}, \ref{excellence}, \ref{looping} and \ref{maxcourbe}, we deduce:
  
  \begin{corollary}
  \label{nary}
  \label{idest}  Let $\upsigma \in \Vcharged{G}$. There exists a thread $T_\upsigma  \in \mathbbmss T(\upsigma, \Omega)$ such that 
  $T_\upsigma \sousgraphe G$.
     \end{corollary}
  
   \subsubsection {Decomposing graphs  into threads and loops and bridges}
  \label{threads}
  Consider a graph  $G \in \mathcal G^+(\Omega)$. Since $V(G)$ is a finite set, we may write  
 $$\Vcharged(G)=\{p_1, \ldots, p_{_{\ell_{\rm c}}}\}, $$
  each point $p_i$ in the collection having multiplicity  $\Rm_i\equiv \charge (p_i, G)\in \N^*$.
 The following result emphasizes the importance of threads in this context:

\begin{proposition} 
\label{decompresse}
 Let $G \in \mathcal G^+(\Omega)$. We may decompose  the graph $G$ as 
\begin{equation}
\label{decortique}
 G=\underset { i \in \{1, \ldots, \ell_{\rm c}\} } \uvee \left(\underset{j\in \{1, \ldots, \Rm_i\}}\uvee T_{i, j}\right)   \uvee T_0,  \,  {\rm \ with \ }
  T_{i, j} \in  \mathbbmss T_{\rm hread}(p_i, \Omega)  {\rm \ and  \ }   T_0\in \mathcal G_0( \Omega).
 \end{equation}
 If moreover $G \in \mathcal G_{\rm sp}^+(\Omega)$, that is if $G$ possesses the single path property,  then decomposition \eqref{decortique} is unique  and, for any $i\in \{1, \ldots, \ell_{\rm c}\}$, we have 
 \begin{equation}
 \label{grosfilex}
 T_{i, j}=T_{i, j'},   {\rm \ for \ }  j {\rm \ and \ } j' {\rm \ in  \ } \{1, \ldots, \Rm_i\}.
 \end{equation}
 \end{proposition}

\begin{proof}
	 We present  first  the  construction of the subgraph $T_{1, 1}$ and then proceed  recursively. \\

\noindent  
{\bf Step 1: Construction of $T_{1, 1}$}.
  Since the point $p_1$ has positive charge $\Rm_1$ with respect to $G$, we may apply  Corollary \ref{nary} and choose $T_{1, 1}=T_{p_1}$, to define  the graph 
$$
G_{1, 1}=  G\setminus { T}_{1, 1},  \  {\rm  and \ hence \ }   G=G_{1, 1} \uvee \, {T}_{1, 1},   {\rm \ with \ } G_1\in \mathcal G^+(\Omega). 
$$
the total charge of $G_{1,1}$ compared to the total charge of $G$ has  decreased by $1$. More precisely, it follows from the rules  \eqref{reglemult}  and \eqref{chargefred}  for charges that  for $i=2, \ldots, \ell_{\rm c}$, we have 
$$ \charge(p_i, G_{1, 1})=\charge (p_i, G),  {\rm \ for \  } i=2, \ldots, \ell_{\rm c},   {\rm \ and \ }  \charge(p_1, G_{1, 1})=\charge (p_1,  G)-1, 
$$
in case $p_1\in V(G_{1, 1})$, which occurs in particular in $p_1$ has multiplicity. In the case $\ell_{\rm c}=1$  and $\Rm_1=1$, we deduce that 
$G_{1, 1}\in \mathcal G_0(\Omega)$, so that setting  $T_0=G_{1, 1}$, we obtain \eqref{decortique}. Otherwise, we proceed recursively.

\medskip
\noindent
{\bf Step 2: Iterating the construction}.  We proceed as in step 1, but  with $G$ replaced by $G_{1,1}$.   If $\Rm_1 >1$, we  invoke Corollary \ref{nary} again to assert that there exists  a thread   $\tilde T_{1, 2}  \in \mathbbmss T_{\rm hread}(p_1, \Omega)$   which is a subgraph of $G_{1, 1}$. We set $G_{1, 2} =G_{1, 1}\setminus { T}_{1, 2}$ so that we have $G_{1, 1}=G_{1, 2} \uvee \, { T}_{1, 1}$ and  
$$
\charge(p_i, G_{1, 2})=\charge (p_i, G),  {\rm \ for \  } i=2, \ldots, \ell_{\rm c},   {\rm \ and \ }  \charge(p_1, G_1)=\charge (p_1,  G)-2.
$$
If $\Rm_1=2$ and $\ell_{\rm c}=1$, then we are done,  we obtain \eqref{decortique} with  $ T_0=G_{1, 2}$. Otherwise, we  proceed with $G_{1, 2}$ and construct iteratively the threads $\tilde T_{1, 3}, \ldots, \tilde T_{1, \Rm_1}$, and  then $\tilde T_{2, 1}, \ldots, T_{2, \Rm_2}$, ... 
$\tilde T_{\ell_{\rm c}, 1}, \ldots, \tilde T_{\ell_{\rm c}, \Rm_{\ell_{\rm c}}}$. Setting $\ T_0=G_{\ell_{\rm c},  \Rm_{\ell_{\rm c}}}$ we obtain formula \eqref{decortique}. 
\end{proof}


 
     \subsubsection{Prescribing charges and the Kirchhoff law}
We are now in position to model connections of discrete sets to the boundary with possible branching points. Consider a finite set     $A \subset \overline{\Omega}$, with points  repeated with multiplicity   ${\Rm}(a) \in \N^{\star}$, so that 
     $\sharp(A)=\underset{a\in A}\sum {\rm M} (a)$.   We restrict ourselves   to  graphs  $G \in \mathcal G^+(\Omega)$ satisfying 
     \begin{equation}
     \label{prescharge}
     \Vcharged (G)=A {\rm \ and \ } \charge (a, G)={\Rm}(a), \ \forall a \in  A \cap \Omega.
     \end{equation}
 Condition \eqref{prescharge} is equivalent to \emph{Kirchhoff's law }   
      \begin{equation}
     \label{kirchhoff}
     \left\{
     \begin{aligned}
     & \sharp \left(E^+(a , G)\right)=
     \sharp \left(
     E^- (a, G)\right)+ {\rm M}(a),  {\rm \, \ for  \ }  a\in A\cap\Omega\subset V(G) \\  
     &\sharp \left(E^+(\upsigma , G)\right)=\sharp \left( E^- (\upsigma, G)\right),  {\rm \,  \ for  \ any \ }  \upsigma \in V(G)\cap \Omega \setminus A.
     \end{aligned}
     \right. 
     \end{equation}
    We introduce  the class of graphs aimed to model connections  of points in $A$ to the boundary
     \begin{equation} 
     \label{mathcalG}
     \mathcal G(A, \partial \Omega)=\{ G \in \mathcal  G^+(\Omega) \  {\rm such \ that \ }  A\subset V(G) {\rm \ and \ }\eqref{prescharge}  {\rm \ holds }\}.
     \end{equation}
     If $G$ belongs to $\mathcal G(A, \partial \Omega)$, then the points   of $A$ are the only "source" points of the graph inside $\Omega$, with charge ${\rm M}(a)$, whereas all the other points have charge $0$.   The simplest  example $G_0$ of an element in $\mathcal G(A, \partial \Omega)$  when $\Omega$ is convex is provided by the graph for which each element $a$ in $A$  is connected  by a segment to an element of the boundary $b$ so that in this case $\displaystyle{V(G_0)=\underset {a \in A}\cup\{a, b\}}$ and   
     $\displaystyle{E(G)=\underset {a \in A}\cup\{[a, b]\}}$.  
      Going back to \eqref{decortique}, if $G \in  \mathcal G(A, \partial \Omega)$,  then we have 
     \begin{equation}
     \label{gthread}
     G_{\rm thread}\equiv \underset { i \in \{1, \ldots, \ell_{\rm c}\} } \uvee \left(\underset{j\in \{1, \ldots, \Rm_i\}}\uvee T_{i, j}\right)  \in G(A, \partial \Omega). 
     \end{equation}

       \begin{figure}[h]
	\centering
     	\includegraphics[height=7cm]{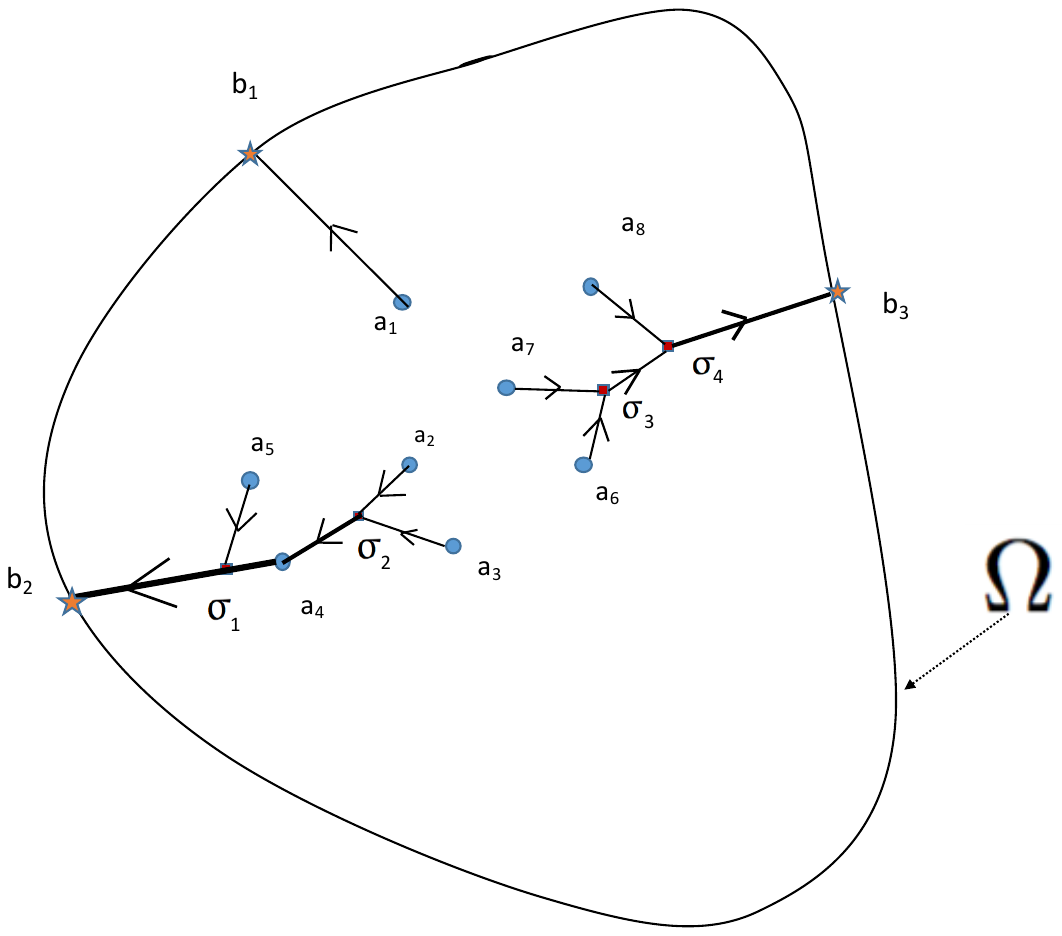}
     	\caption{  {\it Branched transport of the points $a_i$.	}}
     	\label{graphetrans}
     \end{figure}
     
\begin{remark}
\label{auchoix}
{\rm In  definition \eqref{mathcalG} we  allow points in $A$ to belong to $\partial \Omega$: This, perhaps unnatural choice,  is motivated by the fact that we face such a situation in Subsection \ref{zertyu},  the chosen  convention simplifying  somewhat the presentation. However, one may verify that 
\begin{equation}
\label{follow}
  \mathcal G(A, \partial \Omega)=\mathcal G(A\setminus \partial \Omega, \partial \Omega).
  \end{equation}
}
\end{remark}

 \subsection{The functional and minimal branched connections to the boundary}
Given $0\leq \upalpha \leq 1$, we  consider the functional  ${\rm W}_\upalpha$ defined on  the set  $\mathcal G( \Omega)$ by 
 \begin{equation}
 \label{walpha}
 {\rm W}_\upalpha(G)= \underset { e\in E(G)}\sum (d(e))^\upalpha  \mathcal H^1 (e),  \, {\rm \ for \ }    G \in \mathcal G(\Omega), 
 \end{equation}
 and the non-negative quantity
\begin{equation}
\label{waluigi}
\Lbra(A, \partial \Omega)= \inf\left\{ \rW_\upalpha(G), G \in \mathcal G (A, \partial \Omega)\right\}, 
\end{equation}
 which we will term the branched connection  of order $\upalpha$ of the set $A$ to the boundary $\partial \Omega$. Notice that the case $\upalpha =1$  has already been introduced in \cite{BCL} as minimal connection to the boundary. 
Using, among other arguments, the fact that 
$${\rm W}_\upalpha (G) \leq {\rm W}_\upalpha (G_{\rm thread}), \,   {\rm \ where \ } G_{\rm thread}  {\rm \ is \ defined \ in  \ } \eqref{gthread}, $$
with equality if and only if $T_0$ in \eqref{decortique} is empty, 
  it can be proved, as in \cite{xia}:
 
 \begin{lemma} 
 \label{xialemma}
 The infimum in \eqref{waluigi} is achieved by some graph 
 $G_{\rm opt} ^{^\upalpha } \in  \mathcal G (A, \partial \Omega)$. The graph $\Gopt$ has no loops and we  may  therefore write
 \begin{equation}
 \label{grouic0}
 \Gopt=  G=\underset { a \in A } \uvee \left(\underset{j\in \{1, \ldots, \Rm(a)\}}\uvee T_{a, j}\right),
   {\rm \ with \ } T_{a, j} \in  \mathbbmss T_{\rm hread}(a,  \Omega).
\end{equation}
Moreover, we have  $d(e)\leq \sharp (A)$ for any $e \in E(\Gopt).$
 \end{lemma}

 We notice that, as a  straightforward consequence of  \eqref{follow}, we have 
 \begin{equation}
 \label{follow2}
   \Lbra(A, \partial \Omega)=\Lbra(A\setminus \partial \Omega, \partial \Omega).
 \end{equation}
 We next show, similar to results in \cite{xia, xia2, bermocas}:
 \begin{lemma}
 \label{singlemalt}
 The graph $\Gopt$ possesses the single path property. 
 \end{lemma}
 
 \begin{proof} We  argue by contradiction and assume that there exists some vertex $\upsigma_0\in V(\Gopt)$ and two  distinct vertices $\upsigma_1$ and $\upsigma_2$ in
  $V(\Gopt)$ such that $[\upsigma_0, \upsigma_i] \in E(\Gopt) $  for $ i=1, 2$.   In view of the decomposition \eqref{grouic0}, there are two points  $a_1$ and $a_2$ in $A$ such that, for $i=1, 2,$ the segment $[\upsigma_0, \upsigma_i]$ belongs to $E(T_i)$
 where  $T_i$ is a thread  of the form $T_{a_i, j_i}$  appearing in \eqref{grouic0}.  
 %
   In the case the two threads have a common  vertex  $b_{\rm com}$ past $\upsigma_0$ we decompose $T_i$  as
  \begin{equation}
  \label{calo}
 T_i=\Grp(B_i)=\Grp(B_{0, i}) \uvee \Grp(B_{1, i})\uvee \Grp(B_{2, i}),
\end{equation}
 where $B_{0, i}=\{a_i, \ldots, \upsigma_0\}$, $B_{1, i}=\{\upsigma_0,  \upsigma_i, b_{i, 2}, \ldots, b_{i, \ell_i}=b_{\rm com}\}$ and $B_{2, i}=\{b_{\rm com}, \ldots, b_{\rq_i}\}$, with 
 $b_{\rq_i} \in \partial \Omega$. In the case the two threads have no vertex in common past the vertex $\upsigma_0$,  we use also decomposition \eqref{calo}, $B_{2, i}$ being void.
In order to obtain a contradiction, we compare the energy of  $\Gopt$ with the energy of two comparison graphs $\tilde G_1$ and $\tilde G_2$, which correspond  to an interchange of the threads $T_1$ and $T_2$, or more precisely to the parts $B_{1, 1}$ and  $B_{1, 2}$.  We first consider the modified threads
 $$
 \tilde T_1=\Grp(B_{0, 1}) \uvee \Grp(B_{1, 2})\uvee \Grp(B_{2, 1})  {\rm \ and \ } 
  \tilde T_2=\Grp(B_{0, 2}) \uvee \Grp(B_{1, 1})\uvee \Grp(B_{2, 2}). 
 $$
  We then define 
  \begin{equation*}
  \tilde G_1=\left(G \setminus T_1\right) \uvee \tilde T_1  {\rm \ and \ }   \tilde G_2=\left(G \setminus T_2\right) \uvee \tilde T_2. 
  \end{equation*}
  One verifies that $\tilde G_i \in \mathcal G(A, \partial \Omega)$  for $i=1, 2$.  For $i=1, 2$ and   $j=0, \ldots, \ell_i-1$ we set 
  $e_{i, j}\equiv  [b_{i, j}, b_{i, j+1}]$ where $b_{i, 0}=\upsigma_0$ and $b_{i, 1}=\upsigma_i$. Setting $d_{i, j}=d(e_{i, j}, G)$, we observe that  
  \begin{equation*}
  \left\{
  \begin{aligned}
  d(e_{1,  j}, \tilde G_1)&= d_{1, j} -1, {\rm \ for \ } j=1, \ldots, \ell_1,  {\rm \ and \  } 
  d(e_{2,  j}, \tilde G_2)=d_{2, j} +1, {\rm \ for \ } j=1, \ldots, \ell_2, \\
    d(e_{2,  j}, \tilde G_1)&= d_{i, j} +1, {\rm \ for \ } j=1, \ldots, \ell_1,  {\rm \ and \  } d(e_{1,  j}, \tilde G_2)=d_{2, j} -1,  
    {\rm \ for \ } j=1, \ldots, \ell_2. 
    \end{aligned}
    \right.
 \end{equation*}
  All other segments have the same density as for $\Gopt$. It follows : 
    \begin{equation*}
  \left\{
  \begin{aligned}
   W_\upalpha(\tilde G_1)- W_\upalpha(\Gopt)&=
   \underset{j=0} {\overset {\ell_1-1}\sum}\left[ \left(d_{1, j}-1\right)^\upalpha- d_{1, j}^\upalpha \right]\vert e_{1, j} \vert
   + 
   \underset{j=0} {\overset {\ell_2-1}\sum} \left[(d_{2, j}+1)^\upalpha -d_{2, j}^\upalpha  \right] \vert e_{2, j} \vert \geq 0\\
     W_\upalpha(\tilde G_2)- W_\upalpha(\Gopt)&=
   \underset{j=0} {\overset {\ell_1-1}\sum}\left[ \left(d_{1, j}+1\right)^\upalpha- d_{1, j}^\upalpha \right]\vert e_{1, j} \vert
   + 
   \underset{j=0} {\overset {\ell_2-1}\sum} \left[(d_{2, j}-1)^\upalpha -d_{2, j}^\upalpha  \right] \vert e_{2, j} \vert \geq 0.
   \end{aligned}
 \right.
  \end{equation*}
  Adding these inequalities we obtain
  $$
 { \underset{i=1}{\overset{2} \sum}} {\underset {j=0} {\overset{\ell_i-1} \sum}}
\left [\left(d_{i, j}+1\right)^\upalpha+\left(d_{i, j}-1\right)^\upalpha-2 d_{i, j}^\upalpha \right] \geq 0.
  $$ 
  By concavity of the density function $\phi(d)=d^\upalpha$, we have however  for $d\geq 1$
  $$ \displaystyle{(d+1)^\upalpha+(d-1)^\upalpha-2d^\upalpha<0}, $$
   so that we have reached  a contradiction which establishes the announced result. 
   \end{proof}

\begin{remark} 
\label{brancion}
{\rm 
 Using simple comparison arguments, one may easily prove that if $A$ and $B$  are disjoint finite subsets of $\Omega$,  then 
  \begin{equation}
 \label{lbrasomme}
 \Lbra (A \cup B, \partial \Omega) \leq\Lbra (A, \partial \Omega)+
 \Lbra (B, \partial \Omega).
 \end{equation} 
   Similarily,  if $0<\upalpha' \leq \upalpha$  and  $A \subset \Omega \subset \Omega'$,   then we have
\begin{equation}
   {\rm L}^{\upalpha'}_{\rm branch} (A, \partial \Omega)  \leq  \Lbra  (A, \partial \Omega)
 {\rm  \ and \ }  \Lbra (A, \partial \Omega) \leq\Lbra (A, \partial \Omega').
 \end{equation}
}
\end{remark}

 \begin{remark}
 \label{planine}
 {\rm
We will be led to consider the case $\Omega$ is a polytope,  $\bar \Omega=\Omega_1 \cup \Omega_2$, where $\Omega_1 \cap \Omega_2=\emptyset$, $\Omega_1$ and $\Omega_2$ being polytopes.  Given a graph $G\in \mathcal G^+(\Omega)$, one verifies that 
\begin{equation}
\label{platane}
W_\upalpha (G)=W_\upalpha (G_1) + W_\upalpha (G_2), {\rm \ where \ }  G_\fp= G\rest  \Omega_\fp {\rm \ for \ } \fp=1,2.
\end{equation}
  Assume next that $G\in \mathcal G(A, \Omega)$, where $A  \subset  \Omega$  is a finite set.  We have, for $\fp=1, 2$
  \begin{equation}
  \label{platine}
  G_\fp  \in \mathcal G(A_\fp, \Omega_\fp), {\rm \ so  \ that \ } W_\upalpha (G_\fp ) \geq  \Lbra (A_\fp, \partial \Omega_\fp),
{\rm \ where \ } G_\fp= G \rest \Omega_\fp. 
\end{equation}
}
 \end{remark}
 
  In the next subsection, we will be  concerned with the asymptotic behavior of  $\Lbra  (A, \partial \Omega)$ as the number of elements in $A$ tends to $+\infty$, specially in the case they are equi-distributed. Our methods rely on various decompositions.

 \subsection{Decomposing the domain and the graphs}
  We discuss here  issues   related to partitions of the domain $\Omega$, assuming it is a polytope. We  consider the case where the set $\Omega$ is decomposed as a finite  union 
 \begin{equation}
 \label{decompga}
 \bar \Omega=\underset{\mathfrak p \in \fP}\cup \bar  \Omega_{\mathfrak p}, 
 {\rm \ where \  the  \  sets \ } \Omega_\fp {\rm \  are \  disjoint \ polytopes \ i. e. \  }  \Omega_\fp\cap \Omega_\fp'=\emptyset {\rm  \ for \ }
 \fp\not = \fp'. 
 \end{equation}
Given    a finite subset $A$ of $\Omega$,  
 we have the lower bound, in view of Remark \ref{brancion}
\begin{equation}
\label{decomlbra}
\Lbra  (A, \partial \Omega)\geq \underset {\fp\in \fP}  \sum \Lbra  (A_\fp, \partial \Omega_\fp), 
{\rm \ where \ } 
A_\fp= \bar \Omega_\fp \cap A.
\end{equation}
Indeed, if $G$ is a graph in $\mathcal G(A, \partial \Omega)$, then  the restriction $G_\fp$ to the subset $\Omega_\fp$  belongs to $\mathcal G(A_\fp, \partial \Omega_\fp)$. On the other hand, we have
$$
\rW_\upalpha (G)=\underset{\fp \in \fP} \sum \rW_\upalpha (G_\fp), 
$$
 from which the conclusion \eqref{decomlbra} is deduced.   We assume next that  $\fP=\{1,2\}$, that is $\bar \Omega=\overline{\Omega_1} \cup \overline{\Omega_2}$, where $\Omega_1 \cap \Omega_2=\emptyset$, $\Omega_1$ and $\Omega_2$ being polytopes. 
 Our next result,  is an improvement of \eqref{decomlbra} for this case. 
 
 \begin{proposition}
 \label{decompp}
  Assume that \eqref{decompga} holds, with $\fP=\{1,2\}$, so that $\bar \Omega=\bar \Omega_1 \cup \bar \Omega_2$, with $\Omega_1\cap \Omega_2=\emptyset$. We have the lower bound
 \begin{equation}
\label{decomlbra2}
\Lbra  (A, \partial \Omega)\geq  \Lbra  (A_1, \partial \Omega_1)
+\Lbra  (A_2, \partial \Omega_2)
+ \upkappa_\upalpha
\frac{ \sharp (A_1)}{ \sharp (A)} (\Nel )^\upalpha   {\rm \dist}(\Omega_1, \partial \Omega), 
\end{equation}
where $\upkappa_\upalpha>0$ depends only on $\upalpha$ and where $\Nel=\sharp (A)$ denotes the number of elements in $A$.
 \end{proposition}
 
 The previous result is only of interest in the case  ${\rm \dist}(\Omega_1, \partial \Omega)\not =0$, that is when 
 $\bar \Omega_1 \subset \Omega$. 
   The proof involves     concavity properties, in particular the next elementary  result. 

   \begin{lemma} 
   \label{concavite}
   Let $0<\upalpha\leq 1 $, $a\geq 1$ and $b\geq 1$ be two given numbers. There exists some universal constant $\upkappa_\upalpha >0$ depending only on $\upalpha$ such that 
$$
   (a+b)^\upalpha \geq a^\upalpha+   \upkappa_\upalpha  \inf\{b^{\upalpha}, \,  b\,   a^{\upalpha-1} \}.
   $$
 \end{lemma}
   \begin{proof}[Proof of Lemma \ref{concavite}]  
   We rely on the  Taylor  expansion of the expression  $(1+s)^\upalpha$. It  yields
   \begin{equation}
\label{fasto}
  ( 1+s)^\upalpha \geq 1 + \upalpha s  +\frac 12 \upalpha (\upalpha-1)  s^2
   \geq 1 + \frac 12 \upalpha s  (1+ (\upalpha-1)s )\geq 1+  \frac 12  \upalpha s (1-s) {\rm \ for \ }  s \in [0, 1] .
   \end{equation}
We distinguish three cases. 
   
   \medskip
   \noindent
   {\it Case  1:} $\displaystyle{b\leq \frac a 2   }$.     We apply \eqref{fasto} with  $\displaystyle{s= \frac {b} {a}\leq \frac 12}$,  so that $\displaystyle{1-s \geq \frac 12}$, leading to the inequality
  \begin{equation}
  \label{bielo} 
   (a+b)^\upalpha   \geq a^\upalpha (1+\frac 14 \upalpha s)\geq a^\upalpha + \frac 1 4 \,   \upalpha b \, a^{\upalpha-1}.
    \end{equation}  
  
  \noindent 
  {\it Case  2:} $\displaystyle{ 8a\geq b\geq \frac a 2 }$. In this case, we obtain invoking \eqref{fasto} with $\displaystyle{s=\frac 12}$ 
 \begin{equation}
 \label{bielo2}
   (a+b)^\upalpha \geq (\frac 3 2 a)^\upalpha =(1+\frac 12)^\upalpha a^\upalpha \geq (1+\frac 1 8 \upalpha ) a^\upalpha   \geq a^\upalpha + \frac 1 8 \upalpha \left( \frac{ b}{8}\right)^\upalpha\geq  a^\upalpha + \upalpha \left (\frac 1 8\right)^{\upalpha+1}b^\upalpha.
 \end{equation}

  \noindent 
  {\it Case  3:} $\displaystyle{ 8a\leq b}.$  In this case,  we write
 \begin{equation}
 \label{bielo3}
 ( a+b)^\upalpha\geq b^\upalpha \geq  \frac{1}{8^\upalpha}b^\upalpha  +  ( 1-\frac {1} {8^\upalpha} ) b^\upalpha 
 \geq  a^\upalpha +( 1-\frac {1} {8^\upalpha} ) b^\upalpha.
\end{equation}
We set $\displaystyle{ \upkappa_\upalpha =\inf\{ \upalpha  \slash 4,  \upalpha \left(  1 \slash 8\right)^{\upalpha +1}, \left(1- {1}\slash{8^\upalpha}\right) \}.}$
 Combining  \eqref{bielo}, \eqref{bielo2} and \eqref{bielo3} in the three cases, we  complete the proof of the lemma.
 \end{proof}  

  We  use Lemma \ref{concavite}  in the case we have    the additional assumption
  \begin{equation}
  \label{number}
   	a +b \leq \Nbre,  
   \end{equation}
   where  $\Nbre\gg 1$ is some large number.    It follows from \eqref{number} that 
   $\displaystyle{b^\upalpha \geq b (\Nbre)^{\upalpha-1}}$ and
   $\displaystyle{a^{\upalpha-1} \geq  (\Nbre)^{\upalpha-1}}$ 
   so that in this case,  \eqref{bielo} leads to the inequality
   \begin{equation}
   \label{spartacus}
    (a+b)^\upalpha \geq a^\upalpha+   \upkappa_\upalpha  \,  b (\Nbre)^{\upalpha-1},
   \end{equation}
 and hence the right hand side of \eqref{spartacus} \emph{behaves linearily} with  respect to $b$.
\medskip   
\begin{proof}[Proof of Proposition \ref{decompp}] 
 We  assume that $\bar \Omega_1 \subset \Omega$, since otherwise the result  \eqref{decomlbra2}  is a  immediate consequence of \eqref{decomlbra}. In view of the assumptions,  we have $\Omega_2=\Omega \setminus \bar \Omega_1$. Let $\Gopt$ be an optimal graph  for  $\Lbra  (A, \partial \Omega)$,  assuming   for simplicity that all multiplicities  in $A$ are equal to one.  We proceed first to a spatial decomposition of this graph, introducing the subgraphs
$\displaystyle{G_\fp= G \rest \Omega_\fp.}$
Going back to Remark \ref{planine}, we have 
\begin{equation}
\label{renarde}
W_\upalpha (G)=W_\upalpha (G_1) + W_\upalpha (G_2) \geq  \Lbra  (A_1, \partial \Omega_1) + W_\upalpha (G_2). 
\end{equation}
In order to estimate $W_\upalpha (G_2)$, we rely on  the next Lemma:

\begin{lemma} 
\label{nardre}
We have  the improved  lower bound for $W_\upalpha (G_2) $
\begin{equation}
\label{poularde}
W_\upalpha (G_2)  \geq   \Lbra  (A_2, \partial \Omega_2)   + \upkappa_\upalpha
\frac{ \sharp (A_1)}{ \sharp (A)} (\Nel )^\upalpha   {\rm \dist}(\Omega_1, \partial \Omega), 
\end{equation}
 where $\upkappa_\upalpha >0$ is the constant  provided by Lemma \ref{concavite}. 
\end{lemma}
Combining Lemma \ref{nardre} with inequality \eqref{renarde}, we derive inequality \eqref{decomlbra2} which completes the proof of Proposition \ref{decompp}.
\end{proof}

\begin{proof}[Proof of Lemma \ref{nardre}]
 In view of Lemma \ref{xialemma},  we   may decompose  $\Gopt$    as in \eqref{grouic0}, that is
 $$ 
 \Gopt=\underset {a \in A} \curlyvee T_a=\left( \underset {a \in A_1} \curlyvee T_a\right) \curlyvee \left( \underset {a \in A_2} \curlyvee T_a\right), {\rm  \  where  \  } 
T_a  \in  \Thread(a,   \Omega).
 $$
 We may decompose the graph $G_2=G \rest \Omega_2 $ as 
 $\displaystyle {G_2=  G_{2,1}\uvee G_{2,2}}$ with,  for $\fq=1,2$, 
$$
G_{2,\fq}=\left(\underset {a \in A_\fq} \curlyvee T_a\right)\rest \Omega_2, {\rm  \  where  \ the \ threads \  } 
T_a  \in  \Thread(a,   \Omega)
{\rm \ satisfy \ } \eqref{enplus}.
$$
We notice that $G_{2,2} \in \mathcal G(A_2, \Omega_2)$, so that 
\begin{equation}
\label{voilavoila}
W_\upalpha ( G_{2,2}) \geq \Lbra (A_2, \partial \Omega_2), 
\end{equation}
whereas $G_{2,1} \in \mathcal G_0( \Omega_2)$.  
Given a segment $e$ of the graph $G_2$, we denote by $d_{2,1}(e)$ (resp. $ d_{2,2}(e)$) its multiplicity according to the graph $ G_{2,1}$ (resp $ G_{2,2}$), with the convention that $ d_{2,1}(e)=0$ (resp. $d_{2,2}(e)=0$) if the segment does not belong to $E( G_{2,1})$  (resp. $E(G_{2,2}))$. 
 It follows from the last statement in Lemma \ref{xialemma} that 
 \begin{equation}
 \label{xiasuite}
  d(E, G)= d(E, G_2)=d_{2,2}(e) +  d_{2,1} (e) \leq \Nel, 
 \end{equation}
  and the definition of $W_\upalpha$ leads to the identity
$$
W_\upalpha(G_2)=W_\upalpha ( G_{2,2} \curlyvee  G_{2,1})=
\underset {e \in E( G_2)} \sum  \left( d_{2,2}(e) +  d_{2,1} (e)\right)^\upalpha \mathcal H^1(e).
$$
We  split the remaining of the proof into three steps.

\smallskip
\noindent
{\bf  Step 1}. {\it  We have  the lower bound }
 \begin{equation}
\label{antipasti}
  W_\upalpha(G_2)=W_\upalpha (G_{2,2} \curlyvee G_{2,1})  \geq W_\upalpha (  G_{2,2} )+
  \upkappa_\upalpha \, \left(\Nel\right)^{\upalpha-1}  \,  \underset {e \in  E(G_2)} \sum
d_{2,1}(e) \mathcal H^1(e), 
\end{equation}
{\it Proof of \eqref{antipasti}}.
 We invoke  inequality \eqref{spartacus}  of Lemma \ref{concavite}  with $\Nbre=\Nel$, $a= d_{2,2}(e)$ and $b=  d_{2,1} (e)$.  Since \eqref{xiasuite} yields \eqref{number} in the case considered, we obtain 
\begin{equation}
\label{lise}
\begin{aligned}
W_\upalpha ( G_{2,2} \curlyvee \ G_{2,1}) & \geq \underset {e \in E(G_2)} \sum 
 \left( d_{2,2}(e)^\upalpha  + \upkappa_\upalpha  d_{2,1}(e)\left(\Nel\right)^{\upalpha-1} \right)\mathcal H^1(e) \\
&\geq \underset {e \in E(G_2)} \sum d_{2,2}(e)^\upalpha \mathcal H^1(e)+
\upkappa_\upalpha \, \left(\Nel\right)^{\upalpha-1}  \,  \underset {e \in  E(G_2)} \sum
d_{2,1}(e) \mathcal H^1(e).
\end{aligned}
\end{equation}
Since, by definition, we have
 $\displaystyle{
 W_\upalpha (G_{2,2})=\underset {e \in E( G_2)} \sum d_{2,2}(e)^\upalpha \mathcal H^1(e)
 }, $
 we obtain \eqref{antipasti}.

\medskip
\noindent
{\bf  Step 2}.  We have the lower bound
 \begin{equation}
  \label{rosso}
   \underset {e \in E(G_{2})} \sum d_{2,1}(e) \mathcal H^1(e)  \geq 
   \sharp (A_1) {\rm dist} (\Omega_1, \partial \Omega).
 \end{equation}
{\it Proof of \eqref{rosso}}. We take advantage of the linearity of the l.h.s with respect to  multiplicity.  Indeed, we notice that 
$$
\underset {  e \in E(G_{2})} \sum d_{2,1}(e) \mathcal H^1(e)  =
\underset {  a \in A_1} \sum \mathcal H^1(\mathcal C_a \cap \Omega_2), 
$$
where $\mathcal C_a$ denotes the polygonal curve related to the thread $T_a$. Since any thread   $T_a$, joins a point in $\Omega_1$ to the boundary $\partial \Omega$, we have
$$
H^1(\mathcal C_a \cap \Omega_2) \geq {\rm dist} (\Omega_1, \partial \Omega), 
$$
 so that the conclusion  \eqref{rosso}  follows combining the two previous relations. 

\medskip
\noindent
{\bf  Step 3}. {\it Proof of Lemma \ref{nardre} completed}.  Combining the lower-bound \eqref{antipasti}, \eqref{rosso}  with \eqref{voilavoila}, we derive the lower bound \eqref{poularde}, which completes the proof of Lemma \ref{nardre}.
 \end{proof}

 \subsubsection{Estimates for minimal  branched connections}
 \label{zertyu}
  An important observation made\footnote{Here we refer to  Proposition 3.1 in \cite{xia}. Although the statement there is slightly different from ours, the reader may easily adapt the proof.} in \cite{xia} to obtain  the following: 
 
 \begin{proposition} 
 \label{qlxia}
 Assume that $\upalpha \in (\upalpha_m,  1]$, where $\upalpha_m=1-\frac 1 m$. Then we have, for some constant $C(\Omega, \upalpha)$ depending only on $\Omega$ and $\upalpha$, 
\begin{equation}
\label{fff}
 \Lbra (A, \partial \Omega) \leq C(\Omega, \upalpha) \left (\sharp (A)\right)^\upalpha. 
 \end{equation}
 \end{proposition} 
 
 The proof is obvious for  $\upalpha=1$. Indeed in this case, one   obtain an upper bound for 
 $\mathfrak L_{\rm brbd}^{1} (A, \partial \Omega)$ estimating $W_1(G_0)$,  where $G_0$ is constructed as  in subsection \ref{ukraine}  connecting each point in $A$ to its nearest point on the boundary. We obtain
 $\displaystyle{ W_1(G_0)\leq {\rm  diam }(\Omega)\left( \sharp (A)\right), }$
 yielding the result in the case considered.
 
  In the  case $\upalpha_m\leq \upalpha<1$, estimate \eqref{fff} yields an improvement  on the growth in terms of $\sharp A$.   This is achieved in \cite{xia} replacing the elementary comparison graph $G_0$ by graphs   having  branching points obtained through a dyadic decomposition.   
 \begin{remark}{\rm The result of Proposition \eqref{qlxia} is  optimal: One may find simple distributions of points for which the asymptotic behavior is of order $\left (\sharp (A)\right)^\upalpha$, for instance putting all points at the same location, far from the boundary.
 }
 \end{remark}
 
 \subsection{The case of a uniform grid}
 \label{unigrid}
 
 We next focus on  behavior of $\Lbra$ in the special case  $\Omega$ is the $m$-dimensional unit cube that is 
 $\Omega=(0,1)^m$ and the points of $A$ are located on an uniform grid. We consider therefore for  an integer $\rk$ in $\N^*$  the distance  $h=\frac 1 \rk$ and  the set of points 
 $${\rm \bf A}^\rk_m\equiv \boxplus^\rk_m(h)=\left\{ a^k_\rI\equiv h\, I=h\, (i_1, i_2,\ldots,   i_m),{\rm \ for \ } I \in \left\{1, \ldots,\rk\right\}^m \right \}, 
 $$
so that  $\displaystyle {\sharp(\rbA^\rk_m )=\rk^m}$. Notice that  $\rbA^\rk_m \cap \partial \left((0,1)^m\right)\neq   \emptyset$ (see Remark
\ref{auchoix}). We set
\begin{equation*}
\Uplambda_m^\upalpha (\rk)= \Lbra (\rbA^\rk_m, \partial (0,1)^m) {\rm   \ and \ } 
\Uplambda^{m,\upalpha}_{\rm norm}(\rk)\equiv \rk^{-m\upalpha} \Lbda (\rk). 
\end{equation*}
 We  are interested in the asymptotic behavior  of the quantities  $\Uplambda_m^\upalpha(\rk)$  and $\Upxia(\rk)$ as $\rk \to + \infty$. It follows  from Proposition \ref{qlxia} that,  if $\upalpha >\upalpha_m$ then, we have the upper bound
\begin{equation}
\label{lawofthewest}
\Lbda (\rk) \leq C_\upalpha \rk^{m \upalpha} \ {\rm  \  i.e.\ }   \  \Upxia (\rk)\leq C_\upalpha, 
\end{equation} 
where the constant $C_\upalpha>0$ does not depend on $\rk$.   In the critical case $\upalpha=\upalpha_m$,  the upper bound \eqref{lawofthewest} no longer holds,  as our next result shows.

\begin{theorem}
\label{droppy}  There exists some constant $\rC_m>0$ such that we have the lower bound 
$$ \Lbdam (\rk) \geq  \rC_m  \rk^{m \upalpha_m}\log \rk=\rC_m  \rk^{m-1}\log \rk,   \ {\rm i. e. \ }  \Upxiam (\rk)\geq \rC_m \log \rk, {\rm \ for \ } k \in \N^*. $$
\end{theorem}
\begin{remark} {\rm The fact that the quantity $ \Upxiam (\rk)=\rk^{1-m} \Lbdam (\rk) $ does not remain bounded as $\rk \to + \infty$  is related  to and may  also presumably  be deduced from the fact that the Lebesgue measure \emph{is not irrigible} for the  critical value $\upalpha=\upalpha_m$, a result  proved in \cite{devisolo} (see also \cite{bermocas}).
}\end{remark}

  The proof of Theorem \ref{droppy} will  rely on several preliminary results we present first, starting with elementary scaling laws. Let $\rq \in \N^*$ be given, and consider for $\rk \in \N$ the set 
  $$
   \frac{1}{\rq}\rbA^{ k}_m =\rbA^{\rq k}_m\cap [0, \frac{1}{\rq}]^m=\boxplus^\rk_m(\frac h \rq)=\left\{ a^k_\rI\equiv \frac{1}{\rq \rk}\,  I, I  \in \left\{1, \ldots,\rk\right\}^m \right \}. 
  $$
  The set  $\frac{1}{\rq}\rbA^{k}_m$ hence contains $\rk^m$ distinct  elements.  The scaling law writes as 
  \begin{equation}
  \label{scalitude}
   \Lbra \left(  \frac{1}{\rq}\rbA^{ k}_m, \partial \left(  [0, \frac{1}{\rq}]^m\right)\right)=
  \rq^{-1}  \Lbra \left(  \rbA^{ k}_m, \partial \left(  [0, 1]^m\right)\right)=
    \rq^{-1} \Lbda (\rk). 
 \end{equation}
The main ingredient in the proof of Theorem \ref{droppy} is  a consequence of  Proposition \ref{decompp}:

  \begin{lemma}
  \label{crazy}  Let $\rq \in \N^*$ be given. There exists some constant  $\rC_\rq^\upalpha >0$ such that 
  $$
  \Upxia (\rq  \, \rk) \geq  \rq^{m(\upalpha_m-\upalpha)} \, \Upxia (\rk)+ \rC_\rq^\upalpha,  \ \ {\rm  \ for  \  any \ } \rk \in \N^* .
  $$
  \end{lemma}

\begin{proof}    We consider the set $\rbA^{\rq k}_m$ and   decompose   the domain $\overline{\Omega} =[0,1]^m$ as an union of cubes $\overline{{\rm Q}_\rbJ}$,  with $\rbJ\equiv (j_1, j_2, \dots, j_m)\in \mathfrak J\equiv\{0,\ldots, \rq-1\}^m$,   and 
$$ 
{\rm Q}_\rbJ= \frac 1 \rq \rbJ+ (0, \frac{1}{\rq})^m =\frac 1 \rq (j_1, j_2, \dots, j_m) + (0, \frac{1}{\rq})^m, 
$$
 so that $ {\rm Q}_\rbJ \cap {\rm Q}_{\rbJ'} \not = \emptyset$ if $\rbJ\not = \rbJ'$ and 
 $\displaystyle{[0,1]^m=\underset{ \rbJ  \in \mathfrak J}\cup\bar {{\rm Q}}_\rbJ}.$ We  set 
  $$
  A_{\rbJ} \equiv A^{\rq \rk}_m \cap {\rm Q}_\rbJ, {\rm \  so \   that  \ } 
  A_{\rbJ}= \frac 1 \rq \rbJ+ \frac {1} {\rq} \rbA_m^k.
  $$ 
 It follows from the scaling law \eqref{scalitude} and translation invariance  that
 \begin{equation}
 \label{scalitude2}
  \Lbra (A_\rbJ, \partial{\rm Q}_{\rbJ})=\rq^{-1} \Lbda (\rk), {\rm \ for \ } 
  \rbJ \in \mathfrak J. 
\end{equation}
   We next single out a cube ${\rm Q}_{\rbJ_0}$ which is far from the boundary. For that purpose, we consider the integer $\displaystyle{
   \rq_0 \equiv \left [\frac {\rq} {2}\right]}$, the multi-index $\rbJ_0=(\rq_0, \rq _0, \ldots, \rq_0)$  and the sets
$$
\Omega_1 ={\rm Q}_{\rbJ_0}   {\rm  \  and \ }
\Omega_2=  \underset{ \rbJ \in \mathfrak J \setminus \{\rbJ_0\}}  \cup   {\rm Q}_\rbJ,
{\rm \ so \ that \ }
  {\rm  dist } (\Omega_1, \partial \Omega) \geq \frac 1 4  {\rm \ for \ } \rq \geq  3. 
  $$
Applying inequality \eqref{decomlbra2} of Proposition \ref{fff}, we are led to 
  \begin{equation}
  \label{colonne}
  \begin{aligned}
  \Uplambda_m^\upalpha (\rq\rk)= \Lbra (\rbA_m^{\rq\rk}, \partial (0,1)^m)  \geq &\Lbra (A_{\rbJ_0}, \partial {\rm Q}_{J_0})+ 
  \Lbra (\Omega_2 \cap  \rbA^{\rq \rk}_m, \partial  (0,1)^m) \\
  &+\frac 1 4\upkappa_\upalpha \rk^{m \upalpha}  
 \rq^{m(\upalpha-1)}.
\end{aligned}
  \end{equation}
We deduce from inequality \eqref{decomlbra} and \eqref{scalitude} that 
 \begin{equation}
 \label{almostf}
 \Lbra (\Omega_2 \cap  A^{\rq \rk}_m, \partial  (0,1)^m) \geq \underset{ \rbJ\in \mathfrak J \setminus \{\rbJ_0\}} \sum 
   \Lbra (A_\rbJ, \partial {\rm Q}_\rbJ)=\left[\rq^{m}-1\right] \rq^{-1} \Lbda (\rk).
\end{equation}
 Combining \eqref{colonne}, \eqref{almostf} with \eqref{scalitude2} for $\rbJ=\rbJ_0$, we are led to the lower bound
 \begin{equation*}
  \Uplambda_m^\upalpha (\rq\rk) \geq \rq^{m-1}  \Lbda (\rk)  + \frac 1 4\upkappa_\upalpha \rk^{m \upalpha}  
 \rq^{m(\upalpha-1)}.
 \end{equation*}
Multiplying both sides  by $(\rq\rk)^{-m\upalpha}$, we obtain the desired  result with $\rC^\upalpha_\rq=\frac 14 \upkappa_\upalpha \rq^{-1}.$ 
\end{proof}

\begin{lemma}
\label{inferieuritude}
 We have,  for  any integers $1 \leq \rk' \leq \rk$, 
 $$\Lbda (\rk')  \leq  \frac {\rk'}{\rk}  \Lbda (\rk) \  {\rm \ and \ hence \ \  } 
 \Upxia (\rk') \leq  \left(  \frac {\rk'}{\rk} \right)^{m\upalpha+1} \Upxia (\rk).
$$
\end{lemma}
\begin{proof}   consider the cube $ \mathcal Q_\rk'=(0, \frac {\rk'}{ \rk} )^m \subset (0,1)^m $ and the  set $A'=\rbA^\rk_m \cap \mathcal Q_\rk'$. It follows from 
inequality \eqref{decomlbra} that 
$$ \Lbra (A', \partial \mathcal Q_\rk') \leq \Lbra (A^\rk_m, \partial (0,1)^m)=\Lbda (\rk), 
$$
 whereas the scaling property yields
 $\displaystyle{\Lbra (A', \partial \mathcal Q_\rk')=\frac {\rk'}{\rk} \Lbda (\rk'). }$
  The conclusion follows combining the previous inequalities. 
\end{proof}

\begin{proof} [Proof of Theorem \ref{droppy} completed]  In the special case $\upalpha=\upalpha_m$, the exponent of $\rq$ in the r.h.s of the  inequality of  Lemma \ref{crazy} vanishes, so that we obtain
$$
 \Upxiam (\rq  \, \rk) \geq \Upxiam (\rk)+ \rC_\rq^{\upalpha_m}, {\rm \ for \ any \ integer \ }
\rq \geq 3.
$$
 Iterating this lower bound, we obtain, for any any integer $\ell >0$,  to the lower bound
\begin{equation}
\label{loweritude}
\Upxiam (\rq^\ell)\geq \rC^{\upalpha_m}_\rq \ell.
 \end{equation}
On the other hand, it follows from Lemma \ref{inferieuritude} that for any $\rq^\ell \leq  \rk \leq \rq^{\ell+1}$ we have 
\begin{equation}
\label{facilitude}
\Upxiam (\rk) \geq   \rq^{-m} \  \Upxiam (\rq^\ell). 
\end{equation}
Combining \eqref{facilitude} with \eqref{loweritude} we deduce that, for any $\rk \in \N^*$, we obtain the inequality
$$
\Upxiam (\rk) \geq   \rq^{-m} \rC^{\upalpha_m}_\mq \left[ \frac {\log \rk}{\log \rq}\right], 
$$
which leads immediately to the conclusion,   fixing the value of $\rq$ for instance as $\rq=5$. 
\end{proof}

\begin{remark} 
{\rm
 For $\upalpha <\upalpha_m$ the same type of argument show that 
$\displaystyle{\Upxia(\rk) \to + \infty {\rm \ as \ } \rk \to + \infty.
}$
}
\end{remark}
\subsection{Charges  with opposite  signs }
We consider here, as in the introduction,  the possibility of having also points with negative charges. We consider therefore  a collection of  points $\rbP=\{P_i\}_{i \in I}$ in $\Omega$  with positive charge $+1$,   a collection of points  $\rbQ=\{Q_j\}_{j \in J}$   in $\Omega$,  with negative charge $-1$, and set  $A=\rbP \cup \rbQ$. We define the set 
$\mathcal 
G (\rbP, \rbQ, \Omega)$ of graph  satisfying $A\subset V(G)\subset  \bar \Omega, $
and the \emph{Kirchhoff law}  \eqref{prescharge}, setting  ${\rm M}(P_i)=+1$ and  ${\rm M}(Q_j)=-1$.  For $0\leq \upalpha\leq 1$, we set
\begin{equation}
\label{frument}
\Lbra( \rbP, \rbQ, \partial \Omega)=\Lbra(A,  \partial \Omega)=\inf \left\{ {\rm  W}_\upalpha(G), G \in \mathcal G\left (\rbP, \rbQ, \partial \Omega\right) \right \}.
\end{equation}
Next assume that we are given a   collections of disjoint  subdomains $(\Omega_\fp)_{\fp \in \fP}$ of $\Omega$ such that 
\begin{equation}
\label{nopasaran}
\rbQ \cap \Omega_\fp =\emptyset 
{\rm \  and  \  set  \  } \  A_\fp= \rbP \cap  \Omega_\fp, 
\end{equation}  
so that the subdomains $\Omega_\fp$ contain only possibly \emph{positive} charges.

\begin{lemma}
\label{toufou}
If \eqref{nopasaran} is satisfied, then we have the  inequality
$$
\Lbra\left( \rbP, \rbQ,  \partial  \Omega\right)  \geq \underset {\fp \in \fP}\sum \Lbra ( A_\fp, \partial \Omega_\fp).
$$
\end{lemma}

\begin{proof}  Let $G$ be a graph in $\mathcal G \left(A,  \partial \Omega\right)$ and set 
$\displaystyle{G_\fp=G \cap\Omega_\fp}$, for $\fp\in \fP$.  Since there are no negative charges in $\Omega_\fp$,  it turns out that  $G_\fp \in \mathcal G(\mathfrak A_\fp, \partial \Omega_\fp)$, so that 
$\displaystyle{
{\rW}_\upalpha(G_\fp) \geq \Lbra  (\mathfrak A_\fp, \partial \Omega_\fp).
}$
On the other hand, we have $\displaystyle{ \rW_\upalpha(G) \geq  \underset {\fp \in \fP} \sum{\rW}_\upalpha(G_\fp)}$ so  that the conclusion follows.
\end{proof}



\begin{thebibliography}{99}
 
 \bibitem{ABL} F.Almgreen, F.Browder and E.Lieb, {\it Co-area, liquid crystals, and minimal surfaces}, in   Partial differential equations (Tianjin, 1986), Springer (1988), 1--22. 
 
 \bibitem{ABO} G.Alberti, S.Baldo and G.Orlandi,  {\it Variational convergence for functionals of Ginzburg-Landau type}, Indiana Univ. Math. J. {\bf  54} (2005), 1411--1472.
 
 
\bibitem{AK} D.Auckly and L. Kapitanski, {\it  The Pontrjagin-Hopf invariants for Sobolev maps},  Commun. Contemp. Math {\bf 12} (2010), 121--181. 

\bibitem{bermocas} M.Bernot, V.Caselles and J.M.Morel,  {\it  Optimal transportation networks. Models and theory. } Springer  (2009). 200 p.

\bibitem{Be1} F. Bethuel, {\it  A caracterization of maps in $H^1(\B^3,\S^2)$ which can be approximated by smooth maps},  Ann. Inst. H. Poincar\'e  Anal. Non Lin\'eaire {\bf 7} (1990), 269--286.
\bibitem{Be2} F.  Bethuel, {\it The approximation problem for Sobolev maps between two manifolds},  Acta Math. {\bf 167} (1991), 153--206. 

\bibitem{betjac}  F.Bethuel, {\it The Jacobian, the square root and the set $\infty$},  J. Math. Anal. Appl. {\bf  352} (2009), 548--572.

\bibitem {BBC} F. Bethuel, H. Brezis and J.M. Coron,  {\it Relaxed energies for harmonic maps},  {\it in}  Variational methods (Paris, 1988),  Springer (1990),  37--52. 

\bibitem{BeChi}  F.Bethuel and D. Chiron, {\it Some questions related to the lifting problem in Sobolev spaces}, {\it  in }Perspectives in nonlinear partial differential equations,  Contemp. Math., {\bf 446} (2007),  125--152. 



\bibitem{BCDH} F. Bethuel, J.M. Coron, F.Demengel and F. H\'elein, {\it  A cohomological criterion for density of smooth maps in Sobolev spaces between two manifolds}, {\it in}  Nematics (Orsay, 1990),  Kluwer Acad. (1991) 15--23.

\bibitem{BeZ} F. Bethuel and X. Zheng,  {\it Density of smooth functions between two manifolds in Sobolev spaces},  J. Funct. Anal. {\bf 80 } (1988),  60--75. 


\bibitem{BCL} H.Brezis, J.M. Coron and E.H. Lieb, {\it  Harmonic maps with defects},  Comm. Math. Phys. {\bf 107} (1986), 649--705.

\bibitem{devisolo} G.Devillanova and S.Solimini, { \it  On the dimension of an irrigable measure},  Rend. Semin. Mat. Univ. Padova {\bf 117} (2007), 1--49.



\bibitem{giamosou1} M.Giaquinta, G.Modica and J.Soucek, {\it The Dirichlet energy of mappings with values into the sphere}, Manuscripta Math. {\bf 65 }(1989), 489--507.

 \bibitem{giamosou2} M.Giaquinta, G.Modica and J.Soucek, {\it  Cartesian currents in the calculus of variations. I. Cartesian currents},  Springer  (1998).

\bibitem{H} P.Hajlasz, {\it  Approximation of Sobolev mappings},  Nonlinear Anal. {\bf 22 }(1994), 1579--1591.

\bibitem{HL} F.H. Lin and R. Hardt, {\it Mappings minimizing the $L^p$ norm of the gradient},  Comm. Pure Appl. Math. {\bf 40} (1987), 555--588.

\bibitem {HL1} F.Hang and F.H. Lin, {\it Topology of Sobolev mappings},  Math. Res. Lett. {\bf 8} (2001), 321--330.

\bibitem{HL2}  F.Hang and F. H.Lin, {\it  Topology of Sobolev mappings II},  Acta Math. {\bf 191} (2003), 55--107. 

\bibitem{HL3}  F.Hang and F. H.Lin, {\it  Topology of Sobolev mappings III},  Comm. Pure Appl. Math. {\bf 56} (2003), 1383--1415.

\bibitem {HR1} R.Hardt and T.  Rivi\`ere, {\it  Connecting topological Hopf singularities},  Ann. Sc. Norm. Super. Pisa Cl. Sci. {\bf 2} (2003), 287--344.

\bibitem{HR2}  R.Hardt and T.  Rivi\`ere, {\it Ensembles singuliers topologiques dans les espaces fonctionnels entres vari\'etes},  S\'eminaire EDP, Ecole Polytechnique  (2001-2002), 14p. 

\bibitem{HR3}  R.Hardt and T.  Rivi\`ere, {\it Connecting rational homotopy type singularities},  Acta Math.
{\bf  200} (2008),  15--83. 

\bibitem{HR4}  R.Hardt and T.  Rivi\`ere, {\it  Sequential weak approximation of maps of finite Hessian energy},  preprint
 (2014). 


\bibitem{isobe}T. Isobe, {\it  Characterization of the strong closure of $C^\infty(\B^4;S^2)$ in $W^{1,p}(\B^4, \S^2)$ for $\frac{16}{5}‚â§p<4$}. J. Math. Anal. Appl.{\bf  190}  (1995), 361--372.

\bibitem{kosinski}  A. Kosinski, {\it  Differential manifolds},  Pure and Applied Mathematics, Academic Press, (1993). 


\bibitem{mona}M, Monastyrsky, {\it  Topology of gauge fields and condensed matter},  Plenum Press, New York, (1993) 
 
\bibitem{PR}M.Pakzad and  Rivi\`ere, {\it  Weak density of smooth maps for the Dirichlet energy between manifolds},  Geom. Funct. Anal. {\bf 13} (2003),  223--257.

\bibitem {ponte} L. Pontrjagin, A classification of mappings of the three-dimensional complex into the two-dimensional sphere, Rec. Math. [Mat. Sbornik] N.S. 9(51) (1941), 331--363.

\bibitem{Ri}T.Rivi\`ere, {\it  Minimizing fibrations and p-harmonic maps in homotopy classes from $\S^3$ into $\S^2$}. Comm. Anal. Geom.  {\bf 6} (1998), 427--483.

\bibitem{SU} R.Schoen  and K.Uhlenbeck, {\it Boundary regularity and the Dirichlet problem for harmonic maps}, J. Differential Geom.    {\bf 18} (1983), 253--268.

\bibitem{xia} Q.Xia, {\it Optimal paths related to transport problems}, Commun. Contemp. Math.{\bf  5} (2003), 251--279.

\bibitem{xia2}  Q.Xia  and A.Vershynina, {\it  On the transport dimension of measures}, SIAM J. Math. Anal. {\bf  41} (2009/10) 2407--2430.





















\end{thebibliography}
\end{document}